\documentclass[10pt]{article}

\usepackage[utf8]{inputenc}
\usepackage{amsmath}
\usepackage{amssymb}
\usepackage{amsthm}
\usepackage{graphicx}
\usepackage{float}
\usepackage{color}
\usepackage{bm}
\usepackage{geometry}

\usepackage{tikz}
\usetikzlibrary{shapes.geometric, arrows.meta, positioning}

\usepackage[%
  backend=bibtex,bibencoding=ascii,
  style=numeric-comp,
  giveninits=true, uniquename=init, 
  natbib=true,
  url=true,
  doi=true,
  isbn=false,
  backref=false,
  maxnames=99,
  ]{biblatex}
\usepackage{booktabs}
\usepackage{subcaption}
\usepackage{siunitx}
\usepackage{commath} 

\usepackage{hyperref}

\everymath{\displaystyle}

\let\epsilon\varepsilon

\newtheorem{theorem}{Theorem}
\newtheorem{definition}[theorem]{Definition}

\newtheorem{lemma}[theorem]{Lemma}

\newtheorem{remark}[theorem]{Remark}

\title{Structure-preserving approximations of the Serre-Green-Naghdi\\equations in standard and hyperbolic form}
\author{Hendrik Ranocha$^1$  and Mario Ricchiuto$^2$}

\date{{\small $^1$ Institute of Mathematics, Johannes Gutenberg University Mainz, \\Staudingerweg 9, 55130 Mainz, Germany\\ORCID: \url{https://orcid.org/0000-0002-3456-2277}}\\[5pt]
{\small$^2$ INRIA, U. Bordeaux, CNRS, Bordeaux INP, IMB, UMR 5251, \\200 Av. de la Vieille Tour, 33400 Talence, France\\ORCID: \url{https://orcid.org/0000-0002-1679-7339}}}

\begin{document}
\maketitle
\begin{abstract}
    We develop structure-preserving numerical methods for the
    Serre-Green-Naghdi equations, a model for weakly dispersive
    free-surface waves. We consider both the classical form, requiring the inversion of a non-linear elliptic operator,
    and a hyperbolic approximation of the equations, allowing fully explicit time stepping. Systems for both flat and variable topography are studied.
    Our novel numerical methods conserve both the
    total water mass and the total energy. In addition,
    the methods for the original Serre-Green-Naghdi equations
    conserve the total momentum for flat bathymetry.
    For variable topography, all the methods proposed are well-balanced for the lake-at-rest state.
    We provide  a theoretical setting allowing us to construct schemes
    of any kind (finite difference, finite element, discontinuous Galerkin, spectral, etc.)
    as long as summation-by-parts operators are available in the chosen setting.
    Energy-stable variants are proposed by adding a consistent high-order artificial viscosity term.
    The proposed methods are validated through a large set of benchmarks
    to verify all the theoretical properties.
    Whenever possible, comparisons with exact, reference numerical, or  experimental data are carried out.
    The impressive advantage of structure preservation, and in particular energy preservation, to resolve accurately dispersive wave propagation
    on very coarse meshes is demonstrated by several of the tests.
\end{abstract}

\section{Introduction}

This paper is devoted to the structure-preserving numerical approximation of the fully nonlinear, weakly dispersive Serre-Green-Naghdi (SGN)
equations for free-surface hydrodynamics.
Dispersive free-surface waves occur in a wide variety of  phenomena going from the propagation of tsunamis \cite{nhess-13-1507-2013,LEGAL2017257,baba21},
to estuarine dynamics and  wave propagation in natural as well man made environments \cite{treske1994,https://doi.org/10.1002/2014JC010267,chass_etal19,FILIPPINI201955,jouy_etal24}.
Many of these applications involve multi-scale wave physics, with often large domains and long-time/distance propagation.
Depth averaged Boussinesq-type equations are used in many existing operational codes for hazard assessment  (see e.g. \cite{kirby16,kazolea2024} and references therein).
These models can be written as a perturbation of the hyperbolic shallow water equations with a dispersive term, which accounts for
some of the  vertical kinematic lost in the depth averaging \cite{Lannes_2020}.
Among these models, the SGN system   \cite{Serre53,GN76,wk95} accounts for the full nonlinearity
of the wave propagation and transformation, and   is endowed with
a rigorous estimate  of the energy  associated  to the wave dynamics (see e.g. \cite{favrie2017rapid,GS24,Lannes_2020,Kazolea2023})
which can be used as a  rigorous criterion to estimate the dissipation during the propagation process.
In the literature,  two main writings  of the SGN equations  have been used for the purpose of numerical approximation: a classical one involving the inversion of
a nonlinear elliptic operator to evaluate the time variation of the velocity \cite{Lannes_2020,Kazolea2023,GS24}, and a system
based on a hyperbolic relaxation of the dispersion operator  \cite{favrie2017rapid,escalante_2020}. Both forms are considered here.\\

Realistic operational applications  require being able to run multi-scale simulations   on reasonably coarse meshes to obtain fast predictions.
A lot of focus is being put on deploying operational codes on modern high-performance parallel architectures \cite{https://doi.org/10.1029/2019MS001957,TAVAKKOL2020106966,gmd-15-5441-2022}.
In this work, we consider the other end of the process, and propose improved numerical methods.
In the literature we find many different numerical approaches to approximate Boussinesq-type,  dispersive shallow water equations.
Many such techniques involve some high-order approximation in space and time, with possibly ad-hoc treatments for the dispersive terms
(hyperbolization, or  some lower-order approximation) to combine efficiency and good numerical dispersion
\cite{wk95,eskilsson2006unstructured,kazolea2012unstructured,rf14,PANDA2014572,lannes2015new,MAZAHERI2016593,filippini2016flexible,duran2017discontinuous,li2019cdg,busto21,cauquis22,TORLO2023997,Kazolea2023}. Some authors  also propose methods with some strong energy stability/dissipation property \cite{svard2023novel,cpr2},
by imitation of what is usually done for hyperbolic conservation laws with entropy stability. Few recent works focus
on  approximations guaranteeing  exact energy conservation   \cite{ranocha2021broad,mitsotakis2021conservative,lampert2024structure} at the discrete level.\\

The theoretical and numerical results discussed in    \cite{el05,el06,10.1093/imamat/hxad030,https://doi.org/10.1111/sapm.12694}
show that there is a very delicate interaction between dissipation, dispersion, and non-linearity.
In particular, while in the purely dispersive setting many initial conditions lead to  appearance of solitary wave fission,
in presence of finite  dissipation  one obtains travelling waves with finite  stationary wavelength and much lower amplitudes.
In this respect, while recent work has shown that discontinuous solutions for Boussinesq models can be constructed  \cite{gnst20,El05-chaos,hoefer14},
such constructions do not rely, as in the case of hyperbolic balance laws, on the notion of a  dissipative solution.
Admissibility conditions for these problems are formulated based on geometrical considerations in phase space, and relate to the celerities of the solution fronts.
There is no notion on the sign of the energy evolution in such conditions.  It is  thus unclear whether one should use numerical dissipation when solving non-dissipative dispersive equations
as a means of stabilization.
This is a major difference between dispersive models and hyperbolic ones.
In addition, the recent work by Jouy et al.\ \cite{jouy_etal24}  has shown that numerical dissipation plays in practice the exact same role
of a physical dissipative regularization. In particular, when using high-order schemes embedding dissipation a gross underestimation
of the wave amplitudes may be obtained on coarse meshes. This is not the case for non-dissipative methods. The results of  \cite{jouy_etal24}  show
that this issue occurs not only for non-dissipative models,  but   also   in presence of physical dissipation terms in the model, e.g., due to friction terms.
Our objective is thus to investigate the construction and validation of structure-preserving methods for the SGN equations,
namely methods which conserve within machine accuracy as many physical properties as possible, including energy,
and which is well-balanced with respect to the well-known lake-at-rest state.

To this end  we use  the framework of summation-by-parts (SBP) operators and split forms \cite{fisher2013discretely}
as a   systematic approach to build exactly energy-conservative   semidiscretizations.
The idea is to follow step-by-step the continuous derivation of the energy balance, and combine the use of SBP differentiation  operators, which allow to  mimic integration by parts,
and use appropriate split forms of the differential equations to mimic the product and chain rule.
To obtain a fully conservative method,  the resulting ordinary differential equations in time can be integrated using relaxation Runge-Kutta (RRK) schemes,
which are   a small modification of classical Runge-Kutta methods allowing to preserve appropriate invariants \cite{ketcheson2019relaxation,doi:10.1137/19M1263480}.\\

The paper is organized as follows.
In the following Section~\ref{sec:basic-discretizations},
we briefly review the techniques that we use for spatial
and temporal discretizations. In particular, we describe
how split forms of the equations can be used to derive
energy-conserving discretizations using SBP operators.
Next, we review such energy-conserving split forms of the
classical shallow water equations in Section~\ref{sec:split-forms-SWE}.
In Section~\ref{sec:SGN_flat}, we review the classical
SGN equations in flat bathymetry as well as
their hyperbolic approximation. To prepare the remainder of this
article, we also explain how to pass from the hyperbolic approximation
to the original system.
Next, we derive a structure-preserving split form and discretization
of the hyperbolic approximation with flat bathymetry in
Section~\ref{sec:SGN_hyperbolic_flat}.
We begin with the hyperbolic approximation since there are less
complicated higher-derivative terms in this case, making it
easier to derive an energy-conserving split form.
Using the previously
established translation rules, we use these results to derive
corresponding structure-preserving methods for the original
SGN equations in Section~\ref{sec:SGN_original_flat}.
We extend the investigations to the case of variable bathymetry
first for the hyperbolic approximation in
Section~\ref{sec:SGN_hyperbolic_variable}
and to the classical SGN equations in
Sections~\ref{sec:SGN_original_mild} and \ref{sec:SGN_original_full}.
We describe how to add stabilizing artificial viscosity/dissipation
in Section~\ref{sec:AV}.
Afterwards, we validate our implementation and present
numerical experiments in Section~\ref{sec:numerical_experiments}.
Finally, we summarize our results and give an outlook on
future work in Section~\ref{sec:summary}.

\section{Brief review of split forms and discretization techniques}
\label{sec:basic-discretizations}

In this section, we briefly review the techniques that we use for spatial
and temporal discretizations. In general, we will use the method of lines,
starting with a semidiscretization in space followed by a time integration.

\subsection{Split forms}

Consider Burgers' equation
\begin{equation}
    u_t + \left(\frac{1}{2} u^2\right)_x = 0
\end{equation}
with periodic boundary conditions. As is well known, a smooth solution
satisfies the energy equality
\begin{equation}
    \left( \frac{1}{2} u^2 \right)_t + \left( \frac{1}{3} u^3 \right)_x = 0.
\end{equation}
Thus, the total energy (squared $L^2$ norm) is conserved.
To prove this, one would typically multiply the PDE by $u$ and use
the chain rule. However, discrete derivative operators can in general
not satisfy a discrete version of the chain rule, in particular for
higher-order operators \cite{ranocha2019mimetic}. Thus, we will use
the split form \cite[eq. (6.40)]{richtmyer1967difference}
\begin{equation}
    u_t + \frac{1}{3} (u^2)_x + \frac{1}{3} u u_x = 0.
\end{equation}
Indeed, multiplying Burgers' equation by the solution $u$
and integrating over the domain yields
\begin{equation}
    \frac{1}{2} \frac{\dif}{\dif t} \int u^2 \dif x
    =
    \int u u_t \dif x
    =
    - \frac{1}{3} \int u (u^2)_x \dif x
    - \frac{1}{3} \int u^2 u_x \dif x
    =
    0.
\end{equation}
Hence, energy conservation can be shown using only integration by parts.
To obtain a semidiscretization satisfying the energy conservation law
at the discrete level, we just need to use the split form derived at the
continuous level and apply discrete derivative operators satisfying a
discrete equivalent of integration by parts.

Although the split form looks like it results in a non-conservative
discretization, one can show that the total mass is still conserved since
\begin{equation}
    \frac{\dif}{\dif t} \int u \dif x
    =
    - \frac{1}{3} \int (u^2)_x \dif x
    - \frac{1}{3} \int u_x u \dif x
    =
    0,
\end{equation}
using again only integration by parts. Moreover, it can be shown that
the discretization of the split form is even locally conservative when
discretized with SBP operators \cite{fisher2013discretely}.
This also holds for more general flux differencing discretizations
that can be entropy-stable but do not need to be related to a split
form \cite{fisher2013high}.

\subsection{Summation-by-parts operators}

SBP operators are discrete derivative operators
designed to mimic integration by parts at the discrete level. Originally,
SBP operators were introduced for finite difference methods
\cite{kreiss1974finite,strand1994summation},
but they can also be used for
finite volume \cite{nordstrom2001finite,nordstrom2003finite},
continuous finite element \cite{hicken2016multidimensional,hicken2020entropy,abgrall2020analysisI},
discontinuous Galerkin \cite{gassner2013skew,carpenter2014entropy,chan2018discretely},
and flux reconstruction methods \cite{huynh2007flux,ranocha2016summation}.
A good overview can be obtained from the review articles
\cite{svard2014review,fernandez2014review} and the application of
various methods from the SBP framework to dispersive wave equations in
\cite{ranocha2021broad}.

We consider periodic boundary conditions in this article. Thus, we will only
briefly recap the corresponding properties of periodic SBP operators.
Further discussions and examples can be found in
\cite{svard2014review,fernandez2014review,ranocha2021broad}.
We use a nodal collocation approach and discretize the spatial domain
using point values of the unknowns at given grid points $x_i$. The discretized
version of a function $u$ is denoted by $\pmb{u}$ with $\pmb{u}_i = u(x_i)$.
In particular, $\pmb{1} = (1, \dots, 1)^T$.
Nonlinear operations are performed pointwise, e.g.,
$(\pmb{u}^2)_i = \pmb{u}_i^2$.

\begin{definition}
    A \emph{periodic SBP operator} on the domain
    $[x_\mathrm{min}, x_\mathrm{max}]$ consists of a grid $\pmb{x}$,
    a symmetric and positive definite mass/norm matrix $M$ satisfying
    $\pmb{1}^T M \pmb{1} = x_\mathrm{max} - x_\mathrm{min}$, and
    a consistent derivative operator $D$ such that
    \begin{equation}
    \label{eq:SBP_periodic}
        M D + D^T M = 0.
    \end{equation}
    It is called {diagonal-norm operator} if $M$ is diagonal.
\end{definition}

A consistent derivative operator $D$ differentiaties constants exactly,
i.e., $D \pmb{1} = \pmb{0}$.
A classical example of a periodic SBP operator is given by central
second-order finite differences, where mass matrix is the identity matrix
scaled by the grid spacing $\Delta x$ and the stencil coefficients of the
derivative operator are $\Delta x^{-1} (-1/2, 0, 1/2)$.

The definition above introduces first-derivative SBP operators. Since we
also need higher derivatives for the dispersive terms, we will use
second-derivative SBP operators
\cite{mattsson2004summation,mattsson2012summation} as well.
Since the second-derivative terms of the Serre-Green-Naghdi equations
have variable coefficients, we will use upwind operators to construct
them in a general way \cite{mattsson2017diagonal}. See also
\cite{ranocha2021broad,ortleb2023stability} for discussions of upwind
SBP operators and discontinuous Galerkin methods for second-derivative
terms.

\begin{definition}
    A \emph{periodic upwind SBP operator} on the domain
    $[x_\mathrm{min}, x_\mathrm{max}]$ consists of a grid $\pmb{x}$,
    a symmetric and positive definite mass/norm matrix $M$ satisfying
    $\pmb{1}^T M \pmb{1} = x_\mathrm{max} - x_\mathrm{min}$, and
    two consistent derivative operators $D_\pm$ such that
    \begin{equation}
    \label{eq:SBP_upwind_periodic}
        M D_+ + D_-^T M = 0,
        \qquad
        M (D_+ - D_-) \text{ is negative semidefinite}.
    \end{equation}
    It is called {diagonal-norm operator} if $M$ is diagonal.
\end{definition}

We will frequently use that the average $D = (D_+ + D_-) / 2$ of
upwind SBP operators is a central SBP operator
\cite{mattsson2017diagonal}.
A classical example of periodic upwind SBP operators is given by the
one-sided first-order finite differences, where the mass matrix is again
the identity matrix scaled by the grid spacing $\Delta x$ and the stencil
coefficients of the upwind derivative operator are
$\Delta x^{-1} (-1, 1, 0)$ and $\Delta x^{-1} (0, -1, 1)$. Thus, the
corresponding central SBP operator is given by the coefficients
$\Delta x^{-1} (-1/2, 0, 1/2)$.

To get a second-derivative operator, one can apply a first-derivative
operator twice. For the classical second-order central SBP operator,
this results in the wide-stencil operator with coefficients
$\Delta x^{-2} (1/4, 0, -1/2, 0, 1/4)$. A better approximation is usually
given by the combination of upwind operators, e.g., $D_+ D_-$ with stencil
coefficients $\Delta x^{-2} (1, -2, 1)$.

In this article, we will only use diagonal-normal SBP operators. We will
use the quadrature rule induced by the mass matrix to compute discrete
versions of integrals or the discrete $L^2$ error.

\subsection{Time integration methods using relaxation}

We will use explicit Runge-Kutta methods for time integration.
Since such explicit time integration methods cannot guarantee
conservation (or dissipation) of nonlinear invariants such as
the energy \cite{ranocha2021strong,offner2020analysis,ranocha2020energy,jungel2017entropy,lozano2018entropy,ranocha2018L2stability,sun2019strong},
we will use relaxation to enforce the conservation of the energy
\cite{ketcheson2019relaxation,ranocha2020relaxation,ranocha2020general}.
This approach has its origins in an idea of Sanz-Serna
\cite{sanzserna1982explicit}. It has been applied successfully to
compressible flows \cite{yan2020entropy,ranocha2020fully,waruszewski2022entropy}
and various other systems conserving or dissipating an energy/entropy
functional. It is particularly useful for long-time simulations of
dispersive wave equations and Hamiltonian systems
\cite{ranocha2021broad,ranocha2020relaxationHamiltonian,mitsotakis2021conservative,ranocha2021rate,biswas2023accurate,zhang2020highly}.

The basic idea is as follows. Consider an ODE $u' = f(u)$ and assume
that the ODE conserves an energy $\eta(u)$, i.e.,
$\forall u\colon \eta'(u) f(u) = 0$. Given a one-step method
computing $u^{n+1} \approx u(t^{n+1})$ from $u^n \approx u(t^n)$,
we enforce conservation of the energy $\eta$ by projecting the numerical
solution along the secant line connecting $u^n$ and $u^{n+1}$ onto the
level set $\eta(u) = \eta(u^n)$. Thus, we need to solve the scalar
root finding problem
\begin{equation}
    \eta(u^{n+1}_\gamma) = \eta(u^n),
    \quad
    u^{n+1}_\gamma = u^n + \gamma (u^{n+1} - u^n),
\end{equation}
for the relaxation parameter $\gamma$. The general theory of relaxation
methods shows that there is a unique solution
$\gamma = 1 + \mathcal{O}(\Delta t^{p-1})$ under rather general assumptions
\cite{ranocha2020general}, where $p$ is the order of accuracy of the
time integration method. Continuing the numerical integration with
$u^{n+1}_\gamma$ instead of $u^{n+1}$ guarantees conservation of the
energy $\eta$, of all linear invariants, and at least the same order of
accuracy $p$ as the original method if the relaxed solution is interpreted
as $u^{n+1}_\gamma \approx u(t^{n+1}_\gamma)$, where the relaxed time is
$t^{n+1}_\gamma = t^n + \gamma \Delta t$.

\section{Split forms of the classical shallow water equations}
\label{sec:split-forms-SWE}

The Serre-Green-Naghdi equations \cite{Serre53,GN76,wk95} (see also \cite{Lannes_2020} and references therein)
and their hyperbolic approximation \cite{favrie2017rapid,busto21}
are extensions of the classical shallow water equations by additional
terms modeling dispersive effects. To prepare deriving structure-preserving
discretizations, we first consider some split forms of the classical shallow
water (Saint-Venant) equations.

\subsection{Flat bathymetry}
\label{sec:split-forms-SWE_flat}

Consider the classical shallow water equations with constant bathymetry
\begin{equation}
\label{eq:SWE_flat_cons}
\begin{aligned}
    & h_t + (h u)_x = 0,\\
    & (h u)_t + \left(h u^2 + \frac{1}{2} g h^2\right)_x = 0,
\end{aligned}
\end{equation}
where $h$ denotes the water height, $u$ the velocity, and $g$ the gravitational
constant. The system admits the energy conservation law
\cite{bouchut2004nonlinear,fjordholm2011well}
\begin{equation}
\label{eq:SWE_flat_energy}
    \biggl( \underbrace{\frac{1}{2} g h^2 + \frac{1}{2} h u^2}_{= E} \biggr)_t
    + \biggl( \underbrace{g h^2 u + \frac{1}{2} h u^3}_{= F} \biggr)_x = 0.
\end{equation}
A split form of the equations can be used to prove energy conservation
using only integration by parts, e.g.,
\cite{gassner2016well,wintermeyer2017entropy}. Indeed, there is even a
two-parameter family of energy-conserving split forms of the classical
shallow water equations \cite{ranocha2017shallow}. A simplified version
discarding some higher-order terms of \cite[Section~4.2]{ranocha2017shallow}
reads\footnote{$a_1 = 3 - 4 \alpha$ and $a_2 = (2 - a_1) / 3$ in the notation of
\cite[Section~4.2]{ranocha2017shallow}.}
\begin{equation}
\label{eq:SWE_flat_cons_split1}
\begin{aligned}
    & h_t
        + \alpha (h u)_x
        + (1 - \alpha) h_x u
        + (1 - \alpha) h u_x
        = 0,
    \\
    & (h u)_t
        + (1 - \alpha) g (h^2)_x
        - (1 - 2 \alpha) g h h_x
        \\
        &\qquad
        + \frac{\alpha}{2} (h u^2)_x
        + \frac{1 - \alpha}{2} h (u^2)_x
        + \frac{1 - \alpha}{2} h_x u^2
        + \frac{1}{2} (h u)_x u
        + \frac{1}{2} h u u_x
        = 0.
\end{aligned}
\end{equation}
To simplify the following derivation, we focus on the split form
\eqref{eq:SWE_flat_cons_split1} with $\alpha = 0$, i.e.,
\begin{equation}
\label{eq:SWE_flat_cons_split}
\begin{aligned}
    & h_t
        + h_x u
        + h u_x
        = 0,
    \\
    & (h u)_t
        + g (h^2)_x
        - g h h_x
        + \frac{1}{2} h (u^2)_x
        + \frac{1}{2} h_x u^2
        + \frac{1}{2} (h u)_x u
        + \frac{1}{2} h u u_x
        = 0.
\end{aligned}
\end{equation}
Furthermore, to simplify the treatment of the elliptic terms in the
Serre-Green-Naghdi equations, we will use primitive variables $(h, u)$
instead of the conservative variables $(h, h u)$ in the following.
This is described in more detail in Section~\ref{sec:SGN_flat_introduction}.
Thus, using the product rule in time
\begin{equation}
    (h u)_t = h u_t + h_t u = h u_t - h_x u^2 - h u u_x,
\end{equation}
the split form \eqref{eq:SWE_flat_cons_split} becomes
\begin{equation}
\label{eq:SWE_flat_prim_split}
\begin{aligned}
    & h_t
        + h_x u
        + h u_x
        = 0,
    \\
    & h u_t
        + g (h^2)_x
        - g h h_x
        + \frac{1}{2} h (u^2)_x
        - \frac{1}{2} h_x u^2
        + \frac{1}{2} (h u)_x u
        - \frac{1}{2} h u u_x
        = 0.
\end{aligned}
\end{equation}
To prepare the following arguments, we will demonstrate how to obtain
conservation of total water mass, momentum, and energy.

\subsubsection{Flat bathymetry: conservation of the total water mass}

The spatial terms of the first equation of \eqref{eq:SWE_flat_prim_split}
cancel when integrated over the periodic domain due to integration by parts,
i.e.,
\begin{equation}
    0
    =
    \int h_t \dif x
        + \int h_x u \dif x
        + \int h u_x \dif x
    =
    \frac{\dif}{\dif t} \int h \dif x.
\end{equation}
Note that the two terms vanishing via integration by parts yield exactly
the difference of the flux $h u$ at the periodic boundaries.

\subsubsection{Flat bathymetry: conservation of the total momentum}

The time derivative of the momentum is
\begin{equation}
    (h u)_t = h_t u + h u_t.
\end{equation}
Thus, we multiply the first equation of \eqref{eq:SWE_flat_prim_split}
by $u$, add it to the second equation, integrate over the periodic domain,
and obtain
\begin{equation}
\begin{aligned}
    0
    &=
    \int (h u)_t \dif x
    + \int h_x u^2 \dif x
    + \int h u u_x \dif x
    + \int g (h^2)_x \dif x
    - \int g h h_x \dif x
    \\
    &\quad
    + \frac{1}{2} \int h (u^2)_x \dif x
    - \frac{1}{2} \int h_x u^2 \dif x
    + \frac{1}{2} \int (h u)_x u \dif x
    - \frac{1}{2} \int h u u_x \dif x
    \\
    &=
    \int (h u)_t \dif x
    + \int g (h^2)_x \dif x
    - \int g h h_x \dif x
    \\
    &\quad
    + \frac{1}{2} \int h (u^2)_x \dif x
    + \frac{1}{2} \int h_x u^2 \dif x
    + \frac{1}{2} \int (h u)_x u \dif x
    + \frac{1}{2} \int h u u_x \dif x.
\end{aligned}
\end{equation}
Using integration by parts, the two $g h^2$ terms yield a boundary term
$g h^2 / 2$. The first two terms of the last line as well as the last two
terms cancel due to integration by parts, respectively, resulting in the
boundary terms $h u^2 / 2$ and $h u^2 / 2$.
Thus, all terms canceling via
integration by parts sum up to the expected flux $g h^2 / 2 + h u^2$. This shows that \emph{for exact time integration} momentum is conserved.

\subsubsection{Flat bathymetry: conservation of the total energy}

The time derivative of the energy $E$ \eqref{eq:SWE_flat_energy} is
\begin{equation}
    E_t
    =
    \left( \frac{1}{2} g h^2 + \frac{1}{2} h u^2 \right)_t
    =
    g h h_t + \frac{1}{2} h_t u^2 + h u u_t.
\end{equation}
Thus, we multiply the first equation of \eqref{eq:SWE_flat_prim_split}
by $g h + u^2 / 2$, the second equation by $u$, add them, and obtain
\begin{equation}
\begin{aligned}
    0
    &=
    E_t
    + g h h_x u
    + g h^2 u_x
    + \frac{1}{2} h_x u^3
    + \frac{1}{2} h u^2 u_x
    + g (h^2)_x u
    - g h h_x u
    + \frac{1}{2} h u (u^2)_x
    - \frac{1}{2} h_x u^3
    + \frac{1}{2} (h u)_x u^2
    - \frac{1}{2} h u^2 u_x
    \\
    &=
    E_t
    + g \left( h^2 u_x + (h^2)_x u \right)
    + \frac{1}{2} \left( h u (u^2)_x + (h u)_x u^2 \right).
\end{aligned}
\end{equation}
We see that all pairs of terms in parentheses cancel when integrating over the
periodic domain due to integration by parts, resulting in the expected
flux terms $F = g h^2 u + h u^3 / 2$. Thus, the total energy is conserved, \emph{if the time integration is exact}.

\subsection{Variable bathymetry}
\label{sec:split-forms-SWE_var}

The classical shallow water equations with variable bathymetry are
\begin{equation}
\label{eq:SWE_var_cons}
\begin{aligned}
    & h_t + (h u)_x = 0,\\
    & (h u)_t + \left(h u^2 + \frac{1}{2} g h^2\right)_x + g h b_x = 0,
\end{aligned}
\end{equation}
where $b$ denotes the bathymetry (bottom topography).
The system admits the energy conservation law
\cite{bouchut2004nonlinear,fjordholm2011well,gassner2016well}
\begin{equation}
\label{eq:SWE_var_energy}
    \biggl( \underbrace{\frac{1}{2} g h^2 + g h b + \frac{1}{2} h u^2}_{= E} \biggr)_t
    + \biggl( \underbrace{g h^2 u + g h b u + \frac{1}{2} h u^3}_{= F} \biggr)_x = 0.
\end{equation}
The generalization to variable bathymetry of the split form
\eqref{eq:SWE_flat_prim_split} of \cite[Section~5.3]{ranocha2017shallow} is
\begin{equation}
\label{eq:SWE_var_prim_split}
\begin{aligned}
    & h_t
        + h_x u
        + h u_x
        = 0,
    \\
    & h u_t
        + g \bigl( h (h + b) \bigr)_x
        - g (h + b) h_x
        + \frac{1}{2} h (u^2)_x
        - \frac{1}{2} h_x u^2
        + \frac{1}{2} (h u)_x u
        - \frac{1}{2} h u u_x
        = 0.
\end{aligned}
\end{equation}
Conservation of the total water mass follows as in the case of flat bathymetry.
Since \eqref{eq:SWE_var_prim_split} is the same as the previous split form
\eqref{eq:SWE_flat_prim_split} for flat bathymetry $b = 0$, conservation of
the total momentum follows for constant bottom topography as well. In the
energy rate of change $E_t$, we get the new term $g b h_t$,  leading to the   additional terms
\begin{equation}
    g h_x b u
    + g h b u_x
    + g (h b)_x u
    - g h_x b u
    =
    g \left( h b u_x + (h b)_x u \right),
\end{equation}
which vanish due to integration by parts when integrating over the periodic
domain. Thus, the total energy is conserved for variable bathymetry as well.

Moreover, the split form \eqref{eq:SWE_var_prim_split} is well-balanced, i.e.,
it preserves the lake-at-rest steady state $h + b = \mathrm{const}$, $u = 0$.
Indeed, the time derivative of $h$ vanishes for $u = 0$, and the time derivative
of the velocity is given by
\begin{equation}
    0
    =
    h u_t
    + g \bigl( h (h + b) \bigr)_x
    - g (h + b) h_x
    =
    h u_t
    g (h + b) h_x
    - g (h + b) h_x
    =
    h u_t
\end{equation}
if $h + b = \mathrm{const}$ additionally. All of these properties still hold
for semidiscretizations using periodic SBP operators.

\section{Review of the Serre-Green-Naghdi equations for flat bathymetry}
\label{sec:SGN_flat}

In this section, we review the classical form of the Serre-Green-Naghdi
equations for flat bathymetry as well as their hyperbolic approximation.
We present the associated energy conservation laws and describe how to pass
from one set of equations to the other to prepare the derivations later
in this paper.

\subsection{Equations in classical form and elliptic operator}
\label{sec:SGN_flat_introduction}

On a flat bathymetry, the SGN equations can be written as
\begin{equation}
\label{eq:SGN_original_flat_cons}
\begin{aligned}
    & h_t + (h u)_x = 0,\\
    & (h u)_t + \left(h u^2 + \frac{1}{2} g h^2 + \tilde p \right)_x = 0, \\
    & \tilde p = - \dfrac{1}{3} \left( h^3 (\dot u)_x  - 2 h^3 u_x^2 \right),
\end{aligned}
\end{equation}
with the classical notation for the material derivative
\begin{equation}
\label{eq:material_derivative}
    \dot a = a_t + u a_x.
\end{equation}
As in Section~\ref{sec:split-forms-SWE}, $h$ is the water height,
$u$ the velocity, and $g$ the gravitational constant. Compared to the
shallow water equations \eqref{eq:SWE_flat_cons}, the SGN equations
\eqref{eq:SGN_original_flat_cons} contain an additional non-hydrostatic pressure
$\tilde p$.
This system is known to admit an energy conservation law reading
\begin{equation}
\label{eq:SGN_original_flat_energy}
    \biggl( \underbrace{\frac{1}{2} g h^2 + \frac{1}{2} h u^2 + \frac{1}{6} h (\dot h)^2}_{= E} \biggr)_t
    + \biggl( \underbrace{g h^2 u + \frac{1}{2} h u^3 + \frac{1}{6} h (\dot h)^2 u + \tilde p u}_{= F} \biggr)_x = 0.
\end{equation}
Note that both the energy $E$ and the energy flux $F$ are extensions of
the corresponding quantities for the shallow water equations
\eqref{eq:SWE_flat_energy}.
Our objective is to construct discrete approximations of the SGN system
preserving exactly the energy conservation law
\eqref{eq:SGN_original_flat_energy}.

Note that to advance in time system \eqref{eq:SGN_original_flat_cons}
requires the inversion of the operator  $\mathcal{T}(u_t)$, where
\begin{equation}
\label{eq:SGN_original_flat_elliptic_operator}
    \mathcal{T}(v) =  h v - \dfrac{1}{3} ( h^3 v_{x})_x.
\end{equation}
Indeed, writing   system \eqref{eq:SGN_original_flat_cons} in  primitive variables we have
\begin{equation}
\label{eq:SGN_original_flat_prim}
\begin{aligned}
    & h_t + (h u)_x = 0,\\
    & h u_t - \frac{1}{3} (h^3 u_{tx})_x + \frac{1}{2} g (h^2)_x + h u u_x + p_x = 0, \\
    & p = \frac{1}{3} h^3 u_x^2 - \frac{1}{3} h^3 u u_{xx},
\end{aligned}
\end{equation}
where we have rewritten the non-hydrostatic pressure as
\begin{equation}
    \tilde p = -\frac{1}{3} (h^3 u_{tx})_x + p,
\end{equation}
and calculated
\begin{equation}
\begin{aligned}
    \tilde p
    &=
    - \frac{1}{3} \left( h^3 (\dot u)_x - 2 h^3 u_x^2 \right)
    =
    - \frac{1}{3} \left( h^3 (u_t + u u_x)_x - 2 h^3 u_x^2 \right)
    \\
    &=
    - \frac{1}{3} h^3 u_{tx}
    - \frac{1}{3} h^3 u u_{xx}
    + \frac{1}{3} h^3 u_x^2.
\end{aligned}
\end{equation}
This  shows the appearance of  the elliptic operator
$\mathcal{T}(u_t)$ \eqref{eq:SGN_original_flat_elliptic_operator}
whose inversion is required  to obtain $u_t$.

\subsection{1-D augmented Lagrangian hyperbolic system}

In this work we will also study the structure-preserving approximation of the
hyperbolic augmented Lagrangian formulation of the fully-nonlinear
SGN equations reading \cite{favrie2017rapid,TKACHENKO2023111901}
\begin{equation}
\label{eq:SGN_hyperbolic_flat_cons}
\begin{aligned}
    & h_t + (h u)_x = 0,\\
    & (h u)_t + \bigg(hu^2 + \frac{1}{2} gh^2 + \frac{\lambda}{3} \eta (1 - \eta / h)  \bigg)_x = 0, \\
    & (h w)_t + ( h w u )_x =   \lambda (1 - \eta / h),\\
    & (h \eta)_t + ( h \eta u )_x = h w.
\end{aligned}
\end{equation}
For $\lambda = 0$, we recover the usual hyperbolic shallow
water system \eqref{eq:SWE_flat_cons} from the first two equations.
The augmented system  with $\lambda >0$ can be shown to be
an approximation of order $\mathcal{O}(1/\lambda)$ of the original one. The rigorous
justification of this fact can be found in \cite{duchene}.
This justifies the fact that the value of  $\lambda$ should be large.
In compact form, we can write the above system as
(with obvious definitions)
\begin{equation*}
    \mathbf{U}_t + \mathbf{F}_x = \mathbf{S}.
\end{equation*}
This system is hyperbolic, with characteristic speeds $\lambda_{1/4}= u\mp c$
and $\lambda_{2/3}=u$, where the celerity is given by
$$
c^2 = gh + \lambda \dfrac{\eta^2}{h^2}.
$$
The model can be shown to admit the mathematical entropy (energy)
conservation law
\begin{equation}
\label{eq:SGN_hyperbolic_flat_energy}
    \biggl( \underbrace{\frac{1}{2} g h^2 + \frac{1}{2} h u^2 + \frac{1}{6} h w^2 + \frac{\lambda}{6} h (1 - \eta / h)^2 }_{= E} \biggr)_t
    + \biggl( \underbrace{g h^2 u + \frac{1}{2} h u^3 + \frac{1}{6} h w^2 u + \frac{\lambda}{6} h (1 - \eta^2 / h^2) u}_{= F} \biggr)_x = 0.
\end{equation}
Please note again that the energy $E$ and the energy flux $F$ are extensions
of the corresponding quantities for the shallow water equations
\eqref{eq:SWE_flat_energy}.
In primitive variables, the system \eqref{eq:SGN_hyperbolic_flat_cons} reads
\begin{equation}
\label{eq:SGN_hyperbolic_flat_prim}
\begin{aligned}
    & h_t + (h u)_x = 0,\\
    & h u_t + \frac{1}{2} g (h^2)_x + h u u_x + \left( \frac{\lambda}{3} \eta (1 - \eta / h) \right)_x = 0, \\
    & h w_t + h u w_x = \lambda (1 - \eta / h),\\
    & \eta_t + \eta_x u = w.
\end{aligned}
\end{equation}

\subsection{Passing from the hyperbolic to the classical form}
\label{sec:from_hyperbolic_to_classical}

It will be very useful later in the paper to use the existing relations between
the hyperbolic and classical systems
since it is easier for us to derive appropriate structure-preserving
methods for the hyperbolic system \eqref{eq:SGN_hyperbolic_flat_prim} and then
deduce the corresponding methods for the classical system \eqref{eq:SGN_original_flat_prim}.
In particular, we can easily pass from the hyperbolic approximation
\eqref{eq:SGN_hyperbolic_flat_prim} to the classical system
\eqref{eq:SGN_original_flat_prim} as follows.
First, we note that
$$
    \left\{\begin{array}{ll}
        \eta_t + u \eta_x = \dot{\eta} = w, \\
        h w_t + h u w_x = h \dot{w} = \lambda (1 - \eta / h)
    \end{array}\right.
    \quad \implies \quad
    \lambda (1 - \eta / h) = h \ddot{\eta}.
$$
We then set
\begin{equation}
\label{eq:pi_def}
    \pi = \lambda (1 - \eta / h) / 3 = \dfrac{h}{3} \ddot{\eta}
\end{equation}
and note that the hyperbolic system \eqref{eq:SGN_hyperbolic_flat_prim}
can also be written as
\begin{equation}
\label{eq:SGN_hyperbolic_flat_prim_pi}
\begin{aligned}
    & h_t + (h u)_x = 0,\\
    & h u_t + g h h_x + h u u_x + (\eta \pi)_x = 0,\\
    & h w_t + h u w_x = 3 \pi,\\
    & h \eta_t + h \eta_x u = h w.
\end{aligned}
\end{equation}
Considering the relaxed limit $\lambda\rightarrow\infty$ in which $\eta\rightarrow h$  \cite{favrie2017rapid,TKACHENKO2023111901}, we set
$$
    \tilde p = \lim\limits_{\eta \to h} \eta \pi = \dfrac{h^2}{3}\ddot{h},
$$
which can be readily shown to be equivalent to the last equation in
\eqref{eq:SGN_original_flat_cons}.
In this limit, the first two equations reduce to the classical SGN system.
Note that we also have $w = \dot{h}$ in the limit $\eta \to h$.
Using this, we can write the classical system \eqref{eq:SGN_original_flat_prim}
in the alternative first-order form (see also \cite{cpr2})
\begin{equation}
\label{eq:SGN_classical_first_order}
\begin{aligned}
    & h_t + (h u)_x = 0,\\
    & h u_t + g h h_x + h u u_x + (h \pi)_x = 0,\\
    & h w_t + h u w_x = 3 \pi,\\
    & w + h u_x = 0,
\end{aligned}
\end{equation}
in which the last two equations define   $w$ and $\pi$.
System \eqref{eq:SGN_classical_first_order} can be directly obtained
from \eqref{eq:SGN_hyperbolic_flat_prim} in the limit $\eta \to h$
(or simply replacing $\eta$ by $h$), and using the mass equation to express
$\dot{h}$.
Note that to march the system in time, we still need to combine the last two equations and replace $\pi$ in the mometum equation,
which at the continuous level still leads to the need to invert operator \eqref{eq:SGN_original_flat_elliptic_operator}.
However, given a discretization for the hyperbolic system, with this correspondence we can   deduce one for the classical  system in first order form \eqref{eq:SGN_classical_first_order}.
Moreover, the SGN energy from \eqref{eq:SGN_original_flat_energy}
can be written as
\begin{equation}
    E = \frac{1}{2} g h^2 + \frac{1}{2} h u^2 + \frac{1}{6} h w^2
\end{equation}
and we have that
$$
 \dif E
 =
 \left( g h + \dfrac{1}{2} u^2 + \dfrac{1}{2} w^2 \right) \dif h
 + u\, (h\,\dif u) + w\, (h\,\dif w)
$$
and
\begin{equation}
    F_x =
    \left( g h + \dfrac{1}{2} u^2 + \dfrac{1}{2} w^2 \right) (hu)_x
    + u\, \left( h u_x + g h h_x + (h\pi)_x \right)
    + w\, \left( huw_x - 3 \pi \right)
    + \pi \left( w + h u_x \right).
\end{equation}
So from \eqref{eq:SGN_classical_first_order},  one obtains   energy conservation using only integration by parts
by multiplying by the transpose of the dual variables $\mathbf{W}^*$
\begin{equation}\label{eq:wstar_flat_noncons}
(    \mathbf{W}^* )^T:=
    \left(
        g h + \frac{1}{2} u^2 + \frac{1}{2} w^2,\;
        u,\;
        w,\;
        \pi
\right).
\end{equation}

\section{Hyperbolic approximation with flat bathymetry}
\label{sec:SGN_hyperbolic_flat}

In this section, we will derive an energy-conservative split form and
corresponding structure-preserving numerical methods for the hyperbolic
approximation \eqref{eq:SGN_hyperbolic_flat_prim} of the Serre-Green-Naghdi
equations with flat bathymetry. We will first derive the split form and
then present the numerical methods.

\subsection{Energy equation}
\label{sec:SGN_hyperbolic_flat_energy}

The energy conservation law \eqref{eq:SGN_hyperbolic_flat_energy}
can be obtained as usual by multiplying the equations
\eqref{eq:SGN_hyperbolic_flat_cons} by the entropy variables
\begin{equation}\label{de1}
    \mathbf{U}^*
    =
    \frac{\partial E}{\partial \mathbf{U}}
    =
    \left(\begin{array}{c}
        g h - \frac{1}{2} h u^2 - \frac{1}{6} h w^2 + \dfrac{\lambda}{2}( 1 -\dfrac{\eta}{h} ) \left( \dfrac{1}{3} +\dfrac{\eta}{h}   \right) \\
        u \\
        \dfrac{1}{3} w \\
        -\dfrac{\lambda}{3 h}(1 -\dfrac{\eta}{h} )
    \end{array}\right)
\end{equation}
and summing them up since
\begin{equation}\label{de0}
    \dif E = (\mathbf{U}^*)^T \dif \mathbf{U}.
\end{equation}
A similar expression can be obtained for the physical/primitive
variables $\mathbf{V} := ( h, u, w, \eta)^T$:
\begin{equation}\label{de2}
    \dif E = (\mathbf{V}^*)^T  \text{diag}(1,h,h,1)\dif \mathbf{V},
    \quad
    \mathbf{V}^* =
    \left(\begin{array}{c}
        g h + \frac{1}{2} u^2 + \frac{1}{6} w^2 + \dfrac{\lambda}{6}(1 - \eta^2 / h^2) \\
         u \\
        \dfrac{1}{3}  w \\
        -\dfrac{\lambda}{3}(1 - \eta / h)
    \end{array}\right).
\end{equation}
Setting $P_h=  \text{diag}(1,h,h,1)$,
we have  that
\begin{equation}
    (\mathbf{V}^*)^T P_h \mathbf{V}_t
    =
    (\mathbf{U}^*)^T \mathbf{U}_t
    =
    E_t
    =
    - F_x.
\end{equation}
Indeed, multiplying \eqref{eq:SGN_hyperbolic_flat_prim} by $\mathbf{V}^*$
and summing up, we get
\begin{equation}
\label{eq:calculate_energy_SGN_hyperbolic_flat_prim}
\begin{aligned}
    -E_t
    &=
    -(\mathbf{V}^*)^T P_h\mathbf{V}_t
    \\
    &=
    \left( g h + \frac{1}{2} u^2 + \frac{1}{6} w^2 + \dfrac{\lambda}{6}(1 - \eta^2 / h^2) \right) (h u)_x
    \\
    &\quad
    + u \left( \frac{1}{2} g (h^2)_x + h u u_x + \frac{\lambda}{3} \left( \eta (1 - \eta / h) \right)_x \right)
    + \frac{w}{3} \left( h u w_x - \lambda (1 - \eta / h) \right)
    - \frac{\lambda}{3} (1 - \eta / h) \left( \eta_x u - w \right)
    \\
    &=
    \left( g h (h u)_x + \frac{1}{2} g (h^2)_x u \right)
    + \left( \frac{1}{2} (h u)_x u^2 + h u^2 u_x \right)
    + \left( \frac{1}{6} (h u)_x w^2 + \frac{1}{3} h u w w_x \right)
    \\
    &\quad
    + \frac{\lambda}{6} \left(
        (1 - \eta^2 / h^2) (h u)_x
        + 2 \left( \eta (1 - \eta / h) \right)_x u
        - 2 (1 - \eta / h) \eta_x u
    \right)
    \\
    &=
    \left( g h^2 u + \frac{1}{2} h u^3 + \frac{1}{6} h u w^2 + \frac{\lambda}{6} h (1 - \eta^2 / h^2) u \right)_x
    =
    F_x.
\end{aligned}
\end{equation}

\subsection{Split form}

To develop an energy-conservative split form of the hyperbolic system
\eqref{eq:SGN_hyperbolic_flat_prim}, we  rewrite each nonlinear term
as a linear combination of terms equivalent when using the product/chain
rule. To reduce the number of parameters we have to deal with simultanesously,
we start by looking at a split form of the mass terms and non-hydrostatic
pressure terms allowing to obtain boundary terms upon integration by parts.
Thus, we start from the hyperbolic system \eqref{eq:SGN_hyperbolic_flat_prim}
and introduce parameters $\alpha, \beta, \gamma \in \mathbb{R}$
to obtain the split form
\begin{equation}
\label{eq:SGN_hyperbolic_flat_split1}
\begin{aligned}
    & h_t + \alpha (h u)_x + (1 - \alpha) h_x u + (1 - \alpha) h u_x = 0,\\
    & h u_t + \frac{1}{2} g (h^2)_x + h u u_x + \frac{\lambda}{3} \beta \left( \eta (1 - \eta / h) \right)_x + \frac{\lambda}{3} (1 - \beta) \frac{\eta^2}{h^2} h_x + \frac{\lambda}{3} (1 - \beta) (1 - 2 \eta / h) \eta_x = 0, \\
    & h w_t + h u w_x = \lambda (1 - \eta / h),\\
    & \eta_t + \gamma \eta_x u + (1 - \gamma) (\eta u)_x - (1 - \gamma) \eta u_x = w.
\end{aligned}
\end{equation}
Multiplying the equations by the corresponding entropy variables
$\mathbf{V}^*$ in primitive variables and summing up the terms contributing
to the non-hydrostatic pressure part of the energy flux $F$, we obtain
\begin{equation}
\begin{aligned}
    &\quad
    \frac{\lambda}{6} (1 - \eta^2 / h^2)
    \left( \alpha (h u)_x + (1 - \alpha) h_x u + (1 - \alpha) h u_x \right)
    \\
    &\quad
    + \frac{\lambda}{3} \beta \left( \eta (1 - \eta / h) \right)_x u
    + \frac{\lambda}{3} (1 - \beta) \frac{\eta^2}{h^2} h_x u
    + \frac{\lambda}{3} (1 - \beta) (1 - 2 \eta / h) \eta_x u
    \\
    &\quad
    - \frac{\lambda}{3}(1 - \eta / h) \gamma \eta_x u
    - \frac{\lambda}{3}(1 - \eta / h) (1 - \gamma) (\eta u)_x
    + \frac{\lambda}{3}(1 - \eta / h) (1 - \gamma) \eta u_x
    \\
    &=
    \alpha \frac{\lambda}{6} (h u)_x
    - \alpha \frac{\lambda}{6} \frac{\eta^2}{h^2} (h u)_x
    + (1 - \alpha) \frac{\lambda}{6} \left( h_x u + h u_x \right)
    + \left( 1 + \alpha - 2 \beta \right) \frac{\lambda}{6} \frac{\eta^2}{h^2} h_x u
    \\
    &\quad
    + \left( 2 (1 - \gamma) - (1 - \alpha) \right) \frac{\lambda}{6} \frac{\eta^2}{h} u_x
    - \beta \frac{\lambda}{3} \left( \frac{\eta^2}{h} \right)_x u
    + \left( \beta + (1 - \beta) - \gamma - (1 - \gamma) \right) \frac{\lambda}{3} \eta_x u
    \\
    &\quad
    + \left( \gamma - 2 (1 - \beta) \right) \frac{\lambda}{3} \frac{\eta}{h} \eta_x u
    + (1 - \gamma) \frac{\lambda}{3} \frac{\eta}{h} (\eta u)_x.
\end{aligned}
\end{equation}
We can compare terms involving the same monomials, which gives the conditions
to be verified to obtain only boundary terms upon integration by parts
of one of them:
\begin{equation*}
\begin{split}
    (hu)_x \; :\; & \textsf{ok} \\
    \dfrac{\eta^2}{h^2}(hu)_x\; :\; & \alpha =0\\
    uh_x \; :\; &   \textsf{ok} \\
    \dfrac{\eta^2}{h^2} h_x u  \; :\; & 1 + \alpha - 2 \beta = 0\\
    \dfrac{\eta^2}{h}u_x  \; :\; & - (1-\alpha) + 2 \beta + 2 (1-\gamma) =0 \\
    \eta_x u\; :\; & \beta + (1-\beta) - \gamma - (1-\gamma) =0  \\
    \dfrac{\eta}{h} \eta_x u\; :\; &     -2 (1-\beta) + \gamma =0 \\
    (\eta u)_x\; :\; &  1-\gamma =0
\end{split}
\end{equation*}
Luckily enough, the above system admits the unique solution
$$
    \alpha = 0, \quad \beta = 1/2, \quad \gamma = 1,
$$
giving the split forms
\begin{equation*}
\begin{aligned}
    (h u)_x
        &\to
        h_x u + h u_x,
    \\
    \dfrac{\lambda}{3} \dfrac{\eta^2}{h^2} h_x + \dfrac{\lambda}{3} (1-2\dfrac{\eta}{h})\eta_x
        &\to
        \dfrac{\lambda}{6} \dfrac{\eta^2}{h^2} h_x +
        \dfrac{\lambda}{6} (1 - 2 \eta / h) \eta_x
        + \dfrac{\lambda}{6} (\eta (1 - \eta / h))_x,
    \\
    \eta_x u
        &\to
        \eta_x u.
\end{aligned}
\end{equation*}
We are thus left now with the shallow water terms plus the vertical mass
equation on $w$. We can choose the split form of the shallow water equation
\eqref{eq:SWE_var_prim_split} of \cite{ranocha2017shallow} for the remaining
shallow water terms.
Finally, we need to consider the terms leading to the
$h u w^2 / 6$ term of the energy flux \eqref{eq:SWE_flat_energy}.
We use the ansatz
\begin{equation}
    \delta h u w_x
    + (1 - \delta) \left(
        (h u w)_x
        - \epsilon (h u)_x w
        - (1 - \epsilon) h_x u w
        - (1 - \epsilon) h u_x w
    \right)
\end{equation}
for the  term $h u w_x$ in the third equation of
\eqref{eq:SGN_hyperbolic_flat_prim}. Assembling   the  terms resulting
in the energy equation we get
\begin{equation}
\begin{aligned}
    &\quad
    \frac{1}{6} h_x u w^2
    + \frac{1}{6} h u_x w^2
    + \delta \frac{1}{3} h u w w_x
    + (1 - \delta) \frac{1}{3} (h u w)_x w
    \\
    &\quad
    - (1 - \delta) \epsilon \frac{1}{3} (h u)_x w^2
    - (1 - \delta) (1 - \epsilon) \frac{1}{3} h_x u w^2
    - (1 - \delta) (1 - \epsilon) \frac{1}{3} h u_x w^2
    \\
    &=
    \frac{1}{6} \left( 1 - 2 (1 - \delta) (1 - \epsilon) \right) h_x u w^2
    + \frac{1}{6}\left( 1 - 2 (1 - \delta) (1 - \epsilon) \right) h u_x w^2
    \\
    &\quad
    + \delta \frac{1}{3} h u w w_x
    + (1 - \delta) \frac{1}{3} (h u w)_x w
    - (1 - \delta) \epsilon \frac{1}{3} (h u)_x w^2.
\end{aligned}
\end{equation}
To get only boundary terms from integration by parts, we need to solve the
system
\begin{equation}
\begin{aligned}
    & h_x u w^2 \colon& \quad 1 - 2 (1 - \delta) (1 - \epsilon) &= 0, \\
    & h u_x w^2 \colon& \quad 1 - 2 (1 - \delta) (1 - \epsilon) &= 0, \\
    & h u w w_x \colon& \quad \delta - (1 - \delta) &= 0, \\
    & (h u)_x w^2 \colon& \quad (1 - \delta) \epsilon &= 0,
\end{aligned}
\end{equation}
with unique   solution $    (\delta,\epsilon) =( 1 / 2, 0)$.
Assembling  all the results,    we obtain the following split form
for   \eqref{eq:SGN_hyperbolic_flat_prim}:
\begin{equation}
\label{eq:SGN_hyperbolic_flat_prim_split0}
\begin{aligned}
    &h_t
        + h_x u + h u_x
        = 0,
    \\
    &h u_t
        + g (h^2)_x - g h h_x
        + \frac{1}{2} h (u^2)_x - \frac{1}{2} h_x u^2
        + \frac{1}{2} (h u)_x u - \frac{1}{2} h u u_x
        \\
        &\qquad
        + \frac{\lambda}{6} \frac{\eta^2}{h^2} h_x
        + \frac{\lambda}{6} (1 - 2 \eta / h) \eta_x
        + \frac{\lambda}{6} (\eta (1 - \eta / h))_x
        = 0,
    \\
    &h w_t
        + \frac{1}{2} (h u w)_x
        + \frac{1}{2} h u w_x
        - \frac{1}{2} h_x u w
        - \frac{1}{2} h u_x w
        = \lambda (1 - \eta / h),
    \\
    &\eta_t
        + \eta_x u
        = w.
\end{aligned}
\end{equation}
By grouping terms that appear multiple times, we obtain
\begin{equation}
\label{eq:SGN_hyperbolic_flat_prim_split}
\begin{aligned}
     &h_t
        + h_x u + h u_x
        = 0,\\
    &h u_t
        + g (h^2)_x - g h h_x
        + \frac{1}{2} h (u^2)_x - \frac{1}{2} h_x u^2
        + \frac{1}{2} (h u)_x u - \frac{1}{2} h u u_x
        \\
        &\qquad
        + \frac{\lambda}{6} \frac{\eta^2}{h^2} h_x
        + \frac{\lambda}{3} \eta_x
        - \frac{\lambda}{3} \frac{\eta}{h} \eta_x
        - \frac{\lambda}{6} \Bigl( \frac{\eta^2}{h} \Bigr)_x = 0, \\
    &h w_t
        + \frac{1}{2} (h u w)_x
        + \frac{1}{2} h u w_x
        - \frac{1}{2} h_x u w
        - \frac{1}{2} h u_x w
        = \lambda - \lambda \frac{\eta}{h},\\
    &\eta_t
        + \eta_x u
        = w.
\end{aligned}
\end{equation}

\subsection{Semidiscretization}

We now consider the semidiscrete form of the non-conservative system
\eqref{eq:SGN_hyperbolic_flat_prim_split}, which reads
\begin{equation}
\label{eq:SGN_hyperbolic_flat_prim_SBP}
\begin{split}
    & \partial_t \pmb{h}
        + \pmb{u} D \pmb{h}
        + \pmb{h} D \pmb{u}
        = \pmb{0},\\
    & \pmb{h} \partial_t \pmb{u}
        + g D \pmb{h}^2
        - g \pmb{h} D \pmb{h}
        + \frac{1}{2} \pmb{h} D \pmb{u}^2
        - \frac{1}{2} \pmb{u}^2 D \pmb{h}
        + \frac{1}{2} \pmb{u} D \pmb{h} \pmb{u}
        - \frac{1}{2} \pmb{h} \pmb{u} D \pmb{u}
        \\
        &\qquad
        + \frac{\lambda}{6} \frac{\pmb{\eta}^2}{\pmb{h}^2} D \pmb{h}
        + \frac{\lambda}{3} D \pmb{\eta}
        - \frac{\lambda}{3} \frac{\pmb{\eta}}{\pmb{h}} D \pmb{\eta}
        - \frac{\lambda}{6} D \frac{\pmb{\eta}^2}{\pmb{h}}
        = \pmb{0}, \\
    & \pmb{h} \partial_t \pmb{w}
        + \frac{1}{2} D \pmb{h} \pmb{u} \pmb{w}
        + \frac{1}{2} \pmb{h} \pmb{u} D \pmb{w}
        - \frac{1}{2} \pmb{u} \pmb{w} D \pmb{h}
        - \frac{1}{2} \pmb{h} \pmb{w} D \pmb{u}
        = \lambda - \lambda \frac{\pmb{\eta}}{\pmb{h}},\\
    & \partial_t \pmb{\eta}
        + \pmb{u} D \pmb{\eta}
        = \pmb{w}.
\end{split}
\end{equation}
Here, the discretized functions are again denoted in boldface. The operator
$D$ is the discrete differentiation operator. Multiplication and division
of vectors is defined pointwise.
The discrete total energy for
\eqref{eq:SGN_hyperbolic_flat_prim_SBP} is $\pmb{1}^T M \pmb{E}$, where
$\pmb{1}^T = (1, \dots, 1)$ is the vector of ones, $M$ is the mass matrix,
and
\begin{equation}
\begin{aligned}
    \pmb{E}
    &=
    \frac{1}{2} g \pmb{h}^2
    + \frac{1}{2} \pmb{h} \pmb{u}^2
    + \frac{1}{6} \pmb{h} \pmb{w}^2
    + \frac{\lambda}{6} \pmb{h}
    - \frac{\lambda}{3} \pmb{\eta}
    + \frac{\lambda}{6} \frac{\pmb{\eta}^2}{\pmb{h}}
\end{aligned}
\end{equation}
is the discrete equivalent of the energy $E$ in
\eqref{eq:SGN_hyperbolic_flat_energy}.

\begin{theorem}
\label{thm:SGN_hyperbolic_flat_prim_SBP}
    Consider the semidiscretization \eqref{eq:SGN_hyperbolic_flat_prim_SBP}
    of the hyperbolic approximation of the Serre-Green-Naghdi
    equations \eqref{eq:SGN_hyperbolic_flat_cons}
    with a periodic first-derivative SBP operator $D$ with
    diagonal mass/norm matrix.
    \begin{enumerate}
        \item The total water mass $\pmb{1}^T M \pmb{h}$ is conserved.
        \item The total energy $\pmb{1}^T M \pmb{E}$ is conserved.
    \end{enumerate}
\end{theorem}
\begin{proof}
    The discrete
    total water mass is $\pmb{1}^T M \pmb{h}$, and its rate of change is
    \begin{equation}
        \partial_t (\pmb{1}^T M \pmb{h})
        = \pmb{1}^T M \partial_t \pmb{h}
        = -\pmb{1}^T M \pmb{u} D \pmb{h} -\pmb{1}^T M \pmb{h} D \pmb{u}
        = -\pmb{u}^T M D \pmb{h} -\pmb{h}^T M D \pmb{u}
        = 0,
    \end{equation}
    hence the result.
    Here, we have used again that $M$ is diagonal and applied the
    periodic SBP property \eqref{eq:SBP_periodic}.

    We now evaluate the rate of change of energy. First we note that
    \begin{equation}
    \begin{aligned}
        \partial_{\pmb{h}} \pmb{E}
        &=
        g \pmb{h}
        + \frac{1}{2} \pmb{u}^2
        + \frac{1}{6} \pmb{w}^2
        + \frac{\lambda}{6} \pmb{1}
        - \frac{\lambda}{6} \frac{\pmb{\eta}^2}{\pmb{h}^2},
        \\
        \partial_{\pmb{u}} \pmb{E}
        &=
        \pmb{h} \pmb{u},
        \\
        \partial_{\pmb{w}} \pmb{E}
        &=
        \frac{1}{3} \pmb{h} \pmb{w},
        \\
        \partial_{\pmb{\eta}} \pmb{E}
        &=
        - \frac{\lambda}{3} \pmb{1}
        + \frac{\lambda}{3} \frac{\pmb{\eta}}{\pmb{h}}.
    \end{aligned}
    \end{equation}
So the semidiscrete rate of change of the energy satisfies
    \begin{equation}
    \begin{aligned}
        -\partial_t (\pmb{1}^T M \pmb{E})
        &=
        -\left(
            \pmb{1}^T M \partial_{\pmb{h}} \pmb{E} \cdot \partial_t \pmb{h}
            + \pmb{1}^T M \partial_{\pmb{u}} \pmb{E} \cdot \partial_t \pmb{u}
            + \pmb{1}^T M \partial_{\pmb{w}} \pmb{E} \cdot \partial_t \pmb{w}
            + \pmb{1}^T M \partial_{\pmb{\eta}} \pmb{E} \cdot \partial_t \pmb{\eta}
        \right)
        \\
        &=
        g (\pmb{h} \pmb{u})^T M D \pmb{h}
        + g (\pmb{h}^2)^T M D \pmb{u}
        + \frac{1}{2} (\pmb{u}^3)^T M D \pmb{h}
        + \frac{1}{2} (\pmb{h} \pmb{u}^2)^T M D \pmb{u}
        \\
        &\quad
        + \frac{1}{6} (\pmb{u} \pmb{w}^2)^T M D \pmb{h}
        + \frac{1}{6} (\pmb{h} \pmb{w}^2)^T M D \pmb{u}
        + \frac{\lambda}{6} (\pmb{u})^T M D \pmb{h}
        + \frac{\lambda}{6} (\pmb{h})^T M D \pmb{u}
        \\
        &\quad
        - \frac{\lambda}{6} \Bigl(\frac{\pmb{\eta}^2}{\pmb{h}^2} \pmb{u}\Bigr)^T M D \pmb{h}
        - \frac{\lambda}{6} \Bigl(\frac{\pmb{\eta}^2}{\pmb{h}}\Bigr)^T M D \pmb{u}
        \\
        &\quad
        + g \pmb{u}^T M D \pmb{h}^2
        - g (\pmb{h} \pmb{u})^T M D \pmb{h}
        + \frac{1}{2} (\pmb{h} \pmb{u})^T M D \pmb{u}^2
        - \frac{1}{2} (\pmb{u}^3)^T M D \pmb{h}
        + \frac{1}{2} (\pmb{u}^2)^T M D \pmb{h} \pmb{u}
        \\
        &\quad
        - \frac{1}{2} (\pmb{h} \pmb{u}^2)^T M D \pmb{u}
        + \frac{\lambda}{6} \frac{\pmb{\eta}^2}{\pmb{h}^2} \pmb{u} D \pmb{h}
        + \frac{\lambda}{3} \pmb{u} D \pmb{\eta}
        - \frac{\lambda}{3} \frac{\pmb{\eta}}{\pmb{h}} \pmb{u} D \pmb{\eta}
        - \frac{\lambda}{6} \pmb{u} D \frac{\pmb{\eta}^2}{\pmb{h}}
        \\
        &\quad
        + \frac{1}{6} \pmb{w}^T M D \pmb{h} \pmb{u} \pmb{w}
        + \frac{1}{6} (\pmb{h} \pmb{u} \pmb{w})^T M D \pmb{w}
        - \frac{1}{6} (\pmb{u} \pmb{w}^2)^T M D \pmb{h}
        - \frac{1}{6} (\pmb{h} \pmb{w}^2)^T M D \pmb{u}
        \\
        &\quad
        - \frac{\lambda}{3} \pmb{w}^T M \pmb{1}
        + \frac{\lambda}{3} \Bigl(\frac{\pmb{\eta}}{\pmb{h}}\Bigr)^T M \pmb{w}
        - \frac{\lambda}{3} \pmb{u}^T M D \pmb{\eta}
        + \frac{\lambda}{3} \pmb{1}^T M \pmb{w}
        + \frac{\lambda}{3} \Bigl(\frac{\pmb{\eta}}{\pmb{h}} \pmb{u}\Bigr)^T M D \pmb{\eta}
        - \frac{\lambda}{3} \Bigl(\frac{\pmb{\eta}}{\pmb{h}}\Bigr)^T M \pmb{w}.
    \end{aligned}
    \end{equation}
    Here, we have used that the mass matrix $M$ is diagonal to simplify
    terms such as
    \begin{equation}
        g \pmb{1}^T M \pmb{h} \pmb{u} D \pmb{h}
        =
        g (\pmb{h} \pmb{u})^T M D \pmb{h}.
    \end{equation}
    Canceling and grouping terms, we obtain
    \begin{equation}
    \begin{aligned}
        -\partial_t E
        &=
        \left(
            g \pmb{u}^T M D \pmb{h}^2
            + g (\pmb{h}^2)^T M D \pmb{u}
        \right)
        + \frac{1}{2} \left(
            (\pmb{u}^2)^T M D \pmb{h} \pmb{u}
            + (\pmb{h} \pmb{u})^T M D \pmb{u}^2
        \right)+ \frac{\lambda}{6} \left(
            \pmb{u}^T M D \pmb{h}
            + \pmb{h}^T M D \pmb{u}
        \right)
        \\
        &
        - \frac{\lambda}{6} \left(
            \Biggl(\frac{\pmb{\eta}^2}{\pmb{h}}\Biggr)^T M D \pmb{u}
            + \pmb{u}^T M D \frac{\pmb{\eta}^2}{\pmb{h}}
        \right)
        + \frac{1}{6} \left(
            \pmb{w}^T M D \pmb{h} \pmb{u} \pmb{w}
            + (\pmb{h} \pmb{u} \pmb{w})^T M D \pmb{w}
        \right).
    \end{aligned}
    \end{equation}
    All terms on the right-hand side are of the general form
    \begin{equation}
        \pmb{a}^T M D \pmb{b} + \pmb{b}^T M D \pmb{a},
    \end{equation}
and  cancel for periodic first-derivative SBP operators $D$,    due to \eqref{eq:SBP_periodic}.
    Hence,  the conservation of total energy.
\end{proof}

\begin{remark}
    The total momentum is not conserved in general by the
    semidiscretization \eqref{eq:SGN_hyperbolic_flat_prim_SBP} of the
    hyperbolic approximation of the Serre-Green-Naghdi equations.
    This is due to the split form of the non-hydrostatic pressure term ---
    the other terms conserve the total momentum since they are the same
    as for the classical shallow water equations, see
    Section~\ref{sec:split-forms-SWE} and \cite{ranocha2017shallow}.
    We have so far not been able to derive a split form of the
    non-hydrostatic pressure term that conserves the total momentum
    (using only integration by parts). For the same reason, there is no
    advantage in considering a split form  using conservative variables.
\end{remark}
%

\section{Original Serre-Green-Naghdi equation with flat bathymetry}
\label{sec:SGN_original_flat}

In this section, we use the results from the previous
Section~\ref{sec:SGN_hyperbolic_flat} to derive an energy-conserving
split form of the original Serre-Green-Naghdi equations with flat bathymetry.
Afterwards, we will present corresponding structure-preserving
semidiscretizations.

\subsection{Deriving a split form from the hyperbolic approximation}

As discussed in Section~\ref{sec:SGN_hyperbolic_flat} we are going to apply the results for the split form \eqref{eq:SGN_hyperbolic_flat_prim_split}
and the connections between the hyperbolic and classical model to devise a split form for the latter.
In practice we can start from  the first three in \eqref{eq:SGN_hyperbolic_flat_prim_split0},
we replace the $1-\eta/h$ terms using   the definition  of $\pi$  \eqref{eq:pi_def}, and replace $\eta$ by $h$ in the remaining ones.
The resulting split form reads
\eqref{eq:SGN_original_flat_prim}:
\begin{equation}
\begin{aligned}
    & h_t
        + h_x u
        + h u_x
        = 0,\\
    & h u_t
        + g (h^2)_x
        - g h h_x
        + \frac{1}{2} h (u^2)_x
        - \frac{1}{2} h_x u^2
        + \frac{1}{2} (h u)_x u
        - \frac{1}{2} h u u_x
        + (h \pi)_x
        = 0,\\
    & h w_t
        + \frac{1}{2} (h u w)_x
        + \frac{1}{2} h u w_x
        - \frac{1}{2} h_x u w
        - \frac{1}{2} h u_x w
        = 3 \pi,\\
    & w
        + h u_x
        =0.
\end{aligned}
\end{equation}
One can easily check that  the conservation of energy \eqref{eq:SGN_original_flat_energy}  can be obtained   only using summation by parts, upon integration with  the
dual variables  \eqref{eq:wstar_flat_noncons}, as discussed in
Section~\ref{sec:from_hyperbolic_to_classical}.\\

To clarify the structure of the elliptic operator, we compute the
part of the non-hydrostatic pressure $h \pi$ containing a time derivative
of $u$ and obtain
\begin{equation}
    h^2 w_t
    =
    -h^3 u_{tx} - h^2 h_t u_x
    =
    -h^3 u_{tx} + h^2 h_x u u_x + h^3 u_x^2.
\end{equation}
Moreover, we   have
\begin{equation}
    h (h u w)_x
    + h^2 u w_x
    - h^2_x u w
    - h^2 u_x w
    =
    - h (h^2 u u_x)_x
    - h^2 u (h u_x)_x
    + h^2 h_x u u_x
    + h^3 u_x^2,
\end{equation}
using which we can recast  the split form of the original Serre-Green-Naghdi equations as
\begin{equation}
\label{eq:SGN_original_flat_prim_split}
\begin{aligned}
    & h_t
        + h_x u
        + h u_x
        = 0,\\
    & h u_t - \frac{1}{3} (h^3 u_{tx})_x
        + g (h^2)_x
        - g h h_x
        + \frac{1}{2} h (u^2)_x
        - \frac{1}{2} h_x u^2
        + \frac{1}{2} (h u)_x u
        - \frac{1}{2} h u u_x
        + p_x
        = 0,\\
    & p
      = \frac{1}{2} h^3 (u_x)^2
        + \frac{1}{2} h^2 h_x u u_x
        - \frac{1}{6} h (h^2 u u_x)_x
        - \frac{1}{6} h^2 u (h u_x)_x.
\end{aligned}
\end{equation}

\begin{remark}
    The split form \eqref{eq:SGN_original_flat_prim_split} of the original
    Serre-Green-Naghdi equations conserves not only the total water mass and
    the energy but also the total momentum, since the
    non-hydrostatic pressure is written in a conservative form and the
    remaining terms are the same as in the split form
    \eqref{eq:SWE_flat_prim_split} of the classical shallow water equations.
\end{remark}

\subsection{More general split forms of the non-hydrostatic pressure}

As described in the previous subsection, the split form
\eqref{eq:SGN_hyperbolic_flat_prim_split} of the hyerbolic approximation
induces an energy-conservative split form of the original Serre-Green-Naghdi
equations \eqref{eq:SGN_original_flat_prim}. However, such energy-conserving
split form is not unique.  Indeed, we can derive other  splittings  of the
non-hydrostatic pressure. We  can start with
\begin{equation}
\begin{aligned}
    & h_t
        + h_x u
        + h u_x
        = 0,\\
    & h u_t - \frac{1}{3} (h^3 u_{tx})_x
        + g (h^2)_x
        - g h h_x
        + \frac{1}{2} h (u^2)_x
        - \frac{1}{2} h_x u^2
        + \frac{1}{2} (h u)_x u
        - \frac{1}{2} h u u_x
        + p_x
        = 0,\\
\end{aligned}
\end{equation}
where $p$ is a general non-hydrostatic pressure term consistent with
\begin{equation}
\begin{aligned}
    p
    &=
    \frac{1}{2} h^3 (u_x)^2
    + \frac{1}{2} h^2 h_x u u_x
    - \frac{1}{6} h (h^2 u u_x)_x
    - \frac{1}{6} h^2 u (h u_x)_x
    =
    \frac{1}{3} h^3 (u_x)^2
    - \frac{1}{3} h^3 u u_{xx}.
\end{aligned}
\end{equation}
Note again that the  energy \eqref{eq:SGN_original_flat_energy}
is the shallow water energy plus the  term $h (h u_x)^2 / 6$. Its
rate of change  is
\begin{equation}
    E_t
    =
    \left( g h + \frac{1}{2} u^2 h_t + \frac{1}{2} h^2 u_x^2 \right) h_t
    + u \, h u_t
    + \frac{1}{3} h^3 u_x u_{tx}.
\end{equation}
Thus, we multiply the first equation of the split form by
$( g h + u^2 / 2 + h^2 u_x^2 / 2)$,
the second equation by $u$, add them and integrate over the periodic domain
to obtain
\begin{equation}
\begin{aligned}
    0
    &=
    \int E_t \dif x
    + \left[ g h^2 u + \frac{1}{2} h u^3 - \frac{1}{3} h^3 u u_{tx} + p u \right]
    + \int \left(
        \frac{1}{2} h^3 u_x^3
        + \frac{1}{2} h^2 h_x u u_x^2
        - p u_x
    \right) \dif x,
\end{aligned}
\end{equation}
where the term $[\cdot]$ is a boundary flux term vanishing for
periodic domains. Thus, we want to choose $p$ such we can use
integration by parts on the last volume term to obtain the
missing boundary flux term, i.e.,
\begin{equation}
    \int \left(
        \frac{1}{2} h^3 u_x^3
        + \frac{1}{2} h^2 h_x u u_x^2
        - p u_x
    \right) \dif x
    =
    \left[
        \frac{1}{6} h^3 u u_x^2
    \right]
    =
    \left[
        \frac{1}{6} h (\dot h)^2 u
    \right],
\end{equation}
where the material derivative of the water height is
$\dot h = h_t + u h_x = -(u h_x + h u_x) + u h_x = - h u_x$.
This is satisfied by the general pressure term
\begin{equation}
    p
    =
    \frac{1}{2} h^3 u_x^2
    + \frac{1}{2} h^2 h_x u u_x
    - \frac{1}{6} \mathfrak{a} (\mathfrak{b} u_x)_x
    - \frac{1}{6} \mathfrak{b} (\mathfrak{a} u_x)_x,
\end{equation}
where
\begin{equation}
    \mathfrak{a} \mathfrak{b} = h^3 u.
\end{equation}
All of these choices yield a consistent pressure term since
\begin{equation}
\begin{aligned}
    p
    &=
    \frac{1}{2} h^3 u_x^2
    + \frac{1}{2} h^2 h_x u u_x
    - \frac{1}{6} \mathfrak{a} (\mathfrak{b} u_x)_x
    - \frac{1}{6} \mathfrak{b} (\mathfrak{a} u_x)_x
    =
    \frac{1}{2} h^3 u_x^2
    + \frac{1}{2} h^2 h_x u u_x
    - \frac{1}{3} \mathfrak{a} \mathfrak{b} u_{xx}
    - \frac{1}{6} \mathfrak{a} \mathfrak{b}_x u_x
    - \frac{1}{6} \mathfrak{a}_x \mathfrak{b} u_x
    \\
    &=
    \frac{1}{2} h^3 u_x^2
    + \frac{1}{2} h^2 h_x u u_x
    - \frac{1}{3} h^3 u u_{xx}
    - \frac{1}{6} (h^3 u)_x u_x
    \\
    &=
    \frac{1}{2} h^3 u_x^2
    + \frac{1}{2} h^2 h_x u u_x
    - \frac{1}{3} h^3 u u_{xx}
    - \frac{1}{2} h^2 h_x u u_x
    - \frac{1}{6} h^3 u_x^2
    =
    \frac{1}{3} h^3 u_x^2
    - \frac{1}{3} h^3 u u_{xx}.
\end{aligned}
\end{equation}

\begin{remark}
\label{rem:future_work_split_form_pressure}
Many choices are possible, here. We take any in $(\mathfrak{a}, \mathfrak{b})  =( h^mu^n, h^{3-m}u^{1-n})$, with  non-negative integers $ m \le 3,\,  n \le 1 $.
The split form \eqref{eq:SGN_original_flat_prim_split} uses  $(\mathfrak{a}, \mathfrak{b})=(h, h^2u)$. Preliminary tests do not show any significant differences when making other
choices.  Investigating this further is left for future work.
\end{remark}

\subsection{Spatial semidiscretizations}

Replacing  continuous   derivatives  in the split form
\eqref{eq:SGN_original_flat_prim_split} by periodic SBP operators
results in the semidiscretization
\begin{equation}
\label{eq:SGN_original_flat_prim_SBP}
\begin{split}
    & \partial_t \pmb{h}
        + \pmb{u} D \pmb{h}
        + \pmb{h} D \pmb{u}
        = \pmb{0},
    \\
    & \pmb{h} \partial_t \pmb{u}
        - \frac{1}{3} D \pmb{h}^3 D \partial_t \pmb{u}
        + g D \pmb{h}^2
        - g \pmb{h} D \pmb{h}
        + \frac{1}{2} \pmb{h} D \pmb{u}^2
        - \frac{1}{2} \pmb{u}^2 D \pmb{h}
        + \frac{1}{2} \pmb{u} D \pmb{h} \pmb{u}
        - \frac{1}{2} \pmb{h} \pmb{u} D \pmb{u}
        + D \pmb{p}
        = \pmb{0},
    \\
    & \pmb{p}
      = \frac{1}{2} \pmb{h}^3 (D \pmb{u})^2
        + \frac{1}{2} \pmb{h}^2 (D \pmb{h}) \pmb{u} D \pmb{u}
        - \frac{1}{6} \pmb{h} D \pmb{h}^2 \pmb{u} D \pmb{u}
        - \frac{1}{6} \pmb{h}^2 \pmb{u} D \pmb{h} D \pmb{u}.
\end{split}
\end{equation}
The discrete total energy for \eqref{eq:SGN_original_flat_prim_SBP}
is $\pmb{1}^T M \pmb{E}$ with
\begin{equation}
    \pmb{E}
    =
    \frac{1}{2} g \pmb{h}^2 + \frac{1}{2} \pmb{h} \pmb{u}^2 + \frac{1}{6} \pmb{h}^3 (D \pmb{u})^2.
\end{equation}

\begin{theorem}
\label{thm:SGN_original_flat_prim_SBP}
    Consider the semidiscretization \eqref{eq:SGN_original_flat_prim_SBP}
    of the original Serre-Green-Naghdi equations
    \eqref{eq:SGN_original_flat_cons}
    with a periodic first-derivative SBP operator $D$ with
    diagonal mass/norm matrix.
    \begin{enumerate}
            \item The total water mass $\pmb{1}^T M \pmb{h}$ is conserved.
        \item The total momentum $\pmb{1}^T M \pmb{h} \pmb{u}$ is conserved.
        \item The total energy $\pmb{1}^T M \pmb{E}$ is conserved.
    \end{enumerate}
\end{theorem}
\begin{proof}
    This is a specialization of the more general result
    Theorem~\ref{thm:SGN_original_flat_prim_SBP_upwind}
    below with $D = D_- = D_+$.
\end{proof}

If the semidiscretization \eqref{eq:SGN_original_flat_prim_SBP}
is discretized in time with an explicit method (like an explicit
Runge-Kutta method), it requires the solution of discretized
elliptic problems of the form
\begin{equation}
    \left( \pmb{h} - \frac{1}{3} D \pmb{h}^3 D \right) \partial_t \pmb{u} = \pmb{y},
\end{equation}
where
\begin{equation}
    \pmb{y}
    =
    - g D \pmb{h}^2
    + g \pmb{h} D \pmb{h}
    - \frac{1}{2} \pmb{h} D \pmb{u}^2
    + \frac{1}{2} \pmb{u}^2 D \pmb{h}
    - \frac{1}{2} \pmb{u} D \pmb{h} \pmb{u}
    + \frac{1}{2} \pmb{h} \pmb{u} D \pmb{u}
    - D \pmb{p}.
\end{equation}

\begin{lemma}
    If if the water height is positive, the discrete operator
    \begin{equation}
        \pmb{h} - \frac{1}{3} D \pmb{h}^3 D
    \end{equation}
    of \eqref{eq:SGN_original_flat_prim_SBP}
    is symmetric and positive definite with respect to the
    diagonal mass matrix $M$.
\end{lemma}
\begin{proof}
    We have
    \begin{equation}
        M \left( \pmb{h} - \frac{1}{3} D \pmb{h}^3 D \right)
        =
        \pmb{h} M + \frac{1}{3} D^T M \pmb{h}^3 D
    \end{equation}
    due to the SBP property \eqref{eq:SBP_periodic} for diagonal $M$.
\end{proof}

Thus, the discrete elliptic problem can be solved uniquely.
However, the second derivative is discretized using a wide-stencil
operator (with variable coefficients). This can lead to stability issues (for under-resolved)
grids and a loss of efficiency. Thus, it would be better to use narrow-stencil
second-derivative operators \cite{mattsson2012summation} or upwind SBP operators
\cite{mattsson2017diagonal}. Here, we choose the second option with upwind SBP
operators $D_\pm$ and their corresponding central operator $D = (D_+ + D_-) / 2$,
leading to the semidiscretization
\begin{equation}
\label{eq:SGN_original_flat_prim_SBP_upwind}
\begin{aligned}
     &\partial_t \pmb{h}
        + \pmb{u} D \pmb{h}
        + \pmb{h} D \pmb{u}
        = \pmb{0},
    \\
    &\pmb{h} \partial_t \pmb{u}
        - \frac{1}{3} D_+ \pmb{h}^3 D_- \partial_t \pmb{u}
        + g D \pmb{h}^2
        - g \pmb{h} D \pmb{h}
        + \frac{1}{2} \pmb{h} D \pmb{u}^2
        - \frac{1}{2} \pmb{u}^2 D \pmb{h}
        + \frac{1}{2} \pmb{u} D \pmb{h} \pmb{u}
        - \frac{1}{2} \pmb{h} \pmb{u} D \pmb{u}
        + D_+ \pmb{p}_+
        + D \pmb{p}_0
        = \pmb{0},
    \\
    & \pmb{p}_+
      = \frac{1}{2} \pmb{h}^3 (D \pmb{u}) D_- \pmb{u}
        + \frac{1}{2} \pmb{h}^2 (D \pmb{h}) \pmb{u} D_- \pmb{u},
    \\
    & \pmb{p}_0
       = - \frac{1}{6} \pmb{h} D \pmb{h}^2 \pmb{u} D \pmb{u}
       - \frac{1}{6} \pmb{h}^2 \pmb{u} D \pmb{h} D \pmb{u}.
\end{aligned}
\end{equation}

\begin{lemma}
  If if the water height is positive, the discrete operator
    \begin{equation}
        \pmb{h} - \frac{1}{3} D_+ \pmb{h}^3 D_-
    \end{equation}
    of \eqref{eq:SGN_original_flat_prim_SBP_upwind}
    is symmetric and positive definite with respect to the
    diagonal mass matrix $M$.
\end{lemma}
\begin{proof}
    We have
    \begin{equation}
        M \left( \pmb{h} - \frac{1}{3} D_+ \pmb{h}^3 D_- \right)
        =
        \pmb{h} M + \frac{1}{3} D_-^T M \pmb{h}^3 D_-
    \end{equation}
    due to the upwind SBP property \eqref{eq:SBP_upwind_periodic}
    for diagonal $M$.
\end{proof}
The corresponding discrete total energy for
\eqref{eq:SGN_original_flat_prim_SBP_upwind} is
\begin{equation}
    \frac{1}{2} g \pmb{1}^T M \pmb{h}^2
    + \frac{1}{2} \pmb{1}^T M \pmb{h} \pmb{u}^2
    + \frac{1}{6} \pmb{u}^T D_-^T M \pmb{h}^3 D_- \pmb{u}
    =
    \pmb{1}^T M \left(
        \frac{1}{2} g \pmb{h}^2
        + \frac{1}{2} \pmb{h} \pmb{u}^2
        + \frac{1}{6} \pmb{h}^3 (D_- \pmb{u})^2
    \right)
    =
    \pmb{1}^T M \pmb{E}.
\end{equation}

There could be more options to use upwind SBP operators in
\eqref{eq:SGN_original_flat_prim_SBP_upwind}. Here, we have chosen the
simplest version avoiding the wide-stencil second-derivative
operator for the elliptic problem.

\begin{theorem}
\label{thm:SGN_original_flat_prim_SBP_upwind}
    Consider the semidiscretization \eqref{eq:SGN_original_flat_prim_SBP_upwind}
    of the original Serre-Green-Naghdi equations
    \eqref{eq:SGN_original_flat_cons}
    with periodic first-derivative upwind SBP operators $D_\pm$
    inducing the central operator $D = (D_+ + D_-) / 2$ with
    diagonal mass/norm matrix.
    \begin{enumerate}
     \item The total water mass $\pmb{1}^T M \pmb{h}$ is conserved.
        \item The total momentum $\pmb{1}^T M \pmb{h} \pmb{u}$ is conserved.
        \item The total energy $\pmb{1}^T M \pmb{E}$ is conserved.
    \end{enumerate}
\end{theorem}
\begin{proof}
    Conservation of the total water mass follows from the first equation
    of \eqref{eq:SGN_original_flat_prim_SBP_upwind} as in the proof of
    Theorem~\ref{thm:SGN_hyperbolic_flat_prim_SBP}.
    Conservation of the total meomentum follows since the split form of
    the shallow water part is the same as in Section~\ref{sec:split-forms-SWE}
    and the non-hydrostatic pressure term is discretized in a conservative
    way, i.e.,
    \begin{equation}
        \pmb{1}^T M D_+ \pmb{p}_+ + \pmb{1}^T M D \pmb{p}_0
        =
        - \pmb{1}^T D_-^T M \pmb{p}_+ - \pmb{1}^T D^T M \pmb{p}_0
        =
        0
    \end{equation}
    for consistent derivative operators $D_\pm$ and $D$.
Concerning energy, its rate of change  is
    \begin{equation}
    \begin{aligned}
        &\quad
        \partial_t \left(
            \frac{1}{2} g \pmb{1}^T M \pmb{h}^2
            + \frac{1}{2} \pmb{1}^T M \pmb{h} \pmb{u}^2
            + \frac{1}{6} \pmb{u}^T D_-^T M \pmb{h}^3 D_- \pmb{u}
        \right)
        \\
        &=
        g \pmb{h}^T M \partial_t \pmb{h}
        + \frac{1}{2} (\pmb{u}^2)^T M \partial_t \pmb{h}
        + \pmb{u}^T M \pmb{h} \partial_t \pmb{u}
        + \frac{1}{2} \left( \pmb{h}^2 (D_- \pmb{u})^2 \right)^T M \partial_t \pmb{h}
        + \frac{1}{3} \pmb{u}^T D_-^T M \pmb{h}^3 D_- \partial_t \pmb{u}.
    \end{aligned}
    \end{equation}
To evaluate it  we thus  multiply the first  of \eqref{eq:SGN_original_flat_prim_SBP_upwind}
    by $g \pmb{h}^T M + (1/2) (\pmb{u}^2)^T M + (1/2) \bigl( \pmb{h}^2 (D_- \pmb{u})^2 \bigr)^T M$,
     the second  by $\pmb{u}^T M$, and add the two equations to get
    \begin{equation}
    \begin{aligned}
        &\quad
        g \pmb{h}^T M \partial_t \pmb{h}
        + \frac{1}{2} (\pmb{u}^2)^T M \partial_t \pmb{h}
        + \frac{1}{2} \left( \pmb{h}^2 (D_- \pmb{u})^2 \right)^T M \partial_t \pmb{h}
        + \pmb{u}^T M \pmb{h} \partial_t \pmb{u}
        + \frac{1}{3} \pmb{u}^T D_-^T M \pmb{h}^3 D_- \partial_t \pmb{u}
        \\
        &=
        - g \pmb{h}^T M \pmb{u} D \pmb{h}
        - g \pmb{h}^T M \pmb{h} D \pmb{u}
        - \frac{1}{2} (\pmb{u}^2)^T M \pmb{u} D \pmb{h}
        - \frac{1}{2} (\pmb{u}^2)^T M \pmb{h} D \pmb{u}
        \\
        &\quad
        - \frac{1}{2} \left( \pmb{h}^2 (D_- \pmb{u})^2 \right)^T M \pmb{u} D \pmb{h}
        - \frac{1}{2} \left( \pmb{h}^2 (D_- \pmb{u})^2 \right)^T M \pmb{h} D \pmb{u}
        - \pmb{u}^T M g D \pmb{h}^2
        + \pmb{u}^T M g \pmb{h} D \pmb{h}
        \\
        &\quad
        - \frac{1}{2} (\pmb{h} \pmb{u})^T M D \pmb{u}^2
        + \frac{1}{2} (\pmb{u}^3)^T M D \pmb{h}
        - \frac{1}{2} (\pmb{u}^2)^T M D \pmb{h} \pmb{u}
        + \frac{1}{2} (\pmb{h} \pmb{u}^2)^T M D \pmb{u}
        - \pmb{u}^T M D_+ \pmb{p}_+
        - \pmb{u}^T M D \pmb{p}_0
        \\
        &=
        - g (\pmb{h} \pmb{u})^T M D \pmb{h}
        - g (\pmb{h}^2)^T M D \pmb{u}
        - \frac{1}{2} (\pmb{u}^3)^T M D \pmb{h}
        - \frac{1}{2} (\pmb{h} \pmb{u}^2)^T M D \pmb{u}
        \\
        &\quad
        - \frac{1}{2} \left( \pmb{h}^2 \pmb{u} (D_- \pmb{u})^2 \right)^T M D \pmb{h}
        - \frac{1}{2} \left( \pmb{h}^3 (D_- \pmb{u})^2 \right)^T M D \pmb{u}
        - g \pmb{u}^T M D \pmb{h}^2
        + g (\pmb{h} \pmb{u})^T M D \pmb{h}
        \\
        &\quad
        - \frac{1}{2} (\pmb{h} \pmb{u})^T M D \pmb{u}^2
        + \frac{1}{2} (\pmb{u}^3)^T M D \pmb{h}
        - \frac{1}{2} (\pmb{u}^2)^T M D \pmb{h} \pmb{u}
        + \frac{1}{2} (\pmb{h} \pmb{u}^2)^T M D \pmb{u}
        - \pmb{u}^T M D_+ \pmb{p}_+
        - \pmb{u}^T M D \pmb{p}_0
        \\
        &=
        - \frac{1}{2} \left( \pmb{h}^2 \pmb{u} (D_- \pmb{u})^2 \right)^T M D \pmb{h}
        - \frac{1}{2} \left( \pmb{h}^3 (D_- \pmb{u})^2 \right)^T M D \pmb{u}
        + \pmb{p}_+^T M D_- \pmb{u}
        + \pmb{p}_0^T M D \pmb{u}.
    \end{aligned}
    \end{equation}
    Inserting the pressure terms $\pmb{p}_+$ and $\pmb{p}_0$ yields
    \begin{equation*}
    \begin{aligned}
\partial_t(     \pmb{1}^T M \pmb{E}) = &
        - \frac{1}{2} \left( \pmb{h}^2 \pmb{u} (D_- \pmb{u})^2 \right)^T M D \pmb{h}
        - \frac{1}{2} \left( \pmb{h}^3 (D_- \pmb{u})^2 \right)^T M D \pmb{u}
        + \frac{1}{2} \left( \pmb{h}^3 (D \pmb{u}) D_- \pmb{u} \right)^T M D_- \pmb{u}
                \\
        &
        + \frac{1}{2} \left( \pmb{h}^2 (D \pmb{h}) \pmb{u} D_- \pmb{u} \right)^T M D_- \pmb{u}
        - \frac{1}{6} \left( \pmb{h} D \pmb{h}^2 \pmb{u} D \pmb{u} \right)^T M D \pmb{u}
        - \frac{1}{6} \left( \pmb{h}^2 \pmb{u} D \pmb{h} D \pmb{u} \right)^T M D \pmb{u}
        \\
        &=
        - \frac{1}{6} \left( \pmb{h}^2 \pmb{u} D \pmb{u} \right)^T D^T M \left( \pmb{h} D \pmb{u} \right)
        - \frac{1}{6} \left( \pmb{h} D \pmb{u} \right)^T D^T M \left( \pmb{h}^2 \pmb{u} D \pmb{u} \right) =0,
    \end{aligned}
    \end{equation*}
    since  after using the SBP property \eqref{eq:SBP_periodic} the first and the fourth as well as the second and the third terms
    in the first right hand side cancel for a diagonal    mass matrix $M$, and similarly  the
    last two terms      in the last step.
\end{proof}

There are even more possibilities to discretize the non-hydrostatic
pressure terms with upwind operators. Consider for example the
central pressure term
\begin{equation}
    D \pmb{p}_0
    =
    - \frac{1}{6} D \pmb{h} D \pmb{h}^2 \pmb{u} D \pmb{u}
    - \frac{1}{6} D \pmb{h}^2 \pmb{u} D \pmb{h} D \pmb{u}
\end{equation}
of \eqref{eq:SGN_original_flat_prim_SBP_upwind}. The energy contribution
of this terms vanishes since
\begin{equation*}
\begin{aligned}
    -6 \pmb{u}^T M D \pmb{p}_0
    &=
    \pmb{u}^T M D \pmb{h} D \pmb{h}^2 \pmb{u} D \pmb{u}
    + \pmb{u}^T M D \pmb{h}^2 \pmb{u} D \pmb{h} D \pmb{u}
    =
    (\pmb{h} D \pmb{u})^T M D (\pmb{h}^2 \pmb{u} D \pmb{u})
    + (\pmb{h}^2 \pmb{u} D \pmb{u})^T M D (\pmb{h} D \pmb{u})
    =
    0.
\end{aligned}
\end{equation*}
Let us now consider a more general form
\begin{equation}
    D_\alpha \pmb{h} D_\beta \pmb{h}^2 \pmb{u} D_\gamma \pmb{u}
    + D_\delta \pmb{h}^2 \pmb{u} D_\epsilon \pmb{h} D_\zeta \pmb{u},
\end{equation}
where $\alpha, \beta, \gamma, \delta, \epsilon, \zeta \in \{+, -, 0\}$,
and $D_0 = D$ is the central operator. The energy contribution of this
term is
\begin{equation}
\begin{aligned}
    \pmb{u}^T M D_\alpha \pmb{h} D_\beta \pmb{h}^2 \pmb{u} D_\gamma \pmb{u}
    + \pmb{u}^T M D_\delta \pmb{h}^2 \pmb{u} D_\epsilon \pmb{h} D_\zeta \pmb{u}
    =
    (\pmb{h} D_{\alpha^*} \pmb{u})^T M D_\beta (\pmb{h}^2 \pmb{u} D_\gamma \pmb{u})
    + (\pmb{h}^2 \pmb{u} D_{\delta^*} \pmb{u})^T M D_\epsilon (\pmb{h} D_\zeta \pmb{u}),
\end{aligned}
\end{equation}
where $\pm^* = \mp$ and $0^* = 0$. This term vanishes if
\begin{equation}
    \zeta = \alpha^*, \quad \epsilon = \beta^*, \quad \gamma = \delta^*.
\end{equation}
Even if we do not want to introduce a clear directional bias by choosing
the same number of $+$ and $-$ signs in $(\alpha, \beta, \gamma)$, there
are many possibilities different from $\alpha = \beta = \gamma = 0$
as in \eqref{eq:SGN_original_flat_prim_SBP_upwind}, e.g.,
$\alpha = +$, $\beta = 0$, and $\gamma = -$.\\

    Preliminary tests do not show any significant differences between
    the different choices of applying (upwind) derivative operators
    to the non-hydrostatic pressure term. Investigating this further
    is left for future work.

\section{Hyperbolic approximation with variable bathymetry}
\label{sec:SGN_hyperbolic_variable}

Following the notation of \cite{busto21}, we introduce the bottom topography
$b = b(x)$ and write the generalization of the hyperbolic system
\eqref{eq:SGN_hyperbolic_flat_prim} as
\begin{equation}
\label{eq:SGN_hyperbolic_variable_prim}
\begin{split}
    & h_t
        + (h u)_x = 0,\\
    & h u_t
        + g h h_x
        + h u u_x
        + (h \Pi)_x
        + \left( g h + \dfrac{3}{2} \dfrac{h}{\eta} \Pi \right) b_x
        = 0, \\
    & h w_t
        + h u w_x
        = \lambda (1 - \eta / h),\\
    & \eta_t
        + \eta_x u
        + \dfrac{3}{2} b_x u
        = w,\\
    & \Pi = \dfrac{\lambda}{3} \dfrac{\eta}{h} \left(1 - \dfrac{\eta}{h}\right).
\end{split}
\end{equation}
This system satisfies the energy conservation law
\begin{multline}
\label{eq:SGN_hyperbolic_variable_energy}
    \biggl( \overbrace{ \frac{1}{2} g (h + b)^2 + \frac{1}{2} h u^2 + \frac{1}{6} h w^2 + \frac{\lambda}{6} h (1 - \eta / h)^2 }^{= E} \biggr)_t
    \\
    + \biggl( \underbrace{ g h (h + b) u +\frac{1}{2} h u^3 + \frac{1}{6} h u w^2 + \frac{\lambda}{6} h (1 - \eta / h)^2 u + h \Pi u }_{= F} \biggr)_x = 0.
\end{multline}
We will derive an energy-conservative split form of this generalization of
\eqref{eq:SGN_hyperbolic_flat_prim} in the following.

\subsection{Energy equation}
\label{sec:SGN_hyperbolic_variable_energy}

The energy conservation law \eqref{eq:SGN_hyperbolic_variable_energy} can
be derived by multiplying the time derivatives of the primitive variables
$\mathbf{V} = (h, u, w, \eta)^T$ by
\begin{equation}
    \mathbf{V}^*
    =
    \frac{\partial E}{\partial \mathbf{V}}
    =
    \begin{pmatrix}
        g (h + b) + \frac{1}{2} u^2 + \frac{1}{6} w^2 + \frac{\lambda}{6} (1 - \eta^2 / h^2) \\
        h u \\
        h w \\
        - \frac{\lambda}{3} (1 - \eta / h)
    \end{pmatrix}.
\end{equation}
Indeed, the computation is very similar to the one from
Section~\ref{sec:SGN_hyperbolic_flat_energy}; the additional terms are
\begin{equation}
\begin{aligned}
    &\quad
    g (h u)_x b
    + \left( g h + \dfrac{3}{2} \dfrac{h}{\eta} \Pi \right) b_x u
    - \frac{\lambda}{3} (1 - \eta / h) \frac{3}{2} b_x u
    \\
    &=
    \left( g h b u \right)_x
    + \frac{\lambda}{2} (1 - \eta / h) b_x u
    - \frac{\lambda}{2} (1 - \eta / h) b_x u
    =
    \left( g h b u \right)_x.
\end{aligned}
\end{equation}

\subsection{Split form}

The analysis shows that the bathymetry terms cancel identically
when evaluating the energy balance. Hence, the split form
\eqref{eq:SGN_hyperbolic_flat_prim_split} generalizes readily
to this case if the bathymetry is included appropriately in the
hydrostatic pressure term.
In particular, the non-conservative split form becomes
\begin{equation}
\label{eq:SGN_hyperbolic_variable_prim_split}
\begin{split}
    & h_t
        + h_x u + h u_x
        = 0,\\
    & h u_t
        + g \bigl( h (h + b) \bigr)_x - g (h + b) h_x
        + \frac{1}{2} h (u^2)_x - \frac{1}{2} h_x u^2
        + \frac{1}{2} (h u)_x u - \frac{1}{2} h u u_x
        \\
        &\qquad
        + \frac{\lambda}{6} \frac{\eta^2}{h^2} h_x
        + \frac{\lambda}{3} \eta_x
        - \frac{\lambda}{3} \frac{\eta}{h} \eta_x
        - \frac{\lambda}{6} \Bigl( \frac{\eta^2}{h} \Bigr)_x
        + \frac{\lambda}{2} \left( 1 - \frac{\eta}{h} \right) b_x
        = 0, \\
    & h w_t
        + \frac{1}{2} (h u w)_x
        + \frac{1}{2} h u w_x
        - \frac{1}{2} h_x u w
        - \frac{1}{2} h u_x w
        = \lambda - \lambda \frac{\eta}{h},\\
    & \eta_t
        + \eta_x u
        + \frac{3}{2} b_x u
        = w.
\end{split}
\end{equation}

\subsection{Semidiscretization}

The split form \eqref{eq:SGN_hyperbolic_variable_prim_split}
induces the semidiscretization
\begin{equation}
\label{eq:SGN_hyperbolic_variable_prim_SBP}
\begin{split}
    & \partial_t \pmb{h} + \pmb{u} D \pmb{h} + \pmb{h} D \pmb{u} = \pmb{0},\\
    & \pmb{h} \partial_t \pmb{u}
        + g D \pmb{h} (\pmb{h} + \pmb{b})
        - g (\pmb{h} + \pmb{b}) D \pmb{h}
        - \frac{1}{2} \pmb{h} D \pmb{u}^2
        + \frac{1}{2} \pmb{u}^2 D \pmb{h}
        - \frac{1}{2} \pmb{u} D \pmb{h} \pmb{u}
        + \frac{1}{2} \pmb{h} \pmb{u} D \pmb{u}
        \\
        &\qquad
        + \frac{\lambda}{6} \frac{\pmb{\eta}^2}{\pmb{h}^2} D \pmb{h}
        + \frac{\lambda}{3} D \pmb{\eta}
        - \frac{\lambda}{3} \frac{\pmb{\eta}}{\pmb{h}} D \pmb{\eta}
        - \frac{\lambda}{6} D \frac{\pmb{\eta}^2}{\pmb{h}}
        + \frac{\lambda}{2} D \pmb{b}
        - \frac{\lambda}{2} \frac{\pmb{\eta}}{\pmb{h}} D \pmb{b} = \pmb{0}, \\
    & \pmb{h} \partial_t \pmb{w}
        + \frac{1}{2} D \pmb{h} \pmb{u} \pmb{w}
        + \frac{1}{2} \pmb{h} \pmb{u} D \pmb{w}
        - \frac{1}{2} \pmb{u} \pmb{w} D \pmb{h}
        - \frac{1}{2} \pmb{h} \pmb{w} D \pmb{u} = \lambda - \lambda \frac{\pmb{\eta}}{\pmb{h}},\\
    & \partial_t \pmb{\eta}
        + \pmb{u} D \pmb{\eta}
         + \frac{3}{2} \pmb{u} D \pmb{b} = \pmb{w}.
\end{split}
\end{equation}
The corresponding discrete total energy is $\pmb{1}^T M \pmb{E}$, where
\begin{equation}
\begin{aligned}
    \pmb{E}
    &=
    \frac{1}{2} g (\pmb{h} + \pmb{b})^2
    + \frac{1}{2} \pmb{h} \pmb{u}^2
    + \frac{1}{6} \pmb{h} \pmb{w}^2
    + \frac{\lambda}{6} \pmb{h}
    - \frac{\lambda}{3} \pmb{\eta}
    + \frac{\lambda}{6} \frac{\pmb{\eta}^2}{\pmb{h}}.
\end{aligned}
\end{equation}

\begin{theorem}
\label{thm:SGN_hyperbolic_variable_prim_SBP}
    Consider the semidiscretization \eqref{eq:SGN_hyperbolic_variable_prim_SBP}
    of the hyperbolic approximation of the Serre-Green-Naghdi
    equations with varying bathymetry \eqref{eq:SGN_hyperbolic_variable_prim}
    using a periodic first-derivative SBP operator $D$ with
    diagonal mass/norm matrix. The following properties are true.
    \begin{enumerate}
        \item The total water mass $\pmb{1}^T M \pmb{h}$ is conserved.
                \item The total energy $\pmb{1}^T M \pmb{E}$ is conserved.
        \item The semidiscretization is well-balanced w.r.t.\
              the steady state $h + b \equiv \mathrm{const}$, $u \equiv 0$,
              $w \equiv 0$, $\eta \equiv h$.
    \end{enumerate}
\end{theorem}
\begin{proof}
    Conservation of the total water mass follows as in the proof of
    Theorem~\ref{thm:SGN_hyperbolic_flat_prim_SBP}.
Concerning energy, we compute
    \begin{equation}
    \begin{aligned}
        \partial_{\pmb{h}} \pmb{E}
        &=
        g (\pmb{h} + \pmb{b})
        + \frac{1}{2} \pmb{u}^2
        + \frac{1}{6} \pmb{w}^2
        + \frac{\lambda}{6}
        - \frac{\lambda}{6} \frac{\pmb{\eta}^2}{\pmb{h}^2}.
    \end{aligned}
    \end{equation}
    The other derivatives are the same as in the proof of
    Theorem~\ref{thm:SGN_hyperbolic_flat_prim_SBP}. Thus, the differences in
     the semidiscrete rate of change of the total energy
    compared to the case of constant bathymetry in
    Theorem~\ref{thm:SGN_hyperbolic_flat_prim_SBP} are as follows:
    \begin{itemize}
        \item additional term $g \pmb{b}^T M$ multiplied to $\partial_t \pmb{h}$
        \item additional term $g D \pmb{h} \pmb{b} - g \pmb{b} D \pmb{h} + \frac{\lambda}{2} D \pmb{b} - \frac{\lambda}{2} \frac{\pmb{\eta}}{\pmb{h}} D \pmb{b}$ added to $\pmb{h} \partial_t \pmb{u}$
        \item additional term $\frac{3}{2} \pmb{u} D \pmb{b}$ added to $\partial_t \pmb{\eta}$
    \end{itemize}
    These additional terms assembled   give
    \begin{equation*}
    \begin{aligned}
        &\quad
        g \pmb{b}^T M \pmb{u} D \pmb{h}
        + g \pmb{b}^T M \pmb{h} D \pmb{u}
        + g \pmb{u}^T M D \pmb{h} \pmb{b}
        - g \pmb{u}^T M \pmb{b} D \pmb{h}
        \\
        &\quad
        + \frac{\lambda}{2} \pmb{u}^T M D \pmb{b}
        - \frac{\lambda}{2} \pmb{u}^T M \frac{\pmb{\eta}}{\pmb{h}} D \pmb{b}
        - \frac{\lambda}{2} \pmb{1}^T M \pmb{u} D \pmb{b}
        + \frac{\lambda}{2} \left(\frac{\pmb{\eta}}{\pmb{h}}\right)^T M \pmb{u} D \pmb{b}
        \\
        &=
        g (\pmb{b} \pmb{u})^T M D \pmb{h}
        + g (\pmb{h} \pmb{b})^T M D \pmb{u}
        + g \pmb{u}^T M D \pmb{h} \pmb{b}
        - g (\pmb{b} \pmb{u})^T M D \pmb{h}
        \\
       &\quad
        + \frac{\lambda}{2} \pmb{u}^T M D \pmb{b}
        - \frac{\lambda}{2} \left(\frac{\pmb{\eta}}{\pmb{h}} \pmb{u}\right)^T M D \pmb{b}
        - \frac{\lambda}{2} \pmb{u}^T M D \pmb{b}
        + \frac{\lambda}{2} \left(\frac{\pmb{\eta}}{\pmb{h}} \pmb{u}\right)^T M D \pmb{b}=0,
    \end{aligned}
    \end{equation*}
    where we have used that the mass matrix $M$ is diagonal in the first
    step and applied the SBP property \eqref{eq:SBP_periodic}
    in the second step. This plus the elements in the proof of     Theorem~\ref{thm:SGN_hyperbolic_flat_prim_SBP}
    show   that   total energy is conserved.

    Finally, we observe that the method is well-balanced since  for
$\pmb{u} = \pmb{0} = \pmb{w}$, we have both
    $\partial_t \pmb{h} = \pmb{0}$ and  $\partial_t \pmb{\eta} = \pmb{0}$.
Moreover since also     $\pmb{\eta} = \pmb{h}$, we have  $\partial_t \pmb{w} = \pmb{0}$.
    Finally,
    \begin{equation}
    \begin{aligned}
        -\pmb{h} \partial_t \pmb{u}
        &=
        \underbrace{g D \pmb{h} (\pmb{h} + \pmb{b})
        - g (\pmb{h} + \pmb{b}) D \pmb{h}}
        + \underbrace{\frac{\lambda}{6} \frac{\pmb{\eta}^2}{\pmb{h}^2} D \pmb{h}
        - \frac{\lambda}{6} D \frac{\pmb{\eta}^2}{\pmb{h}}}
        + \underbrace{\frac{\lambda}{3} D \pmb{\eta}
        - \frac{\lambda}{3} \frac{\pmb{\eta}}{\pmb{h}} D \pmb{\eta}}
        + \underbrace{\frac{\lambda}{2} D \pmb{b}
        - \frac{\lambda}{2} \frac{\pmb{\eta}}{\pmb{h}} D \pmb{b}}.
    \end{aligned}
    \end{equation}
    The last three terms grouped in braces cancel due to $\pmb{\eta} = \pmb{h}$.
    The first due to  the fact that $\pmb{h} + \pmb{b}$ is constant.
    \end{proof}

\section{Original Serre-Green-Naghdi equations with variable bathymetry:
         mild-slope approximation}
\label{sec:SGN_original_mild}

We proceed as before and note that we can pass from the hyperbolic formulation
\eqref{eq:SGN_hyperbolic_variable_prim}
to the standard one using
$w = \dot \eta + \dfrac{3}{2} \dot b$ and
$\Pi = \dfrac{\eta}{h} \dfrac{1}{3} h \dot w$,
where $\dot a = a_t + u a_x$ is again the material derivative.
Then, we can easily recast the momentum in \eqref{eq:SGN_hyperbolic_variable_prim} as
\begin{equation}
\begin{aligned}
    h u_t
    + g h h_x
    + h u u_x
    + \left(\frac{h \eta}{3} \dot w \right)_x
    + \left( g h + \frac{1}{2} h \dot w \right) b_x=0.
\end{aligned}
\end{equation}
Taking now the relaxed limit $\eta = h$, and simplifying the last in  \eqref{eq:SGN_hyperbolic_variable_prim}
using the mass equation we obtain
\begin{equation}
\label{eq:SGN_original_mild_prim}
\begin{split}
    & h_t
        + (h u)_x = 0,\\
    & h u_t
        + g h h_x + h u u_x + (h \pi)_x
        + \left( g h + \dfrac{3}{2} \pi \right) b_x = 0,\\
    & h w_t
        + h u w_x
        = 3 \pi,\\
    & w + h u_x - \dfrac{3}{2} b_x u = 0,
\end{split}
\end{equation}
where as before the last two equations are merely the definitions of
$\pi$ and $w$.
Note that the above system does not match exactly
the SGN equations reported in, e.g., \cite{guermond2022hyperbolic,GS24}.
The form obtained is essentially a mild-slope version of the SGN model
in which a quadratic term $b_x^2$ has been removed,
following to the classical mild-slope approximation
\cite{doi:10.1098/rsta.1998.0309}.
This approximation has been made in \cite{busto21} to simplify the Lagrangian
used to derive the hyperbolic formulation. We can nevertheless show that
the system obtained admits an energy conservation law
\begin{equation}
\label{eq:SGN_original_mild_energy}
    \biggl( \underbrace{\frac{1}{2} g (h + b)^2 + \frac{1}{2} h u^2 + \frac{1}{6} h w^2}_{= E} \biggr)_t
    + \biggl( \underbrace{g h (h + b) u + \frac{1}{2} h u^3 + \frac{1}{6} h u w^2 + h \pi u}_{= F} \biggr)_x = 0.
\end{equation}
Expanding the energy terms, we get
\begin{equation}
\begin{aligned}
    E
    =
    \frac{1}{2} g (h + b)^2
    + \frac{1}{2} h u^2
    + \frac{1}{6} h \left( - h u_x + \frac{3}{2} b_x u \right)^2
=    \frac{1}{2} g (h + b)^2
    + \frac{1}{2} h u^2
    + \frac{1}{6} h^3 u_x^2
    - \frac{1}{2} h^2 b_x u u_x
    + \frac{3}{8} h b_x^2 u^2.
\end{aligned}
\end{equation}
We can check easily that a $h \dot b^2 / 8 = h (u b_x)^2 / 8$ term is missing
 compared, e.g., to \cite{guermond2022hyperbolic,GS24}. This
 is precisely the quadratic term neglected in \cite{busto21}.
We can check that the following relations hold for the differentials:
\begin{equation}
\label{eq:dual_mild}
\begin{split}
    \dif E
    =&
    \left( g (h + b) + \frac{1}{2} u^2 + \frac{1}{6} w^2 \right) \dif h
    + hu \, \dif u
    + h \dfrac{w}{3} \dif w,
    \\
    \dif F
    =&
    \left( g (h + b) + \frac{1}{2} u^2 + \frac{1}{6} w^2 \right) \dif\, (hu)
    + u \left(hu\, \dif u + gh\, \dif h + \dif (h\pi) + (g h + \dfrac{3}{2}  \pi ) \dif b\right)
    \\ &\quad +\dfrac{w}{3}( h u \dif w -  3 \pi )
    + \pi ( w +   h \, \dif u -   \dfrac{3}{2}u \,\dif b).
\end{split}
\end{equation}

\subsection{Deriving a split form from the hyperbolic approximation}

With the same approach as in Section~\ref{sec:SGN_original_flat},
we find that a split form induced by
\eqref{eq:SGN_hyperbolic_variable_prim_split}
allowing to recover the energy balance using only integration by parts is
\begin{equation}
\begin{aligned}
    & h_t
        + h_x u
        + h u_x
        = 0,\\
    & h u_t
        + g(h (h + b))_x
        - g (h + b) h_x
        + \frac{1}{2} h (u^2)_x
        - \frac{1}{2} h_x u^2
        + \frac{1}{2} (h u)_x u
        - \frac{1}{2} h u u_x
        + (h \pi)_x
        + \dfrac{3}{2} \pi b_x
        = 0,\\
    & h w_t
        + \dfrac{1}{2} (h u w)_x
        + \dfrac{1}{2} h u w_x
        - \dfrac{1}{2} h_x u w
        - \dfrac{1}{2} h u_x w
        = 3 \pi,\\
    & w + h u_x - \dfrac{3}{2} b_x u = 0.
\end{aligned}
\end{equation}

One can easily check (details omitted for brevity) that the SBP property holds now by multiplying the above system
by the transpose of the dual variable  $(\mathbf{W}^*)^T$ where from \eqref{eq:dual_mild} we have
$$
(\mathbf{W}^*)^T :=
\left( g (h + b) +\dfrac{1}{2} u^2+ \dfrac{1}{6} w^2,\; u,\;  w / 3,\; \pi \right).
$$

To clarify the structure of the elliptic operator that we need to invert to
obtain $u_t$, we compute the part of the non-hydrostatic pressure $h \pi$
containing a time derivative of $u$ and obtain
\begin{equation}
\begin{aligned}
    h^2 w_t
    &=
    h^2 \left( -h u_x + \frac{3}{2} b_x u \right)_t
    =
    - h^3 u_{tx}
    + \frac{3}{2} h^2 b_x u_t
    + h^2 h_x u u_x
    + h^3 u_x^2.
\end{aligned}
\end{equation}
Moreover, as in the flat bathymetry case we exploit the relation
\begin{equation}
\begin{aligned}
    h (h u w)_x
    + h^2 u w_x
    - h h_x u w
    - h^2 u_x w
=&
    - h (h^2 u u_x)_x
    - h^2 u (h u_x)_x
    + h^2 h_x u u_x
    + h^3 u_x^2
    \\
    &
    + \frac{3}{2} h (h b_x u^2)_x
    + \frac{3}{2} h^2 u (b_x u)_x
    - \frac{3}{2} h h_x b_x u^2
    - \frac{3}{2} h^2 b_x u u_x,
\end{aligned}
\end{equation}
obtained with some manipulations of the definition of $w$. This leads to
the split form
\begin{equation}
\label{eq:SGN_original_mild_prim_split}
\begin{aligned}
    & h_t
        + h_x u
        + h u_x
        = 0,\\
    & h u_t
        - \frac{1}{3} (h^3 u_{tx})_x
        + \frac{1}{2} (h^2 b_x u_t)_x
        - \frac{1}{2} h^2 b_x u_{tx}
        + \frac{3}{4} h b_x^2 u_t
        + g(h (h + b))_x
        - g (h + b) h_x
        \\
        &\qquad
        + \frac{1}{2} h (u^2)_x
        - \frac{1}{2} h_x u^2
        + \frac{1}{2} (h u)_x u
        - \frac{1}{2} h u u_x
        + p_x
        + \frac{3}{2} \frac{p}{h} b_x
        = 0,\\
    & p = \frac{1}{2} h^3 u_x^2
        + \frac{1}{2} h^2 h_x u u_x
        - \frac{1}{6} h (h^2 u u_x)_x
        - \frac{1}{6} h^2 u (h u_x)_x
        \\
        &\qquad
        + \frac{1}{4} h (h b_x u^2)_x
        + \frac{1}{4} h^2 u (b_x u)_x
        - \frac{1}{4} h h_x b_x u^2
        - \frac{1}{4} h^2 b_x u u_x.
\end{aligned}
\end{equation}
As described in Remark~\ref{rem:future_work_split_form_pressure},
there are still other options of split forms of the non-hydrostatic
pressure terms that could be considered. We will not pursue this
further here.



\subsection{Semidiscretization}

The split form \eqref{eq:SGN_original_mild_prim_split} leads to the
semidiscretization
\begin{equation}
\label{eq:SGN_original_mild_prim_SBP}
\begin{aligned}
    &
    \partial_t \pmb{h}
    + \pmb{u} D \pmb{h}
    + \pmb{h} D \pmb{u}
    = \pmb{0},
    \\
    &
    \left(
        \pmb{h}
        - \frac{1}{3} D \pmb{h}^3 D
        + \frac{1}{2} D \pmb{h}^2 (D \pmb{b})
        - \frac{1}{2} \pmb{h}^2 (D \pmb{b}) D
        + \frac{3}{4} \pmb{h} (D \pmb{b})^2
    \right) \partial_t \pmb{u}
    + g D \pmb{h} (\pmb{h} + \pmb{b})
    - g (\pmb{h} + \pmb{b}) D \pmb{h}
    \\
    &\qquad
    + \frac{1}{2} \pmb{h} D \pmb{u}^2
    - \frac{1}{2} \pmb{u}^2 D \pmb{h}
    + \frac{1}{2} \pmb{u} D \pmb{h} \pmb{u}
    - \frac{1}{2} \pmb{h} \pmb{u} D \pmb{u}
    + D \pmb{p}
    + \frac{3}{2} \frac{\pmb{p}}{\pmb{h}} D \pmb{b}
    = \pmb{0},
    \\
    &
    \pmb{p}
    =
    \frac{1}{2} \pmb{h}^3 (D \pmb{u})^2
    + \frac{1}{2} \pmb{h}^2 (D \pmb{h}) \pmb{u} D \pmb{u}
    - \frac{1}{6} \pmb{h} D (\pmb{h}^2 \pmb{u} D \pmb{u})
    - \frac{1}{6} \pmb{h}^2 \pmb{u} D \pmb{h} D \pmb{u}
    \\
    &\qquad
    + \frac{1}{4} \pmb{h} D \bigl( \pmb{h} (D \pmb{b}) \pmb{u}^2 \bigr)
    + \frac{1}{4} \pmb{h}^2 \pmb{u} D \bigl( (D \pmb{b}) \pmb{u} \bigr)
    - \frac{1}{4} \pmb{h} (D \pmb{h}) (D \pmb{b}) \pmb{u}^2
    - \frac{1}{4} \pmb{h}^2 (D \pmb{b}) \pmb{u} D \pmb{u}.
\end{aligned}
\end{equation}
The discrete total energy for \eqref{eq:SGN_original_mild_prim_SBP}
is $\pmb{1}^T M \pmb{E}$, where
\begin{equation}
\begin{aligned}
    \pmb{E}
    &=
    \frac{1}{2} g (\pmb{h} + \pmb{b})^2
    + \frac{1}{2} \pmb{h} \pmb{u}^2
    + \frac{1}{6} \pmb{h} \left( -\pmb{h} D \pmb{u} + \frac{3}{2} \pmb{u} D \pmb{b} \right)^2
    \\
    &=
    \frac{1}{2} g (\pmb{h} + \pmb{b})^2
    + \frac{1}{2} \pmb{h} \pmb{u}^2
    + \frac{1}{6} \pmb{h}^3 (D \pmb{u})^2
    - \frac{1}{2} \pmb{h}^2 (D \pmb{b}) (D \pmb{u}) \pmb{u}
    + \frac{3}{8} \pmb{h} (D \pmb{b})^2 \pmb{u}^2.
\end{aligned}
\end{equation}

\begin{theorem}
\label{thm:SGN_original_mild_prim_SBP}
    Consider the semidiscretization \eqref{eq:SGN_original_mild_prim_SBP}
    of the original Serre-Green-Naghdi equations with
    mild-slope approximation      \eqref{eq:SGN_original_mild_prim}
    using a periodic  SBP operator $D$ with  diagonal mass  matrix. The following holds.
    \begin{enumerate}
     \item The total water mass $\pmb{1}^T M \pmb{h}$ is conserved.
        \item The total momentum $\pmb{1}^T M \pmb{h} \pmb{u}$ is conserved if the bathymetry is constant.
        \item The total energy $\pmb{1}^T M \pmb{E}$ is conserved.
        \item The semidiscretization is well-balanced, i.e., it preserves
              the steady state $h + b \equiv \mathrm{const}$, $u \equiv 0$.
    \end{enumerate}
\end{theorem}
\begin{proof}
    This is a special case of the more general
    Theorem~\ref{thm:SGN_original_mild_prim_SBP_upwind} below
    with $D = D_- = D_+$.
\end{proof}

\begin{lemma}
\label{lem:SGN_original_mild_prim_SBP}
If   the water height is positive, then
    the discrete operator
    \begin{equation}
    \label{eq:SGN_original_mild_prim_SBP_operator}
        \pmb{h}
        - \frac{1}{3} D \pmb{h}^3 D
        + \frac{1}{2} D \pmb{h}^2 (D \pmb{b})
        - \frac{1}{2} \pmb{h}^2 (D \pmb{b}) D
        + \frac{3}{4} \pmb{h} (D \pmb{b})^2
    \end{equation}
    of \eqref{eq:SGN_original_mild_prim_SBP}
    is symmetric and positive definite with respect to the
    diagonal mass matrix $M$.
\end{lemma}
\begin{proof}
    This is a special case of the more general
    Lemma~\ref{lem:SGN_original_mild_prim_SBP_upwind} below
    with $D = D_- = D_+$.
\end{proof}

As before, the discrete operator
\eqref{eq:SGN_original_mild_prim_SBP_operator}
of \eqref{eq:SGN_original_mild_prim_SBP} includes a wide-stencil
approximation of the second derivative. Thus, we also use an
upwind version.

\begin{lemma}
\label{lem:SGN_original_mild_prim_SBP_upwind}
If   the water height is positive, then
the discrete operator
    \begin{equation}
        \pmb{h}
        - \frac{1}{3} D_+ \pmb{h}^3 D_-
        + \frac{1}{2} D_+ \pmb{h}^2 (D \pmb{b})
        - \frac{1}{2} \pmb{h}^2 (D \pmb{b}) D_-
        + \frac{3}{4} \pmb{h} (D \pmb{b})^2
    \end{equation}
    is symmetric and positive definite with respect to the
    diagonal mass matrix $M$.
\end{lemma}
\begin{proof}
    We have
    \begin{multline}
        M \left(
            \pmb{h}
            - \frac{1}{3} D_+ \pmb{h}^3 D_-
            + \frac{1}{2} D_+ \pmb{h}^2 (D \pmb{b})
            - \frac{1}{2} \pmb{h}^2 (D \pmb{b}) D_-
            + \frac{3}{4} \pmb{h} (D \pmb{b})^2
        \right)
        \\
        =
        \pmb{h} M
        + \frac{1}{3} D^T_- M \pmb{h}^3 D_-
        + \frac{1}{2} M D_+ \pmb{h}^2 (D \pmb{b})
        - \frac{1}{2} \pmb{h}^2 (D \pmb{b}) M D_-
        + \frac{3}{4} \pmb{h} (D \pmb{b})^2 M
    \end{multline}
    and
    \begin{equation*}
    \begin{aligned}
        \left(
            M D_+ \bigl(\pmb{h}^2 (D \pmb{b})\bigr)
            - \pmb{h}^2 (D \pmb{b}) M D_-
        \right)^T=
        \bigl(\pmb{h}^2 (D \pmb{b})\bigr) D_+^T M
        - D_-^T M \bigl(\pmb{h}^2 (D \pmb{b})\bigr)=
        - \pmb{h}^2 (D \pmb{b}) M D_-
        + M D_+ \bigl(\pmb{h}^2 (D \pmb{b})\bigr).
    \end{aligned}
    \end{equation*}
    Thus, the operator is symmetric with respect to $M$. Now for any given  vector $\pmb{v}$, we have
    \begin{multline}
        \pmb{v}^T M \left(
            \pmb{h}
            - \frac{1}{3} D_+ \pmb{h}^3 D_-
            + \frac{1}{2} D_+ \pmb{h}^2 (D \pmb{b})
            - \frac{1}{2} \pmb{h}^2 (D \pmb{b}) D_-
            + \frac{3}{4} \pmb{h} (D \pmb{b})^2
        \right) \pmb{v}
        \\
        =
        \| \pmb{v} \|_{\pmb{h} M}^2
        + \frac{1}{3} \| \pmb{h} D_- \pmb{v} \|_{\pmb{h} M}^2
        - \langle \pmb{h} D_- \pmb{v}, (D \pmb{b}) \pmb{v} \rangle_{\pmb{h} M}
        + \frac{3}{4} \| (D \pmb{b}) \pmb{v} \|_{\pmb{h} M}^2.
    \end{multline}
    By the Cauchy-Schwarz and Young inequalities, we have
    \begin{equation}
    \label{eq:SGN_original_mild_prim_SBP_upwind_inequality}
    \begin{aligned}
        &\quad
        \| \pmb{v} \|_{\pmb{h} M}^2
        + \frac{1}{3} \| \pmb{h} D_- \pmb{v} \|_{\pmb{h} M}^2
        + \frac{3}{4} \| (D \pmb{b}) \pmb{v} \|_{\pmb{h} M}^2
        - \langle \pmb{h} D_- \pmb{v}, (D \pmb{b}) \pmb{v} \rangle_{\pmb{h} M}
        \\
        &\ge
        \| \pmb{v} \|_{\pmb{h} M}^2
        + \frac{1}{3} \| \pmb{h} D_- \pmb{v} \|_{\pmb{h} M}^2
        + \frac{3}{4} \| (D \pmb{b}) \pmb{v} \|_{\pmb{h} M}^2
        - \| \pmb{h} D_- \pmb{v} \|_{\pmb{h} M} \| (D \pmb{b}) \pmb{v} \|_{\pmb{h} M}
        \\
        &\ge
        \| \pmb{v} \|_{\pmb{h} M}^2
        + \frac{1}{3} \| \pmb{h} D_- \pmb{v} \|_{\pmb{h} M}^2
        + \frac{3}{4} \| (D \pmb{b}) \pmb{v} \|_{\pmb{h} M}^2
        - \frac{1}{2 \epsilon} \| \pmb{h} D_- \pmb{v} \|_{\pmb{h} M}^2
        - \frac{\epsilon}{2} \| (D \pmb{b}) \pmb{v} \|_{\pmb{h} M}^2
        =
        \| \pmb{v} \|_{\pmb{h} M}^2
    \end{aligned}
    \end{equation}
    for $\epsilon = 3 / 2$.
\end{proof}

The use of the above operators    leads to the semidiscretization
\begin{equation}
\label{eq:SGN_original_mild_prim_SBP_upwind}
\begin{aligned}
    &
    \partial_t \pmb{h}
    + \pmb{u} D \pmb{h}
    + \pmb{h} D \pmb{u}
    = \pmb{0},
    \\
    &
    \left(
        \pmb{h}
        - \frac{1}{3} D_+ \pmb{h}^3 D_-
        + \frac{1}{2} D_+ \pmb{h}^2 (D \pmb{b})
        - \frac{1}{2} \pmb{h}^2 (D \pmb{b}) D_-
        + \frac{3}{4} \pmb{h} (D \pmb{b})^2
    \right) \partial_t \pmb{u}
    + g D \pmb{h} (\pmb{h} + \pmb{b})
    - g (\pmb{h} + \pmb{b}) D \pmb{h}
    \\
    &\qquad
    + \frac{1}{2} \pmb{h} D \pmb{u}^2
    - \frac{1}{2} \pmb{u}^2 D \pmb{h}
    + \frac{1}{2} \pmb{u} D \pmb{h} \pmb{u}
    - \frac{1}{2} \pmb{h} \pmb{u} D \pmb{u}
    + D_+ \pmb{p}_+
    + D \pmb{p}_0
    + \frac{3}{2} \frac{\pmb{p}_+ + \pmb{p}_0}{\pmb{h}} D \pmb{b}
    = \pmb{0},
    \\
    &
    \pmb{p}_+
    =
    \frac{1}{2} \pmb{h}^3 (D \pmb{u}) D_- \pmb{u}
    + \frac{1}{2} \pmb{h}^2 (D \pmb{h}) \pmb{u} D_- \pmb{u}
    - \frac{1}{4} \pmb{h}^2 (D \pmb{b}) \pmb{u} D \pmb{u}
    - \frac{1}{4} \pmb{h} (D \pmb{h}) (D \pmb{b}) \pmb{u}^2,
    \\
    &
    \pmb{p}_0
    =
    - \frac{1}{6} \pmb{h} D (\pmb{h}^2 \pmb{u} D \pmb{u})
    - \frac{1}{6} \pmb{h}^2 \pmb{u} D (\pmb{h} D \pmb{u})
    + \frac{1}{4} \pmb{h} D (\pmb{h} (D \pmb{b}) \pmb{u}^2)
    + \frac{1}{4} \pmb{h}^2 \pmb{u} D \bigl((D \pmb{b}) \pmb{u} \bigr).
\end{aligned}
\end{equation}
The discrete total energy for \eqref{eq:SGN_original_mild_prim_SBP_upwind}
is $\pmb{1}^T M \pmb{E}$, where
\begin{equation}
\begin{aligned}
    \pmb{E}
    &=
    \frac{1}{2} g (\pmb{h} + \pmb{b})^2
    + \frac{1}{2} \pmb{h} \pmb{u}^2
    + \frac{1}{6} \pmb{h} \left( -\pmb{h} D_- \pmb{u} + \frac{3}{2} \pmb{u} D \pmb{b} \right)^2
    \\
    &=
    \frac{1}{2} g (\pmb{h} + \pmb{b})^2
    + \frac{1}{2} \pmb{h} \pmb{u}^2
    + \frac{1}{6} \pmb{h}^3 (D_- \pmb{u})^2
    - \frac{1}{2} \pmb{h}^2 (D \pmb{b}) (D_- \pmb{u}) \pmb{u}
    + \frac{3}{8} \pmb{h} (D \pmb{b})^2 \pmb{u}^2.
\end{aligned}
\end{equation}

\begin{theorem}
\label{thm:SGN_original_mild_prim_SBP_upwind}
    Consider the semidiscretization \eqref{eq:SGN_original_mild_prim_SBP_upwind}
    of the original Serre-Green-Naghdi equations with
    mild-slope approximation for varying bathymetry
    \eqref{eq:SGN_original_mild_prim}
    with periodic first-derivative upwind SBP operators $D_\pm$
    inducing the central operator $D = (D_+ + D_-) / 2$ with
    diagonal mass/norm matrix.
    \begin{enumerate}
       \item The total water mass $\pmb{1}^T M \pmb{h}$ is conserved
        \item The total momentum $\pmb{1}^T M \pmb{h} \pmb{u}$ is conserved
              for constant bathymetry.
        \item The total energy $\pmb{1}^T M \pmb{E}$ is conserved.
        \item The semidiscretization is well-balanced, i.e., it preserves
              the steady state $h + b \equiv \mathrm{const}$, $u \equiv 0$.
    \end{enumerate}
\end{theorem}
\begin{proof}
    Conservation of the total water mass follows as in the proof of
    Theorem~\ref{thm:SGN_hyperbolic_flat_prim_SBP} since the equation
    for $\pmb{h}$ is the same.    Conservation of the total momentum for constant bathymetry follows
    from Theorem~\ref{thm:SGN_original_flat_prim_SBP_upwind} since
    \eqref{eq:SGN_original_mild_prim_SBP_upwind} reduces to
    \eqref{eq:SGN_original_flat_prim_SBP_upwind} in this case.
We now evaluate  the time derivative of the  total energy:
    \begin{equation}
    \begin{aligned}
        &\quad
        \partial_t (\pmb{1}^T M \pmb{E})
        \\
        &=
        g (\pmb{h} + \pmb{b})^T M \partial_t \pmb{h}
        + \frac{1}{2} (\pmb{u}^2)^T M \partial_t \pmb{h}
        + \frac{1}{2} (\pmb{h}^2 (D_- \pmb{u})^2)^T M \partial_t \pmb{h}
        - \left( \pmb{h} (D \pmb{b}) (D_- \pmb{u}) \pmb{u} \right)^T M \partial_t \pmb{h}
        \\
        &\quad
        + \frac{3}{8} \left( (D \pmb{b})^2 \pmb{u}^2 \right)^T M \partial_t \pmb{h}
        + \pmb{u}^T M \pmb{h} \partial_t \pmb{u}
        + \frac{1}{3} \left( \pmb{h}^3 (D_- \pmb{u}) \right)^T M D_- \partial_t \pmb{u}
        - \frac{1}{2} \left( \pmb{h}^2 (D \pmb{b}) (D_- \pmb{u}) \right)^T M \partial_t \pmb{u}
        \\
        &\quad
        - \frac{1}{2} \left( \pmb{h}^2 (D \pmb{b}) \pmb{u} \right)^T M D_- \partial_t \pmb{u}
        + \frac{3}{4} \left( \pmb{h} (D \pmb{b})^2 \pmb{u} \right)^T M \partial_t \pmb{u}.
    \end{aligned}
    \end{equation}
    Thus, we multiply the first equation of
    \eqref{eq:SGN_original_mild_prim_SBP_upwind} by
    \begin{equation}
        g (\pmb{h} + \pmb{b})^T M
        + \frac{1}{2} (\pmb{u}^2)^T M
        + \frac{1}{2} (\pmb{h}^2 (D_- \pmb{u})^2)^T M
        - \left( \pmb{h} (D \pmb{b}) (D_- \pmb{u}) \pmb{u} \right)^T M
        + \frac{3}{8} \left( (D \pmb{b})^2 \pmb{u}^2 \right)^T M,
    \end{equation}
    the second equation by $\pmb{u}^T M$, and add them.
    Compared to the case of flat bathymetry in
    Theorem~\ref{thm:SGN_original_flat_prim_SBP_upwind},
    the additional spatial derivative terms are
    \begin{equation}
    \begin{aligned}
        &\quad
        g \pmb{b}^T M \left( \pmb{u} D \pmb{h} + \pmb{h} D \pmb{u} \right)
        - \left( \pmb{h} (D \pmb{b}) (D_- \pmb{u}) \pmb{u} \right)^T M \left( \pmb{u} D \pmb{h} + \pmb{h} D \pmb{u} \right)
        + \frac{3}{8} \left( (D \pmb{b})^2 \pmb{u}^2 \right)^T M \left( \pmb{u} D \pmb{h} + \pmb{h} D \pmb{u} \right)
        \\
        &\quad
        + g \pmb{u}^T M D \pmb{h} \pmb{b}
        - g (\pmb{b} \pmb{u})^T M D \pmb{h}
        + \frac{1}{4} \pmb{u}^T M D \pmb{h} D \bigl( \pmb{h} (D \pmb{b}) \pmb{u}^2 \bigr)
        + \frac{1}{4} \pmb{u}^T M D \pmb{h}^2 \pmb{u} D \bigl( (D \pmb{b}) \pmb{u} \bigr)
        \\
        &\quad
        - \frac{1}{4} \pmb{u}^T M D_+ \pmb{h} (D \pmb{h}) (D \pmb{b}) \pmb{u}^2
        - \frac{1}{4} \pmb{u}^T M D_+ \pmb{h}^2 (D \pmb{b}) \pmb{u} (D \pmb{u})
        + \frac{3}{4} \pmb{1}^T M \pmb{h}^2 (D \pmb{b}) \pmb{u} (D_- \pmb{u}) (D \pmb{u})
        \\
        &\quad
        + \frac{3}{4} \pmb{1}^T M \pmb{h} (D \pmb{h}) (D \pmb{b}) \pmb{u}^2 (D_- \pmb{u})
        - \frac{1}{4} \bigl( (D \pmb{b}) \pmb{u} \bigr)^T M D (\pmb{h}^2 \pmb{u} D \pmb{u})
        - \frac{1}{4} (\pmb{h} (D \pmb{b}) \pmb{u}^2)^T M D (\pmb{h} D \pmb{u})
        \\
        &\quad
        + \frac{3}{8} \bigl( (D \pmb{b}) \pmb{u} \bigr)^T M D \bigl( \pmb{h} (D \pmb{b}) \pmb{u}^2 \bigr)
        + \frac{3}{8} \bigl( \pmb{h} (D \pmb{b}) \pmb{u}^2 \bigr)^T M D \bigl( (D \pmb{b}) \pmb{u} \bigr)
        - \frac{3}{8} \pmb{1}^T M (D \pmb{h}) (D \pmb{b})^2 \pmb{u}^3
        \\
        &\quad
        - \frac{3}{8} \pmb{1}^T M \pmb{h} (D \pmb{b})^2 \pmb{u}^2 (D \pmb{u})
        =
        0,
    \end{aligned}
    \end{equation}
proving  energy conservation.      For well-balanced,  if  $h + b \equiv \mathrm{const}$ and $u \equiv 0$ we have
 $\partial_t \pmb{h} = \pmb{0}$, and      $\pmb{p} = \pmb{0}$ so
    \begin{equation}
        \left(
            \pmb{h}
            - \frac{1}{3} D \pmb{h}^3 D
            + \frac{1}{2} D \pmb{h}^2 (D \pmb{b})
            - \frac{1}{2} \pmb{h}^2 (D \pmb{b}) D
            + \frac{3}{4} \pmb{h} (D \pmb{b})^2
        \right) \partial_t \pmb{u}
        + \underbrace{%
            g D \pmb{h} (\pmb{h} + \pmb{b})
            - g (\pmb{h} + \pmb{b}) D \pmb{h}
        }_{= \pmb{0}}
        = \pmb{0}.
    \end{equation}
\end{proof}

\begin{remark}
\label{rem:future_work_upwind_form_pressure2}
    One could also choose other non-hydrostatic pressure terms.
    In addition to the split form versions mentioned in
    Remark~\ref{rem:future_work_split_form_pressure}, one could
    distribute the upwind terms differently. For example,
    one could use
    \begin{equation}
        D \left( -\frac{1}{4} \pmb{h}^2 (D \pmb{b}) \pmb{u} D_- \pmb{u} \right)
    \end{equation}
    in $D \pmb{p}_0$ instead of
    \begin{equation}
        D_+ \left( -\frac{1}{4} \pmb{h}^2 (D \pmb{b}) \pmb{u} D \pmb{u} \right)
    \end{equation}
    in $D_+ \pmb{p}_+$ since the contribution to the energy rate of both choices
    is the same (when multiplied by $\pmb{u}^T M$).
    Investigating this further is left for future work.
\end{remark}

\section{Original Serre-Green-Naghdi equations with variable bathymetry:
         full system without mild-slope approximation}
\label{sec:SGN_original_full}

As an extension of the previous section, we propose here   a split form of the classical SGN system with full
bathymetric variations,  which we write for the moment as
\begin{equation}
\label{eq:SGN_original_full}
\begin{split}
    & h_t
        + (h u)_x = 0,\\
    & h u_t
        + g h h_x
        + h u u_x
        + (h \pi)_x
        + \left( g h + \dfrac{3}{2} \pi \right) b_x
        + \phi b_x
        = 0,\\
    & h w_t
        + h u w_x = 3 \pi,\\
    & w
        + hu_x - \dfrac{3}{2} \omega = 0,\\
    & h \omega_t
        + h u \omega_x = 4 \phi,\\
    & \omega - u b_x = 0,
\end{split}
\end{equation}
with the last four relations defining
$\pi$, $w$, $\phi$, and $\omega$.
We can now check that the energy conservation law
\begin{equation}
\label{eq:SGN_original_full_energy}
    \biggl( \underbrace{\frac{1}{2} g (h + b)^2 + \frac{1}{2} h u^2 + \frac{1}{6} h w^2 + \frac{1}{8} h \omega^2}_{= E} \biggr)_t
    + \biggl( \underbrace{g h (h + b) u + \frac{1}{2} h u^3 + \frac{1}{6} h u w^2 + \frac{1}{8} h u \omega^2 + h \pi u}_{= F} \biggr)_x = 0
\end{equation}
holds.
We can show that the last two definitions allow to write
\begin{equation}
    E_t
    =
    \left( g (h + b) + \frac{1}{2} u^2 + \frac{1}{6} w^2 + \frac{1}{8} \omega^2 \right) h_t
    + h u u_t
    + \frac{1}{3} h w w_t
    + \frac{1}{4} h \omega \omega_t.
\end{equation}
As before, the flux variation requires exploiting all the auxiliary variables
\begin{equation}
\begin{aligned}
    F_x
    &=
    \left( g (h + b) + \frac{1}{2} u^2 + \frac{1}{6} w^2 + \frac{1}{8} \omega^2 \right) (h u)_x
    + u \left( g h h_x + h u u_x + (h \pi)_x + (g h + 3 \pi / 2 ) b_x + \phi b_x \right)
    \\
    &\quad
    + \dfrac{w}{3} ( h u w_x - 3 \pi )
    + \pi (w + h u_x - 3 \omega / 2)
    + \dfrac{\omega}{4} ( h u \omega_x - 4 \phi ) + \phi (\omega - b_x u).
\end{aligned}
\end{equation}

To mimic these relations, using the results obtained so far,
 we consider  the following split form:
\begin{equation}
\begin{split}
    & h_t
        + h_x u
        + h u_x
        = 0,\\
    & h u_t
        + g(h (h + b))_x
        - g (h + b) h_x
        \\
        &\qquad
        + \frac{1}{2} h (u^2)_x
        - \frac{1}{2} h_x u^2
        + \frac{1}{2} (h u)_x u
        - \frac{1}{2} h u u_x
        + (h \pi)_x
        + \dfrac{3}{2} \pi b_x
        + \phi b_x
        = 0,\\
    & h w_t
        + \dfrac{1}{2} (h u w)_x
        + \dfrac{1}{2} h u w_x
        - \dfrac{1}{2} h_x u w
        - \dfrac{1}{2} h u_x w
        = 3 \pi,\\
    & w + hu_x - \frac{3}{2} \omega = 0,\\
    & h \omega_t
        + \dfrac{1}{2} (h u \omega)_x
        + \dfrac{1}{2} h u \omega_x
        - \dfrac{1}{2} h_x u \omega
        - \dfrac{1}{2} h u_x \omega
        = 4 \phi,\\
    & \omega - b_x u = 0.
\end{split}
\end{equation}
One can easily check that this formulation  is compatible, up to integration by parts,
with  mass and energy conservation, as well as
 well-balanced wrt   states  at rest and   constant $h + b$.
Compared to the mild-slope approximation \eqref{eq:SGN_original_mild_prim},
we only have the additional term $+ \phi b_x$ in the momentum equation,
where
\begin{equation}
\begin{aligned}
    \phi
    &=
    \frac{1}{4} h b_x u_t
    + \frac{1}{8} (h b_x u^2)_x
    + \frac{1}{8} h u (b_x u)_x
    - \frac{1}{8} h_x b_x u^2
    - \frac{1}{8} h b_x u u_x.
\end{aligned}
\end{equation}
Following a similar procedure as done in the flat  and mild slope cases,  we can write
\begin{equation}
\label{eq:SGN_original_full_prim_split}
\begin{aligned}
    & h_t
        + h_x u
        + h u_x
        = 0,\\
    & h u_t
        - \frac{1}{3} (h^3 u_{tx})_x
        + \frac{1}{2} (h^2 b_x u_t)_x
        - \frac{1}{2} h^2 b_x u_{tx}
        + h b_x^2 u_t
        + g(h (h + b))_x
        - g (h + b) h_x
        \\
        &\qquad
        + \frac{1}{2} h (u^2)_x
        - \frac{1}{2} h_x u^2
        + \frac{1}{2} (h u)_x u
        - \frac{1}{2} h u u_x
        + p_x
        + \frac{3}{2} \frac{p}{h} b_x
        + \psi b_x
        = 0,\\
    & p = \frac{1}{2} h^3 u_x^2
        + \frac{1}{2} h^2 h_x u u_x
        - \frac{1}{6} h (h^2 u u_x)_x
        - \frac{1}{6} h^2 u (h u_x)_x
        \\
        &\qquad
        + \frac{1}{4} h (h b_x u^2)_x
        + \frac{1}{4} h^2 u (b_x u)_x
        - \frac{1}{4} h h_x b_x u^2
        - \frac{1}{4} h^2 b_x u u_x,\\
    & \psi =
          \frac{1}{8} (h b_x u^2)_x
        + \frac{1}{8} h u (b_x u)_x
        - \frac{1}{8} h_x b_x u^2
        - \frac{1}{8} h b_x u u_x.
\end{aligned}
\end{equation}
Compared to the split form \eqref{eq:SGN_original_mild_prim_split} of the
mild-slope approximation, we have the following differences:
\begin{itemize}
    \item the term $h b_x^2 u_t$ has a factor of unity instead of $3 / 4$
    \item the terms $(h b_x u^2)_x + h u (b_x u)_x - h_x b_x u^2 - h b_x u u_x$
          appearing in $3 p b_x / (2 h)$
          have the factor $1 / 2$ instead of $3 / 8$
          due to the additional term $\psi b_x$
\end{itemize}

\subsection{Semidiscretization}

The split form \eqref{eq:SGN_original_full_prim_split} leads to the
semidiscretization
\begin{equation}
\label{eq:SGN_original_full_prim_SBP}
\begin{aligned}
    &
    \partial_t \pmb{h}
    + \pmb{u} D \pmb{h}
    + \pmb{h} D \pmb{u}
    = \pmb{0},
    \\
    &
    \left(
        \pmb{h}
        - \frac{1}{3} D \pmb{h}^3 D
        + \frac{1}{2} D \pmb{h}^2 (D \pmb{b})
        - \frac{1}{2} \pmb{h}^2 (D \pmb{b}) D
        + \pmb{h} (D \pmb{b})^2
    \right) \partial_t \pmb{u}
    + g D \pmb{h} (\pmb{h} + \pmb{b})
    - g (\pmb{h} + \pmb{b}) D \pmb{h}
    \\
    &\qquad
    + \frac{1}{2} \pmb{h} D \pmb{u}^2
    - \frac{1}{2} \pmb{u}^2 D \pmb{h}
    + \frac{1}{2} \pmb{u} D \pmb{h} \pmb{u}
    - \frac{1}{2} \pmb{h} \pmb{u} D \pmb{u}
    + D \pmb{p}
    + \frac{3}{2} \frac{\pmb{p}}{\pmb{h}} D \pmb{b}
    + \pmb{\psi} D \pmb{b}
    = \pmb{0},
    \\
    &
    \pmb{p}
    =
    \frac{1}{2} \pmb{h}^3 (D \pmb{u})^2
    + \frac{1}{2} \pmb{h}^2 (D \pmb{h}) \pmb{u} D \pmb{u}
    - \frac{1}{6} \pmb{h} D (\pmb{h}^2 \pmb{u} D \pmb{u})
    - \frac{1}{6} \pmb{h}^2 \pmb{u} D \pmb{h} D \pmb{u}
    \\
    &\qquad
    + \frac{1}{4} \pmb{h} D \bigl( \pmb{h} (D \pmb{b}) \pmb{u}^2 \bigr)
    + \frac{1}{4} \pmb{h}^2 \pmb{u} D \bigl( (D \pmb{b}) \pmb{u} \bigr)
    - \frac{1}{4} \pmb{h} (D \pmb{h}) (D \pmb{b}) \pmb{u}^2
    - \frac{1}{4} \pmb{h}^2 (D \pmb{b}) \pmb{u} D \pmb{u},
    \\
    &
    \pmb{\psi}
    =
      \frac{1}{8} D \bigl( \pmb{h} (D \pmb{b}) \pmb{u}^2 \bigr)
    + \frac{1}{8} \pmb{h} \pmb{u} D \bigl( (D \pmb{b}) \pmb{u} \bigr)
    - \frac{1}{8} (D \pmb{h}) (D \pmb{b}) \pmb{u}^2
    - \frac{1}{8} \pmb{h} (D \pmb{b}) \pmb{u} D \pmb{u}.
\end{aligned}
\end{equation}

\begin{lemma}
\label{lem:SGN_original_full_prim_SBP}
If the water height is positive, then the discrete operator
    \begin{equation}
        \pmb{h}
        - \frac{1}{3} D \pmb{h}^3 D
        + \frac{1}{2} D \pmb{h}^2 (D \pmb{b})
        - \frac{1}{2} \pmb{h}^2 (D \pmb{b}) D
        + \pmb{h} (D \pmb{b})^2
    \end{equation}
    is symmetric and positive definite with respect to the
    diagonal mass matrix $M$.
\end{lemma}
\begin{proof}
    Compared to Lemma~\ref{lem:SGN_original_mild_prim_SBP}, we have an
    additional term
    $\pmb{h} (D \pmb{b})^2 / 4$.
\end{proof}

The discrete total energy for \eqref{eq:SGN_original_full_prim_SBP}
is $\pmb{1}^T M \pmb{E}$, where
\begin{equation}
\begin{aligned}
    \pmb{E}
    &=
    \frac{1}{2} g (\pmb{h} + \pmb{b})^2
    + \frac{1}{2} \pmb{h} \pmb{u}^2
    + \frac{1}{6} \pmb{h} \left( -\pmb{h} D \pmb{u} + \frac{3}{2} \pmb{u} D \pmb{b} \right)^2
    + \frac{1}{8} \pmb{h} (\pmb{u} D \pmb{b})^2
    \\
    &=
    \frac{1}{2} g (\pmb{h} + \pmb{b})^2
    + \frac{1}{2} \pmb{h} \pmb{u}^2
    + \frac{1}{6} \pmb{h}^3 (D \pmb{u})^2
    - \frac{1}{2} \pmb{h}^2 (D \pmb{b}) (D \pmb{u}) \pmb{u}
    + \frac{1}{2} \pmb{h} (D \pmb{b})^2 \pmb{u}^2.
\end{aligned}
\end{equation}

\begin{theorem}
\label{thm:SGN_original_full_prim_SBP}
    Consider the semidiscretization \eqref{eq:SGN_original_full_prim_SBP}
    of the original Serre-Green-Naghdi equations without
    mild-slope approximation for varying bathymetry
    \eqref{eq:SGN_original_full}
    using a periodic first-derivative SBP operator $D$ with
    diagonal mass/norm matrix.
    \begin{enumerate}
     \item The total water mass $\pmb{1}^T M \pmb{h}$ is conserved.
        \item The total momentum $\pmb{1}^T M \pmb{h} \pmb{u}$ is conserved if the bathymetry is constant.
        \item The total energy $\pmb{1}^T M \pmb{E}$ is conserved.
        \item The semidiscretization is well-balanced, i.e., it preserves
              the steady state $h + b \equiv \mathrm{const}$, $u \equiv 0$.
    \end{enumerate}
\end{theorem}
\begin{proof}
    This is a special case of the more general result
    Theorem~\ref{thm:SGN_original_full_prim_SBP_upwind} below.
\end{proof}

Analogously, we can derive the upwind version of the split form as
\begin{equation}
\label{eq:SGN_original_full_prim_SBP_upwind}
\begin{aligned}
    &
    \partial_t \pmb{h}
    + \pmb{u} D \pmb{h}
    + \pmb{h} D \pmb{u}
    = \pmb{0},
    \\
    &
    \left(
        \pmb{h}
        - \frac{1}{3} D_+ \pmb{h}^3 D_-
        + \frac{1}{2} D_+ \pmb{h}^2 (D \pmb{b})
        - \frac{1}{2} \pmb{h}^2 (D \pmb{b}) D_-
        + \pmb{h} (D \pmb{b})^2
    \right) \partial_t \pmb{u}
    + g D \pmb{h} (\pmb{h} + \pmb{b})
    - g (\pmb{h} + \pmb{b}) D \pmb{h}
    \\
    &\qquad
    + \frac{1}{2} \pmb{h} D \pmb{u}^2
    - \frac{1}{2} \pmb{u}^2 D \pmb{h}
    + \frac{1}{2} \pmb{u} D \pmb{h} \pmb{u}
    - \frac{1}{2} \pmb{h} \pmb{u} D \pmb{u}
    + D_+ \pmb{p}_+
    + D \pmb{p}_0
    + \frac{3}{2} \frac{\pmb{p}_+ + \pmb{p}_0}{\pmb{h}} D \pmb{b}
    + \pmb{\psi} D \pmb{b}
    = \pmb{0},
    \\
    &
    \pmb{p}_+
    =
    \frac{1}{2} \pmb{h}^3 (D \pmb{u}) D_- \pmb{u}
    + \frac{1}{2} \pmb{h}^2 (D \pmb{h}) \pmb{u} D_- \pmb{u}
    - \frac{1}{4} \pmb{h}^2 (D \pmb{b}) \pmb{u} D \pmb{u}
    - \frac{1}{4} \pmb{h} (D \pmb{h}) (D \pmb{b}) \pmb{u}^2,
    \\
    &
    \pmb{p}_0
    =
    - \frac{1}{6} \pmb{h} D (\pmb{h}^2 \pmb{u} D \pmb{u})
    - \frac{1}{6} \pmb{h}^2 \pmb{u} D (\pmb{h} D \pmb{u})
    + \frac{1}{4} \pmb{h} D (\pmb{h} (D \pmb{b}) \pmb{u}^2)
    + \frac{1}{4} \pmb{h}^2 \pmb{u} D \bigl((D \pmb{b}) \pmb{u} \bigr),
    \\
    &
    \pmb{\psi}
    =
      \frac{1}{8} D \bigl( \pmb{h} (D \pmb{b}) \pmb{u}^2 \bigr)
    + \frac{1}{8} \pmb{h} \pmb{u} D \bigl( (D \pmb{b}) \pmb{u} \bigr)
    - \frac{1}{8} \pmb{h} (D \pmb{b}) \pmb{u} D \pmb{u}
    - \frac{1}{8} (D \pmb{h}) (D \pmb{b}) \pmb{u}^2.
\end{aligned}
\end{equation}
\begin{lemma}
\label{lem:SGN_original_full_prim_SBP_upwind}
If the water height is positive, then the discrete operator
    \begin{equation}
        \pmb{h}
        - \frac{1}{3} D_+ \pmb{h}^3 D_-
        + \frac{1}{2} D_+ \pmb{h}^2 (D \pmb{b})
        - \frac{1}{2} \pmb{h}^2 (D \pmb{b}) D_-
        + \pmb{h} (D \pmb{b})^2
    \end{equation}
    is symmetric and positive definite with respect to the
    diagonal mass matrix $M$.
\end{lemma}
\begin{proof}
    Compared to Lemma~\ref{lem:SGN_original_mild_prim_SBP_upwind},
    we have an additional term
    $+ \pmb{h} (D \pmb{b})^2 / 4$ in the operator,
    leading to an additional term
    $+ \| (D \pmb{b}) \pmb{v} \|^2_{\pmb{h} M} / 4 \ge 0$ in
    \eqref{eq:SGN_original_mild_prim_SBP_upwind_inequality}.
\end{proof}
The discrete total energy of the semidiscretization is $\pmb{1}^T M \pmb{E}$, where
\begin{equation}
\begin{aligned}
    \pmb{E}
    &=
    \frac{1}{2} g (\pmb{h} + \pmb{b})^2
    + \frac{1}{2} \pmb{h} \pmb{u}^2
    + \frac{1}{6} \pmb{h} \left( -\pmb{h} D_- \pmb{u} + \frac{3}{2} \pmb{u} D \pmb{b} \right)^2
    + \frac{1}{8} \pmb{h} \left( \pmb{u} D \pmb{b} \right)^2
    \\
    &=
    \frac{1}{2} g (\pmb{h} + \pmb{b})^2
    + \frac{1}{2} \pmb{h} \pmb{u}^2
    + \frac{1}{6} \pmb{h}^3 (D_- \pmb{u})^2
    - \frac{1}{2} \pmb{h}^2 (D \pmb{b}) (D_- \pmb{u}) \pmb{u}
    + \frac{1}{2} \pmb{h} (D \pmb{b})^2 \pmb{u}^2.
\end{aligned}
\end{equation}

\begin{theorem}
\label{thm:SGN_original_full_prim_SBP_upwind}
    Consider the semidiscretization \eqref{eq:SGN_original_full_prim_SBP_upwind}
    of the original Serre-Green-Naghdi equations without
    mild-slope approximation for varying bathymetry \eqref{eq:SGN_original_full}
    with periodic first-derivative upwind SBP operators $D_\pm$
    inducing the central operator $D = (D_+ + D_-) / 2$ with
    diagonal mass/norm matrix.
    \begin{enumerate}
      \item The total water mass $\pmb{1}^T M \pmb{h}$ is conserved.
        \item The total momentum $\pmb{1}^T M \pmb{h} \pmb{u}$ is conserved
              if the bathymetry is constant.
        \item The total energy $\pmb{1}^T M \pmb{E}$ is conserved.
        \item The semidiscretization is well-balanced, i.e., it preserves
              the steady state $h + b \equiv \mathrm{const}$, $u \equiv 0$.
    \end{enumerate}
\end{theorem}
\begin{proof}
    Conservation of the total water mass and momentum as well as
    preservation of the steady
    state follow as in Theorem~\ref{thm:SGN_original_mild_prim_SBP_upwind}.
    Thus, we just check the rate of change of the total energy.
    Compared to the mild-slope case in
    Theorem~\ref{thm:SGN_original_mild_prim_SBP_upwind}, we get the following
    additional terms
    \begin{itemize}
        \item $+ (1/8) \bigl( (D \pmb{b})^2 \pmb{u}^2 \bigr)^T M \partial_t \pmb{h}$
              from the additional energy term and the rate of change of $\pmb{h}$
        \item $+ (1/4) \pmb{u}^T M \pmb{h} (D \pmb{b}) \partial_t \pmb{u}$
              from the additional energy term and the rate of change of $\pmb{u}$
        \item $+ \pmb{u}^T M \pmb{\psi} D \pmb{b} = \bigl( (D \pmb{b}) \pmb{u} \bigr)^T M \pmb{\psi}$
              from the additional term involving the bottom topography
    \end{itemize}
    The additional term involving the time derivative of $\pmb{u}$ is included
    in the linear operator that we need to invert. Adding the remaining terms
    yields
    \begin{equation*}
    \begin{aligned}
        \bigl( (D \pmb{b})^2 \pmb{u}^2 \bigr)^T M \partial_t \pmb{h}
        + 8 \pmb{u}^T M \pmb{\psi} D \pmb{b}
       &
        =
        \bigl( (D \pmb{b})^2 \pmb{u}^2 \bigr)^T M \pmb{h} D \pmb{u}
        + \bigl( (D \pmb{b})^2 \pmb{u}^2 \bigr)^T M \pmb{u} D \pmb{h}
        + \bigl( (D \pmb{b}) \pmb{u} \bigr)^T M D \bigl( \pmb{h} (D \pmb{b}) \pmb{u}^2 \bigr)
        \\
        &+ \bigl( (D \pmb{b}) \pmb{u} \bigr)^T M \pmb{h} \pmb{u} D \bigl( \pmb{u} (D \pmb{b}) \bigr)
        - \bigl( (D \pmb{b}) \pmb{u} \bigr)^T M \pmb{h} (D \pmb{b}) \pmb{u} D \pmb{u}
        - \bigl( (D \pmb{b}) \pmb{u} \bigr)^T M (D \pmb{h}) (D \pmb{b}) \pmb{u}^2
        \\
        &=
        \pmb{1}^T M \pmb{h} (D \pmb{b})^2 \pmb{u}^2 (D \pmb{u})
        + \pmb{1}^T M (D \pmb{h}) (D \pmb{b})^2 \pmb{u}^3
        + \bigl( (D \pmb{b}) \pmb{u} \bigr)^T M D \bigl( \pmb{h} (D \pmb{b}) \pmb{u}^2 \bigr)
        \\
        &
        + \bigl( \pmb{h} (D \pmb{b}) \pmb{u}^2 \bigr)^T M D \bigl( \pmb{u} (D \pmb{b}) \bigr)
        - \pmb{1}^T M \pmb{h} (D \pmb{b})^2 \pmb{u}^2 (D \pmb{u})
        - \pmb{1}^T M (D \pmb{h}) (D \pmb{b})^2 \pmb{u}^3
        =
        0.
    \end{aligned}
    \end{equation*}
    Thus, the total energy is conserved.
\end{proof}

\section{Artificial viscosity stabilization}
\label{sec:AV}

When considering structure-preserving methods, and in particular entropy/energy-conservative methods,
it is quite natural to compare them to methods embedding some form of numerical dissipation.
As discussed in the introduction,  it is unclear that the notion of a dissipative weak solution should also  apply to
dispersive equations such as those considered here. However, when working on coarse meshes, as   often  in operational practice,
one must be careful in controlling spurious modes related to under-resolution, and some degree of dissipation may be justified.
This is also the motivation to introduce the upwind SBP operators of Section~\ref{sec:basic-discretizations}.
So inspired by classical and more recent works on spectral and high-order approximations with vanishing viscosity  \cite{tadmor89,maday93,GUERMOND20114248,pasquetti17}
we  consider the use of artificial viscosity (AV) as a  stabilization method.  In particular, for all models studied we consider adding to the momentum equation
a viscous term using a classical formulation reading   \cite{Gerbeau01,bresch02}
\begin{equation}\label{eq:av1}
    h u_t + \dots = (\mu hu_x)_x.
\end{equation}
The viscosity definition is set having in mind the preservation  of the consistency of the underlying operators,
and in particular for a method of accuracy order $p$ we set
\begin{equation}\label{eq:av2}
    \mu = C \dfrac{ \Delta x^p}{p},
\end{equation}
where for simplicity the (dimensional) constant $C$,  has been set to 1 in all experiments.

\subsection{Discretization with  SBP operators}

In general, we will add to the right hand side of our discretization a term $\pmb{F}_{\mu}$:
\begin{equation}\label{eq:av3}
\pmb{h} \partial_t \pmb{u} + \dots =  \pmb{F}_{\mu}.
\end{equation}
When using periodic central SBP operators, the artificial diffusion term is approximated as
\begin{equation}\label{eq:av4}
\pmb{F}_{\mu}=  D(\mu \pmb{h} D \pmb{u}).
\end{equation}
When using periodic upwind SBP operators, the artificial diffusion term is approximated as
\begin{equation}\label{eq:av5}
\pmb{F}_{\mu} =  D_+(\mu \pmb{h} D_- \pmb{u}).
\end{equation}
For these additional terms we can prove the following simple result.
 \begin{theorem}
The discrete artificial  viscosity terms \eqref{eq:av4} and \eqref{eq:av5}  verify the following two properties
    \begin{enumerate}
        \item They do not alter momentum conservation:  total energy $\pmb{1}^T M \pmb{F}_{\mu}=0$.
        \item They provide a dissipative contribution to the energy balance: $\pmb{u}^T M \pmb{F}_{\mu} \ge 0$.
    \end{enumerate}
\end{theorem}
\begin{proof}
Result obtained combining \eqref{eq:av4} with \eqref{eq:SBP_periodic},
and \eqref{eq:av5} with  \eqref{eq:SBP_upwind_periodic}  respectively.
\end{proof}
The above proves  that energy preserving schemes become energy stable when including the AV term.
%
%
%
%

\section{Numerical experiments}
\label{sec:numerical_experiments}

The methods proposed in this work have been implemented in Julia \cite{bezanson2017julia}, using
the packages SummationByPartsOperators.jl \cite{ranocha2021sbp} for
the spatial discretizations and OrdinaryDiffEq.jl
\cite{rackauckas2017differentialequations} for time integration.
The Fourier pseudospectral methods use FFTW wrapped in FFTW.jl
\cite{frigo2005design} in Julia.
The sparse linear systems are solved using a direct solver of SuiteSparse
\cite{chen2008cholmod,amestoy2004amd,davis2004colamd} available in Julia.
We use the ITP method \cite{oliveira2020enhancement}
implemented in SimpleNonlinearSolve.jl
\cite{pal2024nonlinearsolve} to compute the relaxation parameter $\gamma$
for energy-conservative time integration methods.
We use Plots.jl \cite{christ2023plots} to visualize the results.
All source code to reproduce our numerical results is available online
\cite{ranocha2024structureRepro}.

Time integration is performed using explicit Runge-Kutta methods with error-based step size control
\cite{berzins1995temporal,ketcheson2020more,ranocha2021optimized,ranocha2023error}.
If not described otherwise, we choose relative and absolute tolerances
$10^{-5}$. For the original Serre-Green-Naghdi equations, we use the
fifth-order Runge-Kutta method of \cite{tsitouras2011runge}. For the
hyperbolic approximation, we use the third-order Runge-Kutta method of
\cite{ranocha2021optimized} that was optimized for discretizations
of hyperbolic conservation laws when the time step size is constrained by
stability instead of accuracy.

Throughout this article, we use SI units for all quantities. The   gravitational constant
is set to $g = \SI{9.81}{m/s^2}$. We apply periodic boundary conditions for all experiments since we have not
analyzed the energy for other boundary conditions. When  necessary larger domain sizes are used
avoid  the effects of  inconsistent values on the left/right domain boundaries.
We initialize the hyperbolic approximation with the water height and
velocity of the Serre-Green-Naghdi equations, and the auxiliary variables
\begin{equation}
    \pmb{\eta} = \pmb{h}, \quad \pmb{w} = -\pmb{h} D \pmb{u}.
\end{equation}
The main steps of the simulations are summarized in the flowcharts in
Figure~\ref{fig:flowcharts}. The main difference of the computational complexity
comes from the necessity to solve an elliptic system for the original Serre-Green-Naghdi
equations versus having to deal with more variables and stiffer systems (if $\lambda \gg 1$)
for the hyperbolic approximation.

\begin{figure}[htbp]
    \centering
    \begin{tikzpicture}[
        node distance=0mm and 0mm,
        >={Triangle[angle=45:2pt 3]},
        line width=.5pt,
        font=\footnotesize,
        box/.style   ={rectangle, draw, rounded corners, align=center,
                        minimum width=2cm, minimum height=8mm},
        startstop/.style = {box, fill=gray!20},
        process/.style   = {box, fill=blue!5},
        decision/.style  = {diamond, draw, fill=red!5, aspect=2, align=center},
        ]

        \begin{scope}[xshift=-1cm]
            \node at (2.5,1) {\large \bf Original SGN Simulation};

            \node[process] (init) {
                Initialize:\\
                $\pmb{h} = \pmb{h}^0$, $\pmb{u} = \pmb{u}^0$;\\
                $\pmb{b}$, $t=0$, $T_{\mathrm{end}}$
            };

            \node[decision, below=of init, yshift=-5mm] (looptime)
                {$t < T_{\mathrm{end}}?$};

            \node[startstop, below=30mm of looptime] (stop) {Stop};

            \node[process, right=of init, xshift=15mm] (dt) {
                Compute $\Delta t$
            };

            \node[process, below=of dt, yshift=-5mm, align=left] (rk) {
            \begin{minipage}{3.5cm}
                \textbf{Runge-Kutta stages:}
                \begin{itemize}
                \item Compute derivatives ($D \pmb{h}$, $D \pmb{u}$, \dots)
                \item Form ODE RHS for $\partial_t \pmb{h}$
                \item Solve the linear system for $\partial_t \pmb{u}$
                \item Update stage solutions
                \end{itemize}
            \end{minipage}
            };

            \node[process, below=of rk, yshift=-5mm] (relax) {
                Relaxation (optional)\\
                (for energy conservation)
            };

            \node[process, below=of relax, yshift=-5mm] (update) {
                Update time $t$ and\\numerical solution $\pmb{h}, \pmb{u}$
            };

            \draw[->] (init) -- (looptime);
            \draw[->] (looptime) -- node[above=2pt] {Yes} (dt);
            \draw[->] (looptime) -- node[left=2pt] {No} (stop);
            \draw[->] (dt) -- (rk);
            \draw[->] (rk) -- (relax);
            \draw[->] (relax) -- (update);
            \draw[->] (update.west) to (looptime.south east);
        \end{scope}

        \begin{scope}[xshift=7cm]
            \node at (2.5,1) {\large \bf Hyperbolic SGN Simulation};

            \node[process] (init) {
                Initialize:\\
                $\pmb{h} = \pmb{h}^0$, $\pmb{u} = \pmb{u}^0$,\\
                $\pmb{\eta} = \pmb{h}^0$, $\pmb{w} = -\pmb{h} D \pmb{u}$;\\
                $\pmb{b}$, $t=0$, $T_{\mathrm{end}}$
            };

            \node[decision, below=of init, yshift=-5mm] (looptime)
                {$t < T_{\mathrm{end}}?$};

            \node[startstop, below=30mm of looptime] (stop) {Stop};

            \node[process, right=of init, xshift=15mm] (dt) {
                Compute $\Delta t$
            };

            \node[process, below=of dt, yshift=-5mm, align=left] (rk) {
            \begin{minipage}{3.5cm}
                \textbf{Runge-Kutta stages:}
                \begin{itemize}
                \item Compute derivatives
                        ($D \pmb{h}, D \pmb{u}, \dots$)
                \item Form ODE RHS for $\partial_t \pmb{h}$, $\partial_t \pmb{u}$,
                        $\partial_t \pmb{\eta}$, and $\partial_t \pmb{w}$
                \item Update stage solutions
                \end{itemize}
            \end{minipage}
            };

            \node[process, below=of rk, yshift=-5mm] (relax) {
                Relaxation (optional)\\
                (for energy conservation)
            };

            \node[process, below=of relax, yshift=-5mm] (update) {
                Update time $t$ and\\numerical solution $\pmb{h}, \pmb{u}, \pmb{\eta}, \pmb{w}$
            };

            \draw[->] (init) -- (looptime);
            \draw[->] (looptime) -- node[above=2pt] {Yes} (dt);
            \draw[->] (looptime) -- node[left=2pt] {No} (stop);
            \draw[->] (dt) -- (rk);
            \draw[->] (rk) -- (relax);
            \draw[->] (relax) -- (update);
            \draw[->] (update.west) to (looptime.south east);
        \end{scope}
    \end{tikzpicture}
    \caption{Flowcharts summarizing the main algorithmic steps of the simulations
             based on the original Serre-Green-Naghdi equations (left) and the hyperbolic
             approximation (right).}
    \label{fig:flowcharts}
\end{figure}

\subsection{Convergence studies}

Since we are interested in the spatial error of the methods, we use the
fifth-order Runge-Kutta method of \cite{tsitouras2011runge} with stricter
tolerances $10^{-9}$ for the convergence experiments reported in this section.
To compute the experimental order of convergence (EOC), we use the formula
\begin{equation}
    \text{EOC} = -\frac{\log(e_2 / e_1)}{\log(N_2 / N_1)},
\end{equation}
where $e_i$ and $N_i$ are measures of the error and the discretization size
for two consecutive grid refinements. For finite difference methods,
we use the number of nodes as discretization size $N$.

\subsubsection{Solitary waves of the Serre-Green-Naghdi equations}
\label{sec:convergence_soliton}

The exact solitary wave for the Serre-Green-Naghdi equations has
depth and depth-averaged velocities given by
\begin{equation}\begin{split}
    h_e = h_{\infty} \left( 1 + \epsilon \operatorname{sech}^2\bigl( \kappa(x-Ct) \bigr) \right), \quad
    u_e = C \left( 1 - \dfrac{h_{\infty}}{h} \right),
\end{split}\end{equation}
where $\epsilon = A/h_{\infty}$ with $A$ the soliton amplitude, and where
\begin{equation}
\kappa^2 =\dfrac{3\epsilon}{4h_{\infty}^2(1+\epsilon)}\;,\quad
C^2 = gh_{\infty}(1+\epsilon).
\end{equation}
This is a solution of the classical Serre-Green-Naghdi equations.
When using it for the hyperbolic system, a possible choice to initialize
the auxiliary variables, at least for $\lambda$ large enough, is
\begin{equation}
\eta(t=0,x)=h_e(0)\;,\quad
w(t=0,x)=-h_e(0)u_e'(0).
\end{equation}
If not stated otherwise, we use the following parameters for the solitary wave:
\begin{equation}
    h_{\infty} =  1, \quad A = 0.2 h_{\infty}.
\end{equation}
For the convergence experiments in this section, we choose the domain
$[-50, 50]$ with periodic boundary conditions and a time interval
such that the wave travels through the domain once.

\begin{figure}[htb]
\centering
  \begin{subfigure}{0.49\textwidth}
    \centering
    \includegraphics[width=0.8\textwidth]{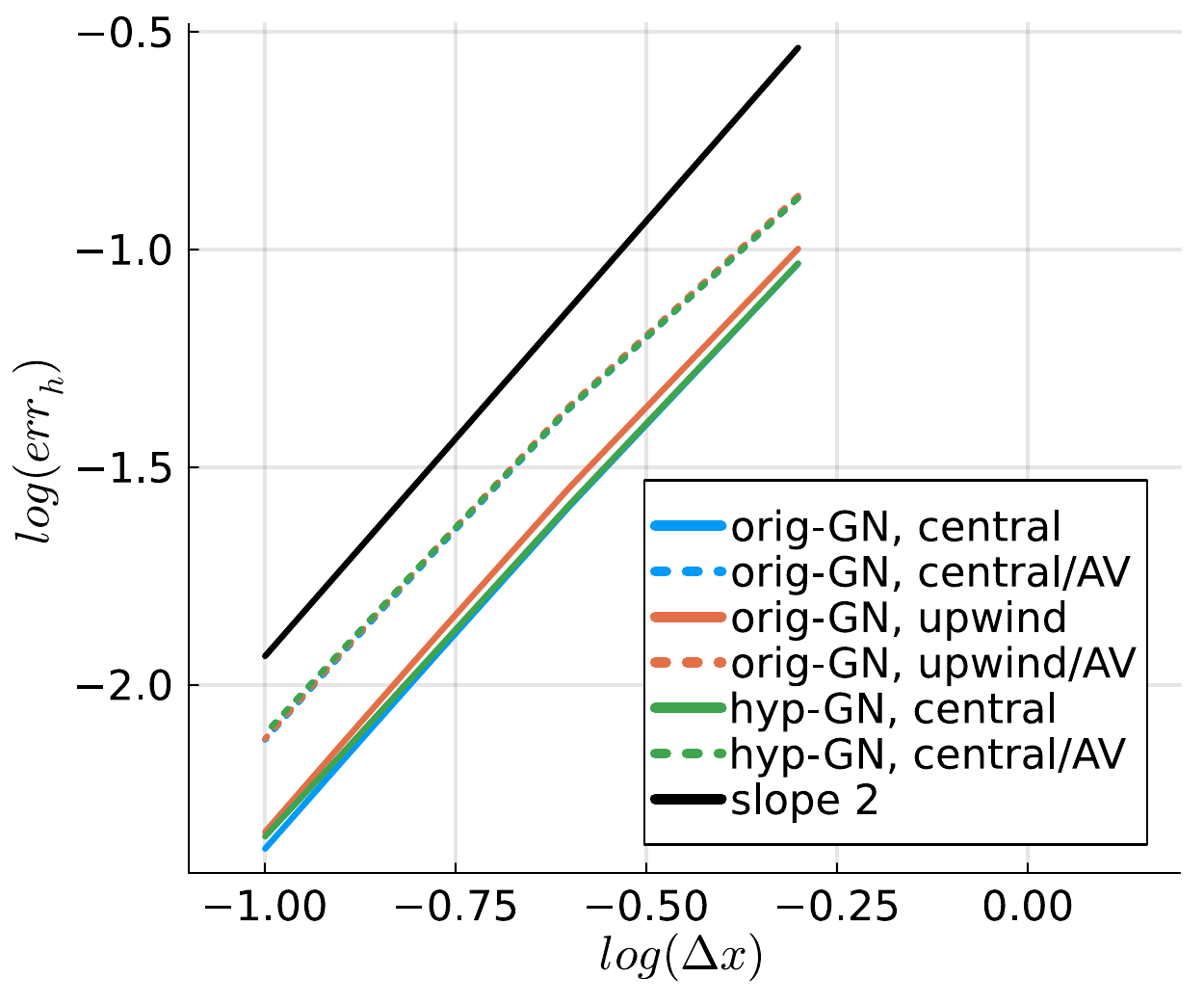}
    \caption{Depth convergence, SBP operators of order  2.}
  \end{subfigure}%
  \hspace{\fill}
  \begin{subfigure}{0.49\textwidth}
    \centering
      \includegraphics[width=0.8\textwidth]{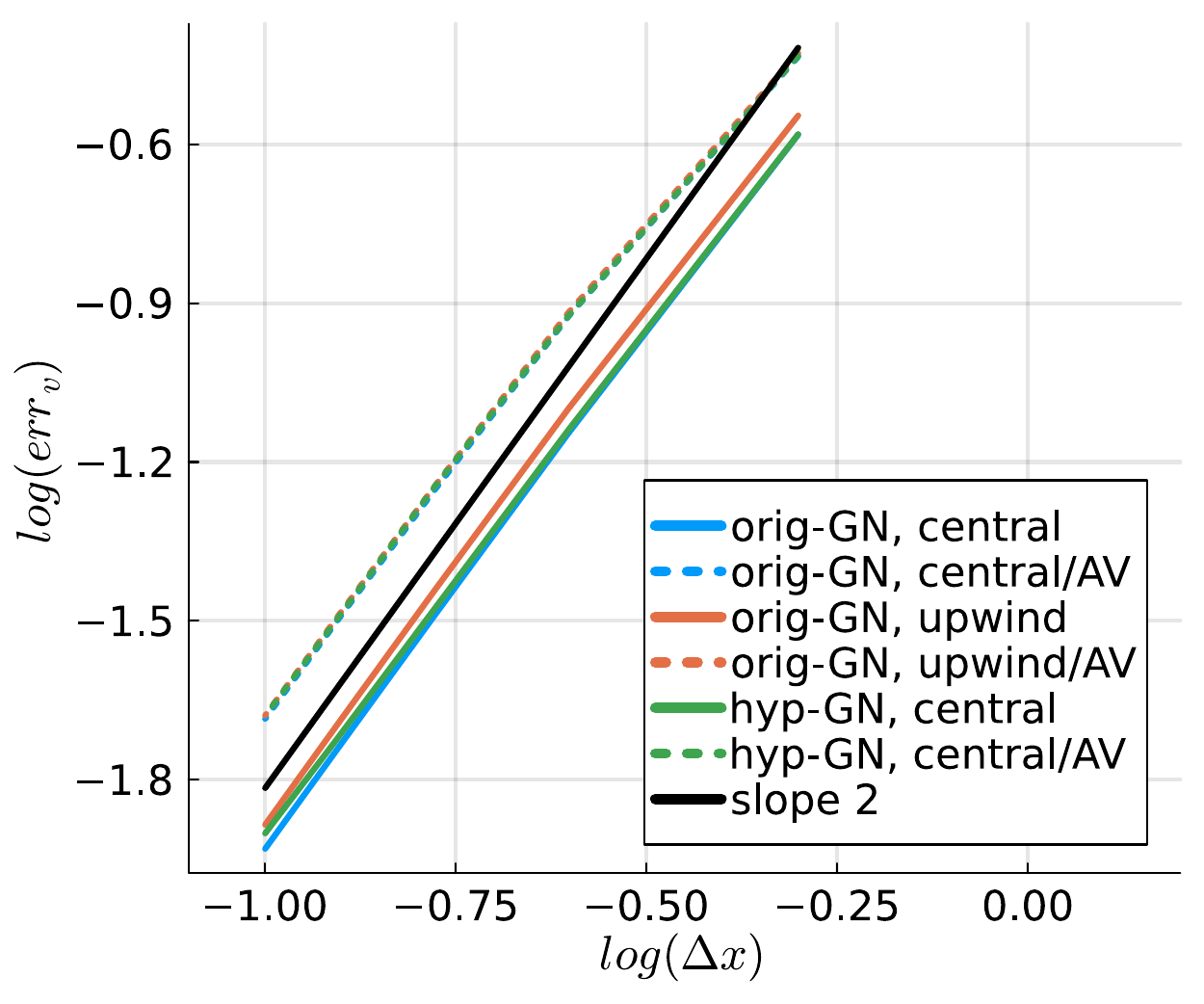}
    \caption{Velocity convergence, SBP operators of order  2.}
  \end{subfigure}
  \\
  \medskip
  \begin{subfigure}{0.49\textwidth}
    \centering
       \includegraphics[width=0.8\textwidth]{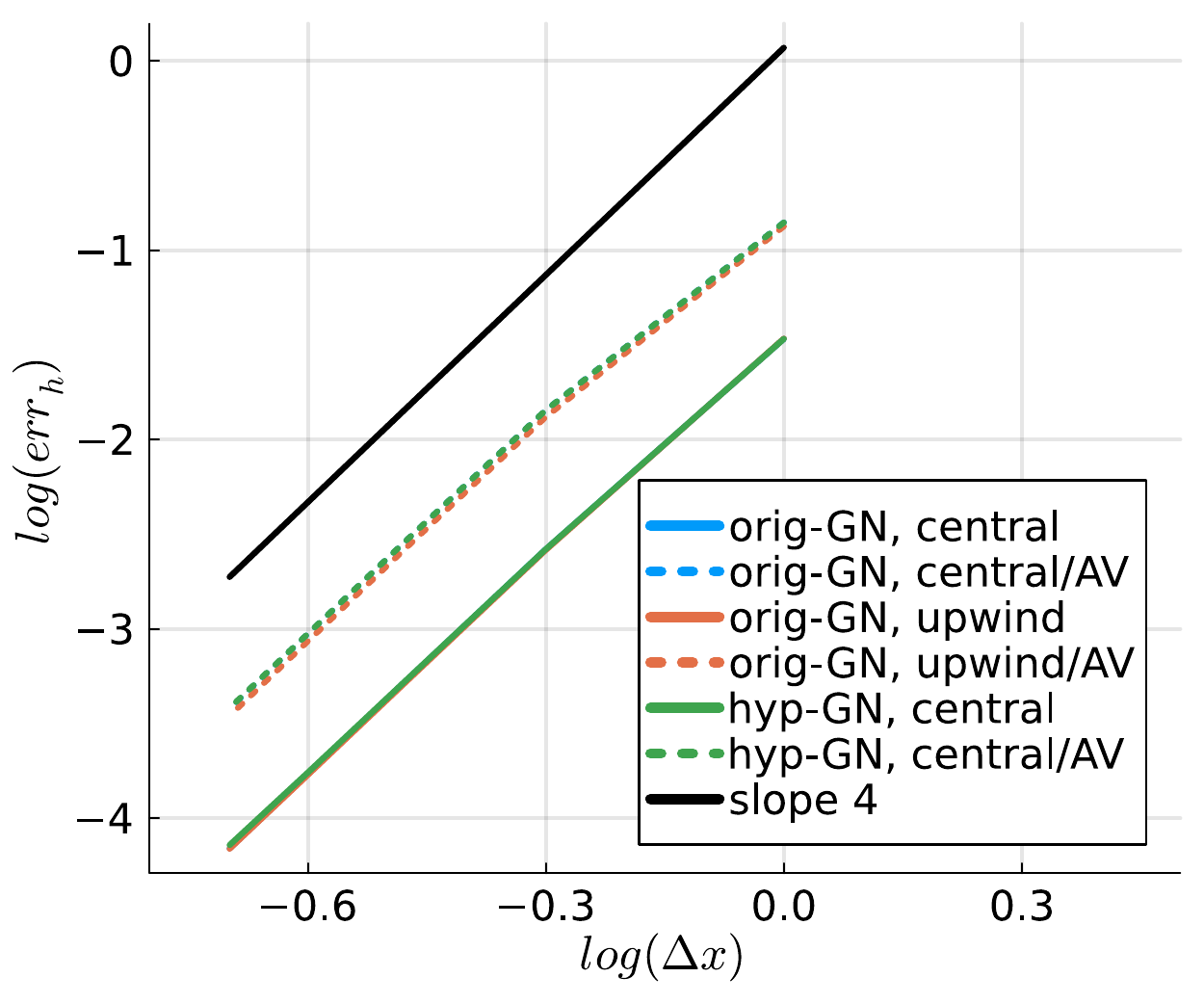}
    \caption{Depth convergence, SBP operators of order  4.}
  \end{subfigure}%
  \hspace{\fill}
  \begin{subfigure}{0.49\textwidth}
    \centering
       \includegraphics[width=0.8\textwidth]{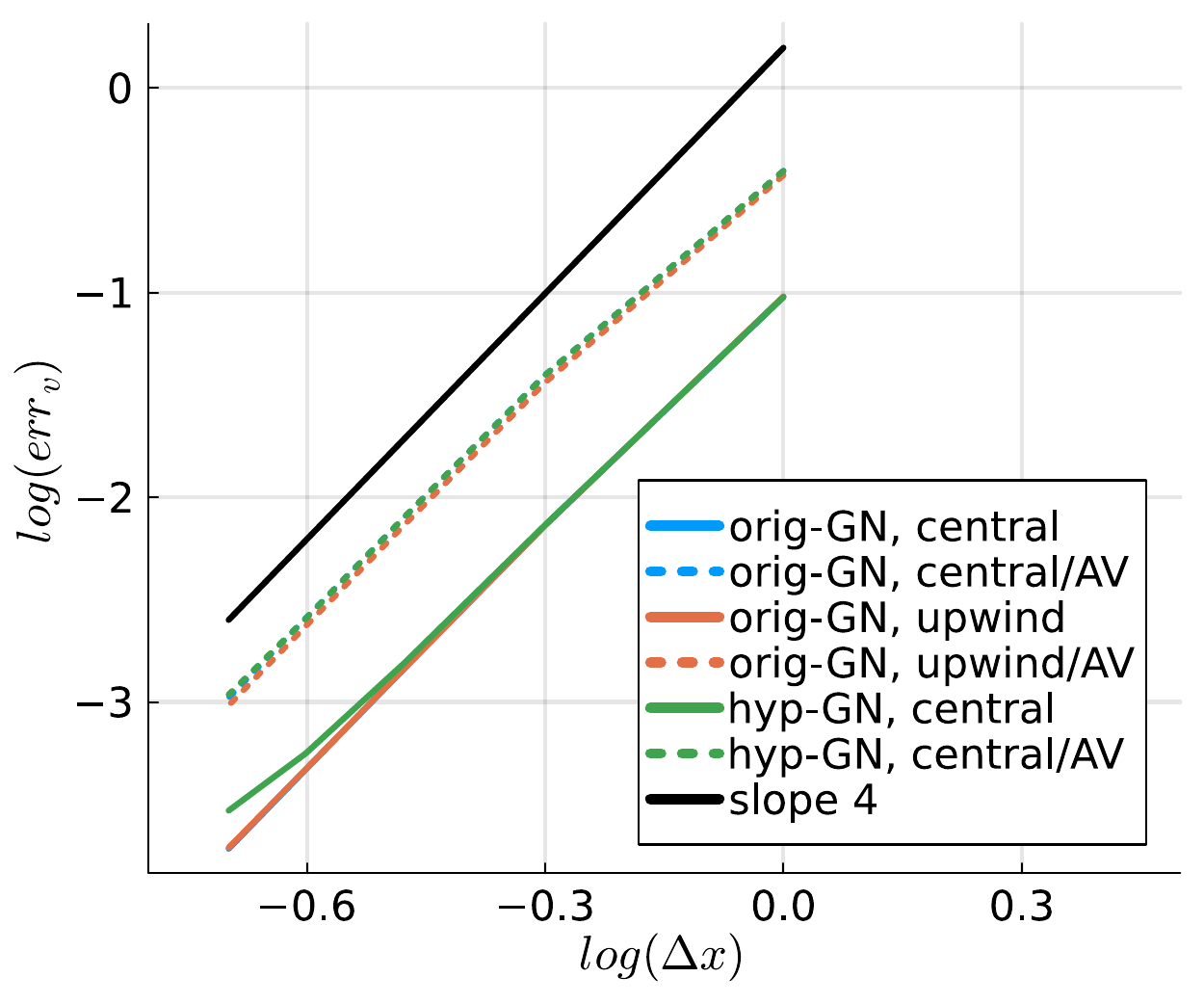}
    \caption{Velocity convergence, SBP operators of order  4.}
  \end{subfigure}
  \caption{Convergence results using finite difference semidiscretizations
           with $N$ nodes
           applied to the solitary wave of the Serre-Green-Naghdi equations.}
  \label{fig:convergence_soliton_serre_green_naghdi}
\end{figure}

Figure~\ref{fig:convergence_soliton_serre_green_naghdi} shows the
errors of the numerical solutions
measured with respect to the soliton
solution of the Serre-Green-Naghdi equations for finite difference
methods with different orders of accuracy, including methods with (high order) artificial viscosity.
The EOC matches the expected order of accuracy.
For this case   artificial viscosity leads to an  increase in the absolute value of the error by a factor in between $\approx$ 2--5.

\begin{table}[htbp]
\centering
  \caption{Convergence results when increasing the parameter $\lambda$
           of the hyperbolic approximation
           using eighth-order finite difference semidiscretizations with
           $N = 500$ nodes
           applied to the solitary wave of the Serre-Green-Naghdi equations.}
  \label{tab:convergence_hyperbolic_approximation}
  \begin{subfigure}{0.49\textwidth}
    \centering
    \caption{Central operators}
    \begin{tabular}{r rr rr}
        \toprule
        $\lambda$ & $L^2$ err. $h$ & EOC $h$ & $L^2$ err. $v$ & EOC $v$ \\
        \midrule
        $10^2$ & 2.85e-02 &  & 8.86e-02 &  \\
        $10^3$ & 2.89e-03 & 0.99 & 8.60e-03 & 1.01 \\
        $10^4$ & 2.91e-04 & 1.00 & 1.02e-03 & 0.92 \\
        $10^5$ & 2.93e-05 & 1.00 & 2.45e-04 & 0.62 \\
        $10^6$ & 2.92e-06 & 1.00 & 1.33e-04 & 0.26 \\
        \bottomrule
    \end{tabular}
\end{subfigure}%
\begin{subfigure}{0.49\textwidth}
    \centering
    \caption{Central operators and artificial dissipation}
    \begin{tabular}{r rr rr}
    \toprule
        $\lambda$ & $L^2$ err. $h$ & EOC $h$ & $L^2$ err. $v$ & EOC $v$ \\
        \midrule
        $10^2$ & 2.85e-02 &  & 8.86e-02 &  \\
        $10^3$ & 2.89e-03 & 0.99 & 8.60e-03 & 1.01 \\
        $10^4$ & 2.91e-04 & 1.00 & 1.02e-03 & 0.92 \\
        $10^5$ & 2.94e-05 & 1.00 & 2.45e-04 & 0.62 \\
        $10^6$ & 3.10e-06 & 0.98 & 1.41e-04 & 0.24 \\
        \bottomrule
    \end{tabular}
    \end{subfigure}%
\end{table}

For the hyperbolic approximation, the EOC matches
the expected order of accuracy for the second-order method, when using $\lambda=10^4$. For the
fourth-order method $\lambda=10^6$, is necessary to obtain the proper rates  for the water height, but even
 with this value the convergence rates are   a bit smaller for the velocity. A larger value of $\lambda$ may allow to correct this.
For completeness we also study the convergence  when increasing the value of the parameter $\lambda$.
The results are summarized in Table~\ref{tab:convergence_hyperbolic_approximation}.
The EOC of the water height is  unity. The velocity converges as well but
with a reduced EOC, especially for large $\lambda$. The  differences due to the artificial diffusion are very small here.

\subsubsection{Manufactured solution for the hyperbolic system}

We use the method of manufactured solutions to check the implementation.
We choose the solution
\begin{equation}
\begin{aligned}
    h(t, x) &= 7 + 2 \cos(2 \pi x) + \cos(2 \pi x - 4 \pi t),
    \\
    u(t, x) &= \sin(2 \pi x - \pi t),
    \\
    \eta(t, x) &= h(t, x),
    \\
    w(t, x) &= -h(t, x) \partial_x u(t, x),
    \\
    b &= -5 - 2 \cos(2 \pi x),
\end{aligned}
\end{equation}
and add source terms to the equations so that the equations are satisfied
for the manufactured solutions.
We consider the grid convergence of the error at the final time $t=1$ on a domain of width $1$.
The results shown in Figure~\ref{fig:convergence_manufactured_hyperbolic}
confirm the expected order of accuracy for the finite difference
semidiscretizations. For this test, the   difference brought by adding artificial dissipation is
orders of magnitude smaller than the errors itself, and the convergence curves are essentially superposed.

\begin{figure}[htbp]
\centering
  \begin{subfigure}{0.49\textwidth}
    \centering
    \includegraphics[width=0.8\textwidth]{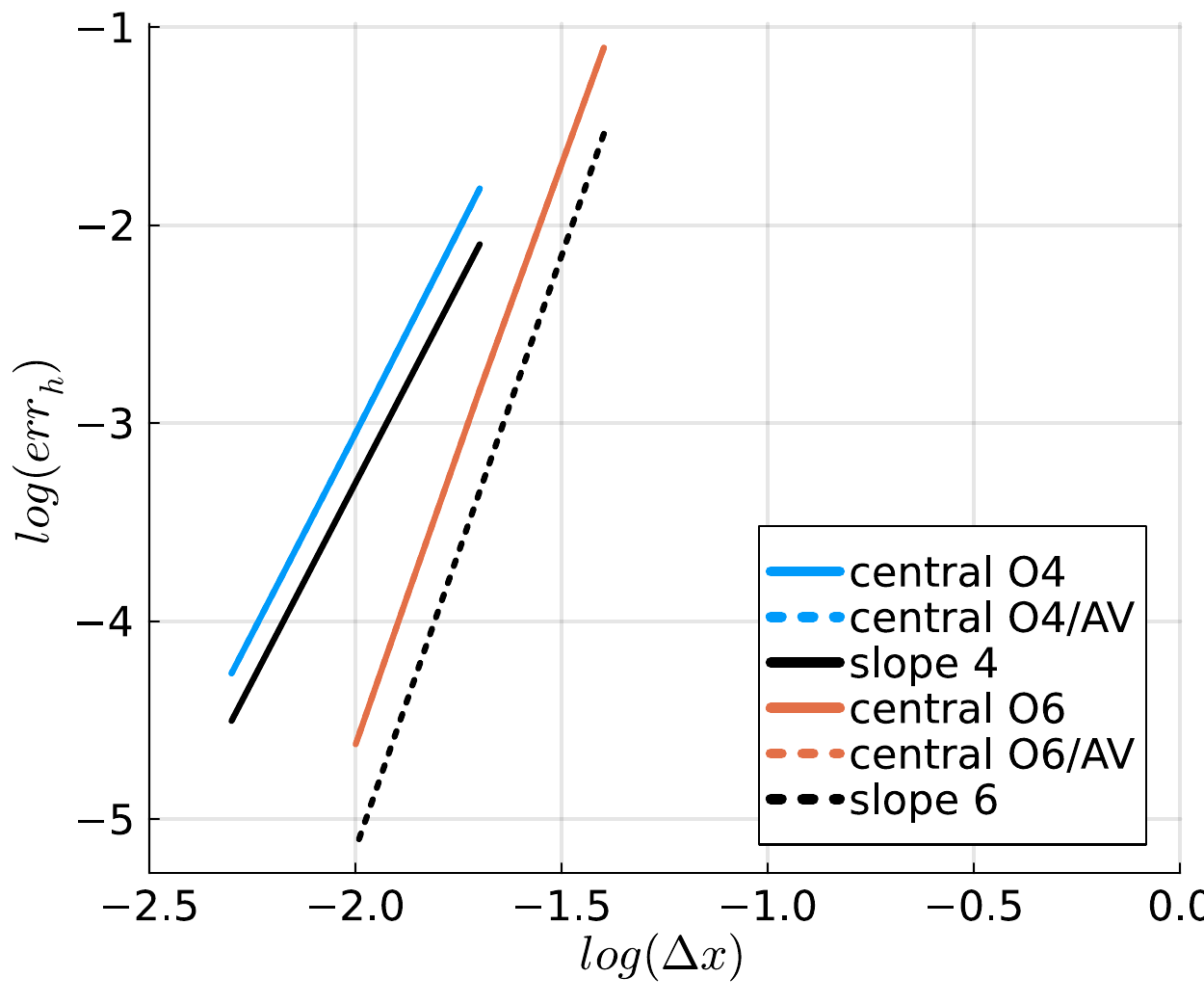}
    \caption{Depth error convergence.}
  \end{subfigure}%
  \hspace{\fill}
  \begin{subfigure}{0.49\textwidth}
    \centering
      \includegraphics[width=0.8\textwidth]{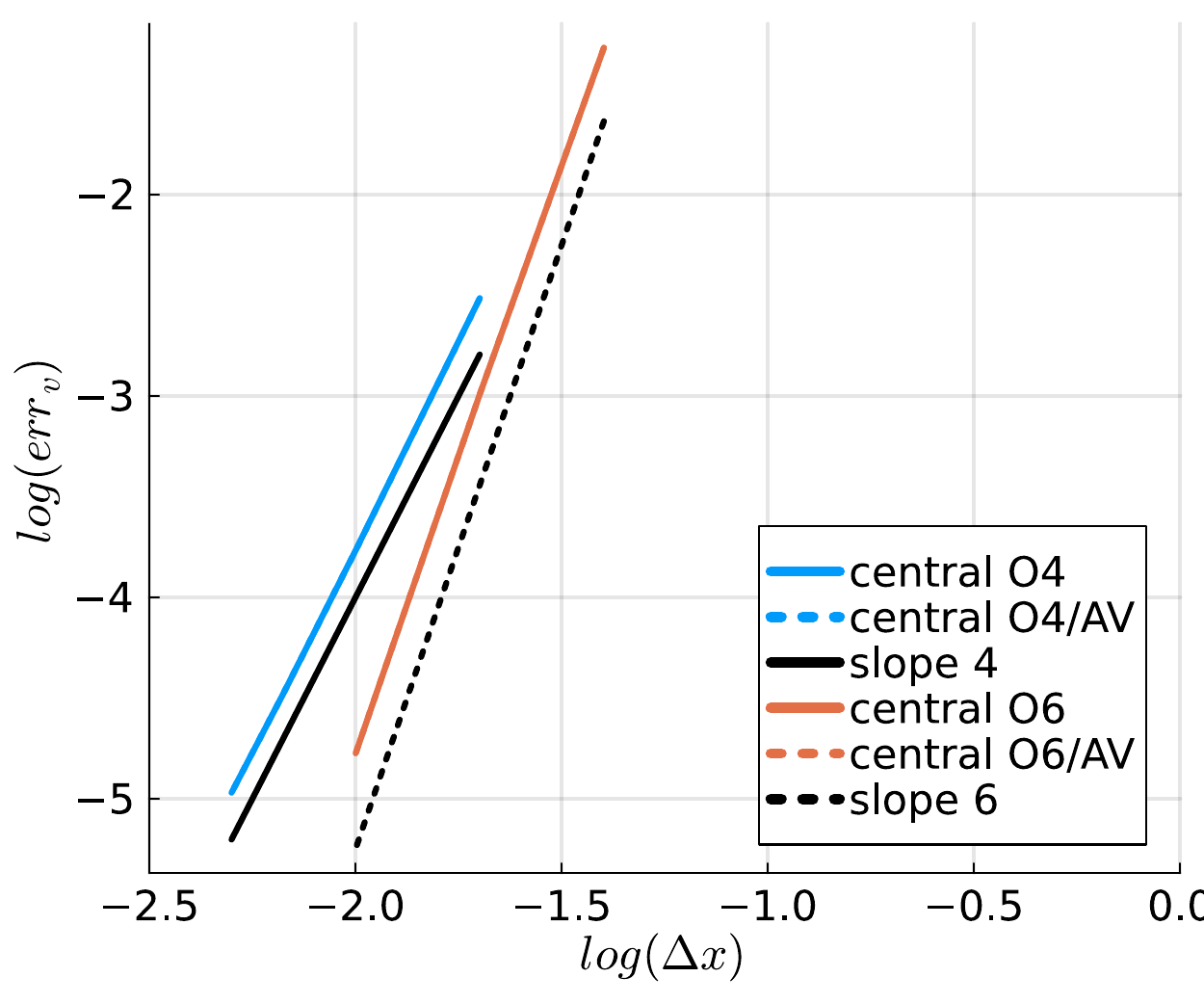}
    \caption{Velocity error convergence.}
  \end{subfigure}
  \caption{Grid convergence to the manufactured solution for the hyperbolic system results using finite difference semidiscretizations
  of different orders and  $\lambda = 500$. Left: depth. Right: velocity. }
    \label{fig:convergence_manufactured_hyperbolic}
\end{figure}

%

\subsection{Qualitative comparison of upwind and central methods}
\label{sec:upwind_vs_central}

To point out the different behavior of upwind and central finite difference
methods for solving the elliptic equations, we compare numerical results
for a Gaussian initial condition.

\subsubsection{Qualitative comparison: flat bathymetry}
\label{sec:upwind_vs_central_flat}

First, we consider the initial condition
\begin{equation}
\label{eq:initial_condition_gaussian}
    h(x, 0) = 1 + \exp(-x^2), \quad u(x, u) = 10^{-2}, \quad b(x) = 0,
\end{equation}
and discretize the domain $[-150, 150]$ with second-order finite differences.
We use the fifth-order Runge-Kutta method of \cite{tsitouras2011runge} with
tolerances $10^{-5}$ for the time interval $[0, 35]$.
Reference solutions  are shown in
Figure~\ref{fig:conservation_solution}. All the methods considered in this section
are  visually fully converged with the same  number of nodes $N = 3000$.

\begin{figure}[htbp]
\centering
    \begin{subfigure}{0.32\textwidth}
    \centering
        \includegraphics[width=\textwidth]{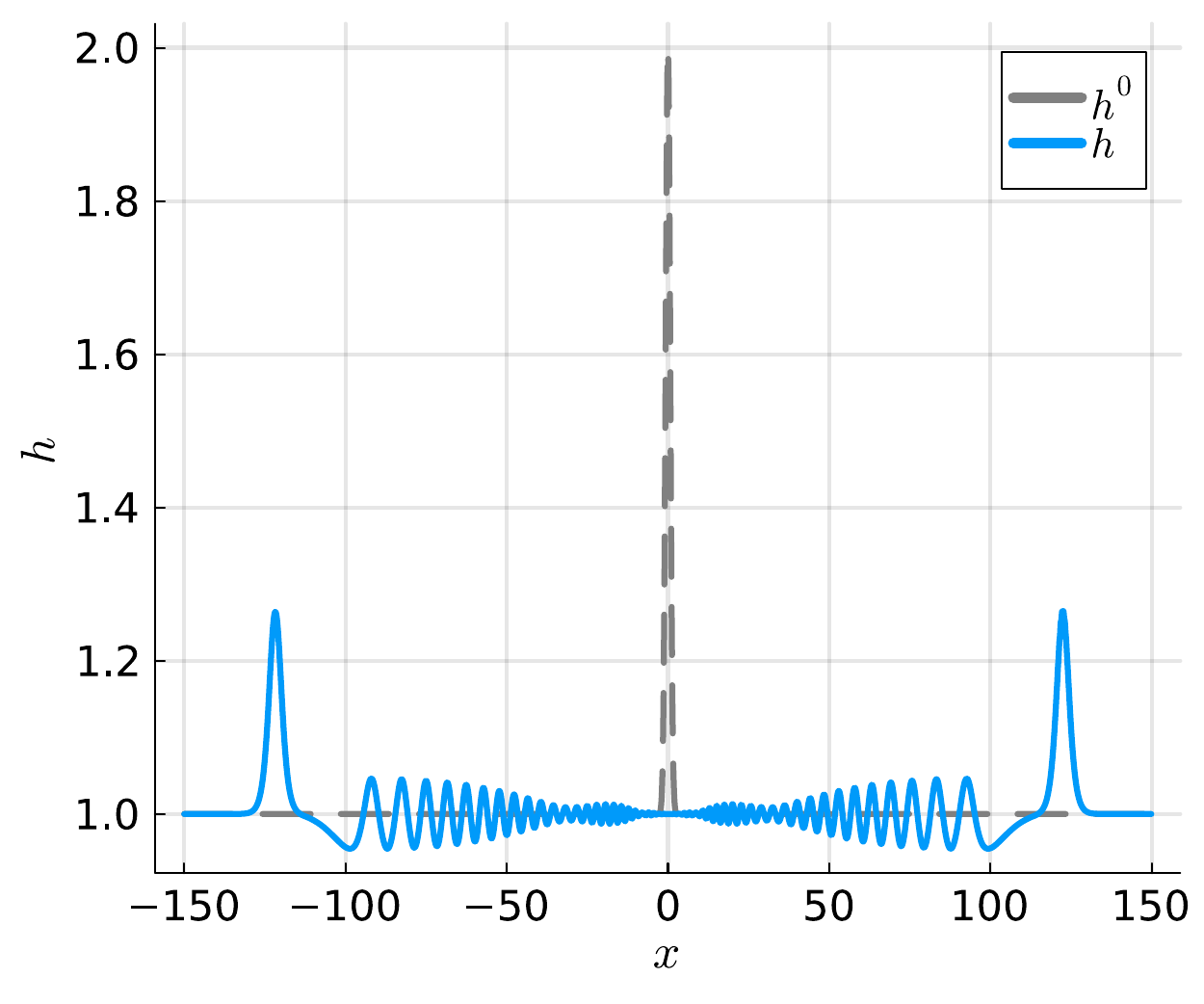}
        \caption{Hyperbolic approximation with $\lambda = 500$.}
    \end{subfigure}%
    \hspace{\fill}
    \begin{subfigure}{0.32\textwidth}
    \centering
        \includegraphics[width=\textwidth]{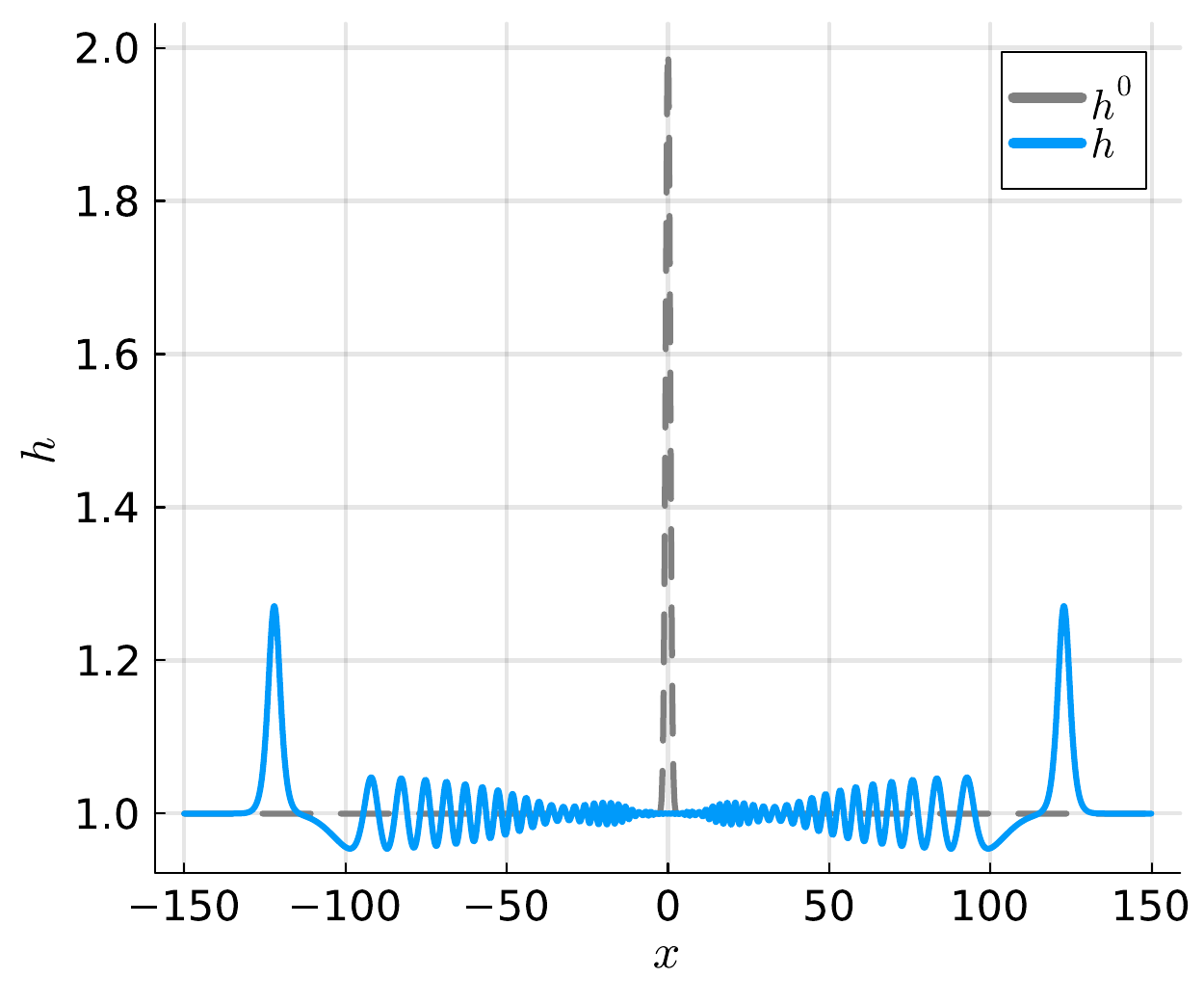}
        \caption{Original Serre-Green-Naghdi equations.}
    \end{subfigure}%
        \hspace{\fill}
        \begin{subfigure}{0.32\textwidth}
    \centering
        \includegraphics[width=\textwidth]{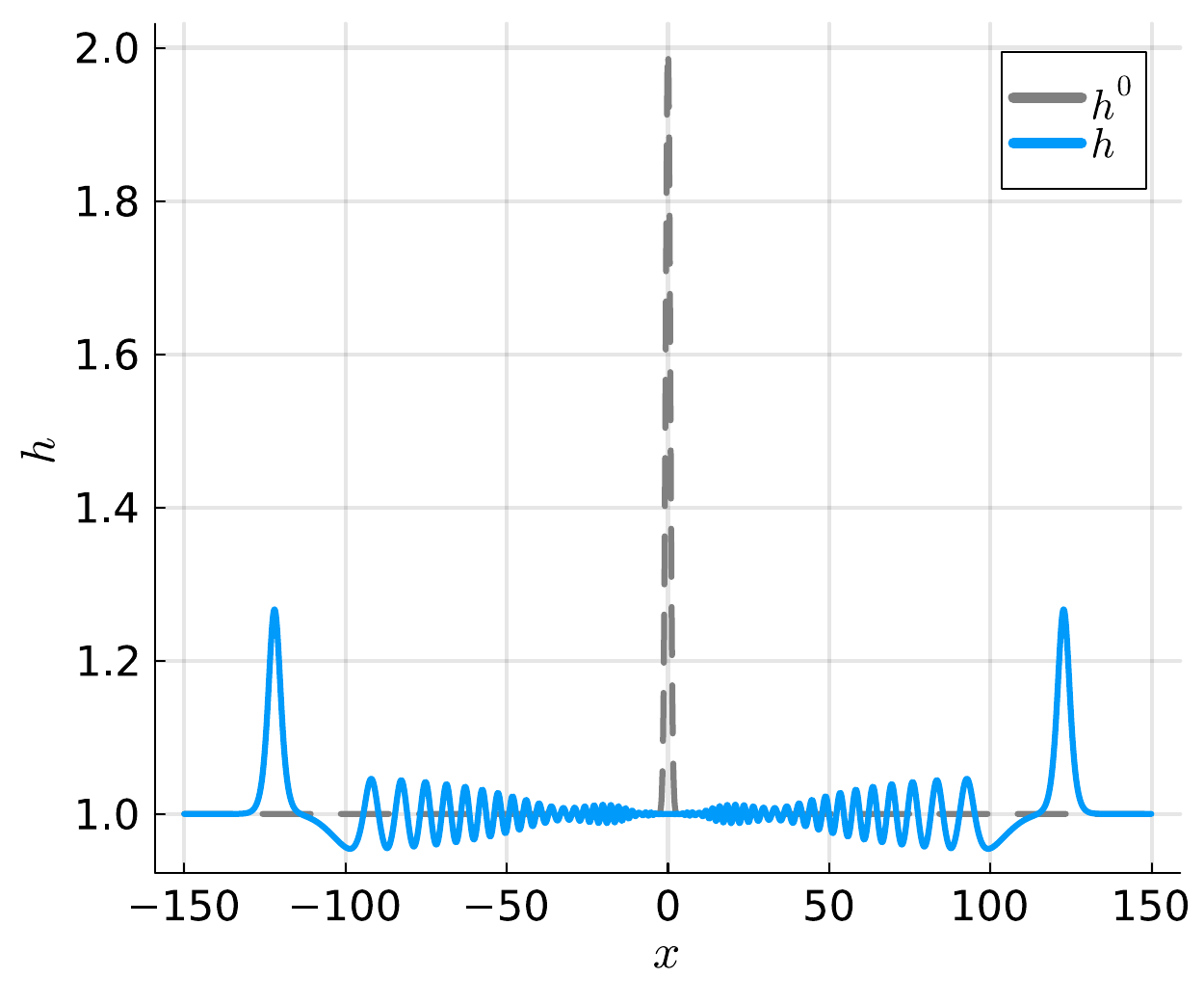}
        \caption{Serre-Green-Naghdi with artificial diffusion.}
    \end{subfigure}%
    \caption{Numerical solutions for the initial condition
             \eqref{eq:initial_condition_gaussian}. These reference solutions
             are all visually converged on the same  grid with $N = 3000$ nodes using
             second-order central finite difference methods.}
    \label{fig:conservation_solution}
\end{figure}

\begin{figure}[htbp]
\centering
    \begin{subfigure}{0.32\textwidth}
    \centering
        \includegraphics[width=\textwidth]{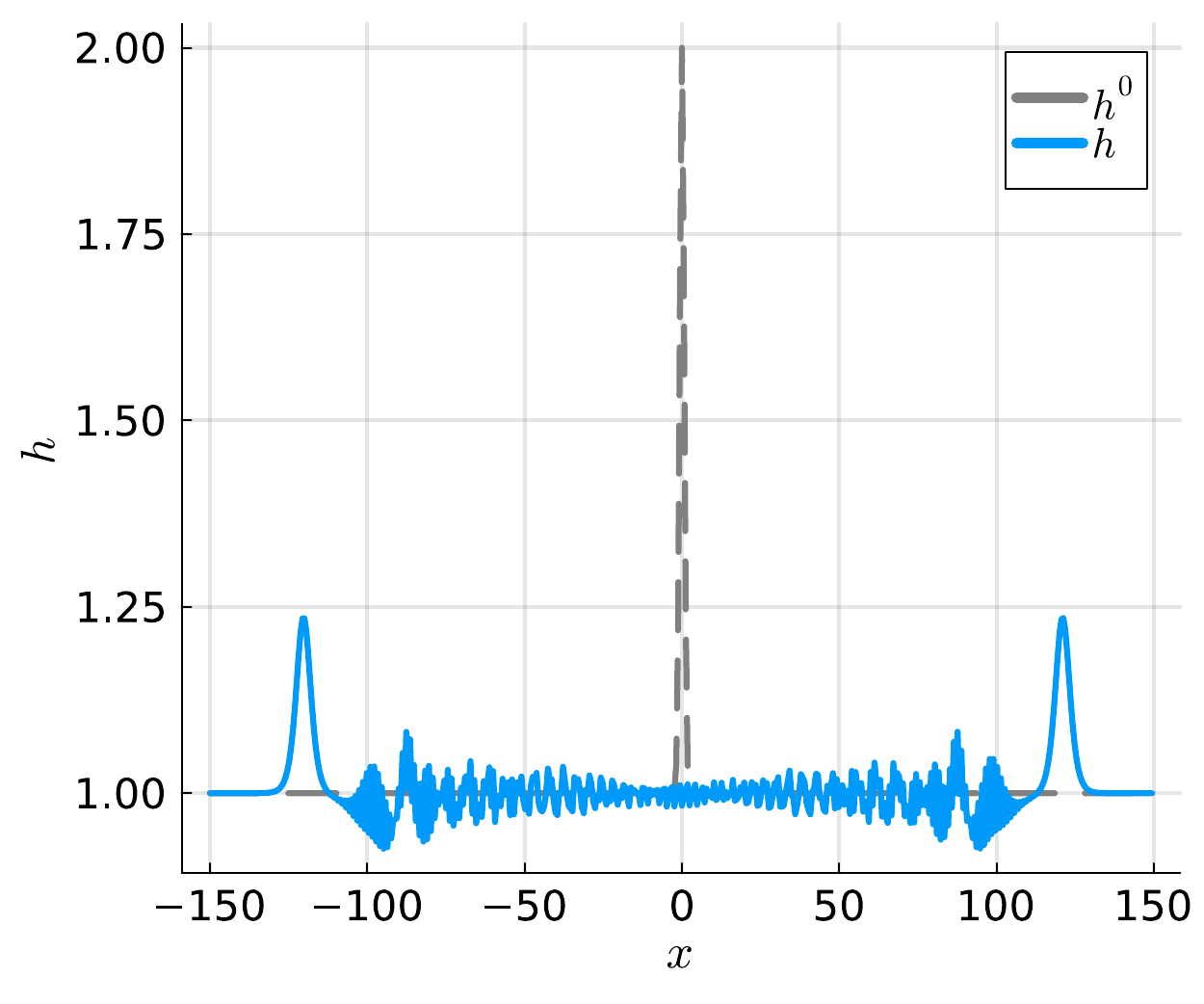}
        \caption{Central operators.}
    \end{subfigure}%
    \hspace{\fill}
    \begin{subfigure}{0.32\textwidth}
    \centering
        \includegraphics[width=\textwidth]{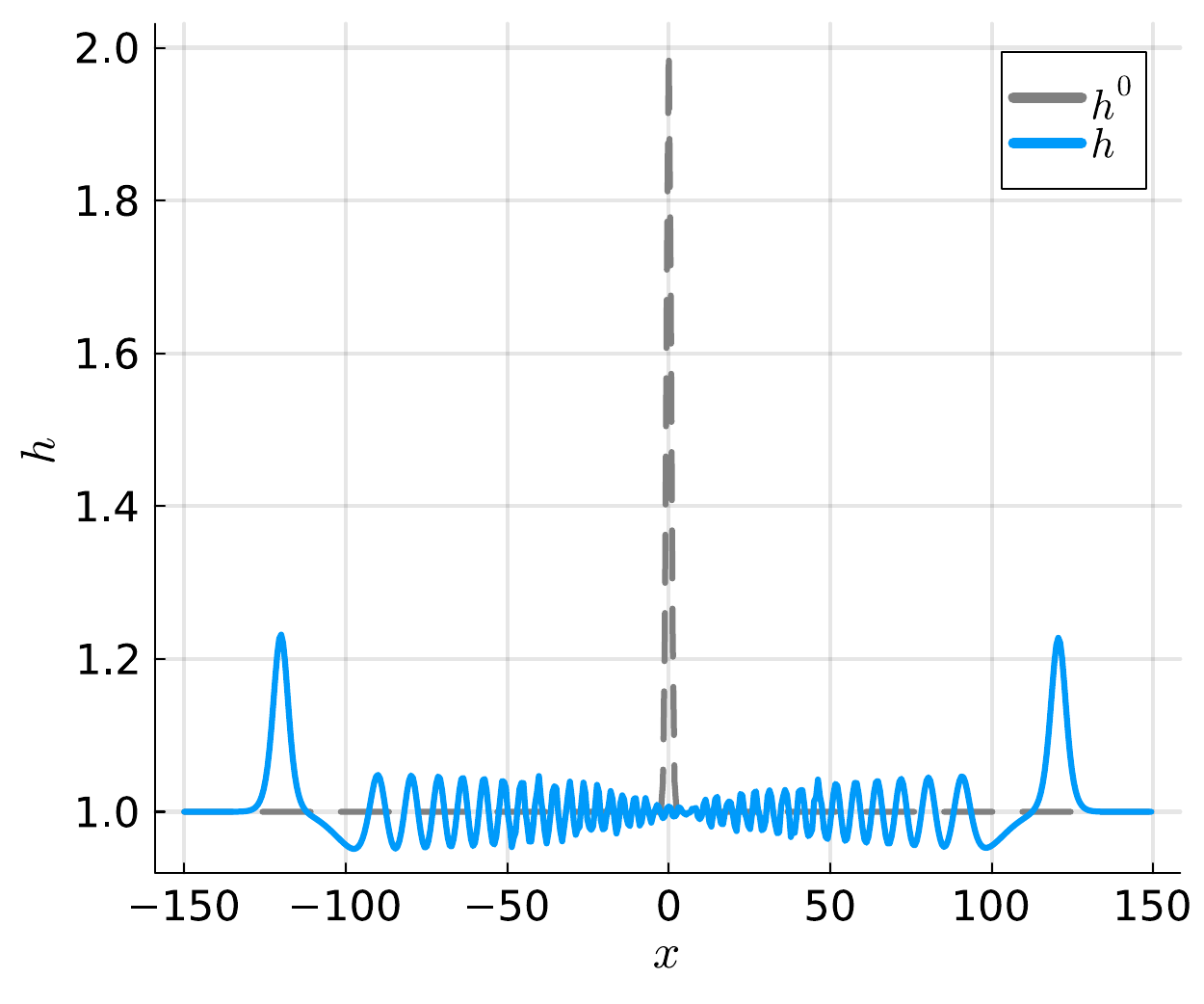}
        \caption{Upwind operators.}
    \end{subfigure}%
    \hspace{\fill}
    \begin{subfigure}{0.32\textwidth}
    \centering
        \includegraphics[width=\textwidth]{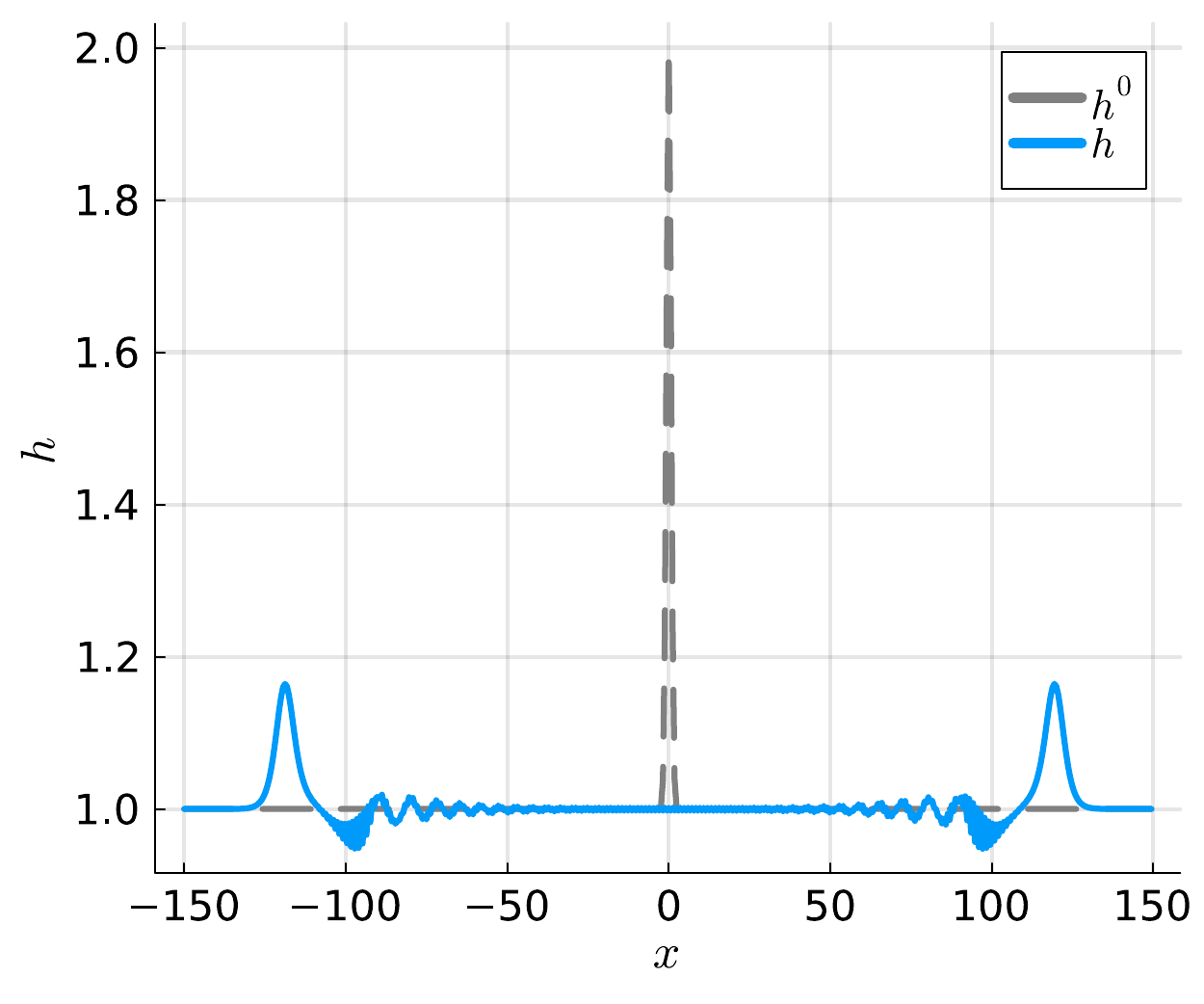}
        \caption{Central operators plus AV.}
    \end{subfigure}%
    \caption{Numerical solutions obtained with second-order finite difference
             methods with $N = 500$ nodes for the original Serre-Green-Naghdi
             equations for the same setup as in
             Figure~\ref{fig:conservation_solution}.}
    \label{fig:upwind_vs_central}
\end{figure}

Next, we compare  three approaches for the original
Serre-Green-Naghdi equations with only $N = 500$ nodes:
central SBP operators, upwind SBP operators, and
central SBP  plus artificial viscosity.
The results are shown in Figure~\ref{fig:upwind_vs_central}.
The central discretization of the Laplacian leads to spurious oscillations   due to under-resolution.
These oscillations, absent in the mesh resolved solutions,  are completely removed by the  upwind structure-preserving discretization.
Artificial viscosity allows to remove some of the oscillations, but also damps the solution everywhere, as visible from the lower
water heights obtained.
This qualitative difference is in accordance with the behavior of
central (wide-stencil) and upwind (narrow-stencil) discretizations of
several time-dependent problems
\cite{mattsson2012summation,mattsson2017diagonal}
including other systems of dispersive wave equations \cite{lampert2024structure}. For the effects of numerical dissipation
one can instead refer to \cite{jouy_etal24}.

\subsubsection{Qualitative comparison: variable bathymetry}

Next, we use a variable bathymetry and
\begin{equation}
\label{eq:initial_condition_gaussian_variable}
    h(x, 0) = 1 + \exp(-x^2) - b(x), \quad u(x, u) = 10^{-2}, \quad b(x) = \frac{\cos(\pi x / 75)}{4}.
\end{equation}
Reference solutions  are shown in
Figure~\ref{fig:conservation_solution_variable}. All methods are visually converged at the same  level
 $N = 3000$ nodes.

\begin{figure}[htbp]
\centering
    \begin{subfigure}{0.49\textwidth}
    \centering
        \includegraphics[width=0.9\textwidth]{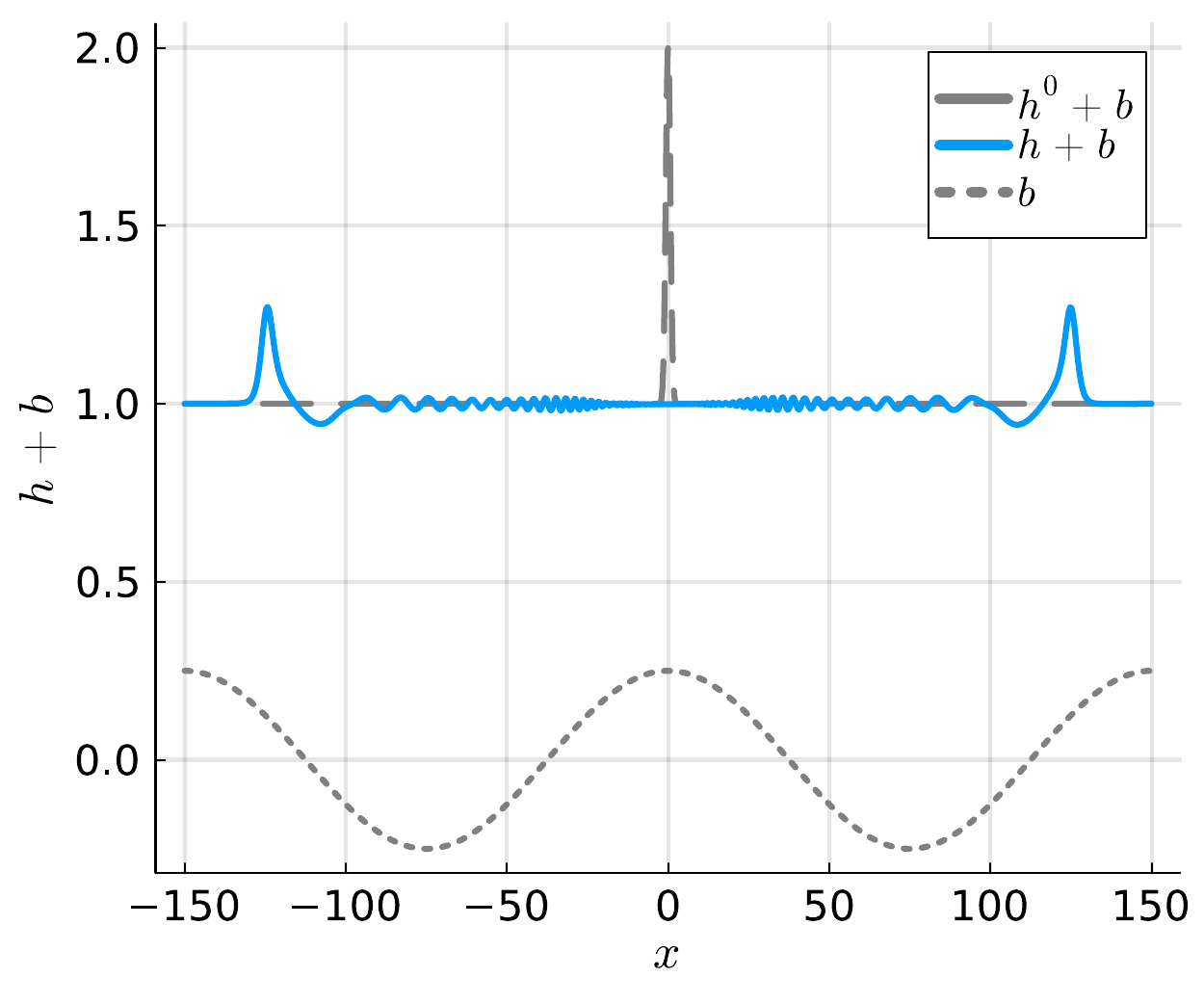}
        \caption{Hyperbolic approximation with $\lambda = 500$.}
    \end{subfigure}%
    \hspace{\fill}
    \begin{subfigure}{0.49\textwidth}
    \centering
        \includegraphics[width=\textwidth]{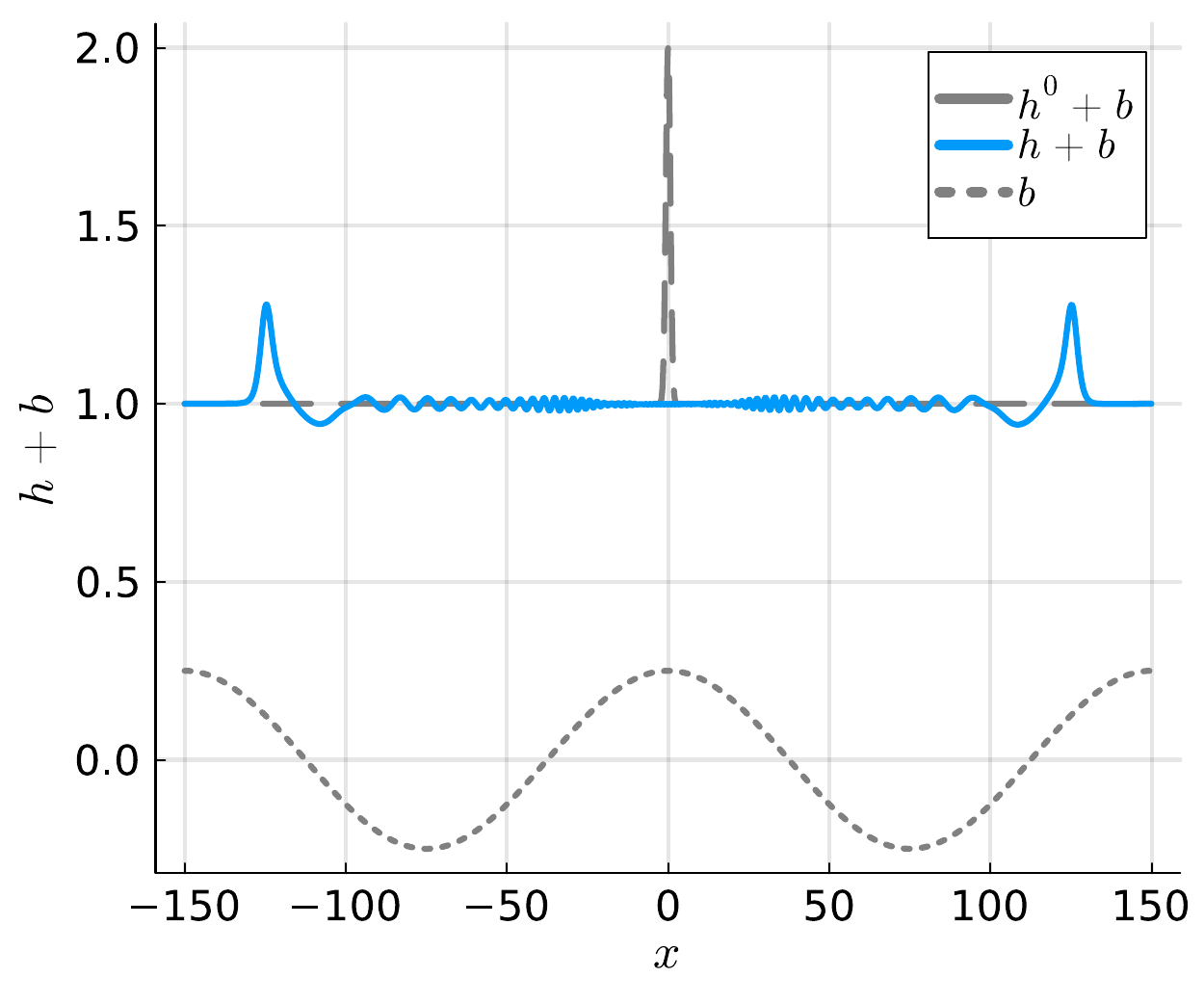}
        \caption{Serre-Green-Naghdi equations (mild slope).}
    \end{subfigure}%
    \\
    \begin{subfigure}{0.49\textwidth}
    \centering
        \includegraphics[width=0.9\textwidth]{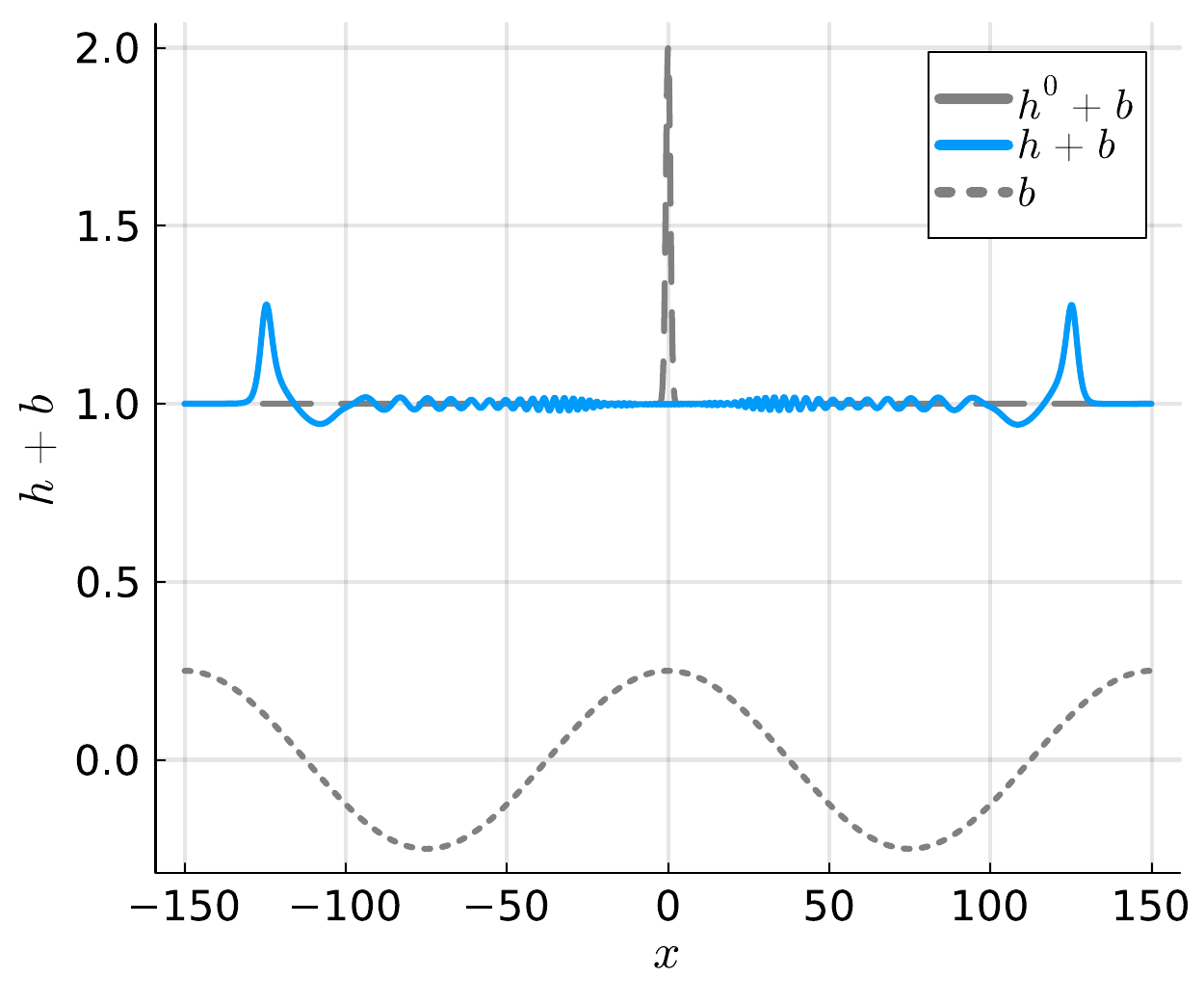}
        \caption{Full Serre-Green-Naghdi equations.}
    \end{subfigure}%
      \hspace{\fill}
    \begin{subfigure}{0.49\textwidth}
    \centering
        \includegraphics[width=0.9\textwidth]{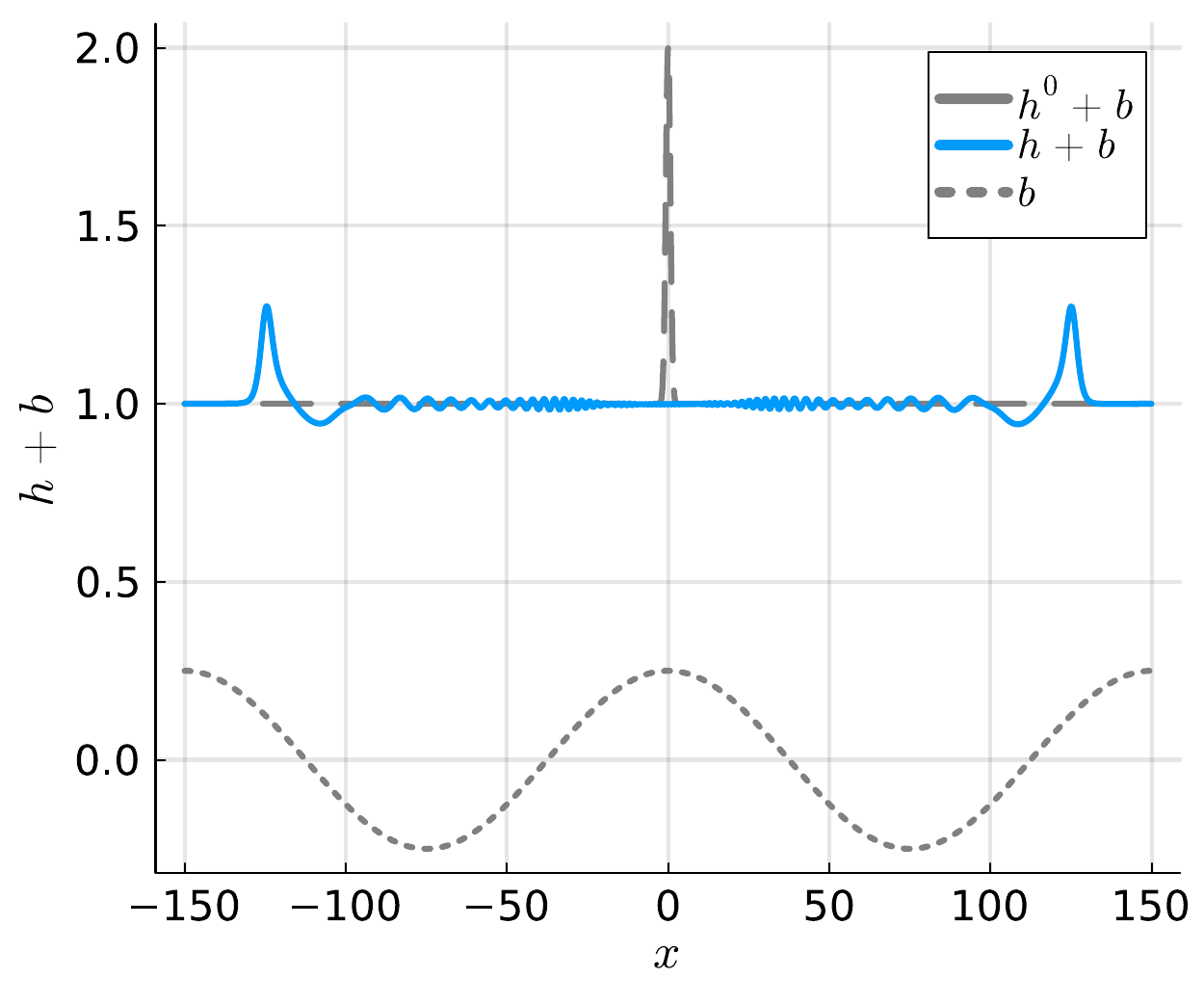}
        \caption{Full Serre-Green-Naghdi plus artificial dissipation.}
    \end{subfigure}
    \caption{Numerical solutions for the initial condition
             \eqref{eq:initial_condition_gaussian_variable}.
             These reference solutions
             are all visually converged on the same grid with $N = 3000$ nodes using
             second-order central finite difference methods.}
    \label{fig:conservation_solution_variable}
\end{figure}

Numerical solutions obtained with second-order finite difference methods
with $N = 500$ nodes for the original Serre-Green-Naghdi equations are shown
in Figure~\ref{fig:upwind_vs_central_variable_mild}. As before, the
central discretization of the Laplacian leads to spurious oscillations.
The oscillations are  in this case removed in both the  upwind SBP method, which is structure-preserving,
and using artificial viscosity. However, as expected, on coarse meshes the latter has dramatic impact on the wave heights obtained,
and on the resolution of secondary waves.

\begin{figure}[htbp]
\centering
    \begin{subfigure}{0.49\textwidth}
    \centering
        \includegraphics[width=0.653\textwidth]{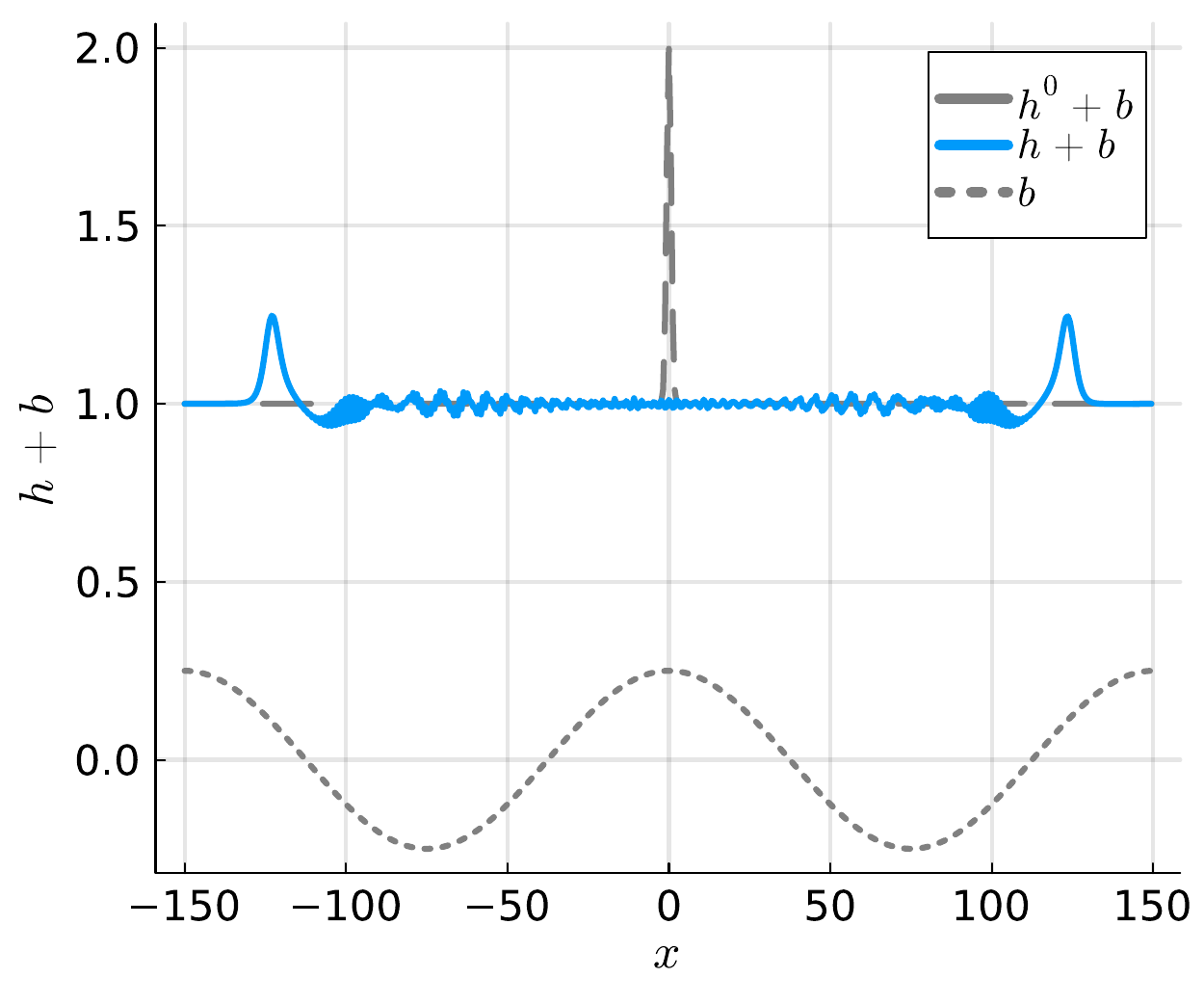}
        \caption{Central operators, mild slope.}
            \end{subfigure}%
    \hspace{\fill}
     \begin{subfigure}{0.49\textwidth}
    \centering
        \includegraphics[width=0.653\textwidth]{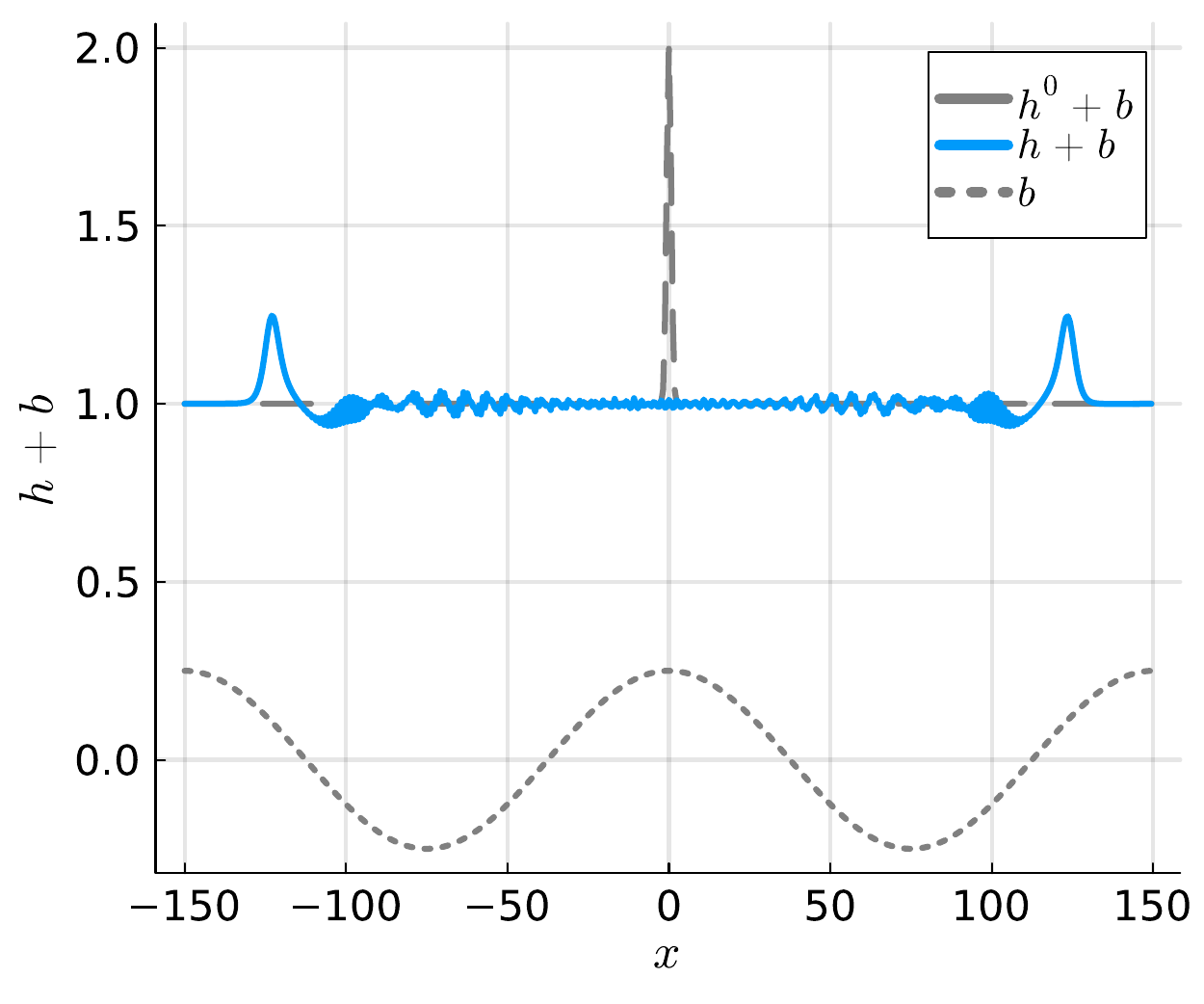}
        \caption{Central operators, full system.}
            \end{subfigure}%
              \\
        \begin{subfigure}{0.32\textwidth}
    \centering
        \includegraphics[width=\textwidth]{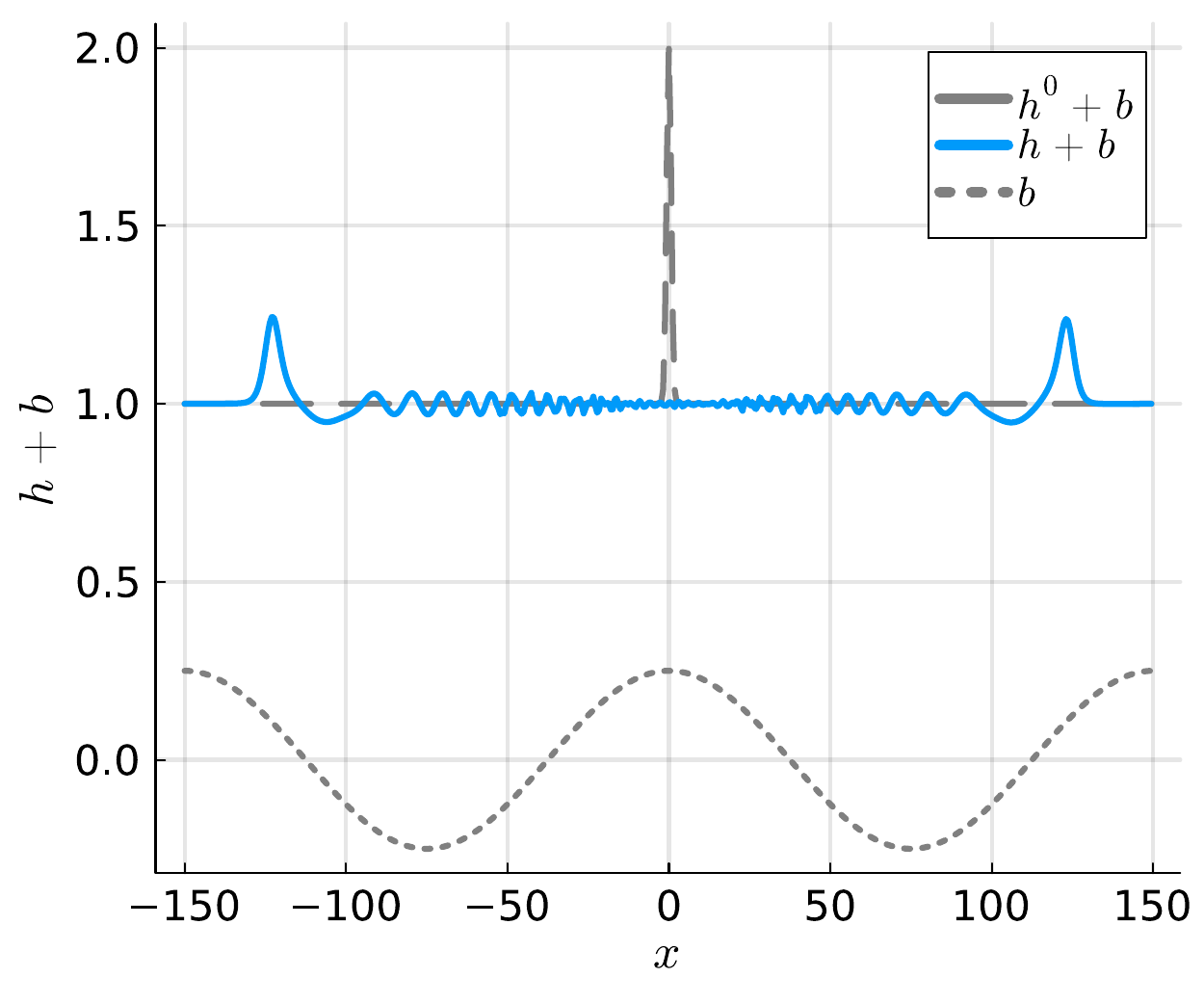}
        \caption{Upwind operators, mild slope.}
    \end{subfigure}%
    \hspace{\fill}
    \begin{subfigure}{0.32\textwidth}
    \centering
        \includegraphics[width=\textwidth]{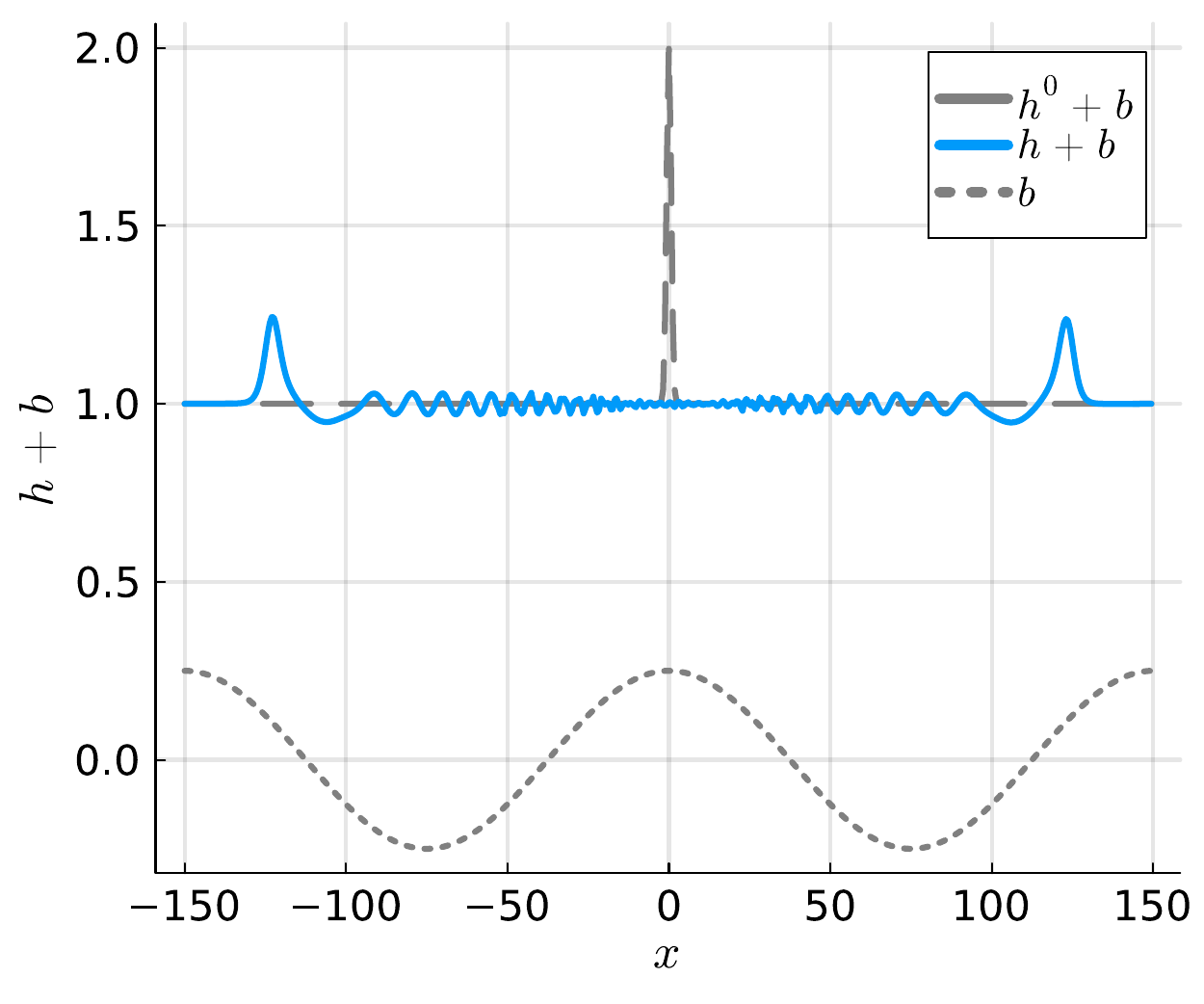}
        \caption{Upwind operators, full system.}
    \end{subfigure}%
\hspace{\fill}
    \begin{subfigure}{0.32\textwidth}
    \centering
        \includegraphics[width=\textwidth]{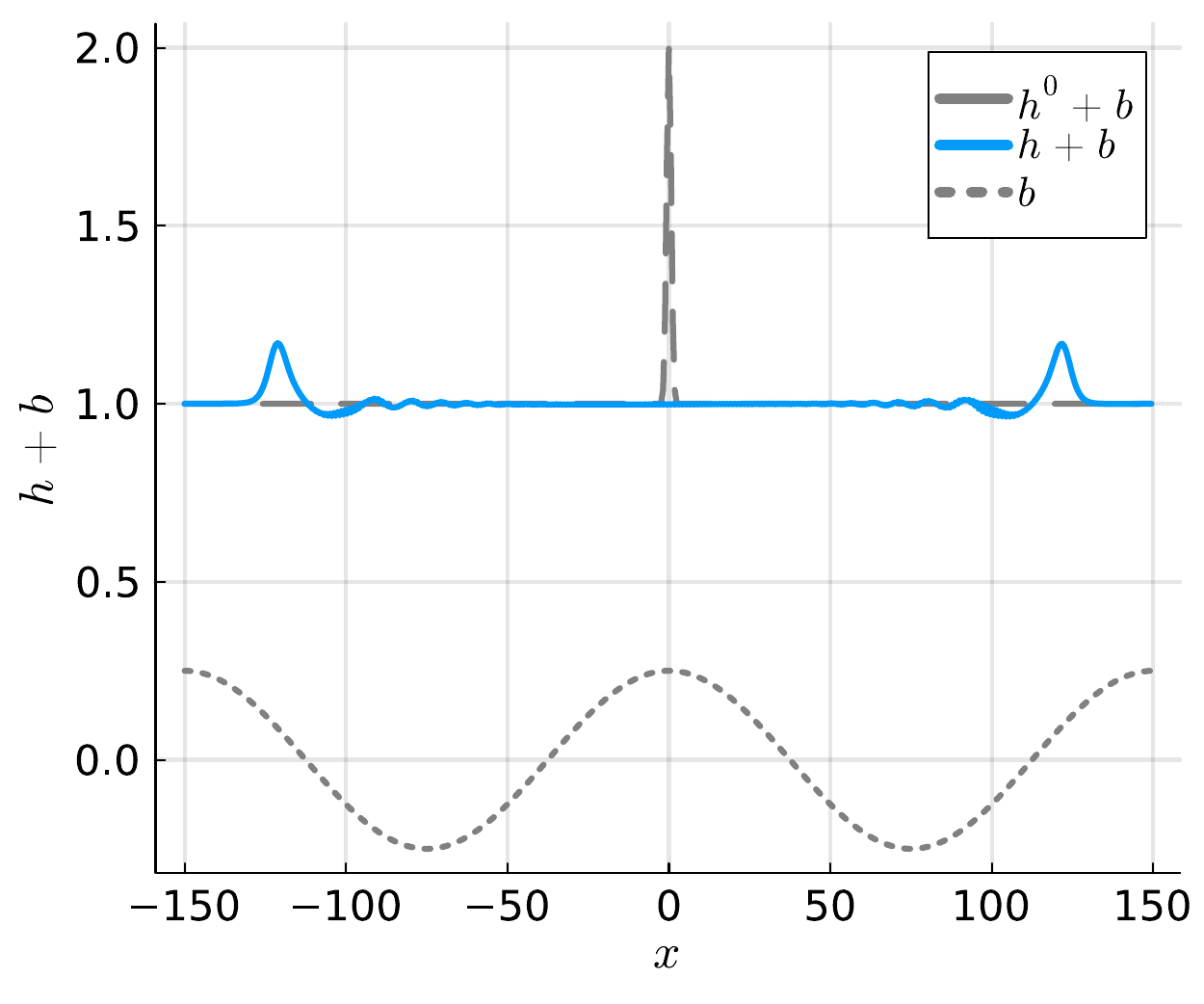}
        \caption{Central  + AV, full system.}
    \end{subfigure}%
    \caption{Numerical solutions obtained with second-order finite difference
             methods with $N = 500$ nodes for the original Serre-Green-Naghdi
             equations for the same setup as in
             Figure~\ref{fig:conservation_solution_variable}.}
    \label{fig:upwind_vs_central_variable_mild}
\end{figure}

\subsection{Conservation of invariants}
\label{sec:conservation_tests}

In this section, we check the conservation of the invariants,
i.e., the total water mass and the total energy. We use again
the initial conditions and setups from Section~\ref{sec:upwind_vs_central}
for constant and variable bathymetry.

\subsubsection{Conservation of invariants: flat bathymetry}

\begin{table}[htbp]
\centering
  \caption{Conservation of invariants
           (error ``Er'' of water mass $\smallint h$,
           momentum $\smallint hu$,
           and energy $\smallint E$)
           for constant bathymetry.}
  \label{tab:conservation_constant_bathymetry}
  \begin{subfigure}{0.49\textwidth}
    \centering
    \caption{Hyperbolic approximation  $\lambda = 500$.\\\textcolor{white}{text here}}
    \begin{small}
    \begin{tabular}{c c c cc}
        \toprule
         $\Delta t$ & Er$\,\smallint\! h$ & Er$\,\smallint\! hu$   & Er$\,\smallint\! E$ & EOC \\
        \midrule
        0.0100 & 2.3e-13 & 1.7e-07 & 6.4e-06 &  \\
        0.0050 & 2.3e-13 & 1.7e-07 & 3.3e-07 & 4.28 \\
        0.0020 & 2.8e-13 & 1.7e-07 & 3.8e-09 & 4.88 \\
        0.0010 & 2.8e-13 & 1.7e-07 & 1.2e-10 & 4.96 \\
        0.0005 & 4.5e-13 & 1.7e-07 & 7.3e-12 & 4.05 \\
        \bottomrule
    \end{tabular}
    \end{small}
  \end{subfigure}
    \hspace{\fill}
  \begin{subfigure}{0.49\textwidth}
    \centering
    \caption{Hyperbolic approximation with $\lambda = 500$ and AV (error of energy with orders 2 and 4).}
    \begin{small}
    \begin{tabular}{c c c cc}
        \toprule
        $\Delta t$ & Er$\,\smallint\! h$ & Er$\,\smallint\! hu$   & \multicolumn{2}{c}{Er$\,\smallint\! E$} \\
        & & & O2 & O4 \\
        \midrule
        0.0100 & 2.3e-13 & 1.5e-07 & 1.4 & 8.9e-02 \\
        0.0050 & 2.3e-13 & 1.5e-07 & 1.4 & 8.9e-02 \\
        0.0020 & 2.8e-13 & 1.5e-07 & 1.4 & 8.9e-02 \\
        0.0010 & 2.8e-13 & 1.5e-07 & 1.4 & 8.9e-02 \\
        0.0005 & 4.0e-13 & 1.5e-07 & 1.4 & 8.9e-02 \\
        \bottomrule
    \end{tabular}
    \end{small}
  \end{subfigure}%
  \\
  \medskip
  \begin{subfigure}{0.49\textwidth}
    \centering
    \caption{Original Serre-Green-Naghdi equations with central operators.}
    \begin{small}
    \begin{tabular}{c c cc cc}
        \toprule
        $\Delta t$ & Er$\,\smallint\! h$ & Er$\,\smallint\! hu$ & EOC & Er$\,\smallint\! E$ & EOC \\
        \midrule
        0.150 & 1.1e-13 & 1.5e-06 &  & 2.8e-04 &  \\
        0.050 & 2.3e-13 & 4.1e-09 & 5.35 & 1.1e-06 & 5.06 \\
        0.020 & 2.8e-13 & 5.4e-11 & 4.73 & 1.4e-08 & 4.76 \\
        0.010 & 4.5e-13 & 1.8e-12 & 4.95 & 4.4e-10 & 4.94 \\
        0.005 & 7.4e-13 & 3.5e-14 & 5.65 & 2.0e-11 & 4.43 \\
        \bottomrule
    \end{tabular}
    \end{small}
  \end{subfigure}%
  \hspace{\fill}
  \begin{subfigure}{0.49\textwidth}
    \centering
   \caption{Original Serre-Green-Naghdi equations with central operators and AV (error of energy with orders 2 and 4).}
    \begin{small}
    \begin{tabular}{c c cc cc}
        \toprule
        $\Delta t$ & Er$\,\smallint\! h$ & Er$\,\smallint\! hu$ & EOC & \multicolumn{2}{c}{Er$\,\smallint\! E$} \\
        & & & & O2 & O4 \\
        \midrule
        0.150 & 1.1e-13 & 8.2e-07 &  & 1.4 & 8.7e-02 \\
        0.050 & 1.7e-13 & 2.7e-09 & 5.19 & 1.4 & 8.8e-02 \\
        0.020 & 2.8e-13 & 3.4e-11 & 4.77 & 1.4 & 8.8e-02 \\
        0.010 & 4.5e-13 & 1.1e-12 & 4.96 & 1.4 & 8.8e-02 \\
        0.005 & 7.4e-13 & 4.1e-14 & 4.76 & 1.4 & 8.8e-02 \\
        \bottomrule
    \end{tabular}
    \end{small}
  \end{subfigure}%
    \\
  \medskip
  \begin{subfigure}{0.49\textwidth}
    \centering
    \caption{Original Serre-Green-Naghdi equations with upwind operators.}
    \begin{small}
    \begin{tabular}{c c cc cc}
        \toprule
        $\Delta t$ & Er$\,\smallint\! h$ & Er$\,\smallint\! hu$ & EOC & Er$\,\smallint\! E$ & EOC \\
        \midrule
        0.150 & 1.1e-13 & 1.7e-05 &  & 2.1e-04 &  \\
        0.050 & 2.3e-13 & 1.2e-08 & 6.57 & 9.7e-07 & 4.88 \\
        0.020 & 2.8e-13 & 9.8e-11 & 5.27 & 1.2e-08 & 4.78 \\
        0.010 & 4.0e-13 & 4.4e-12 & 4.46 & 4.0e-10 & 4.95 \\
        0.005 & 7.4e-13 & 1.7e-13 & 4.69 & 1.9e-11 & 4.39 \\
        \bottomrule
    \end{tabular}
    \end{small}
  \end{subfigure}%
  \hspace{\fill}
  \begin{subfigure}{0.49\textwidth}
    \centering
    \caption{Original Serre-Green-Naghdi equations with upwind operators and AV (error of energy with orders 2 and 4).}
    \begin{small}
    \begin{tabular}{c c cc cc}
        \toprule
        $\Delta t$ & Er$\,\smallint\! h$ & Er$\,\smallint\! hu$ & EOC & \multicolumn{2}{c}{Er$\,\smallint\! E$} \\
        & & & & O2 & O4 \\
        \midrule
        0.150 & 2.3e-13 & 1.4e-05 &   & 1.3 & 8.1e-02 \\
        0.050 & 2.3e-13 & 1.1e-08 & 6.50 & 1.3 & 8.1e-02 \\
        0.020 & 3.4e-13 & 4.4e-11 & 6.03 & 1.3 & 8.1e-02 \\
        0.010 & 4.0e-13 & 2.5e-12 & 4.12 & 1.3 & 8.1e-02 \\
        0.005 & 6.8e-13 & 9.5e-14 & 4.73 & 1.3 & 8.1e-02 \\
        \bottomrule
    \end{tabular}
    \end{small}
  \end{subfigure}%
\end{table}

The conservation of mass, momentum, and energy is  measured on the test of Section~\ref{sec:upwind_vs_central}.
We use  second-order finite differences
with $N = 1000$ nodes, in the spatial domain  $[-150, 150]$, on the time interval $t\in [0,35]$.
The results   are shown in Table~\ref{tab:conservation_constant_bathymetry}.
For all systems, the linear invariant (total water mass) is conserved up to
machine accuracy. For the structure-preserving semidiscretizations, the  total energy error
decreases with the time step size, and the EOC matches the order of accuracy
of the time integration method.  This is expected since,  being a nonlinear invariant, total energy   is
not conserved exactly by the time integration method.   For the original systems  we obtain the same result for the momentum,
which is a nonlinear invariant due to the formulation using  $h u_t$. As already said, this could be avoided  working with $(hu)_t$ but with
considerable overheads in the solution of the elliptic problem. In any case, the momentum EOC matches the order of accuracy
of the time discretization. Conversely, the momentum error for the hyperbolic approximation is the same for all $\Delta t$s
since in this case the semidiscretization in space is non momentum conserving.
The results including AV  show a finite defect in the energy integral, independent of the time step. This error
is  given by $\smallint  \mu(\Delta x,p) h(u_x)^2$ at the final time, and
 reduces when passing from order 2 to order 4 as shown in the tables.

\subsubsection{Conservation of invariants: variable bathymetry}

Next, we use a variable bathymetry as in
\eqref{eq:initial_condition_gaussian_variable}.
The other parameters are still the same as before. The results shown in
Table~\ref{tab:conservation_variable_bathymetry} show similar behavior to
 the constant bathymetry case. To save space we do note report here the errors
when including artificial dissipation, which also behave very similarly as in the previous sub-section.
We do report a table with the energy errors at final time $t=35$ for second- and fourth-order schemes showing some small dependence
of the error on the formulation, but again mostly on the order (and mesh).

\begin{table}[htbp]
\centering
  \caption{Conservation of invariants
           (errors ``Er'' of water mass $\smallint h$
           and energy $\smallint E$)
           for variable bathymetry.
           The spatial discretizations use second-order finite differences
           with $N = 1000$ nodes in the interval $[-150, 150]$.}
  \label{tab:conservation_variable_bathymetry}
    \begin{subfigure}{0.49\textwidth}
    \centering
    \caption{Conservation of  energy for second- and fourth-order finite difference schemes with AV.}
   \begin{tabular}{ccc}
        \toprule
        Formulation  & \multicolumn{2}{c}{Er$\,\smallint\! E$} \\
        & O2 & O4\\
        \midrule
        Hyp. $\lambda = 500$     & 1.20 & 7.49e-02 \\
        Orig. mild slope central & 1.20 & 7.27e-02 \\
        Orig. mild slope upwind  & 1.19 & 7.27e-02 \\
        Orig. full central       & 1.20 & 7.27e-02 \\
        Orig. full upwind        & 1.19 & 7.27e-02 \\
        \bottomrule
    \end{tabular}
   \end{subfigure}%
     \hspace{\fill}
  \begin{subfigure}{0.49\textwidth}
    \centering
    \caption{Hyperbolic approximation with $\lambda = 500$.\\{\color{white}text here}}
    \begin{tabular}{c c cc}
        \toprule
        $\Delta t$ & Er\ $\smallint h$ & Er\ $\smallint E$ & EOC \\
        \midrule
        0.0100 & 2.3e-13 & 6.5e-06 &  \\
        0.0050 & 2.3e-13 & 3.8e-07 & 4.12 \\
        0.0020 & 2.3e-13 & 4.4e-09 & 4.87 \\
        0.0010 & 3.4e-13 & 1.4e-10 & 4.97 \\
        0.0005 & 2.8e-13 & 5.5e-12 & 4.68 \\
        \bottomrule
    \end{tabular}
  \end{subfigure}%
  \\
  \medskip
  \begin{subfigure}{0.49\textwidth}
    \centering
    \caption{Serre-Green-Naghdi equations (mild slope)
             with central operators.}
    \begin{tabular}{c c cc}
        \toprule
        $\Delta t$ & Er\ $\smallint h$ & Er\ $\smallint E$ & EOC \\
        \midrule
        0.150 & 2.3e-13 & 8.9e-04 &  \\
        0.050 & 2.3e-13 & 1.4e-06 & 5.84 \\
        0.020 & 2.3e-13 & 2.1e-08 & 4.62 \\
        0.010 & 2.3e-13 & 6.9e-10 & 4.93 \\
        0.005 & 2.8e-13 & 2.2e-11 & 4.97 \\
        \bottomrule
    \end{tabular}
  \end{subfigure}%
  \hspace{\fill}
  \begin{subfigure}{0.49\textwidth}
    \centering
    \caption{Serre-Green-Naghdi equations (mild slope)
             with upwind operators.}
    \begin{tabular}{c c cc}
        \toprule
        $\Delta t$ & Er\ $\smallint h$ & Er\ $\smallint E$ & EOC \\
        \midrule
        0.150 & 1.7e-13 & 6.5e-04 &  \\
        0.050 & 2.3e-13 & 1.2e-06 & 5.70 \\
        0.020 & 2.3e-13 & 1.7e-08 & 4.65 \\
        0.010 & 2.3e-13 & 5.7e-10 & 4.94 \\
        0.005 & 2.3e-13 & 1.9e-11 & 4.94 \\
        \bottomrule
    \end{tabular}
  \end{subfigure}%
  \\
  \medskip
  \begin{subfigure}{0.49\textwidth}
    \centering
    \caption{Full Serre-Green-Naghdi equations
             with central operators.}
    \begin{tabular}{c c cc}
        \toprule
        $\Delta t$ & Er\ $\smallint h$ & Er\ $\smallint E$ & EOC \\
        \midrule
        0.150 & 2.3e-13 & 8.9e-04 &  \\
        0.050 & 2.3e-13 & 1.4e-06 & 5.84 \\
        0.020 & 2.3e-13 & 2.1e-08 & 4.62 \\
        0.010 & 2.3e-13 & 6.9e-10 & 4.94 \\
        0.005 & 3.4e-13 & 2.2e-11 & 4.94 \\
        \bottomrule
    \end{tabular}
  \end{subfigure}%
  \hspace{\fill}
  \begin{subfigure}{0.49\textwidth}
    \centering
    \caption{Full Serre-Green-Naghdi equations
             with upwind operators.}
    \begin{tabular}{c c cc}
        \toprule
        $\Delta t$ & Er\ $\smallint h$ & Er\ $\smallint E$ & EOC \\
        \midrule
        0.150 & 1.7e-13 & 6.5e-04 &  \\
        0.050 & 2.3e-13 & 1.2e-06 & 5.70 \\
        0.020 & 2.3e-13 & 1.7e-08 & 4.65 \\
        0.010 & 2.3e-13 & 5.7e-10 & 4.93 \\
        0.005 & 2.8e-13 & 1.9e-11 & 4.89 \\
        \bottomrule
    \end{tabular}
  \end{subfigure}%
\end{table}

\subsection{Well-balancedness}

We also check the well-balancedness of the methods. For this, we use the
initial condition
\begin{equation}
    h(x, 0) = 1 - b(x), \quad u(x, u) = 0, \quad b(x) = \frac{\cos(\pi x / 75)}{4},
\end{equation}
in the interval $[-150, 150]$ with $N = 1000$ nodes.
We report in Table~\ref{tab:well_balancedness} the results obtained without artificial viscosity.
The ones obtained with this term added are identical down to machine accuracy and are omitted to save space.
The results shown  confirm the well-balancedness of the semidiscretizations.

\begin{table}[htbp]
\centering
  \caption{Discrete $L^2$ norm of the ODE RHS for the well-balancedness
           test case discretized with central finite difference methods
           of different orders of accuracy $p$.}
  \label{tab:well_balancedness}
  \begin{subfigure}{0.49\textwidth}
    \centering
    \caption{Hyperbolic approximation with $\lambda = 500$.}
    \begin{tabular}{c cccc}
        \toprule
        Order $p$ & $h$ & $v$ & $w$ & $\eta$ \\
        \midrule
        2 & 0.0e+00 & 1.9e-14 & 0.0e+00 & 0.0e+00 \\
        4 & 0.0e+00 & 2.6e-14 & 0.0e+00 & 0.0e+00 \\
        6 & 0.0e+00 & 3.0e-14 & 0.0e+00 & 0.0e+00 \\
        \bottomrule
    \end{tabular}
  \end{subfigure}%
  \hspace{\fill}
  \begin{subfigure}{0.49\textwidth}
    \centering
    \caption{Hyperbolic approximation with $\lambda = 5000$.}
    \begin{tabular}{c cccc}
        \toprule
        Order $p$ & $h$ & $v$ & $w$ & $\eta$ \\
        \midrule
        2 & 0.0e+00 & 1.9e-14 & 0.0e+00 & 0.0e+00 \\
        4 & 0.0e+00 & 2.6e-14 & 0.0e+00 & 0.0e+00 \\
        6 & 0.0e+00 & 3.0e-14 & 0.0e+00 & 0.0e+00 \\
        \bottomrule
    \end{tabular}
  \end{subfigure}%
  \\
  \medskip
  \begin{subfigure}{0.24\textwidth}
    \centering
    \caption{Mild slope Serre-Green-Naghdi,
              central operators}
    \begin{tabular}{c cc}
        \toprule
        $p$ & $h$ & $v$ \\
        \midrule
        2 & 0.0e+00 & 5.2e-15 \\
        4 & 0.0e+00 & 5.5e-15 \\
        6 & 0.0e+00 & 5.9e-15 \\
        \bottomrule
    \end{tabular}
  \end{subfigure}%
  \hspace{\fill}
  \begin{subfigure}{0.24\textwidth}
    \centering
    \caption{Mild slope Serre-Green-Naghdi,
              upwind operators}
    \begin{tabular}{c cc}
        \toprule
        $p$ & $h$ & $v$ \\
        \midrule
        2 & 0.0e+00 & 4.2e-15 \\
        4 & 0.0e+00 & 5.0e-15 \\
        6 & 0.0e+00 & 5.2e-15 \\
        \bottomrule
    \end{tabular}
  \end{subfigure}%
  \hspace{\fill}
  \begin{subfigure}{0.24\textwidth}
    \centering
    \caption{Full Serre-Green-Naghdi,
              central operators}
    \begin{tabular}{c cc}
        \toprule
        $p$ & $h$ & $v$ \\
        \midrule
        2 & 0.0e+00 & 5.2e-15 \\
        4 & 0.0e+00 & 5.5e-15 \\
        6 & 0.0e+00 & 5.9e-15 \\
        \bottomrule
    \end{tabular}
  \end{subfigure}%
  \hspace{\fill}
  \begin{subfigure}{0.24\textwidth}
    \centering
    \caption{Full Serre-Green-Naghdi,
              upwind operators}
    \begin{tabular}{c cc}
        \toprule
        $p$ & $h$ & $v$ \\
        \midrule
        2 & 0.0e+00 & 4.2e-15 \\
        4 & 0.0e+00 & 5.0e-15 \\
        6 & 0.0e+00 & 5.2e-15 \\
        \bottomrule
    \end{tabular}
  \end{subfigure}%
\end{table}

\subsection{Error growth of solitons of the Serre-Green-Naghdi equations}
\label{sec:error_growth}

We study the error growth in long-time simulations of solitary waves.
We use the   setup of Section~\ref{sec:convergence_soliton}, and  apply Fourier pseudospectral
methods in space with $N = 2^7$ nodes. We choose the final time such that
the soliton has traveled through the domain 20 times.
We solve the nonlinear scalar equation for relaxation to conserve the energy
using the ITP method \cite{oliveira2020enhancement}. The results are shown in Figure~\ref{fig:error_growth}.
Energy is only conserved exactly with relaxation, and in this case
it is conserved up to the accuracy of the nonlinear scalar solver.
In this case we also see a linear growth in time of the error, while
the error of the baseline structure-preserving method grows quadratically.
This behavior  has also been observed for other nonlinear
dispersive systems
\cite{frutos1997accuracy,duran2000numerical,ranocha2021broad}.
It can be explained using the theory of relative equilibrium solutions
\cite{duran1998numerical}: the SGN equations can be
expressed as Hamiltonian system \cite{li2002hamiltonian} with the total
energy as Hamiltonian $\mathcal{H}$. However, there is another invariant
$\mathcal{Q}$ of the SGN equations and  solitary wave
solutions are critical points of the functional $\mathcal{H} - c \mathcal{Q}$
\cite{li2002hamiltonian}. Thus, the basic structure of relative equilibrium
solutions of \cite{duran1998numerical} is satisfied and we can expect a
quadratic error growth for general time integration methods and a linear error
growth for methods conserving the total energy.

\begin{figure}[htbp]
\centering
    \begin{subfigure}{0.49\textwidth}
    \centering
        \includegraphics[width=0.9\textwidth]{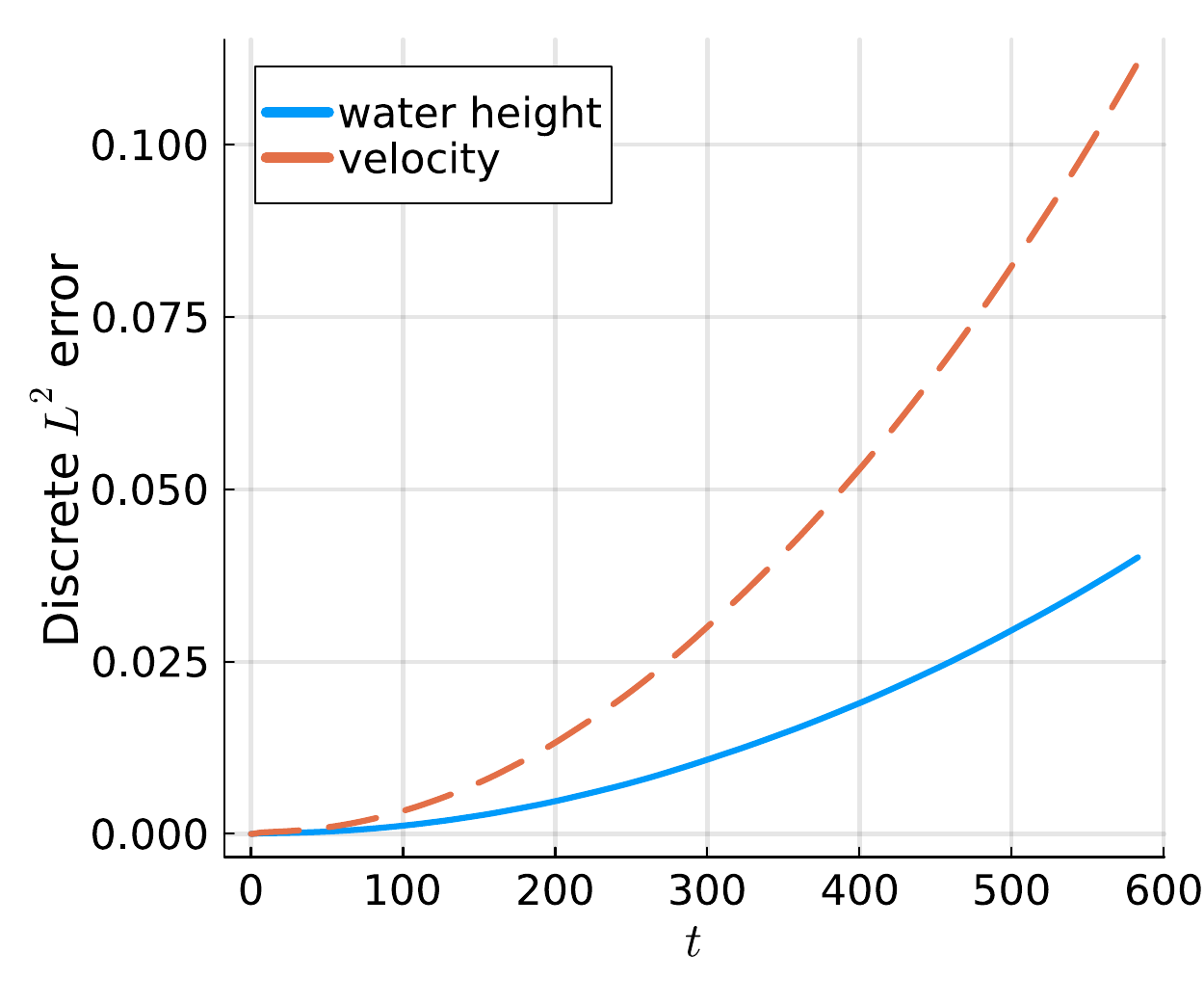}
        \caption{Error without relaxation.}
    \end{subfigure}%
    \hspace{\fill}
    \begin{subfigure}{0.49\textwidth}
    \centering
        \includegraphics[width=0.9\textwidth]{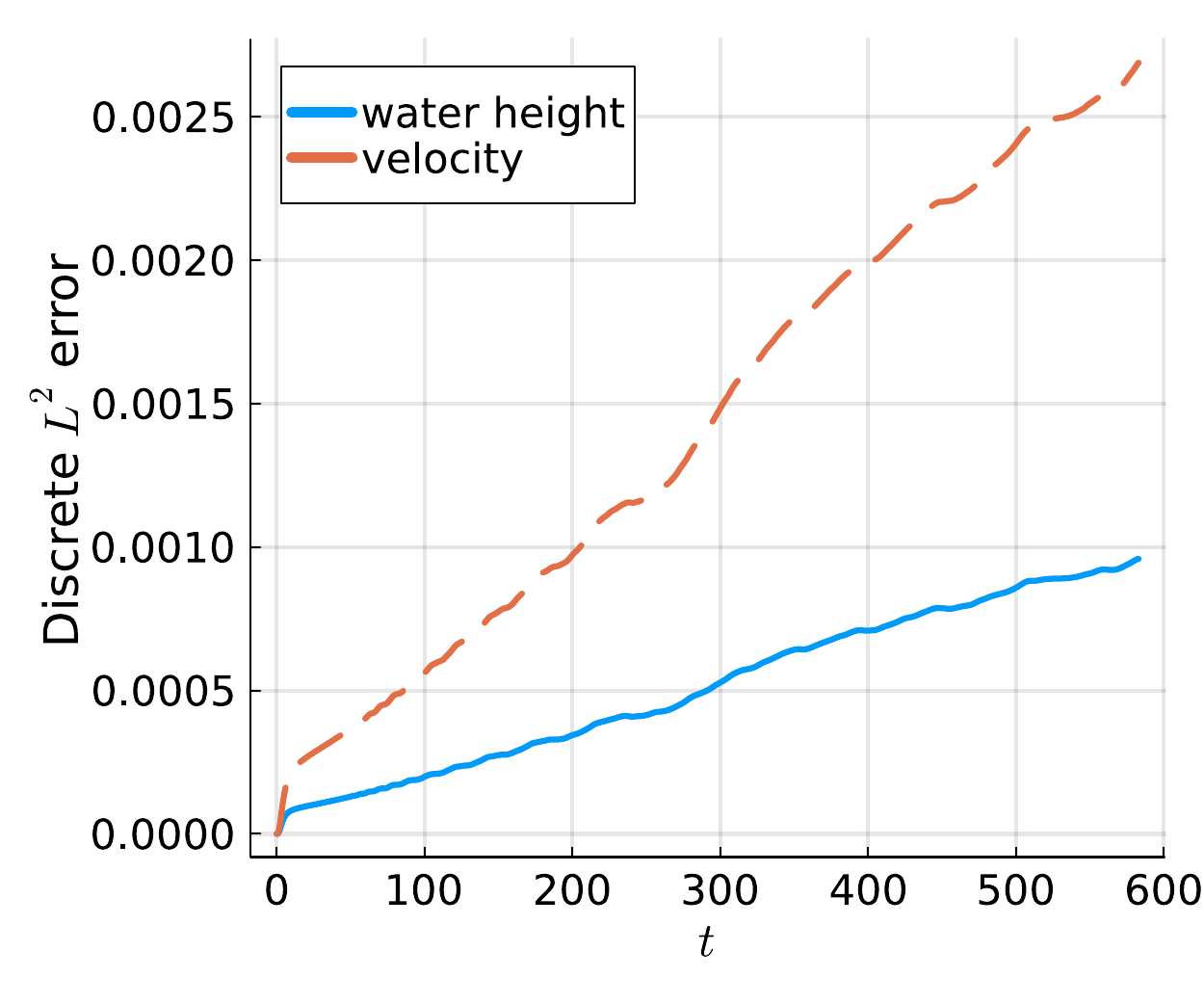}
        \caption{Error with relaxation.}
    \end{subfigure}%
    \\
    \medskip
    \begin{subfigure}{0.49\textwidth}
    \centering
        \includegraphics[width=0.9\textwidth]{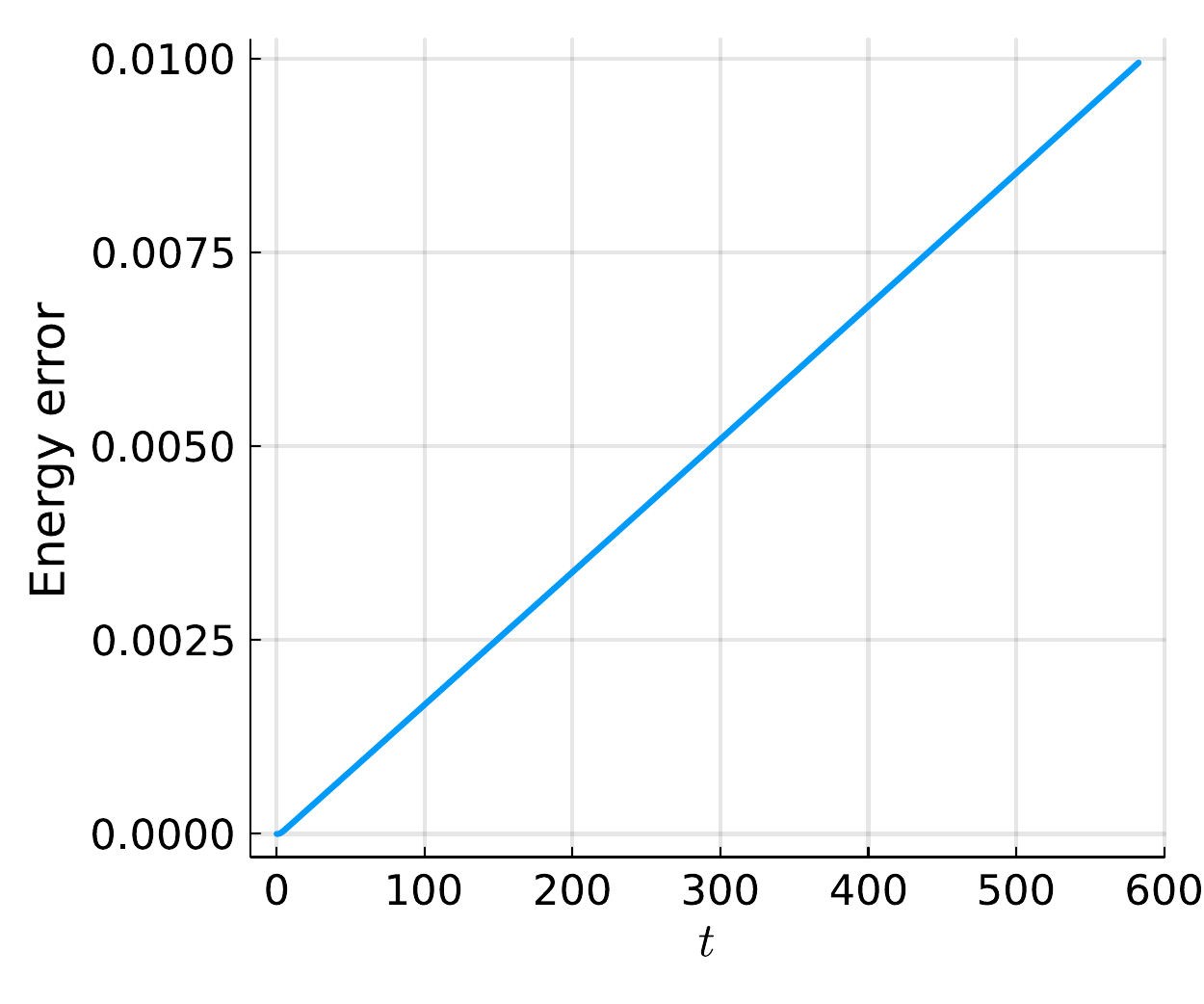}
        \caption{Energy without relaxation.}
    \end{subfigure}%
    \hspace{\fill}
    \begin{subfigure}{0.49\textwidth}
    \centering
        \includegraphics[width=0.9\textwidth]{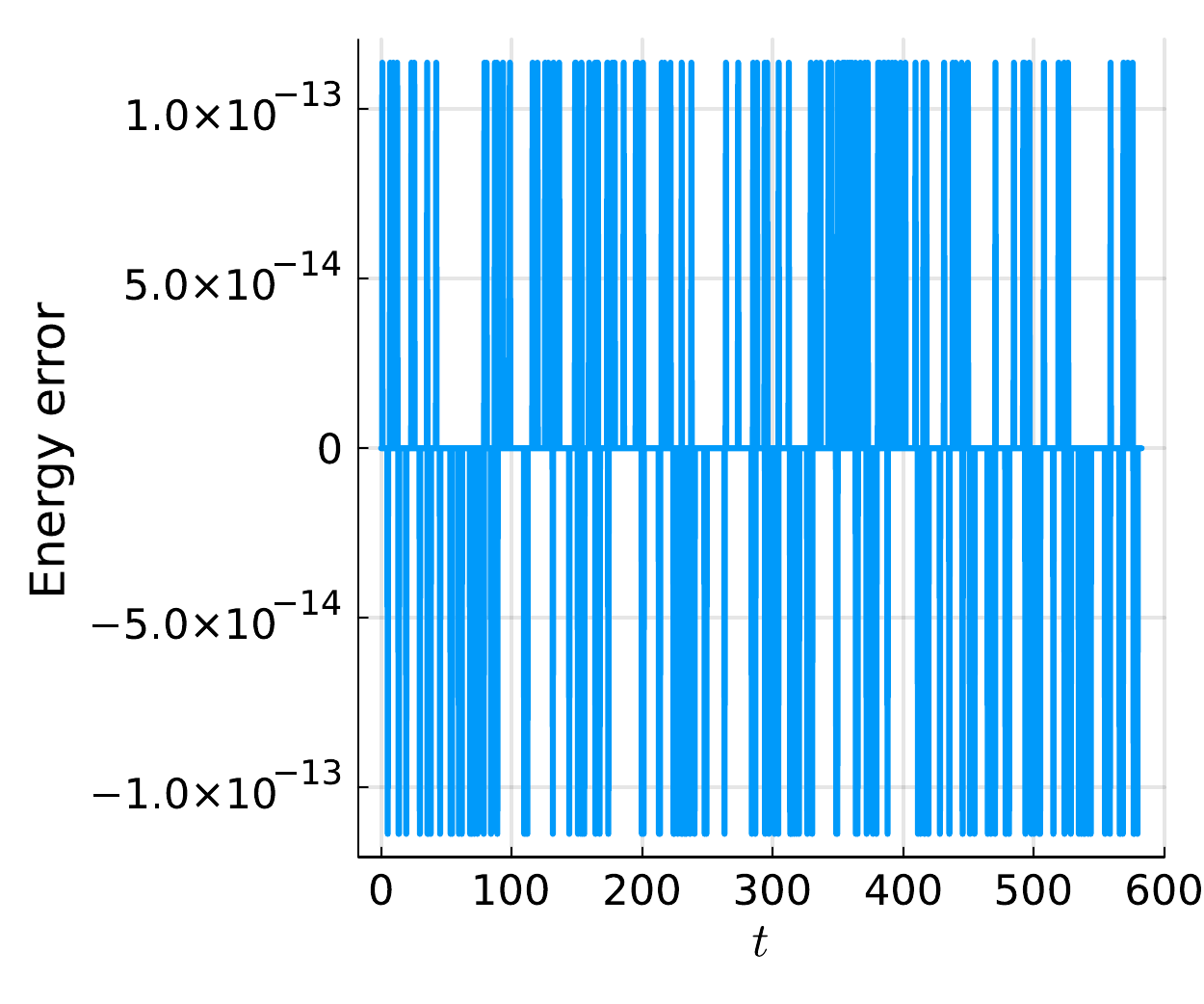}
        \caption{Energy with relaxation.}
    \end{subfigure}%
    \caption{Long-time soliton propagation for the Serre-Green-Naghdi
             equations. Energy-conserving Fourier
             pseudospectral methods, and fifth-order Runge-Kutta method of \cite{tsitouras2011runge}
             with and without relaxation to conserve the energy.
             Top: discrete $L^2$ errors of the numerical
             solutions. Bottom: energy error after 20 periods.}
    \label{fig:error_growth}
\end{figure}

\subsection{Riemann problem}
\label{sec:riemann_problem}

We consider a Riemann problem following the setup by \cite{TKACHENKO2023111901}.
 We use a smoothed initial profile
\begin{equation}
    h(x,0) =h_R+\dfrac{h_L-h_R}{2} \bigl( 1 - \tanh(x / \alpha) \bigr),\quad
        u(x,0) =0,
\end{equation}
with $\alpha = 2$.
The analysis of Riemann invariants of the shallow-water system,
coupled with the analysis of the Whitham system for the
SGN equations \cite{el06,gnst20,TKACHENKO2023111901,el05}
allow to recover the approximate values $(h^*,u^*)$ of the mean flow
dividing the rarefaction wave and the dispersive shock zones as
\begin{equation}
h^* =\dfrac{(\sqrt{h_L}+\sqrt{h_R})^2}{4},
\quad  u^* = 2( \sqrt{gh^*}-\sqrt{gh_R}),\quad a^+=\delta_0-\dfrac{1}{12}\delta_0^2+O(\delta_0^3),
\end{equation}
where $a^+$ is the second-order asymptotic approximation of the amplitude
of the leading soliton
and $\delta_0=|h_R-h_L|$.  We take $h_L = 1.8$ and $h_R = 1.0$, and solve
the problem until  time
$t = 47.434$  on a large domain  $[-600, 600]$. Only
the results in the interval $[-300, 300]$ are retained.
Figure~\ref{fig:riemann_problem} shows  solutions  for the hyperbolic
approximation and the original SGN equations obtained with
structure-preserving central finite difference operators.
The  results from the two systems agree very well with each
other, the analytical predictions, and the numerical results of
\cite{gnst20,pitt2018behaviour}.

\begin{figure}[htbp]
\centering
    \begin{subfigure}{0.49\textwidth}
    \centering
        \includegraphics[width=0.9\textwidth]{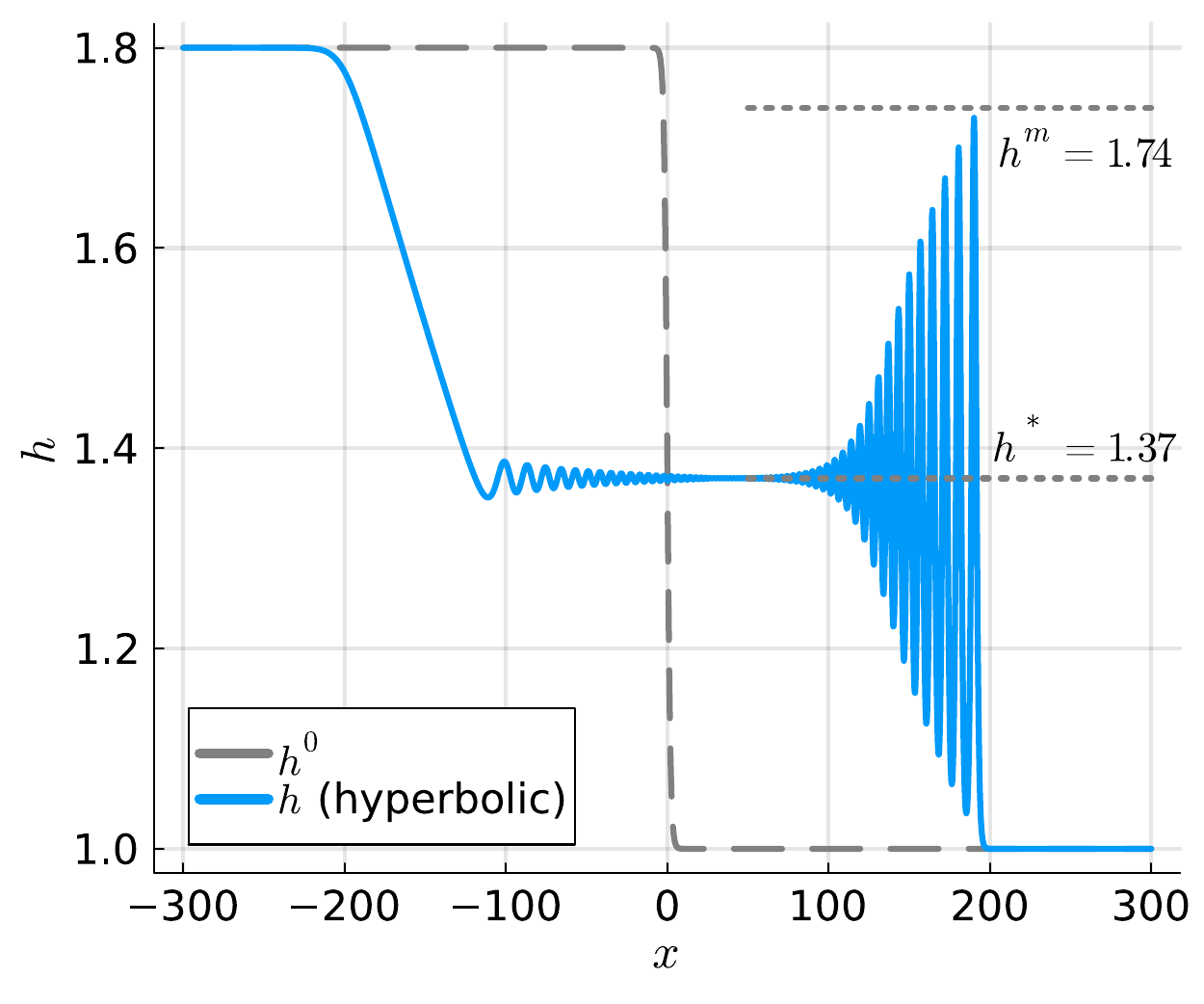}
        \caption{Hyperbolic approximation with $\lambda = 500$.}
    \end{subfigure}%
    \hspace{\fill}
    \begin{subfigure}{0.49\textwidth}
    \centering
        \includegraphics[width=0.9\textwidth]{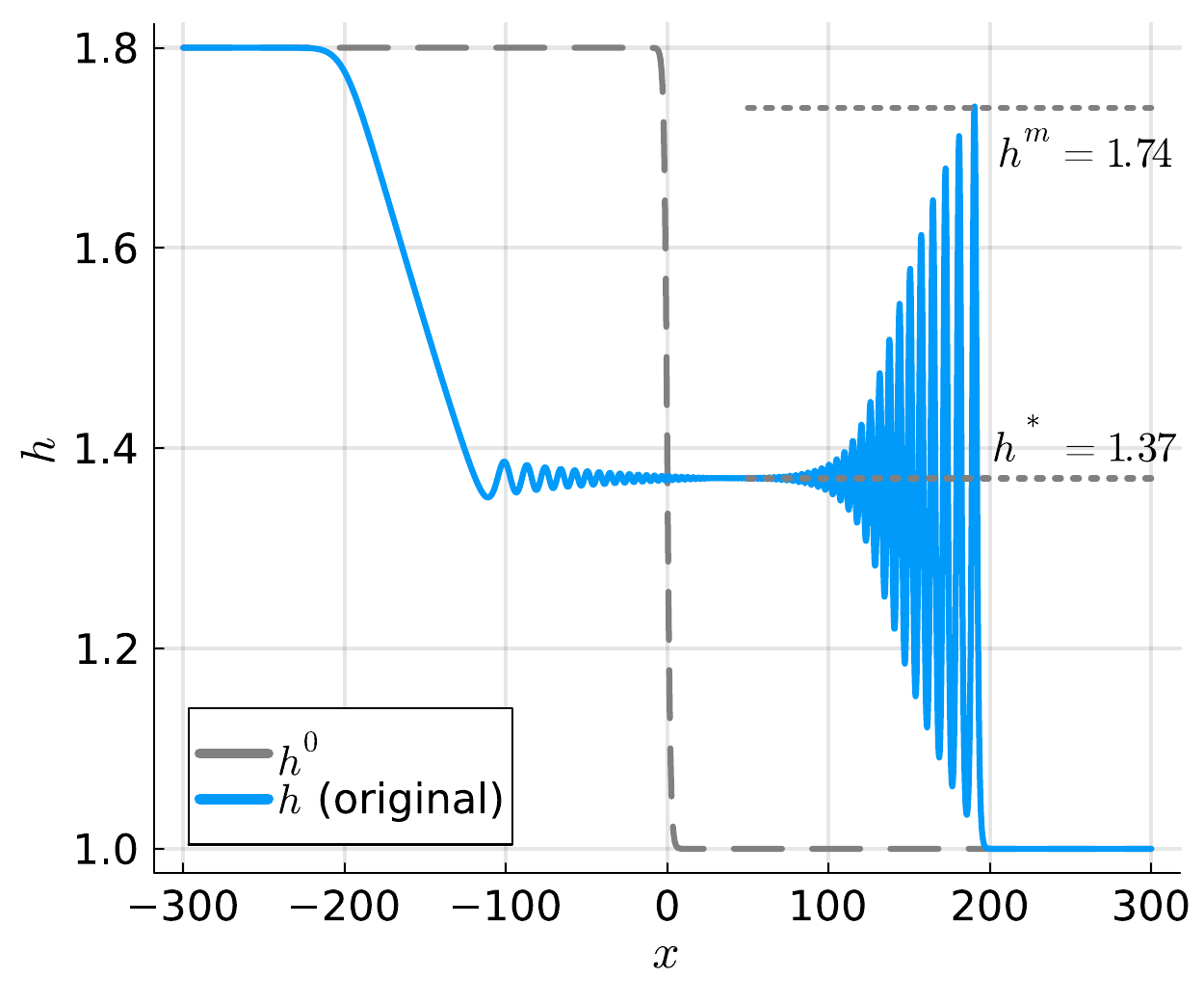}
        \caption{Serre-Green-Naghdi equations.}
    \end{subfigure}%
    \caption{Riemann problem using structure-preserving
             second-order finite differences with $\Delta x=  0.3$.
             The intermediate water height $h^* = 1.37$   and  amplitude
              $h^m = 1.74$ of the leading dispersive wave are in
             accordance with  analytical predictions and  numerical
             results  \cite{gnst20,pitt2018behaviour}.}
    \label{fig:riemann_problem}
\end{figure}

\subsection{Soliton fission}
\label{sec:soliton_fission}

Next, we study the long-time behavior of a dispersive shock wave.
We use the initial condition
\begin{equation}
    h(x,0) = \begin{cases}
        1.8, & |x| < 1, \\
        1.0, & \text{otherwise},
    \end{cases} \quad u(x,0) = 0,
\end{equation}
and discretize the spatial domain $[-500, 500]$ with $10^3$ nodes
using central second-order finite differences.
We integrate the numerical solutions
until $t = 118$, and analyze the leading waves in
the interval $[390, 500]$. We take the values where the water height $h$
is greater than a threshold of $1.001$ and fit analytical  Serre-Green-Naghdi
solitons  to them. For this, we take the median
of the remaining values of $h$ as baseline and use a Nelder-Mead method
\cite{nelder1965simplex,gao2012implementing} implemented in
Optim.jl/Optimization.jl \cite{mogensen2018optim,dixit2023optimization}
to compute a least-squares solution.

\begin{figure}[htb]
\centering
    \begin{subfigure}{0.32\textwidth}
    \centering
        \includegraphics[width=\textwidth]{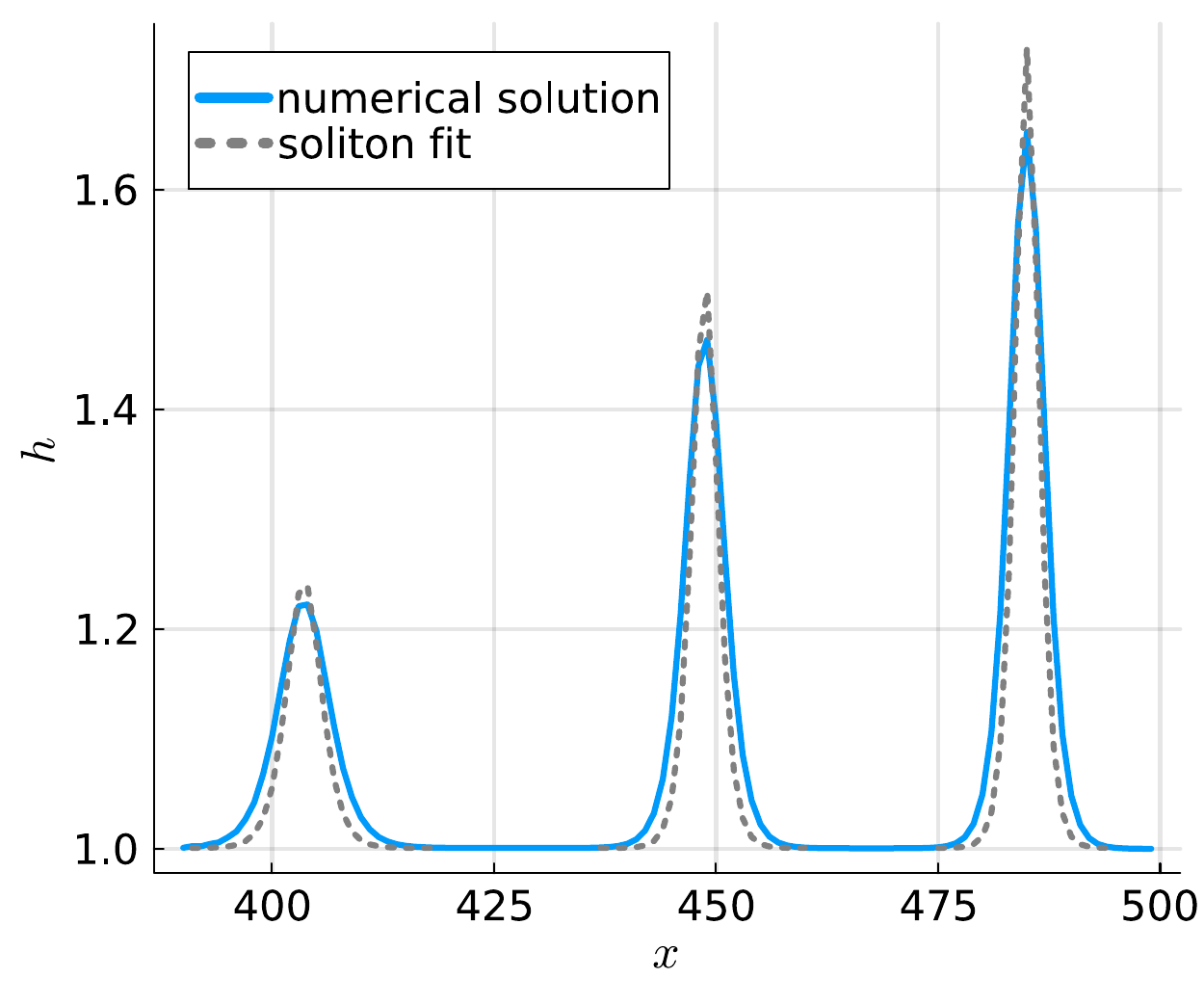}
        \caption{Hyperbolic  with $\lambda = 500$.}
    \end{subfigure}%
    \hspace{\fill}
    \begin{subfigure}{0.32\textwidth}
    \centering
        \includegraphics[width=\textwidth]{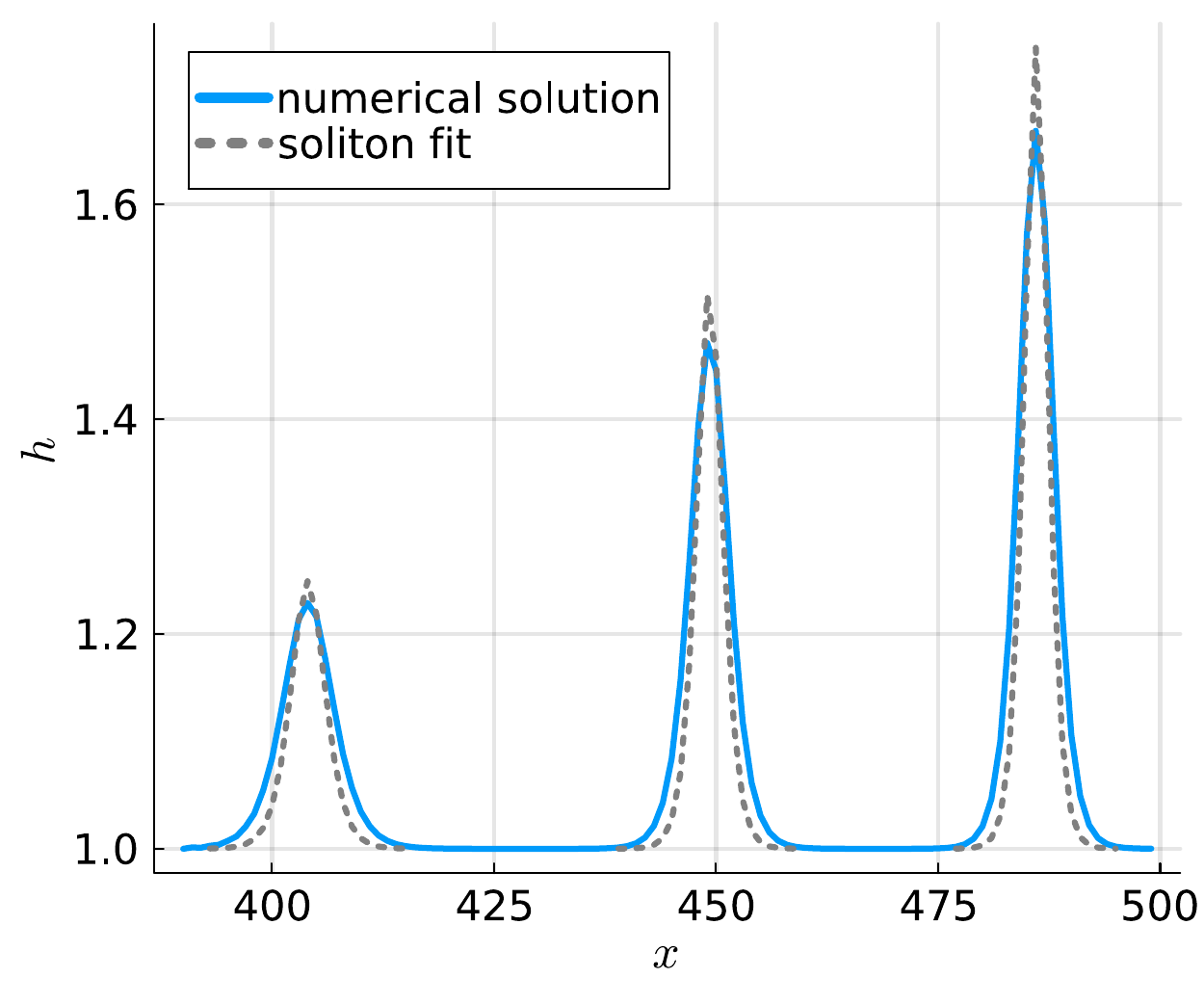}
        \caption{Serre-Green-Naghdi equations.}
    \end{subfigure}%
        \begin{subfigure}{0.32\textwidth}
    \centering
        \includegraphics[width=\textwidth]{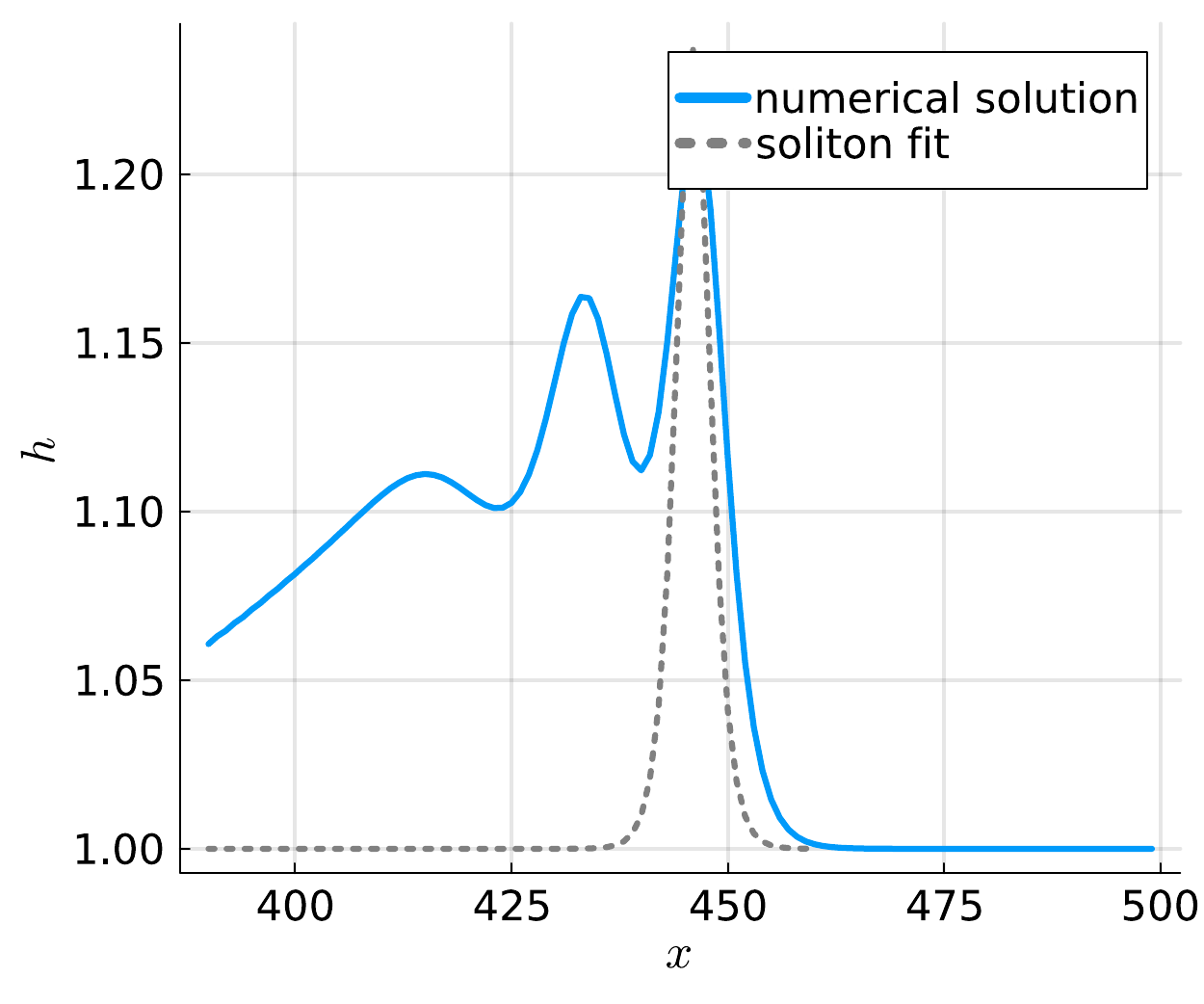}
        \caption{Serre-Green-Naghdi  + AV.}
    \end{subfigure}%
    \\
    \medskip
    \begin{subfigure}{0.32\textwidth}
    \centering
        \includegraphics[width=\textwidth]{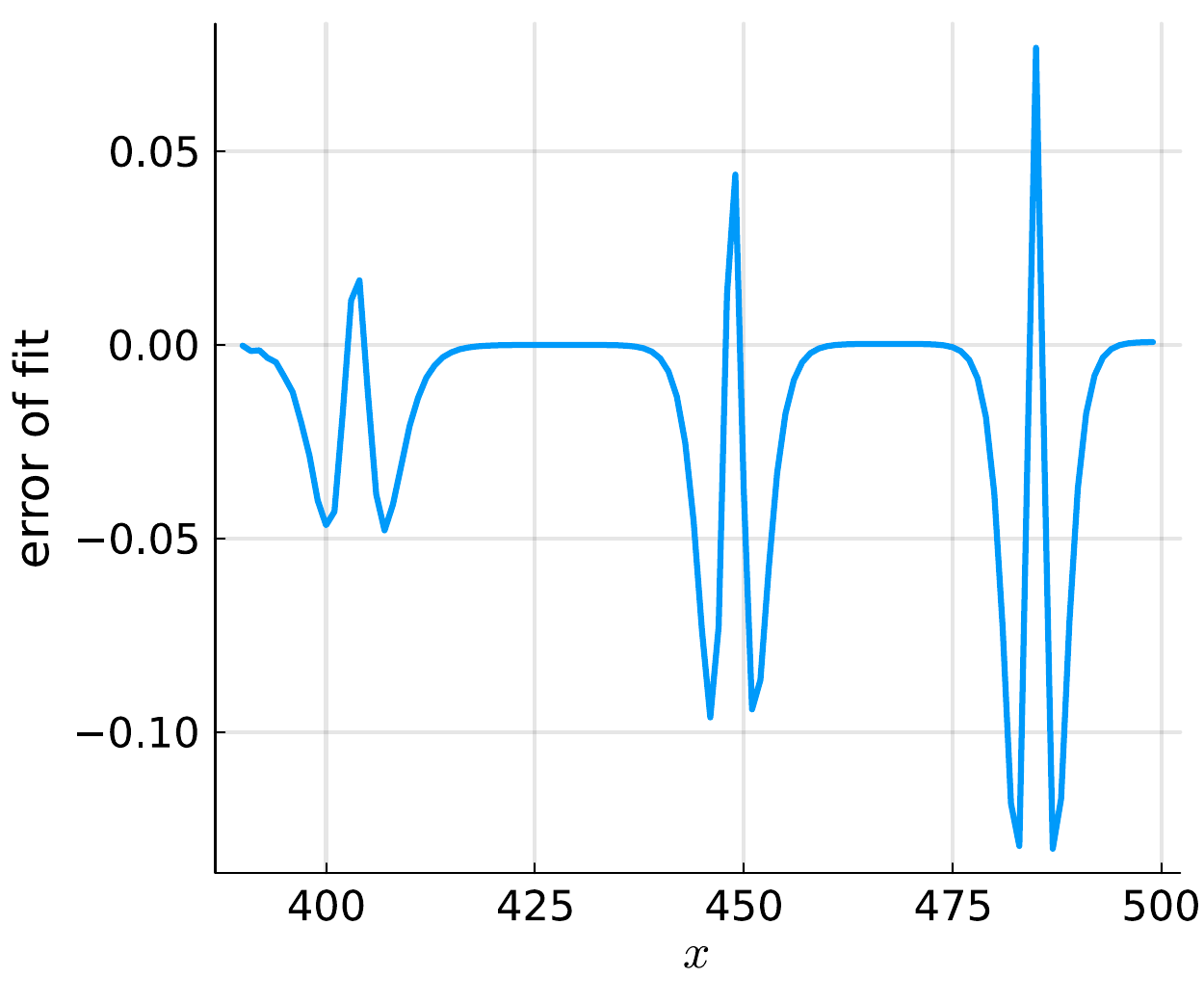}
        \caption{Hyperbolic  with $\lambda = 500$.}
    \end{subfigure}%
    \hspace{\fill}
    \begin{subfigure}{0.32\textwidth}
    \centering
        \includegraphics[width=\textwidth]{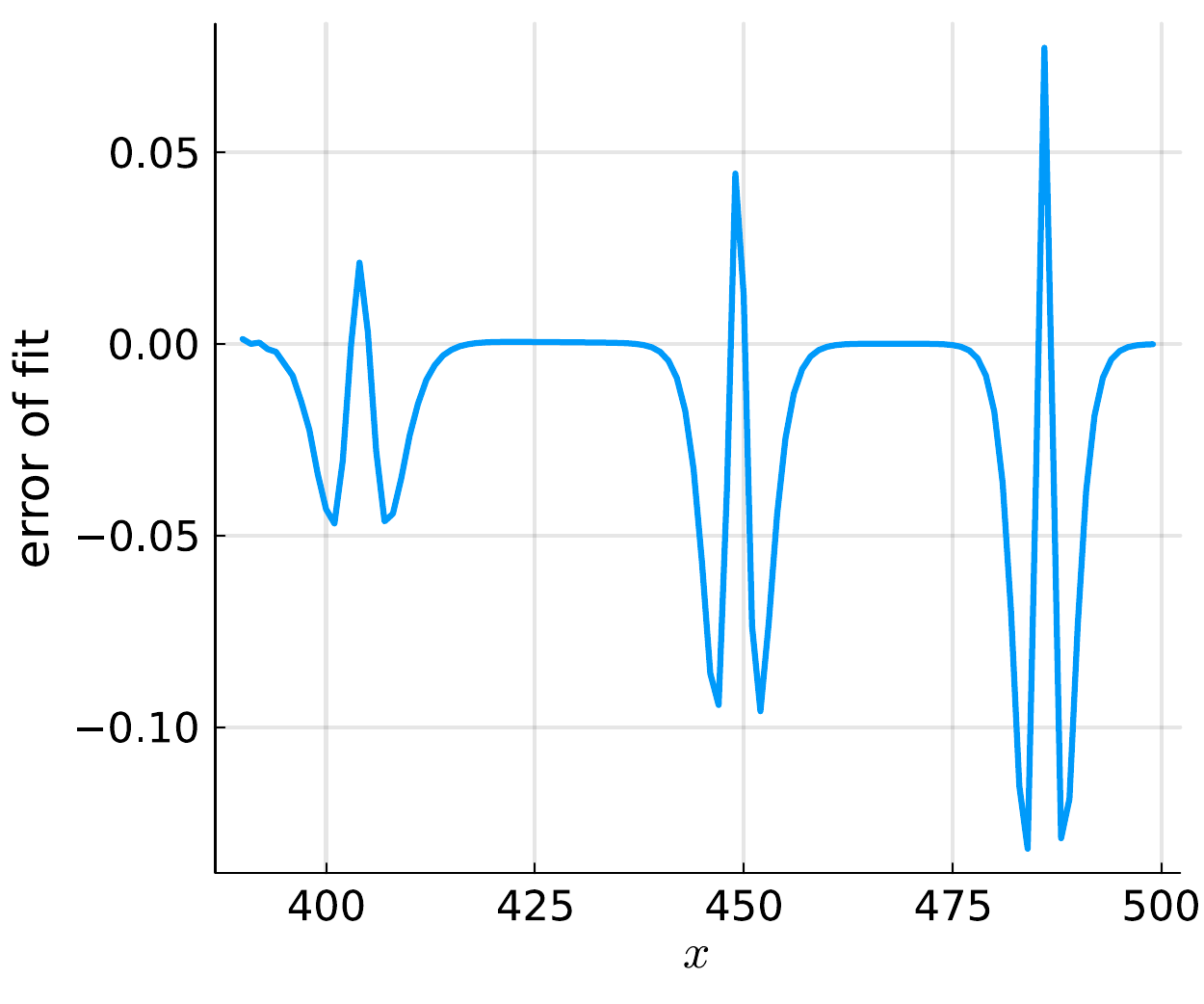}
        \caption{Serre-Green-Naghdi equations.}
    \end{subfigure}%
        \begin{subfigure}{0.32\textwidth}
    \centering
        \includegraphics[width=\textwidth]{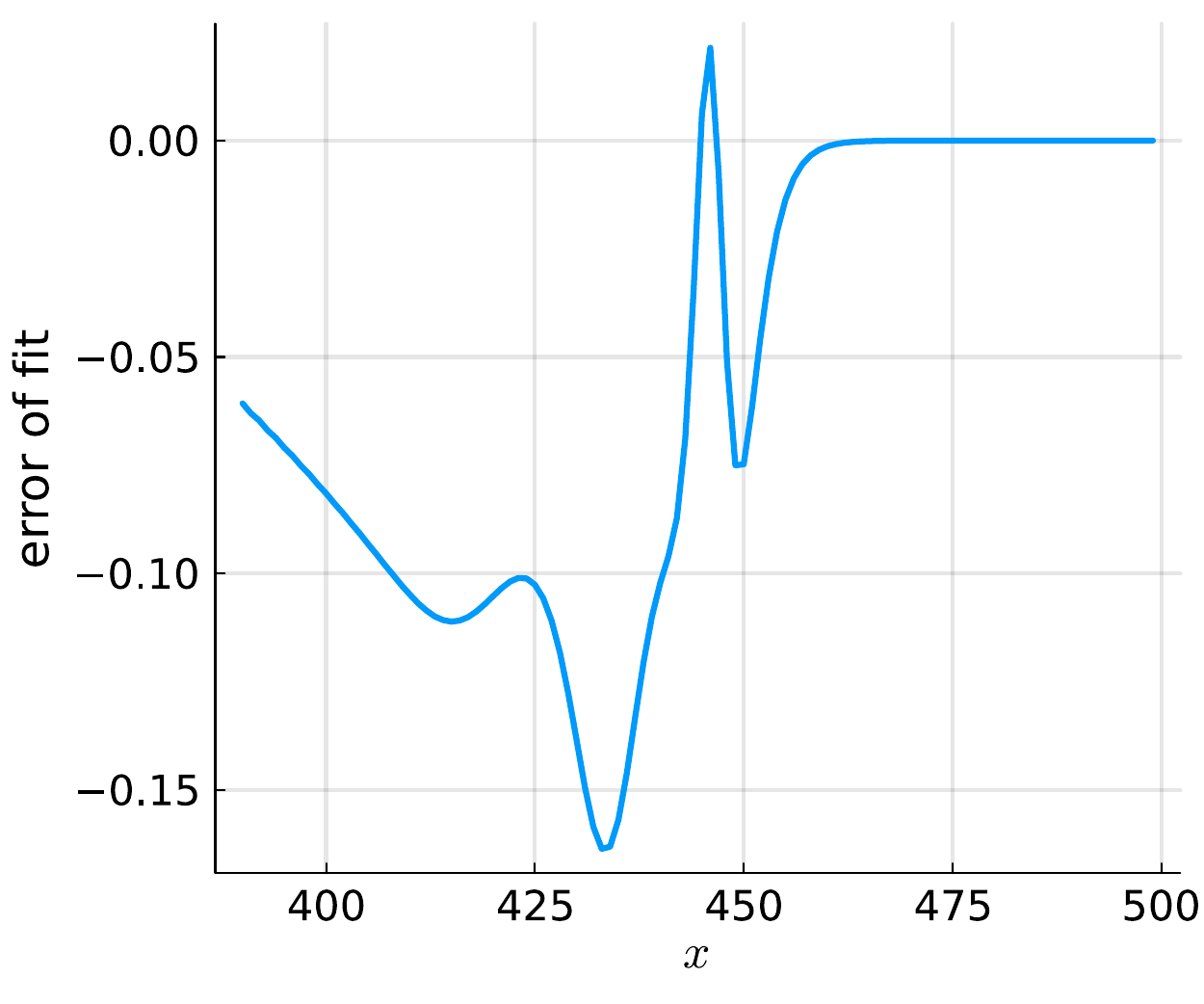}
        \caption{Serre-Green-Naghdi  + AV.}
    \end{subfigure}%
    \caption{Leading soliton waves obtained from the rectangular initial
             condition discretized with second-order finite differences
             with $\Delta x = 1$.
             Left:   hyperbolic approximation with
             $\lambda = 500$. Center: original
             Serre-Green-Naghdi model. Right: original
             Serre-Green-Naghdi model with numerical dissipation. Top: water height and soliton fits.
             Bottom: soliton fit error.}
    \label{fig:soliton_fission}
\end{figure}

The results are shown in Figure~\ref{fig:soliton_fission}. For both the
hyperbolic approximation and the discretization of the original
SGN equations, the numerical solutions agree very well with
the fitted analytical soliton waves. The differences between the numerical
solutions and the fits are roughly two orders of magnitude smaller than the amplitude
of the waves.
To show the impact of numerical viscosity on such long-time computations,
we also plot the results
of the the original SGN equations plus artificial diffusion. The corresponding results with the hyperbolic model are
visually identical. We can see that not only the height of the first wave is much
underestimated, but also its position, certainly due to the dependence of the celerity of the leading wave on its amplitude.
The optimization method does recognize a half soliton shape in the leading front.
This behaviour is further investigated  and commented in the following section.

\subsection{Favre waves}
\label{sec:favre_waves}

The propagation of undular bores, also known as Favre waves, is a classical problem,
see, e.g., \cite{wk95,chass_etal19} and references therein,
for which well-known experiments exist \cite{favre1935,treske1994}.
The initial setup considered here follows, e.g., \cite{chass_etal19,jouy_etal24}.
The initial solution is obtained by a smoothed discontinuity (cf.\ also Figure~\ref{fig.favre-sketch})
$$
\begin{aligned}
h(x,t=0):=   h_0 +  \dfrac{[\![h]\!]}{2}\left\{ 1 - \tanh\bigg(\dfrac{x-x_0}{\alpha}\bigg)  \right\}, \\
u(x,t=0):=   u_0 +  \dfrac{[\![u]\!]}{2}\left\{ 1 - \tanh\bigg(\dfrac{x-x_0}{\alpha}\bigg)  \right\},
\end{aligned}
$$
where
$$
[\![h]\!] := \epsilon h_0,
$$
with $\epsilon$ the nonlinearity, and with  $[\![u]\!]$ satisfying the
shallow-water Rankine-Hugoniot relations,
and in particular
$$
[\![u]\!] =  \sqrt{g\dfrac{h_1+h_0}{2h_0h_1}  }[\![h]\!].
$$
We refer the reader to \cite{chass_etal19,jouy_etal24} for further details on the setup.

\begin{figure}[htbp]
\begin{center}
   \includegraphics[width=0.25\textwidth]{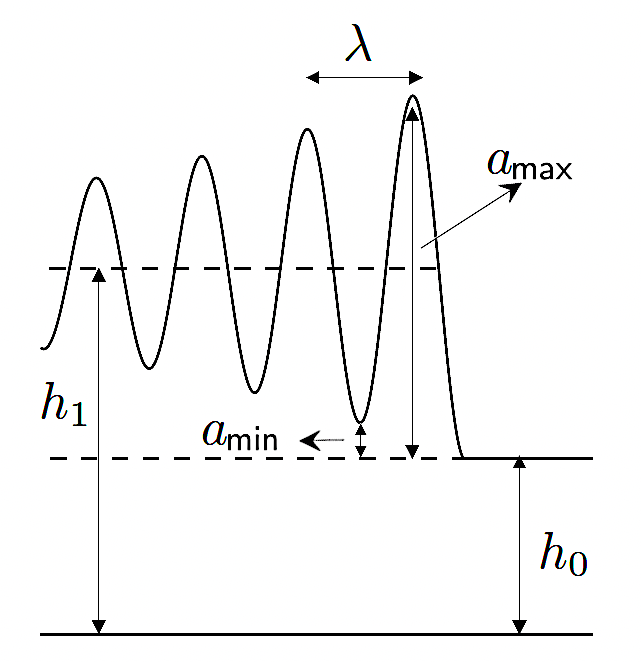}
   \caption{Favre waves:  undular bore sketch, and definition of
           amplitudes $a_\mathrm{min/max}$ and wavelength~$\lambda$.}
   \label{fig.favre-sketch}
\end{center}
\end{figure}

As in   \cite{chass_etal19,jouy_etal24}  we  consider   at first the
short-time bore evolution
 for    three    values of the non-linearity
$\epsilon\in\{0.1,\,0.2,\,0.3\}$.
The  free-surface elevation  for different values of the dimensionless time $\tilde t:= t\sqrt{g/h_0}$
is compared to  fully nonlinear potential solutions from \cite{wk95}.
The results obtained with fourth-order structure-preserving finite differences on a  relatively coarse mesh with $\Delta x = 0.125$ are shown in Figures~\ref{fig:favre_waves-hyp}
and \ref{fig:favre_waves-orig}.
Our results compare well with the fully nonlinear potential solutions, and to those of  \cite{chass_etal19}.
For larger values like $\epsilon = 0.3$ some  limitations, related to the weakly dispersive
character of the model itself, can be seen.  The value  $\lambda = 500$ seems again large enough for
 the hyperbolic approximation and the original formulations  to give visually indistinguishable results.
For completeness, we also report  in each picture the results obtained with  artificial viscosity.
For these short-time simulations we cannot  see any impact of numerical dissipation.

\begin{figure}[htbp]
\centering
    \begin{subfigure}{0.32\textwidth}
    \centering
        \includegraphics[width=\textwidth]{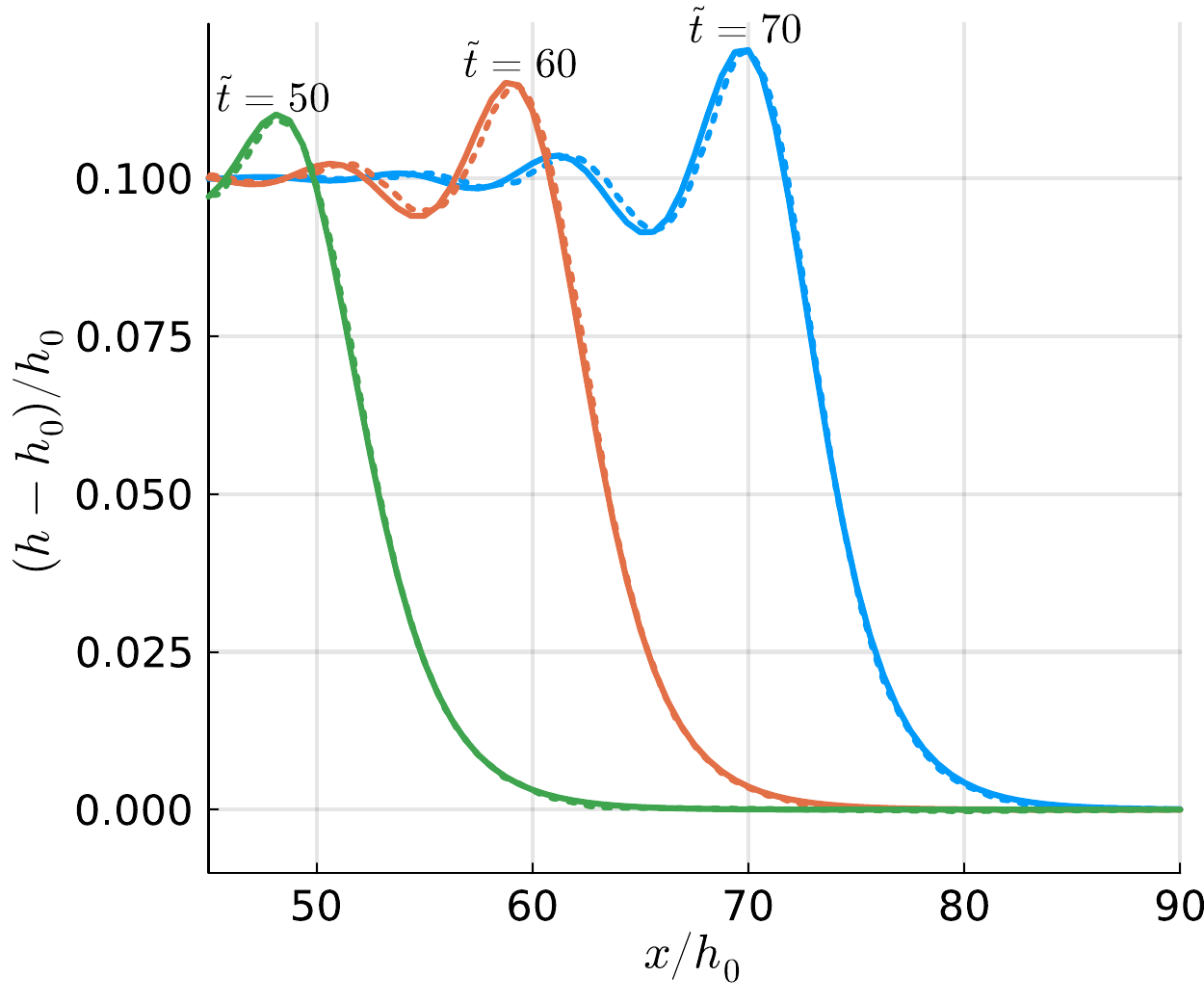}
    \end{subfigure}%
    \hspace{\fill}
    \begin{subfigure}{0.32\textwidth}
    \centering
        \includegraphics[width=\textwidth]{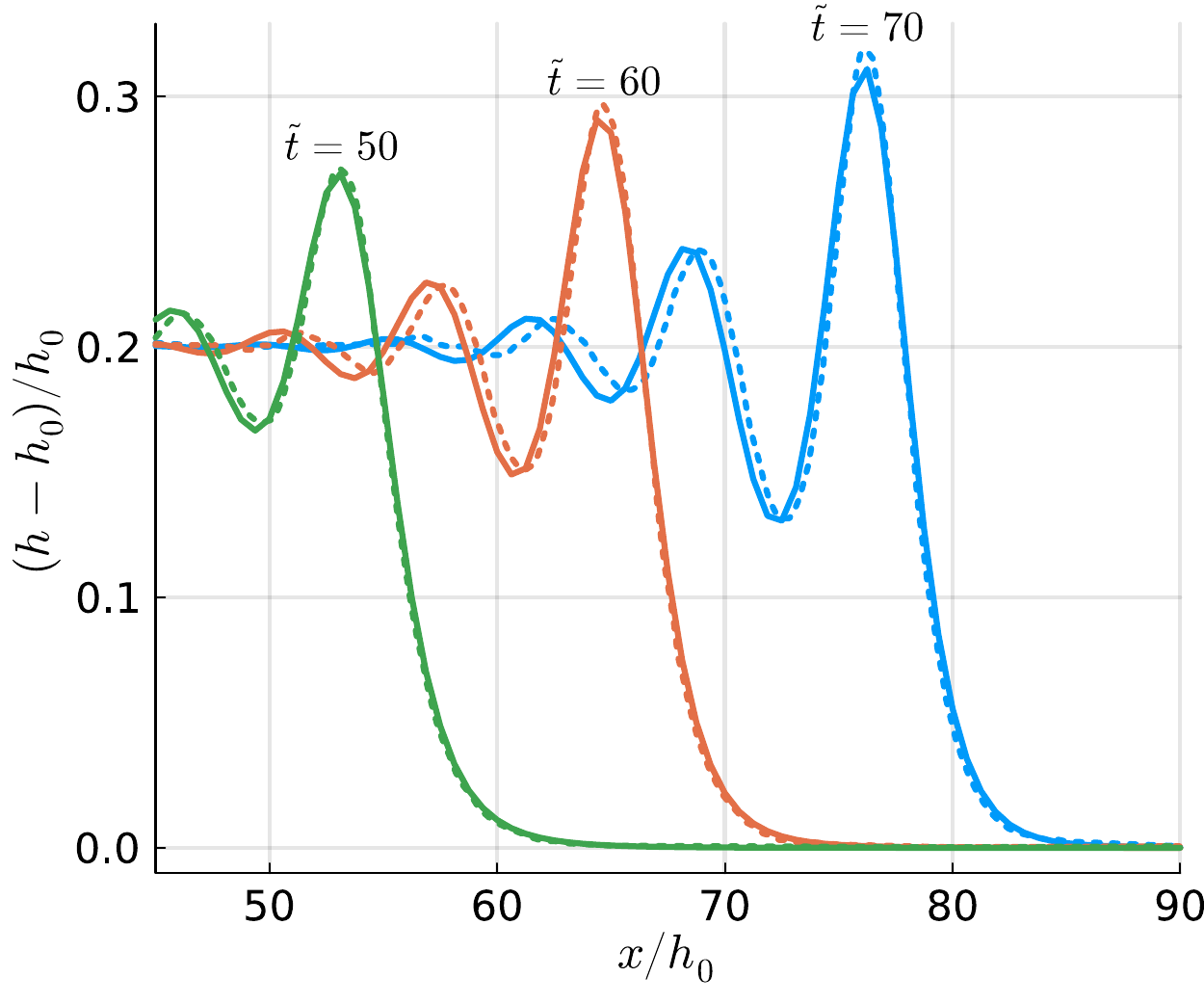}
    \end{subfigure}%
    \hspace{\fill}
    \begin{subfigure}{0.32\textwidth}
    \centering
        \includegraphics[width=\textwidth]{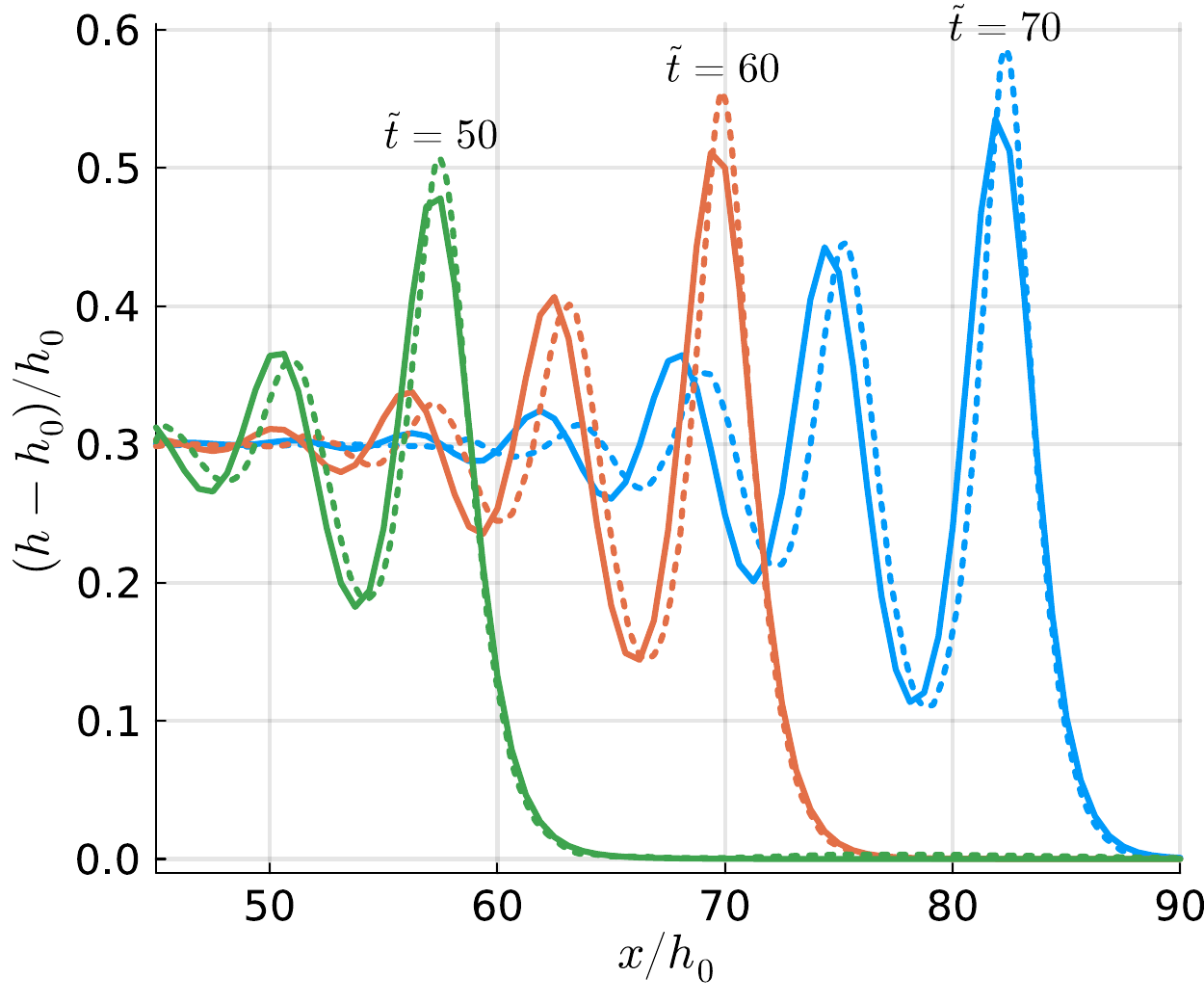}
    \end{subfigure}%
    \\
    \medskip
    \begin{subfigure}{0.32\textwidth}
    \centering
        \includegraphics[width=\textwidth]{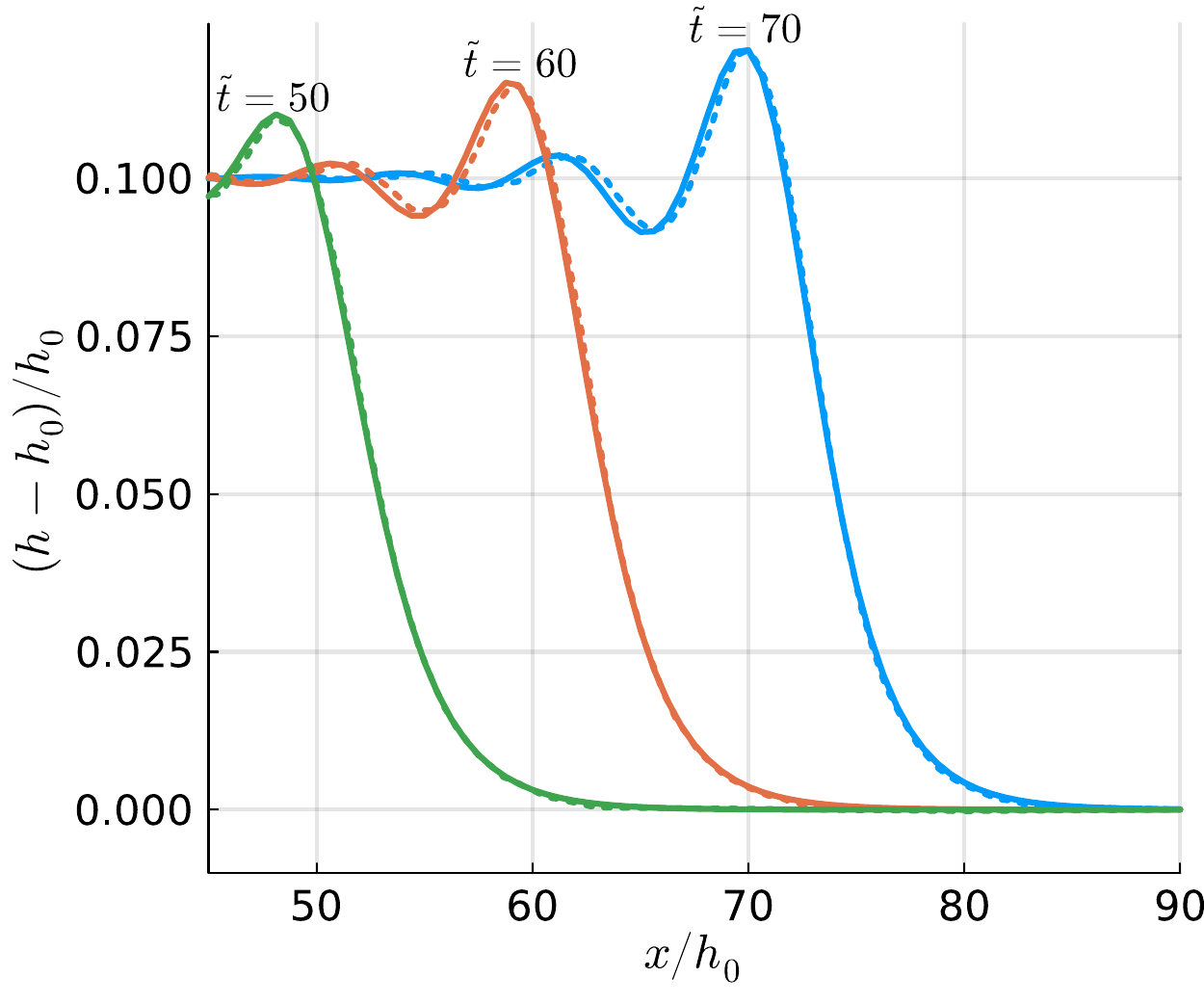}
    \end{subfigure}%
    \hspace{\fill}
    \begin{subfigure}{0.32\textwidth}
    \centering
        \includegraphics[width=\textwidth]{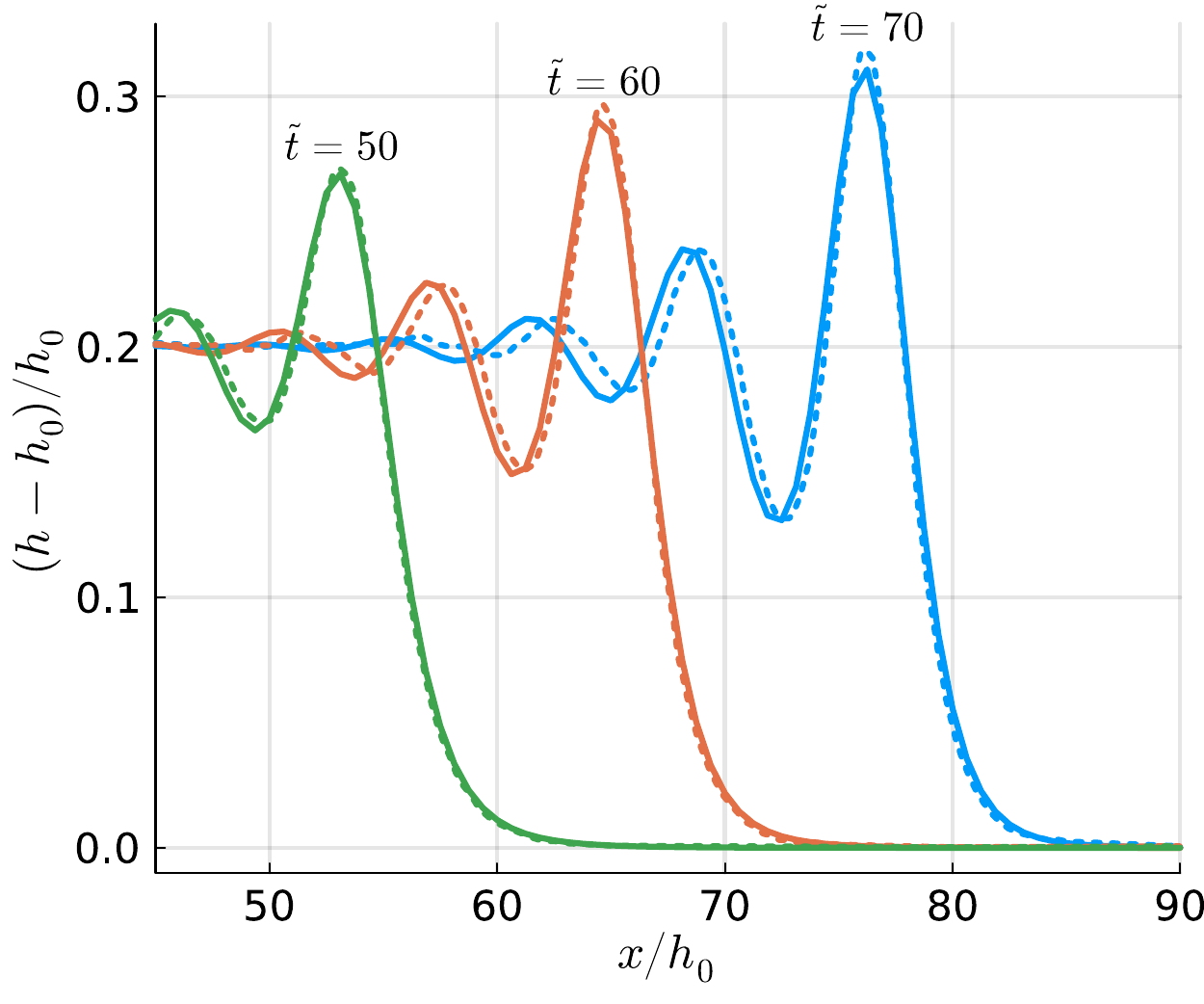}
    \end{subfigure}%
    \hspace{\fill}
    \begin{subfigure}{0.32\textwidth}
    \centering
        \includegraphics[width=\textwidth]{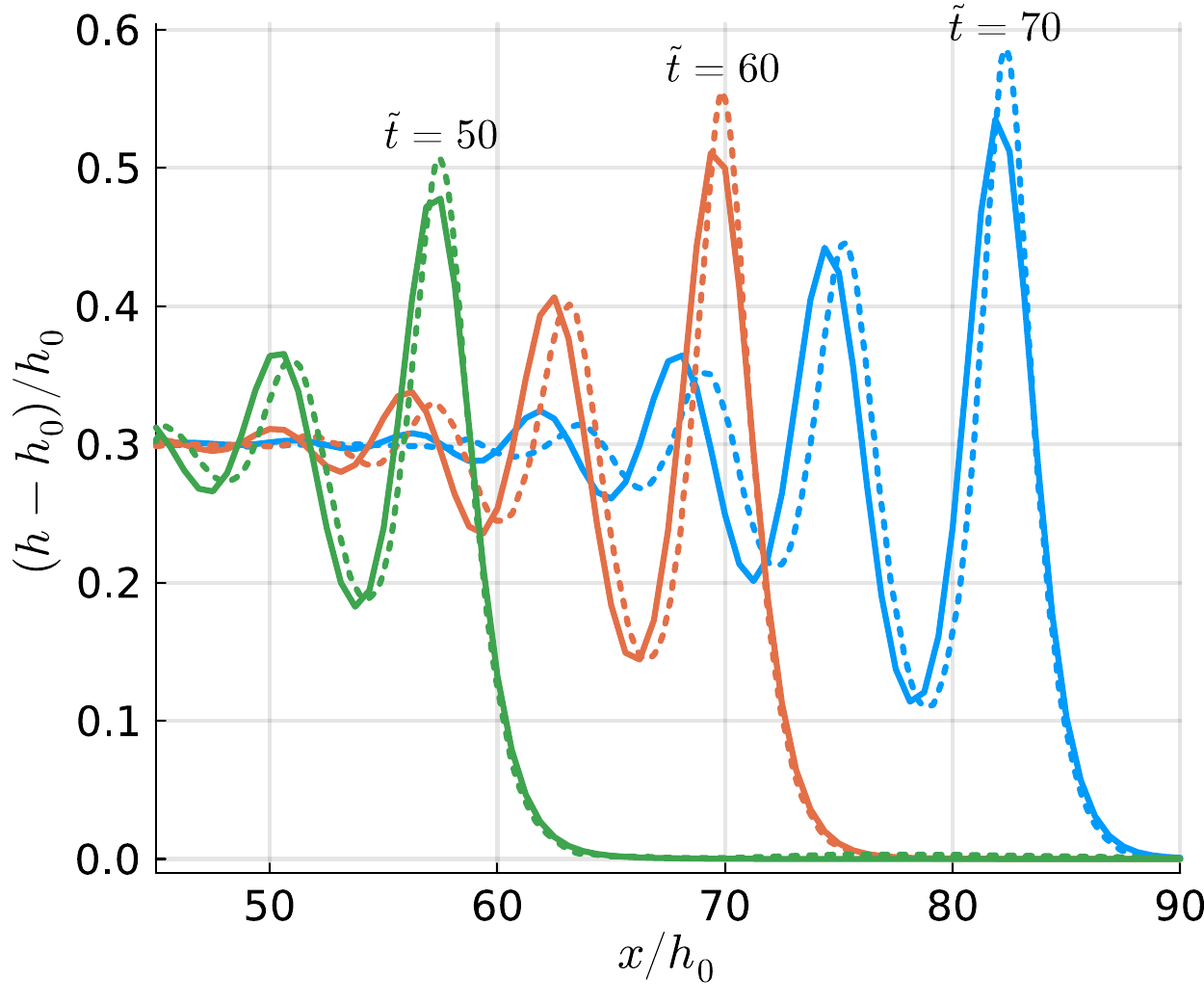}
    \end{subfigure}%
    \caption{Favre wave: hyperbolic formulation with $\lambda = 500$, and  grid spacing $\Delta x = 0.125$.
           Numerical solutions (solid lines)   with  nonlinearity  $\epsilon = 0.1$ (left),  $\epsilon = 0.2$ (center),
             and $\epsilon = 0.3$ (right).
             Top: energy-conservative fourth-order finite differences with central operators.
             Bottom: fourth-order finite differences with central operators and artificial viscosity.
              The dashed lines show the fully nonlinear potential solutions from \cite{wk95}.}
    \label{fig:favre_waves-hyp}
\end{figure}

\begin{figure}[htbp]
\centering
    \begin{subfigure}{0.32\textwidth}
    \centering
        \includegraphics[width=\textwidth]{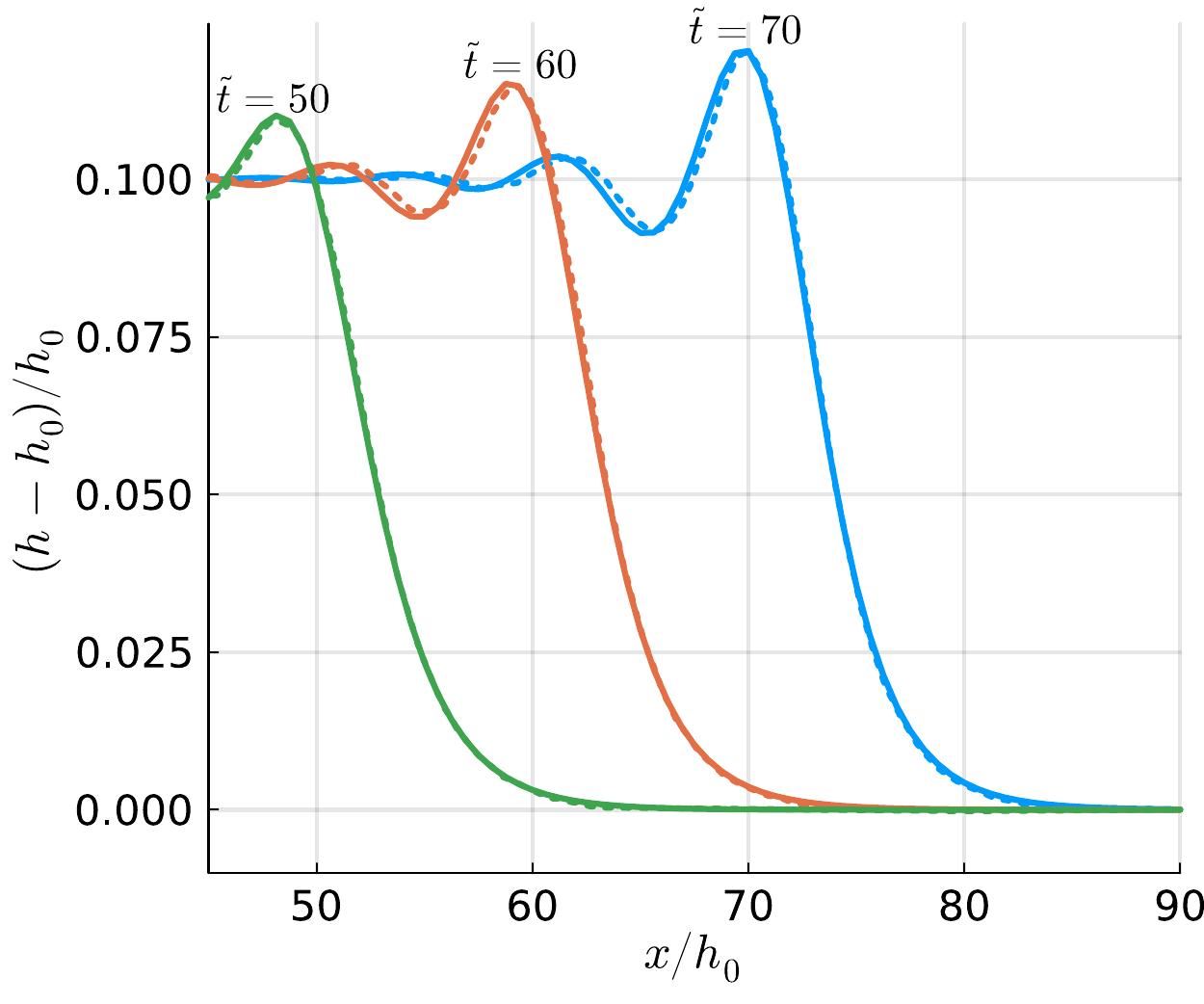}
    \end{subfigure}%
    \hspace{\fill}
    \begin{subfigure}{0.32\textwidth}
    \centering
        \includegraphics[width=\textwidth]{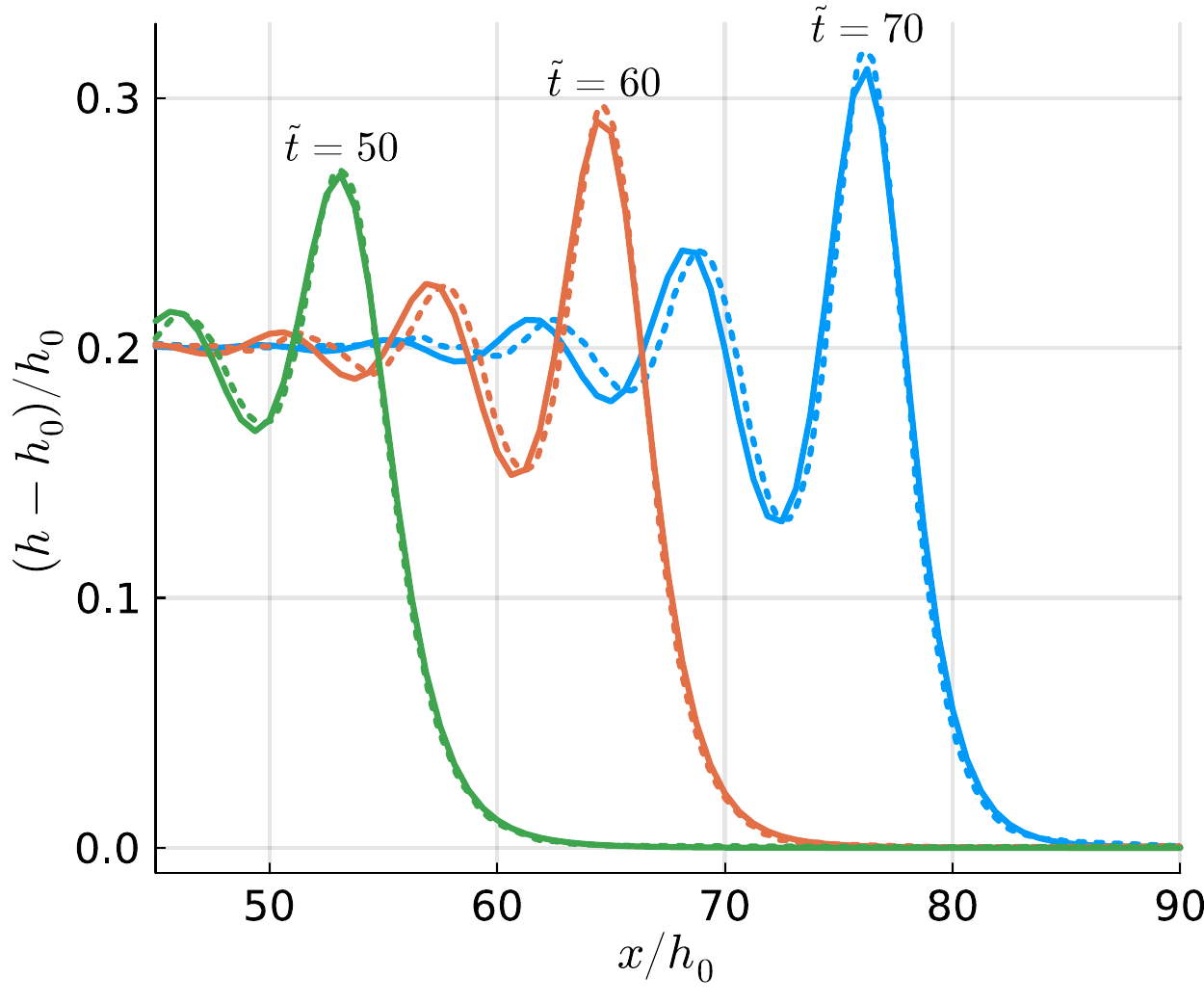}
    \end{subfigure}%
    \hspace{\fill}
    \begin{subfigure}{0.32\textwidth}
    \centering
        \includegraphics[width=\textwidth]{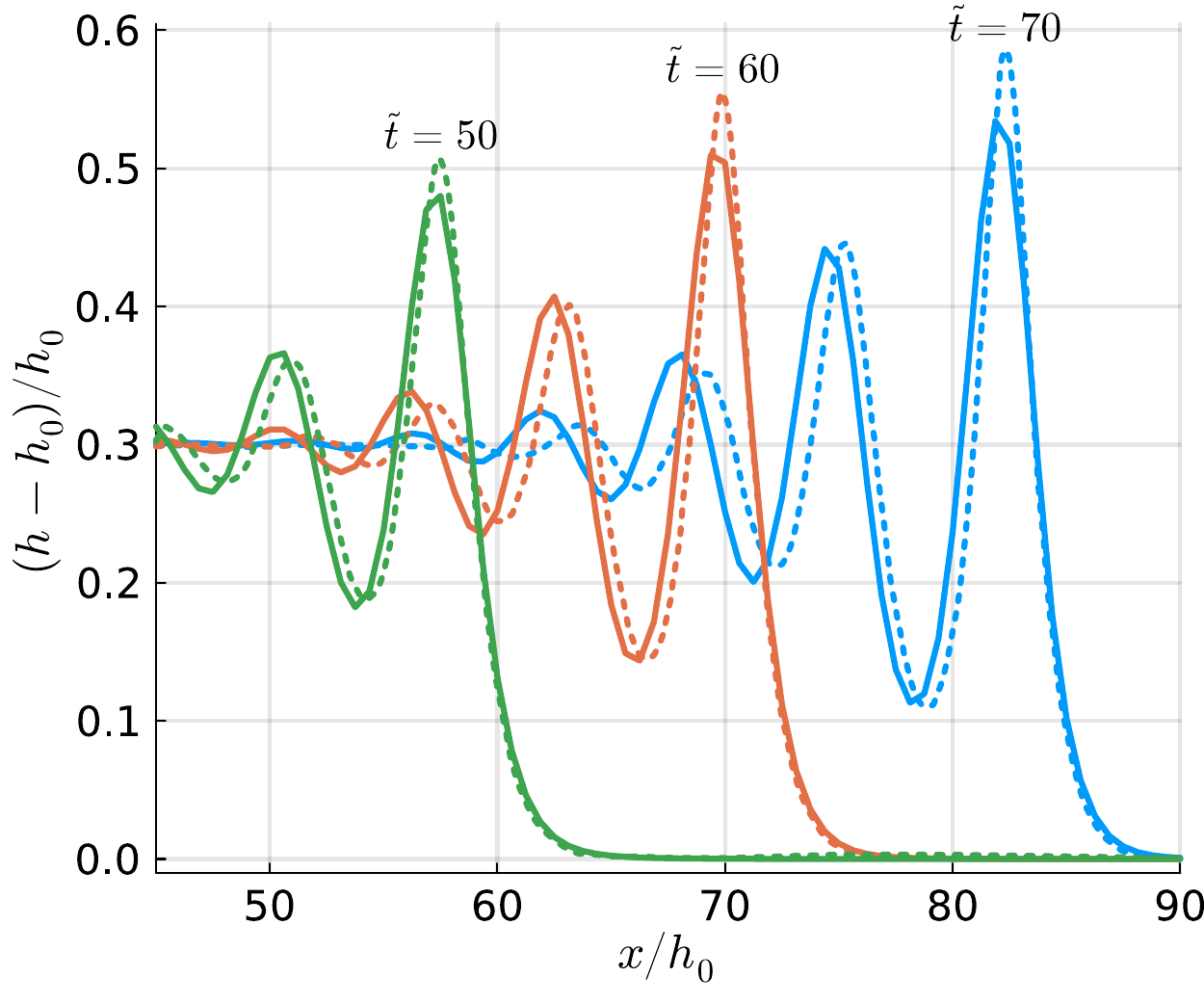}
    \end{subfigure}%
    \\
    \medskip
    \begin{subfigure}{0.32\textwidth}
    \centering
        \includegraphics[width=\textwidth]{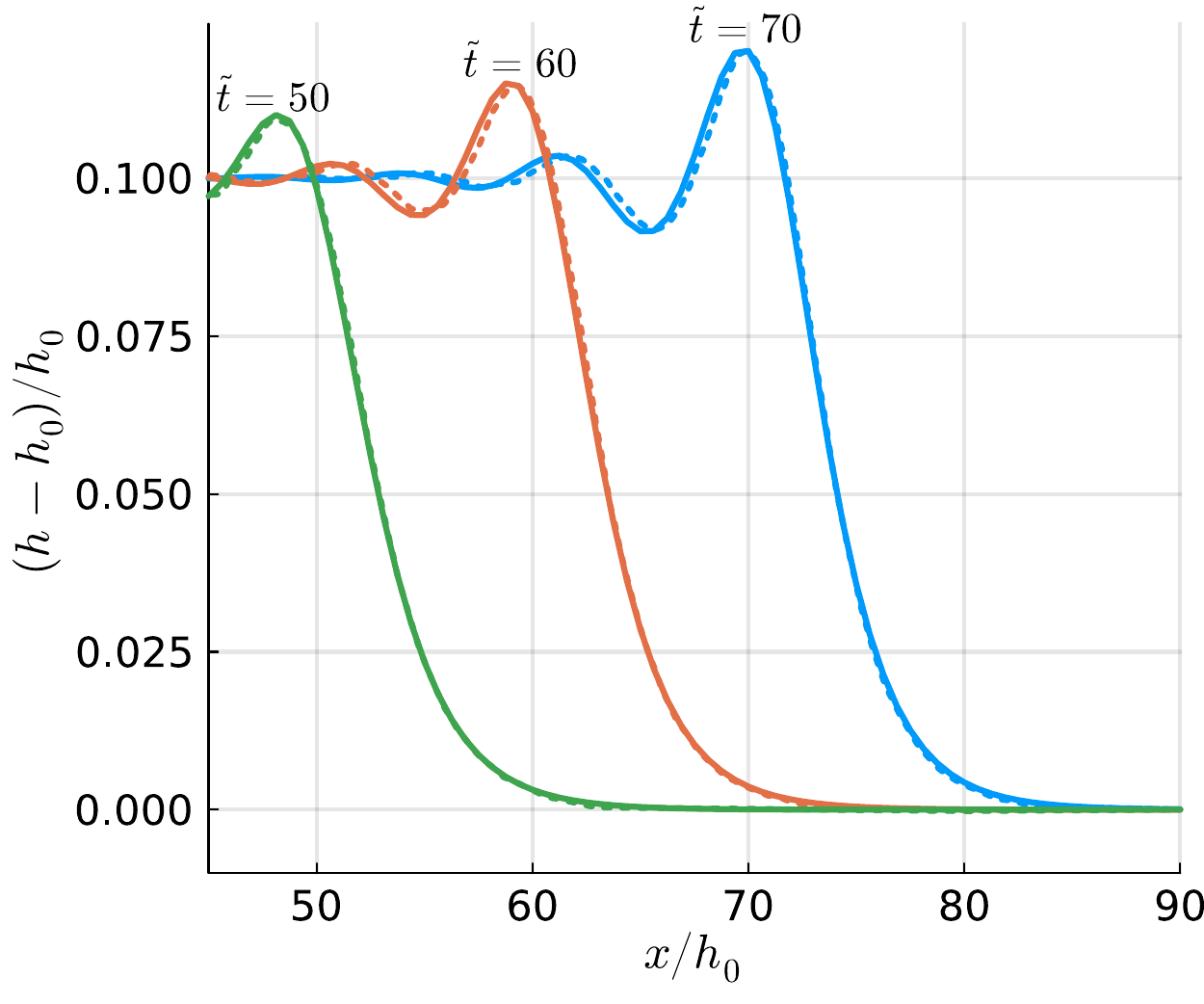}
    \end{subfigure}%
    \hspace{\fill}
    \begin{subfigure}{0.32\textwidth}
    \centering
        \includegraphics[width=\textwidth]{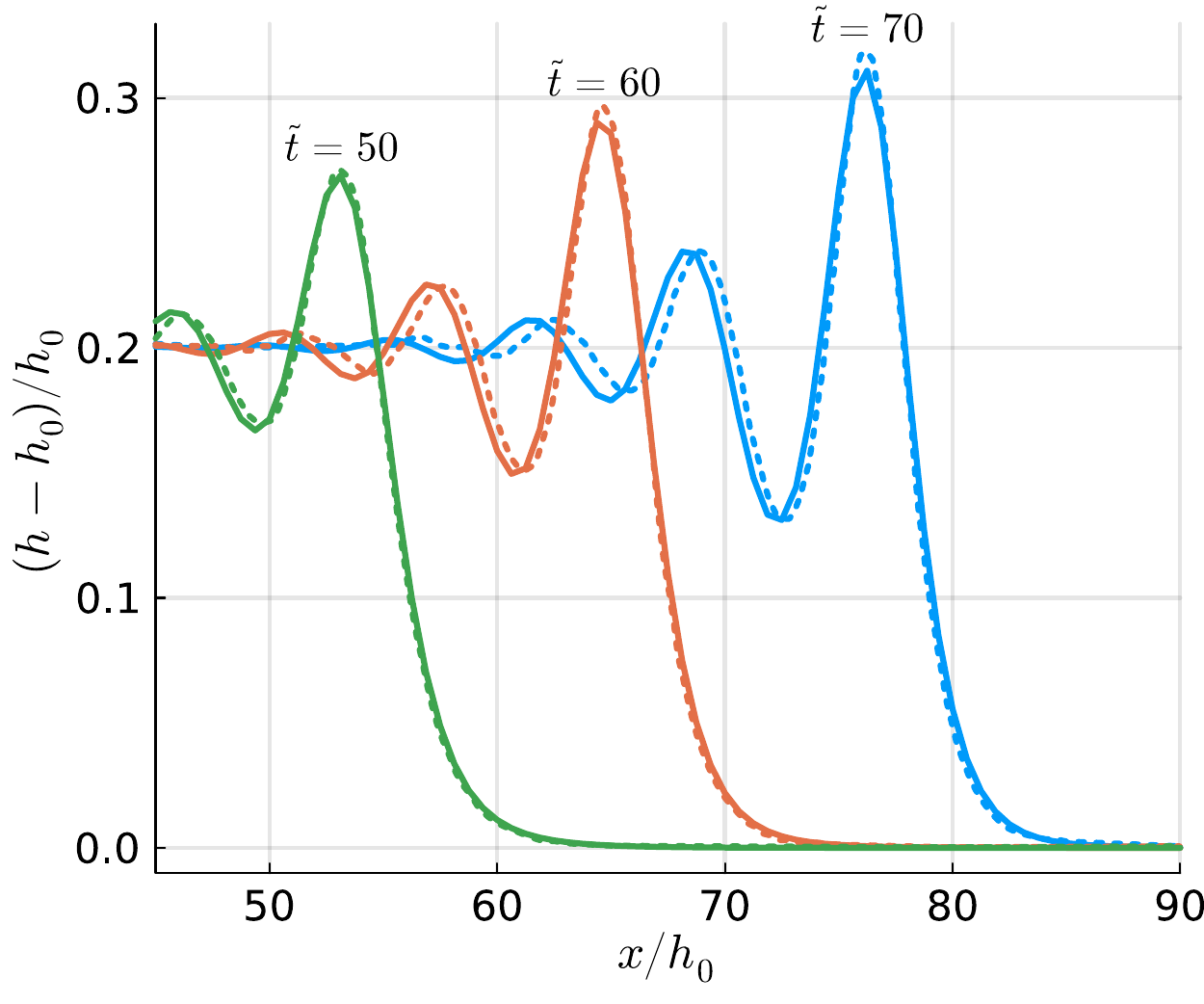}
    \end{subfigure}%
    \hspace{\fill}
    \begin{subfigure}{0.32\textwidth}
    \centering
        \includegraphics[width=\textwidth]{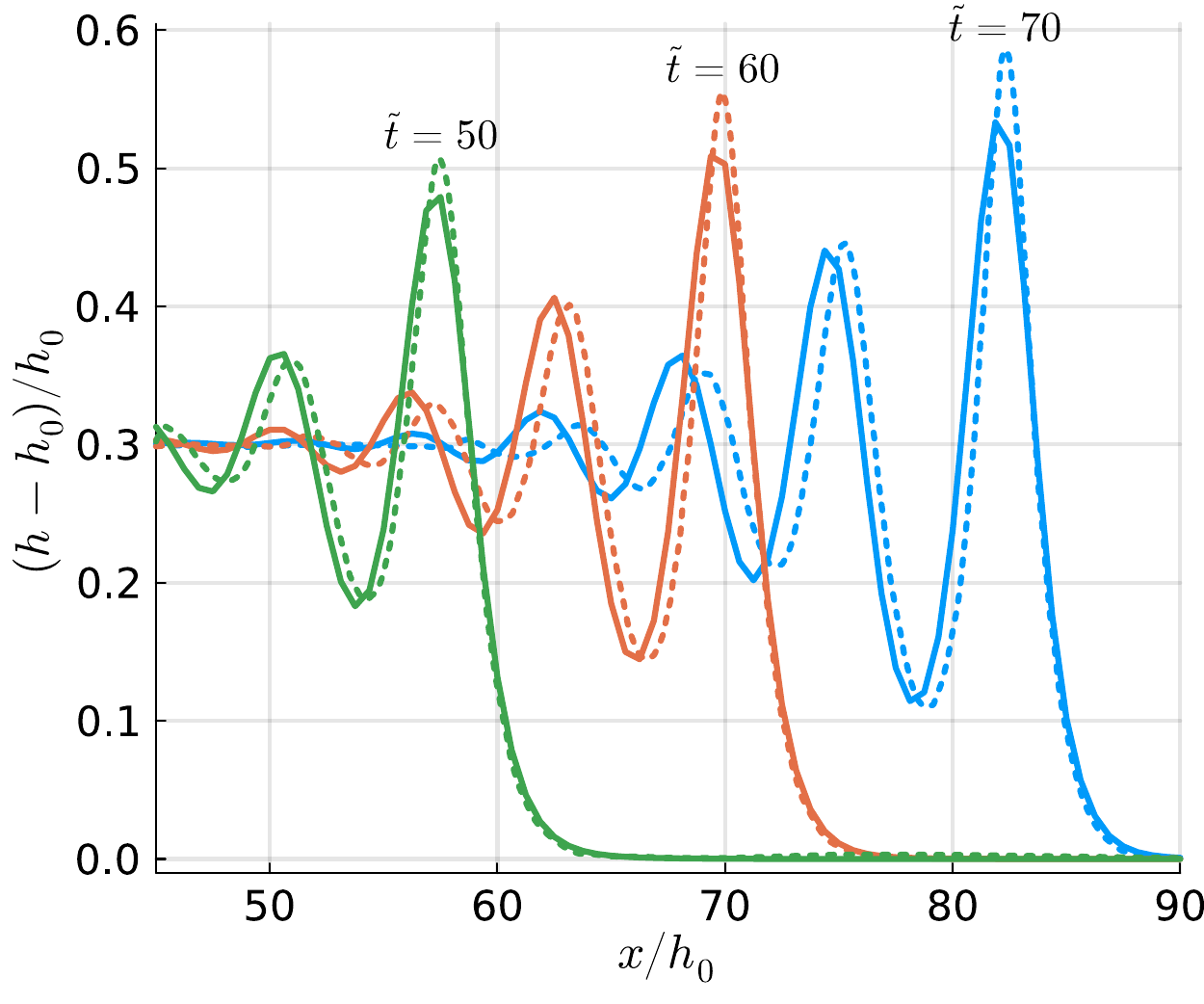}
    \end{subfigure}%
    \caption{Favre wave: original SGN system, and  grid spacing $\Delta x = 0.125$.
           Numerical solutions (solid lines)   with  nonlinearity  $\epsilon = 0.1$ (left),  $\epsilon = 0.2$ (center),
             and $\epsilon = 0.3$ (right).
             Top : energy-conservative fourth-order finite differences with central operators.
             Bottom : fourth-order finite differences with central operators and artificial viscosity.
              The dashed lines show the fully nonlinear potential solutions from \cite{wk95}.}
    \label{fig:favre_waves-orig}
\end{figure}

\subsubsection{Long-time propagation}

As shown in \cite{jouy_etal24,10.1093/imamat/hxad030,https://doi.org/10.1111/sapm.12694}
this problem is extremely sensitive to the presence of dissipative processes such as  friction
or viscous regularization. In absence of dissipation,   soliton fission occurs.
Dissipation generates undular bores of lower amplitudes and finite wavelength.
This fact is also known, for simpler models such as KdV and BBM, from the modulation theory
\cite{EL201611}. As it turns out, numerical dissipation plays exactly the same
role, which may lead  to gross underestimations of the wave heights on coarse meshes
as the results in \cite{jouy_etal24} demonstrate. \\

To investigate this aspect we consider  the propagation for a large  dimensionless time  $\tilde t = 1500$.
To save space we only consider the solution of the system in its original  formulation, but
similar conclusion are obtained when solving the hyperbolic approximation. The spatial domain considered is now $[-3000,\,3000]$
and the initial discontinuity is set at $x_0=-1000$. We plot two sets of results using second- and fourth-order schemes.
The bore front is visualized in Figures~\ref{fig:favre_waves-long1-sol1} for second-order schemes, and \ref{fig:favre_waves-long1-sol2}
for fourth-order ones. From these figures we can see that the first solitary wave is already resolved on the coarsest mesh    for the structure-preserving schemes.
In the second-order case, a phase error is observed for the secondary waves, as one might expect
due to  the impact of discrete dispersion. However, the amplitudes are close to the finer mesh solution.

The fourth-order structure-preserving schemes  have already resolved the solution on the coarsest mesh.
The schemes with artificial viscosity behave much like in presence of a viscous regularization \cite{10.1093/imamat/hxad030,https://doi.org/10.1111/sapm.12694},
with much lower amplitudes and no fission of solitons.
In the fourth-order case,   doubling the number of nodes allows to obtain a reasonable prediction of  wave height and position.
In the second-order case even with two refinements the dissipative method still provides large amplitude underestimations.

\begin{figure}[htbp]
\centering
    \begin{subfigure}{0.32\textwidth}
    \centering
        \includegraphics[width=\textwidth]{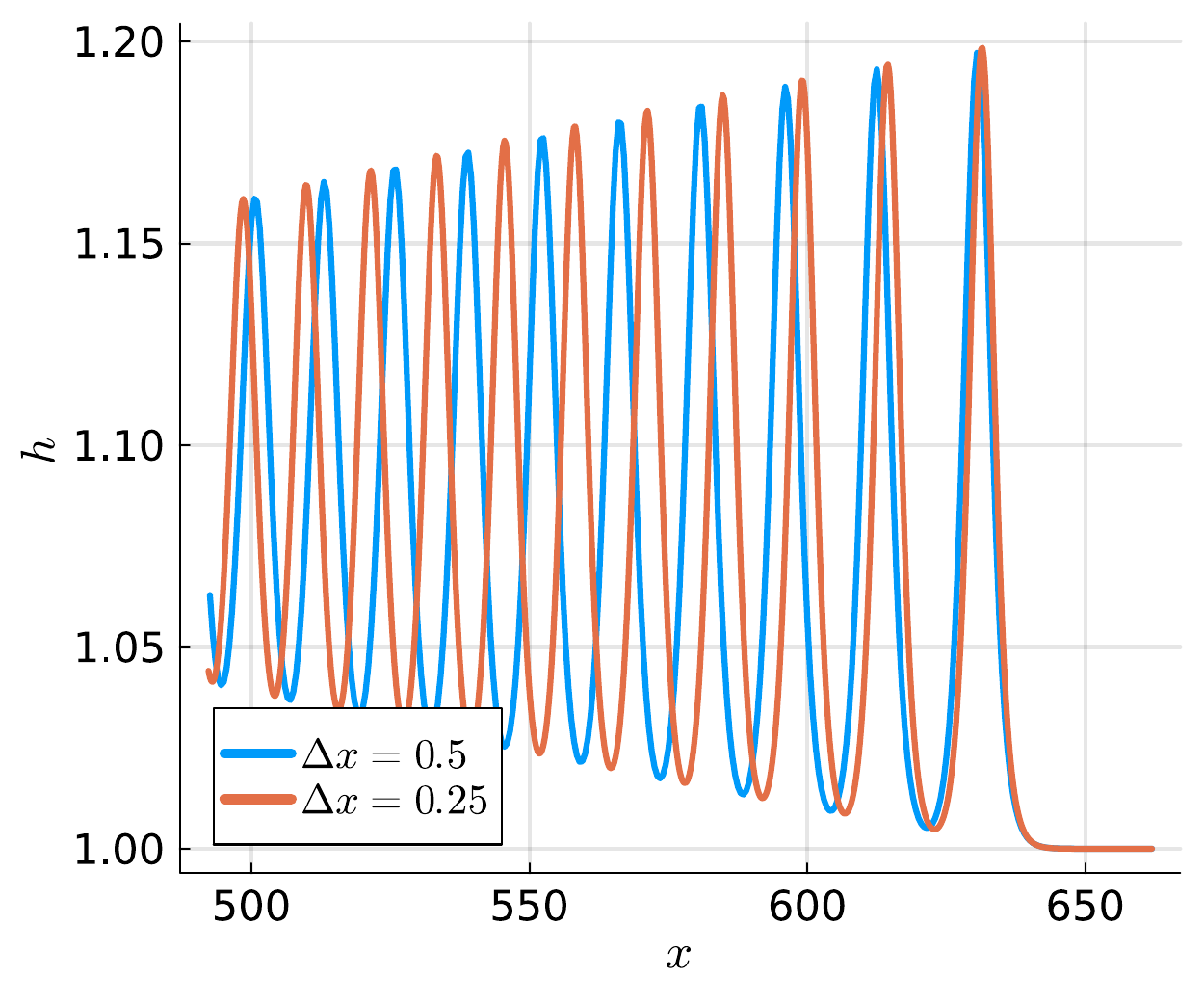}
    \end{subfigure}%
    \hspace{\fill}
    \begin{subfigure}{0.32\textwidth}
    \centering
        \includegraphics[width=\textwidth]{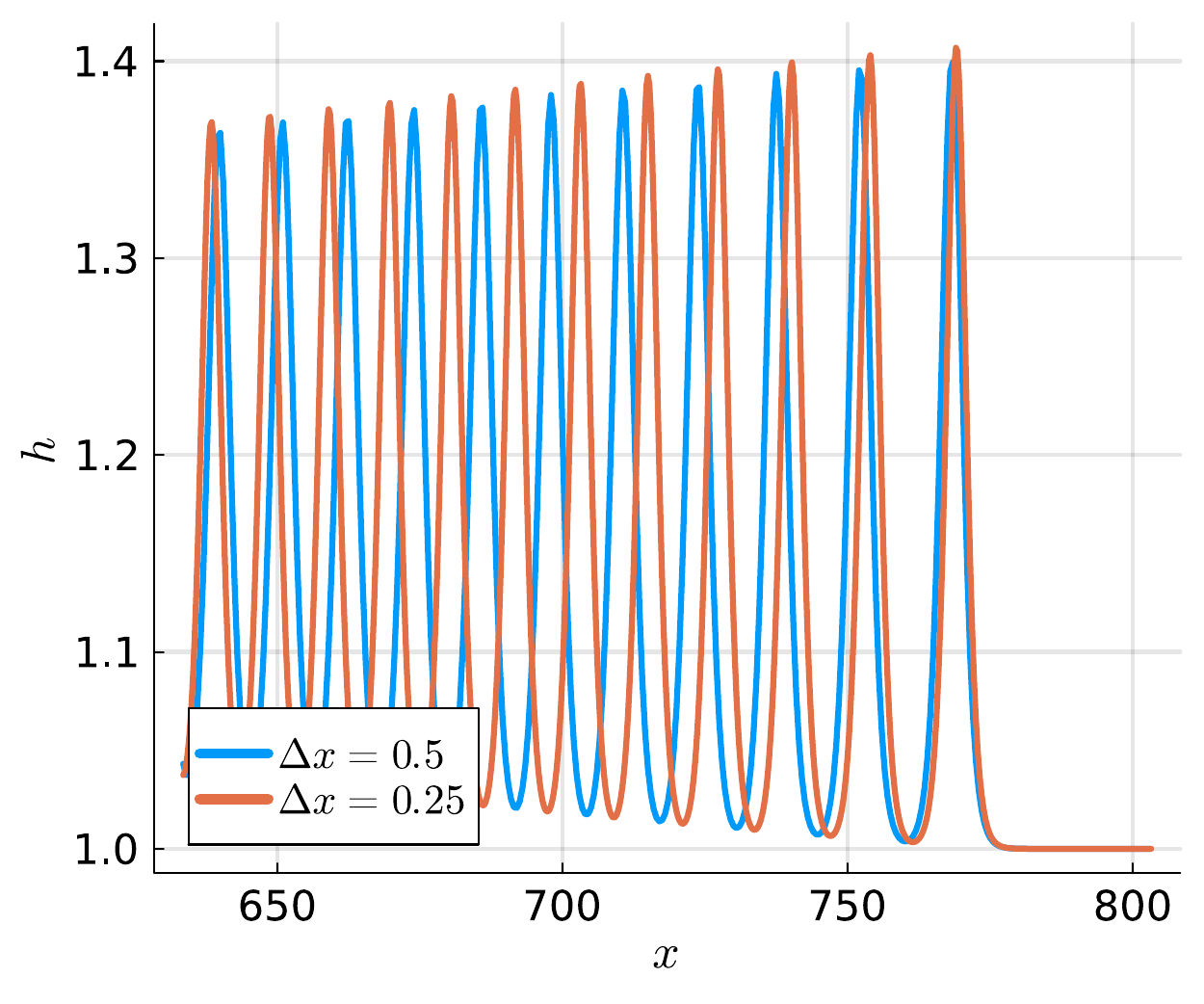}
    \end{subfigure}%
    \hspace{\fill}
    \begin{subfigure}{0.32\textwidth}
    \centering
        \includegraphics[width=\textwidth]{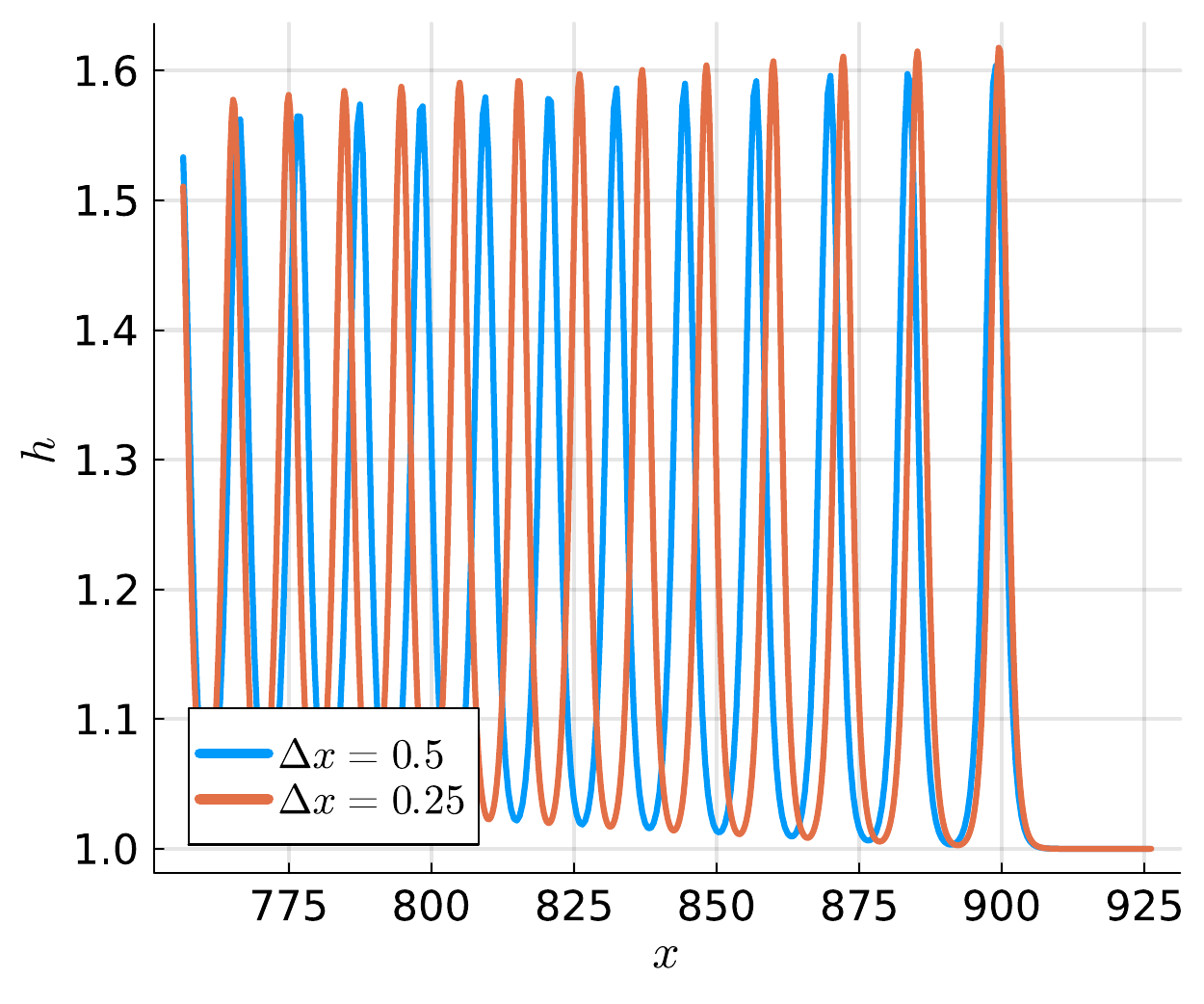}
    \end{subfigure}%
    \\
    \medskip
    \begin{subfigure}{0.32\textwidth}
    \centering
        \includegraphics[width=\textwidth]{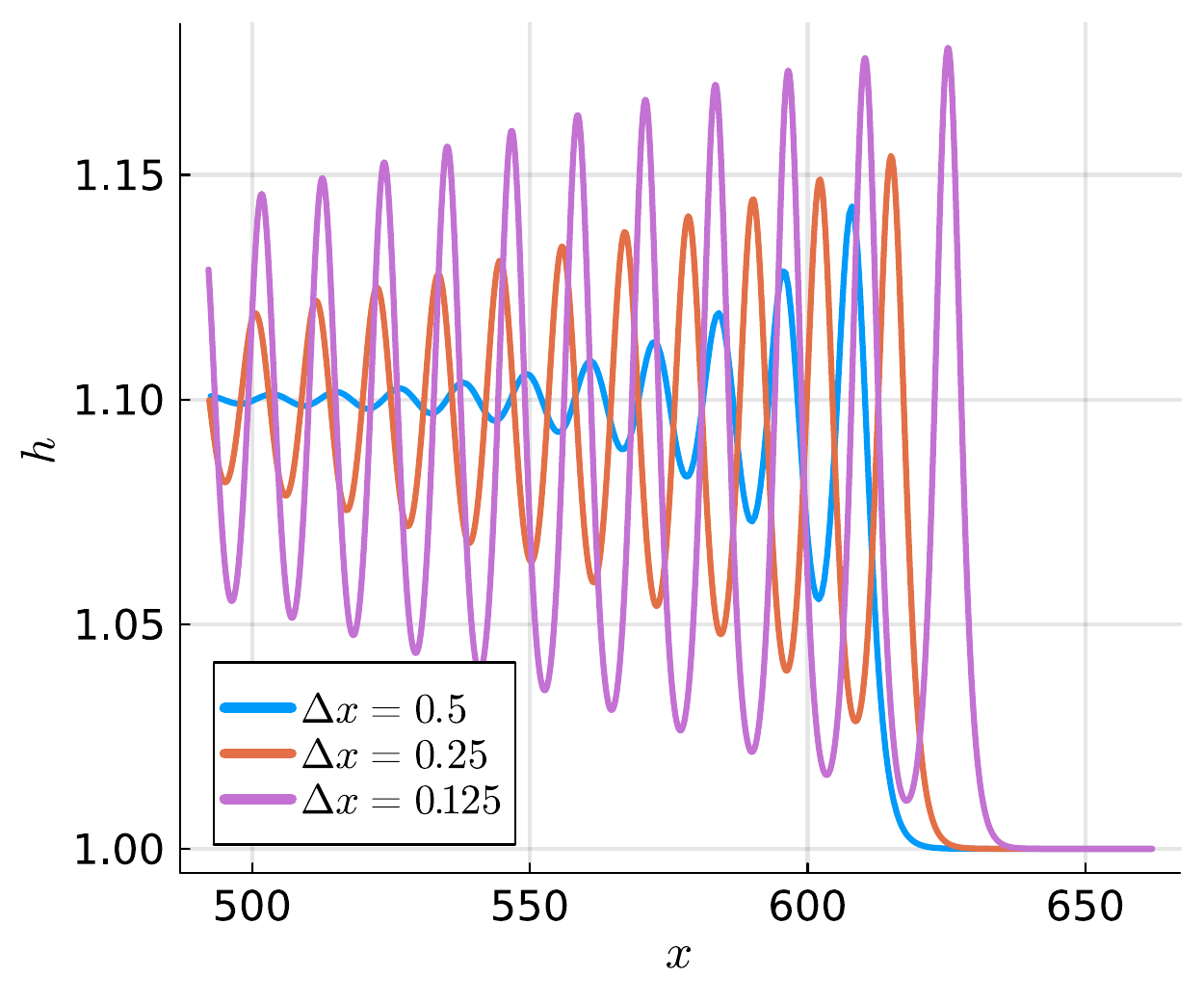}
    \end{subfigure}%
    \hspace{\fill}
    \begin{subfigure}{0.32\textwidth}
    \centering
        \includegraphics[width=\textwidth]{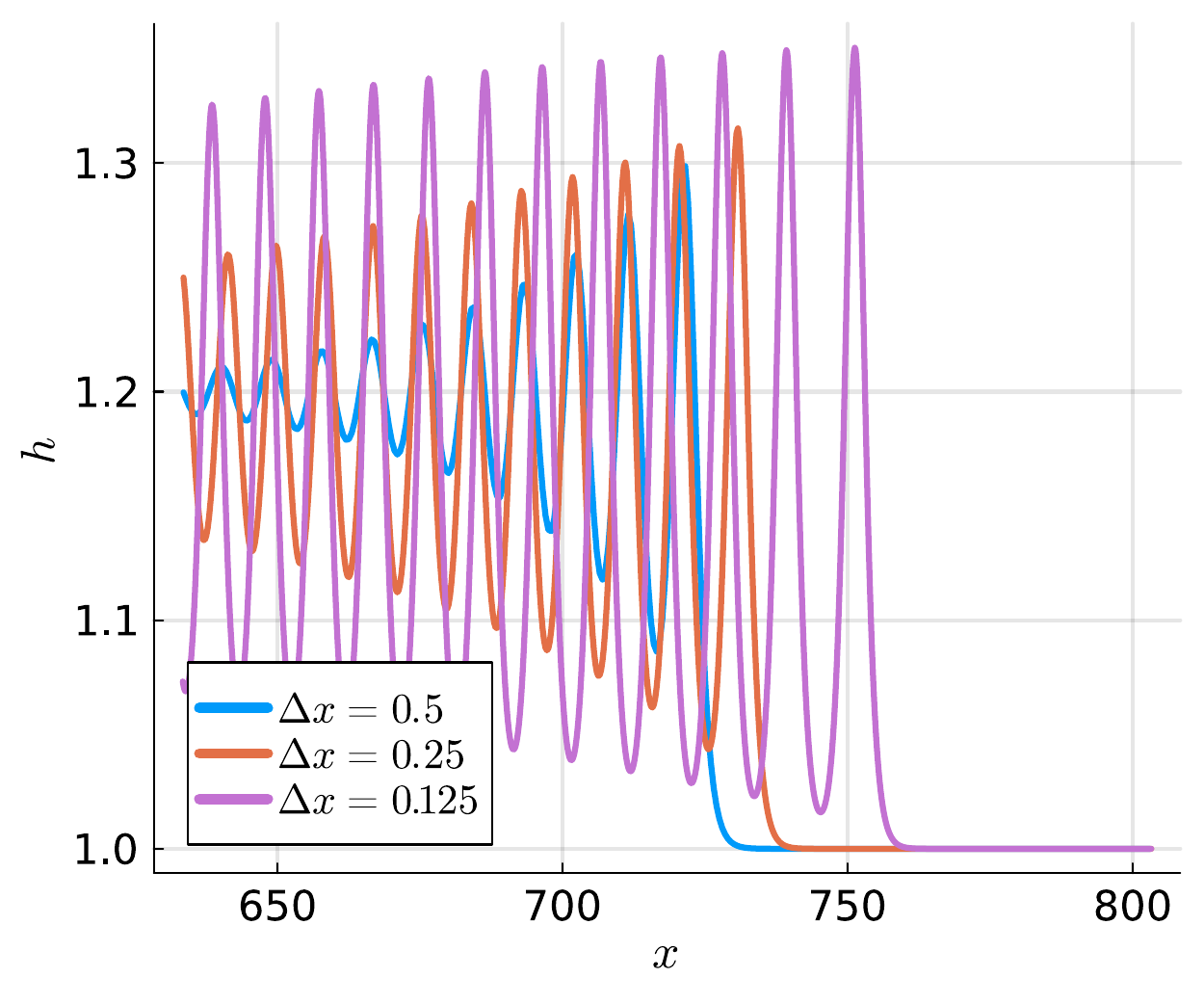}
    \end{subfigure}%
    \hspace{\fill}
    \begin{subfigure}{0.32\textwidth}
    \centering
        \includegraphics[width=\textwidth]{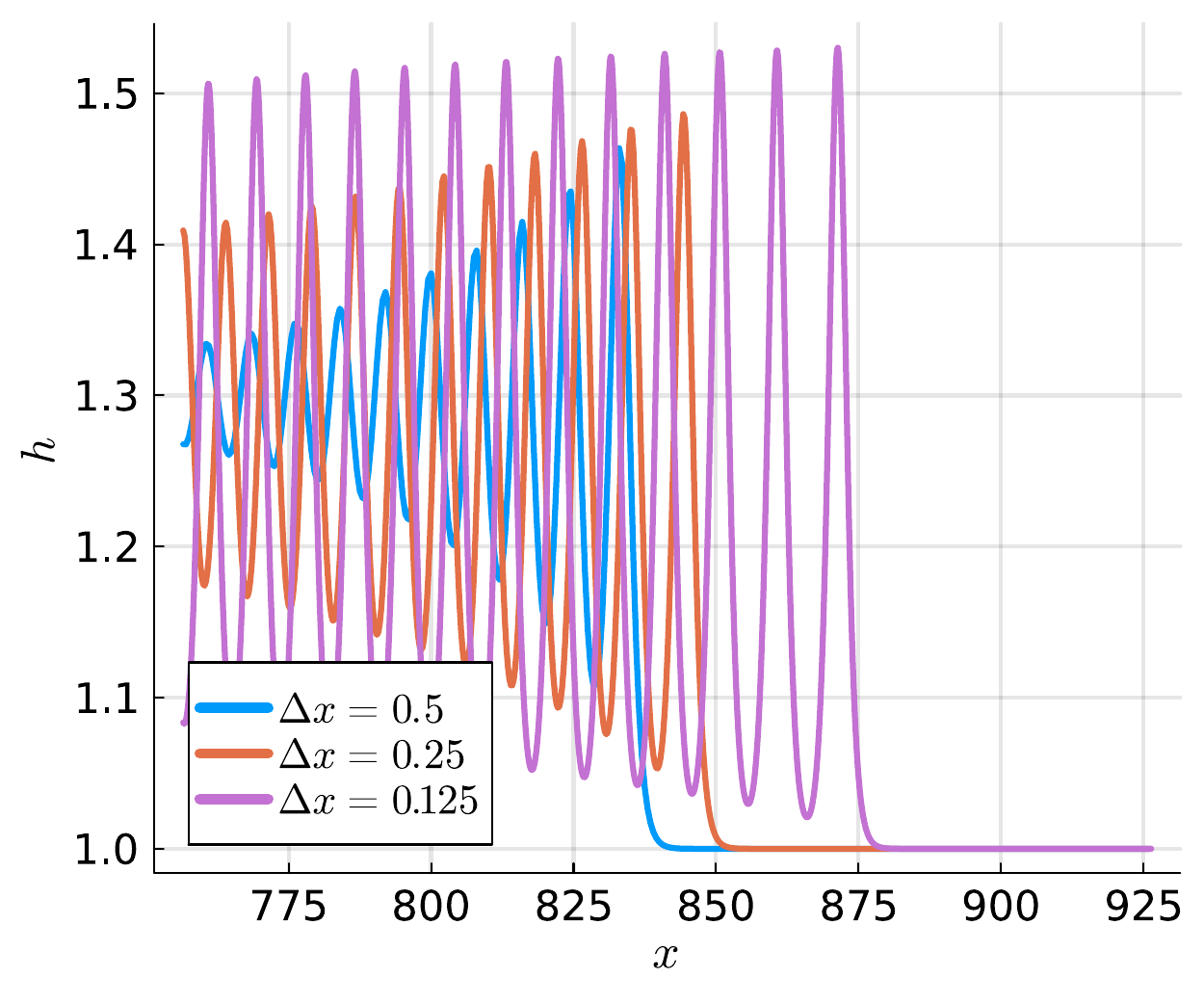}
    \end{subfigure}%
    \caption{Favre waves. Solutions of the original SGN system at dimensionless time $\tilde t=1500$ for nonlinearities $\epsilon =0.1$ (left),
    $\epsilon =0.2$ (center), and $\epsilon =0.3$ (right).
 Top: structure-preserving second-order finite difference scheme. Bottom: second-order finite difference  with artificial viscosity. }
    \label{fig:favre_waves-long1-sol1}
\end{figure}

\begin{figure}[htbp]
\centering
    \begin{subfigure}{0.32\textwidth}
    \centering
        \includegraphics[width=\textwidth]{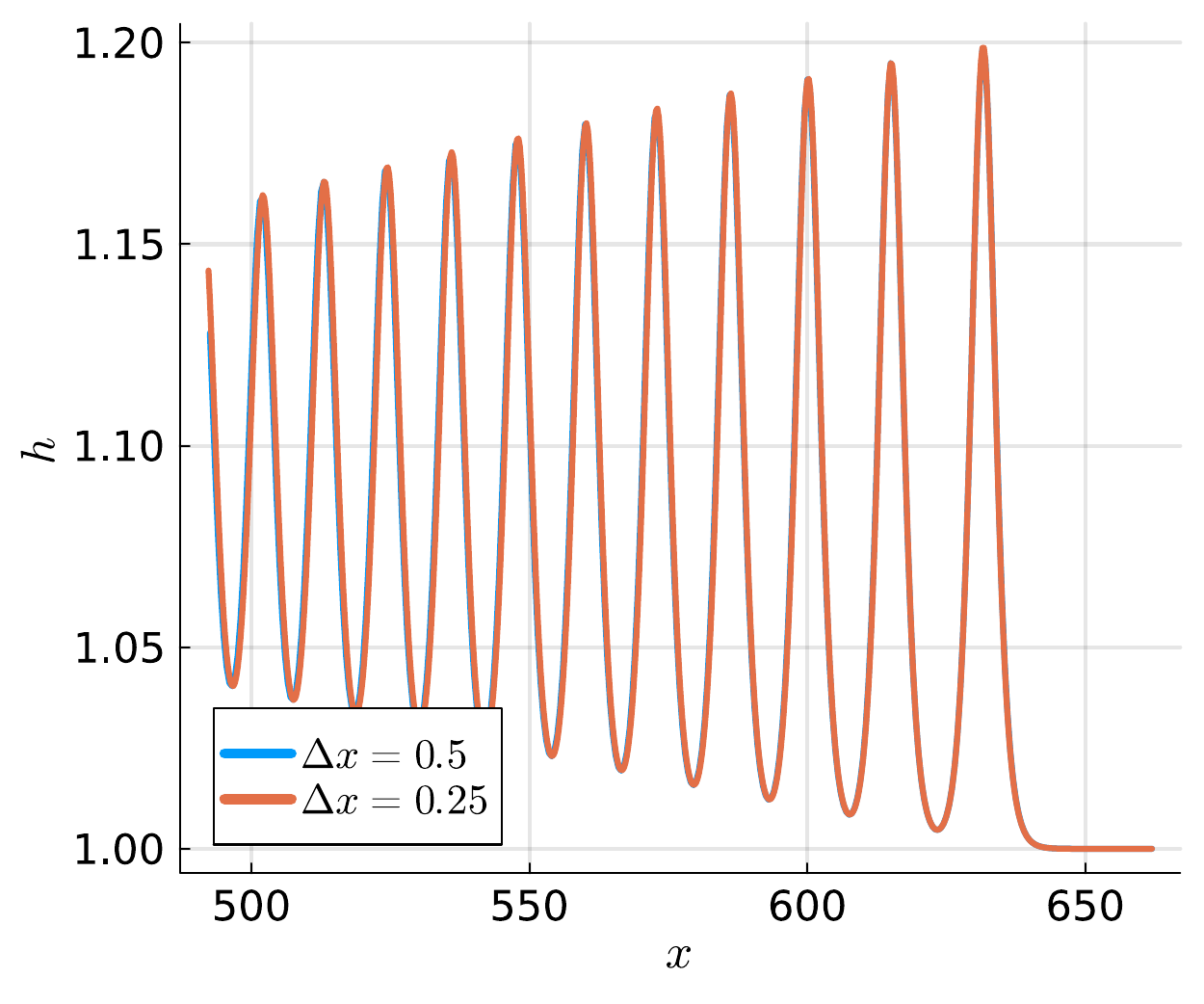}
    \end{subfigure}%
    \hspace{\fill}
    \begin{subfigure}{0.32\textwidth}
    \centering
        \includegraphics[width=\textwidth]{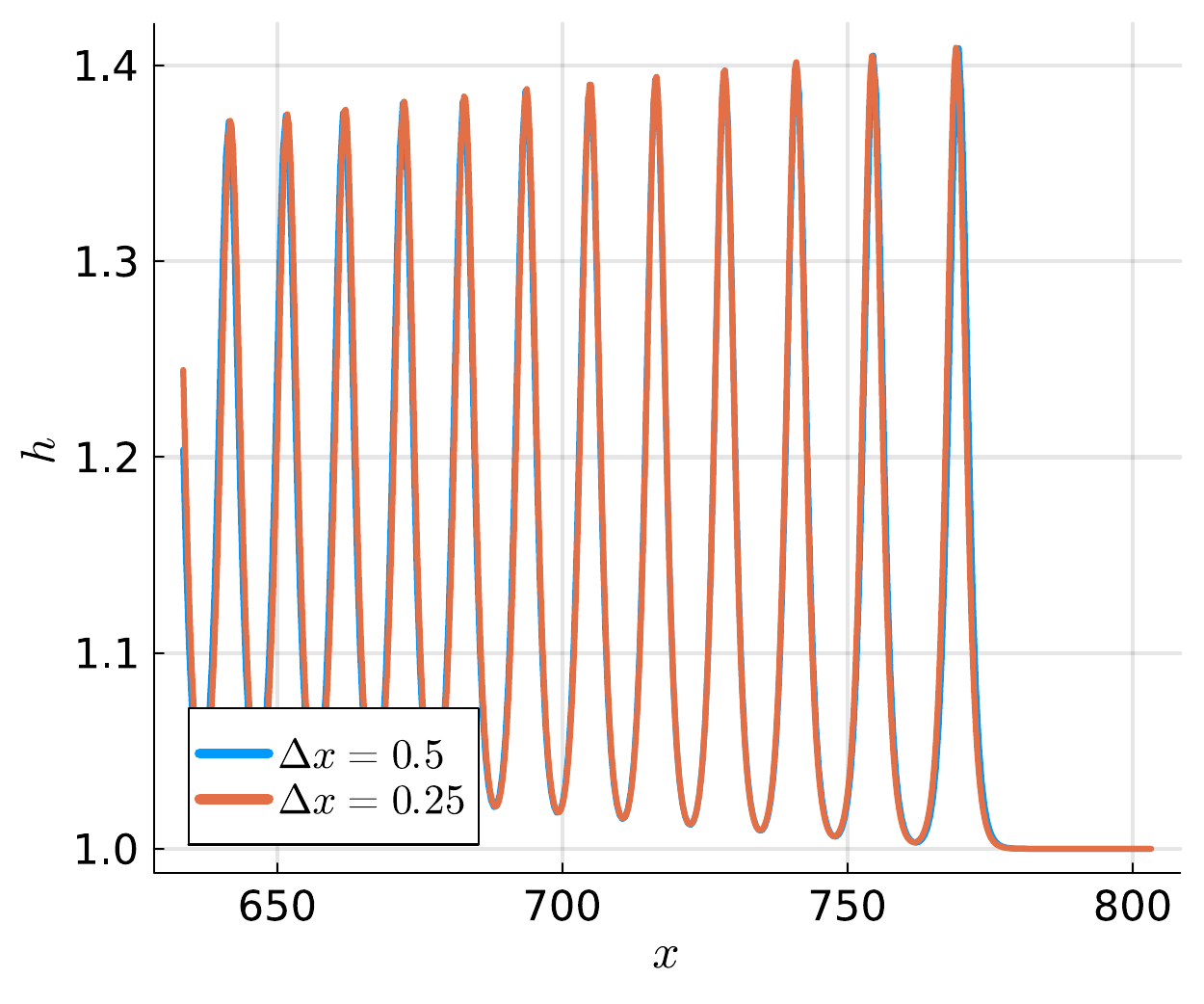}
    \end{subfigure}%
    \hspace{\fill}
    \begin{subfigure}{0.32\textwidth}
    \centering
        \includegraphics[width=\textwidth]{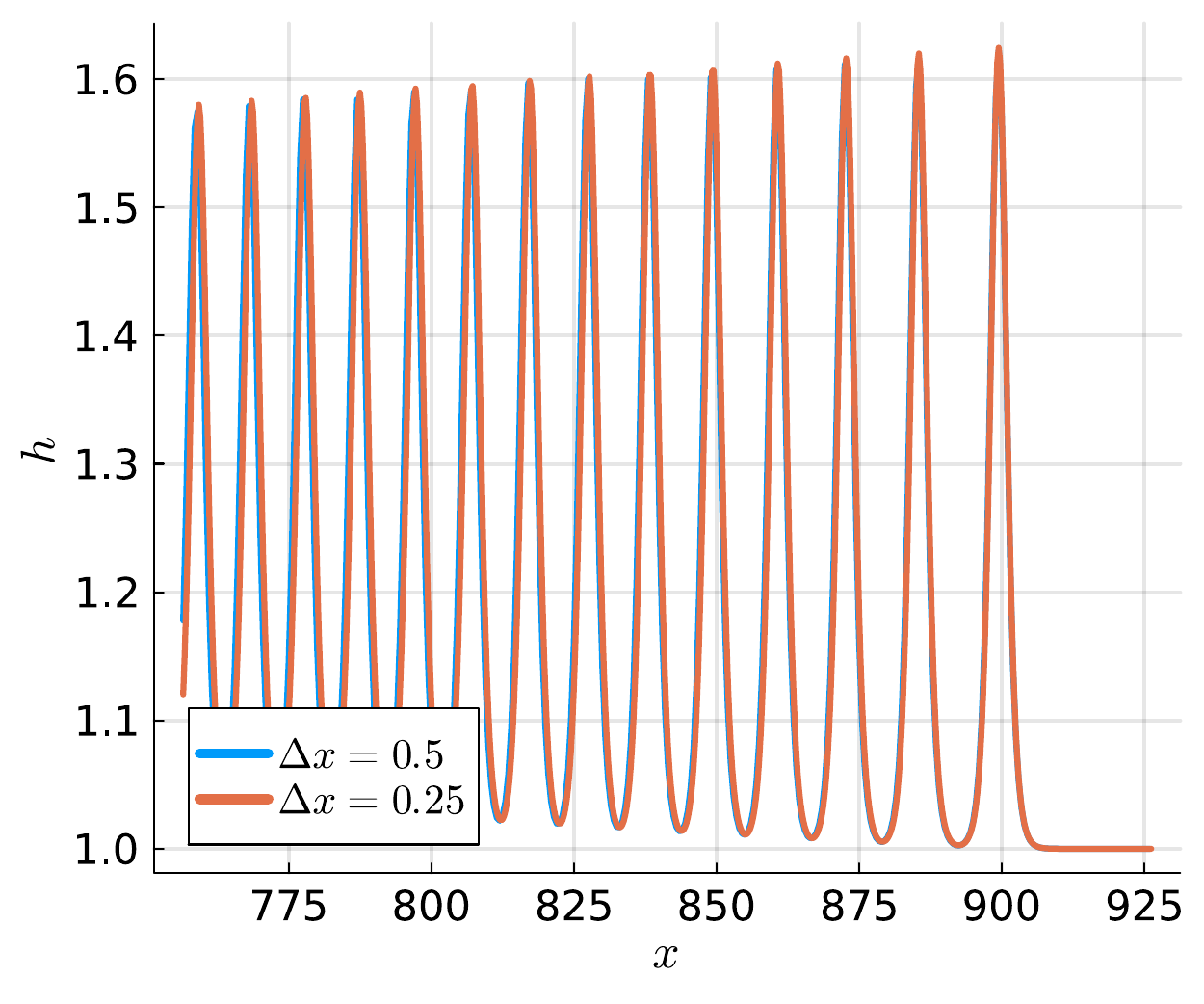}
    \end{subfigure}%
    \\
    \medskip
    \begin{subfigure}{0.32\textwidth}
    \centering
        \includegraphics[width=\textwidth]{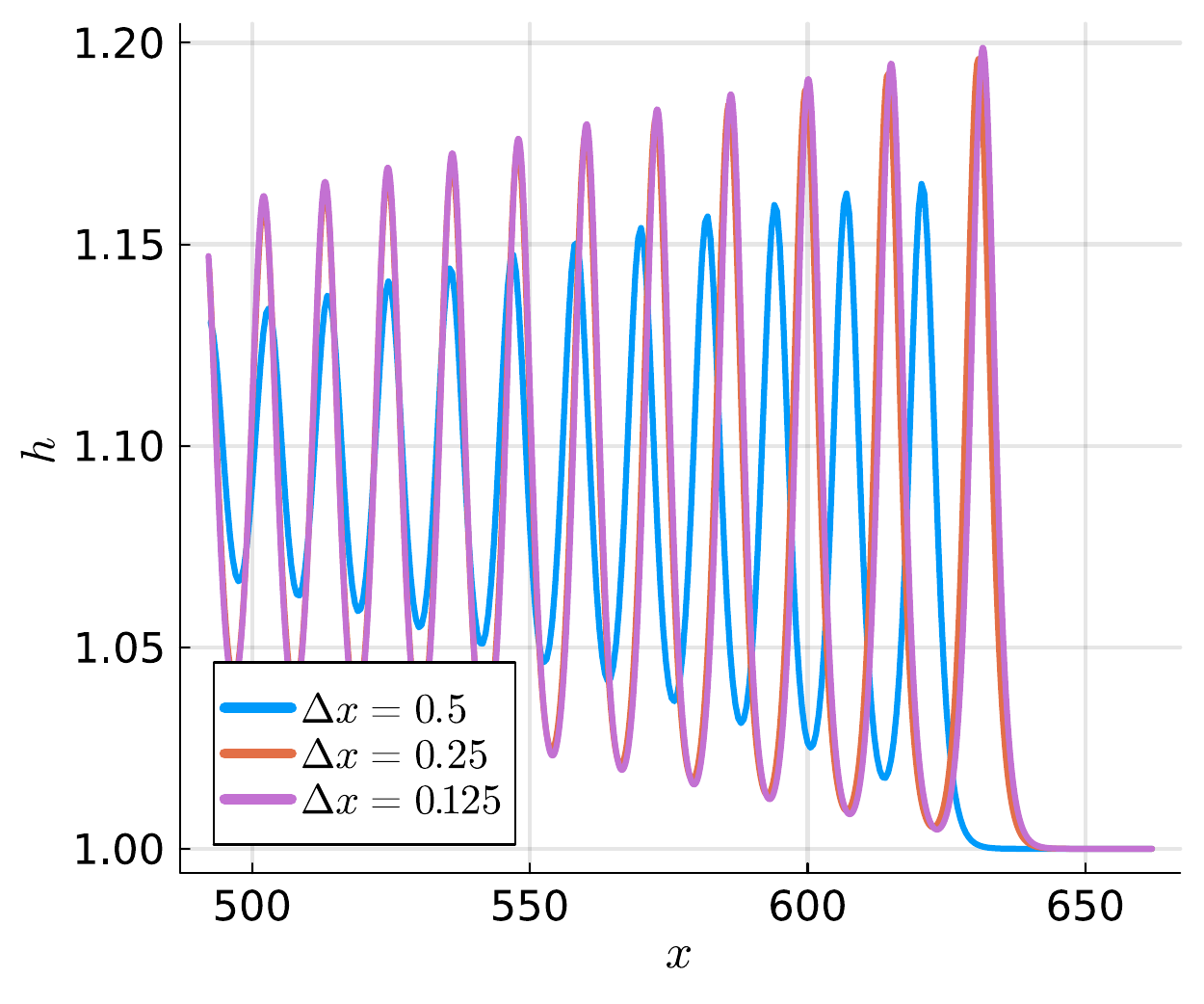}
    \end{subfigure}%
    \hspace{\fill}
    \begin{subfigure}{0.32\textwidth}
    \centering
        \includegraphics[width=\textwidth]{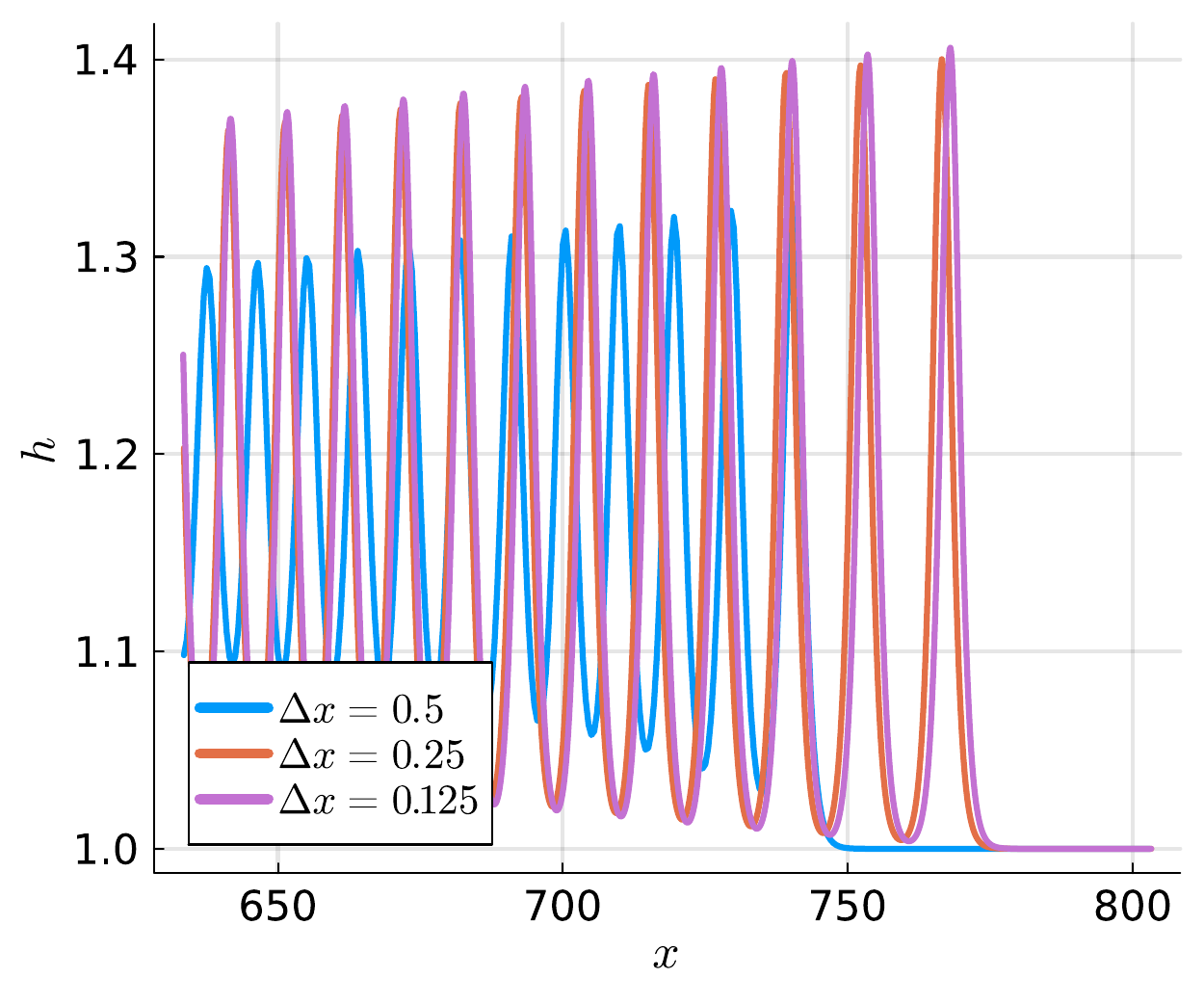}
    \end{subfigure}%
    \hspace{\fill}
    \begin{subfigure}{0.32\textwidth}
    \centering
        \includegraphics[width=\textwidth]{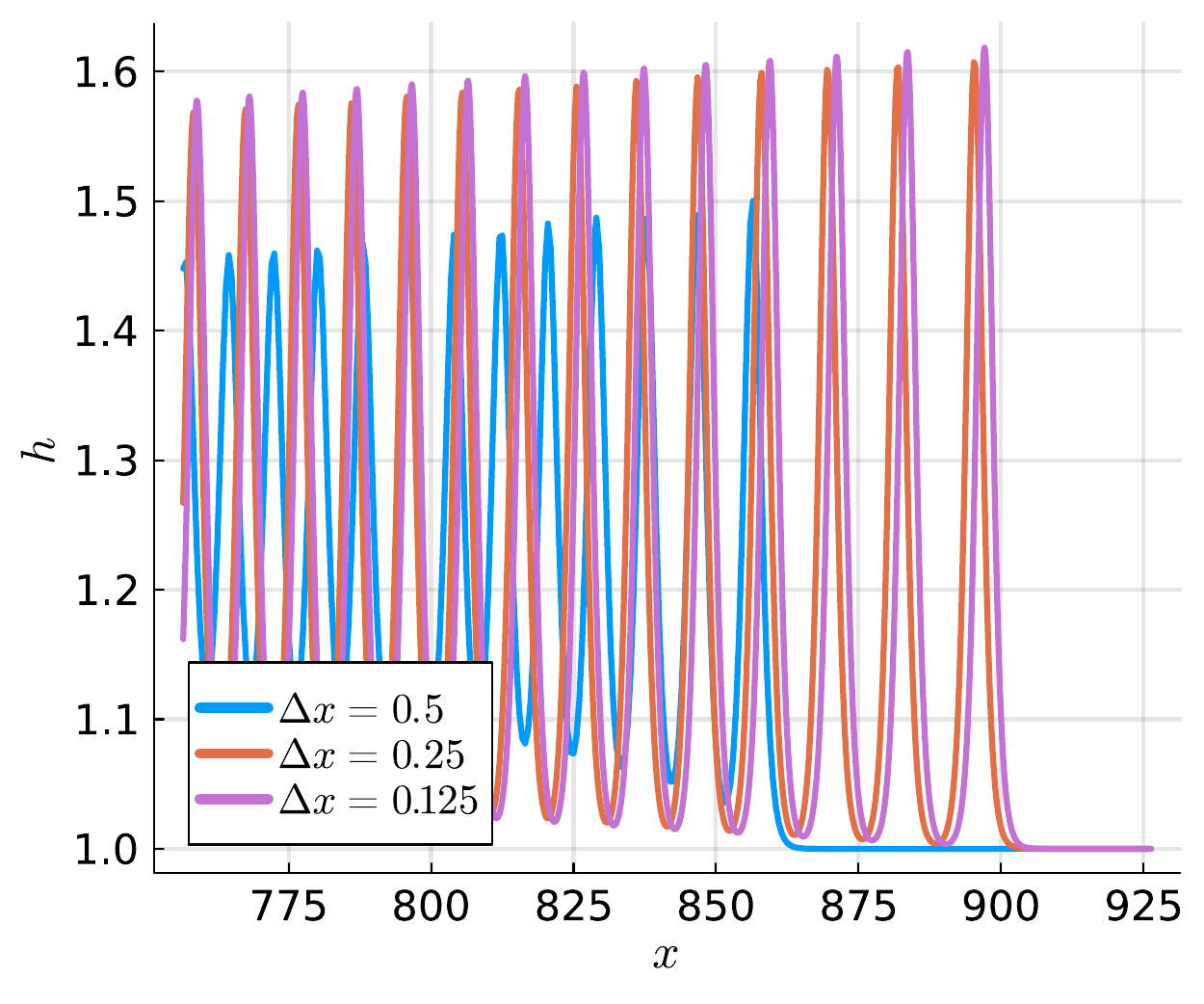}
    \end{subfigure}%
    \caption{Favre waves. Solutions of the original SGN system at dimensionless time $\tilde t=1500$ for nonlinearities $\epsilon =0.1$ (left),
    $\epsilon =0.2$ (center), and $\epsilon =0.3$ (right).
 Top: structure-preserving fourth-order finite difference scheme. Bottom: fourth-order finite difference  with artificial viscosity. }
    \label{fig:favre_waves-long1-sol2}
\end{figure}

\begin{figure}[htbp]
\centering
    \begin{subfigure}{0.32\textwidth}
    \centering
        \includegraphics[width=\textwidth]{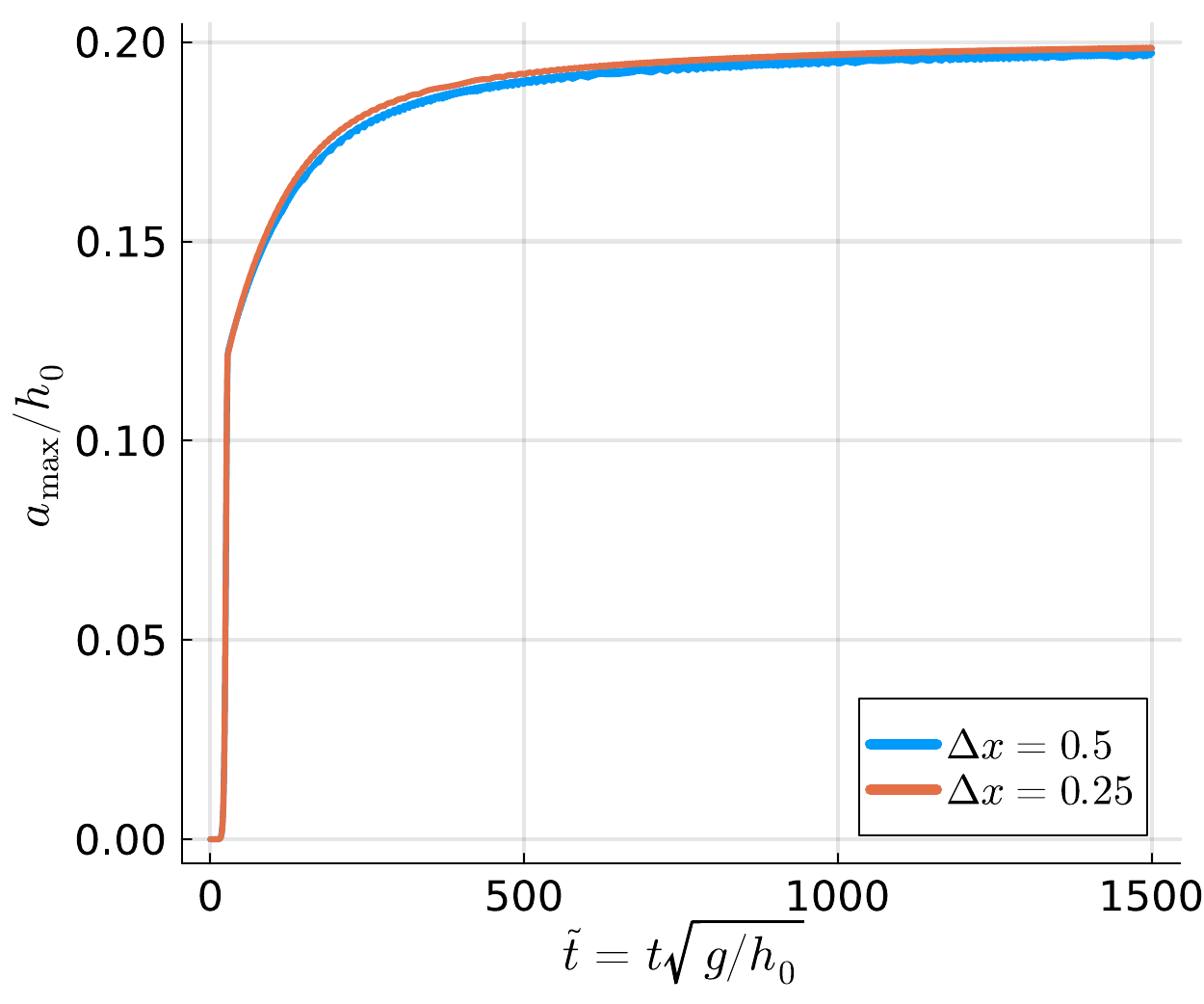}
    \end{subfigure}%
    \hspace{\fill}
    \begin{subfigure}{0.32\textwidth}
    \centering
        \includegraphics[width=\textwidth]{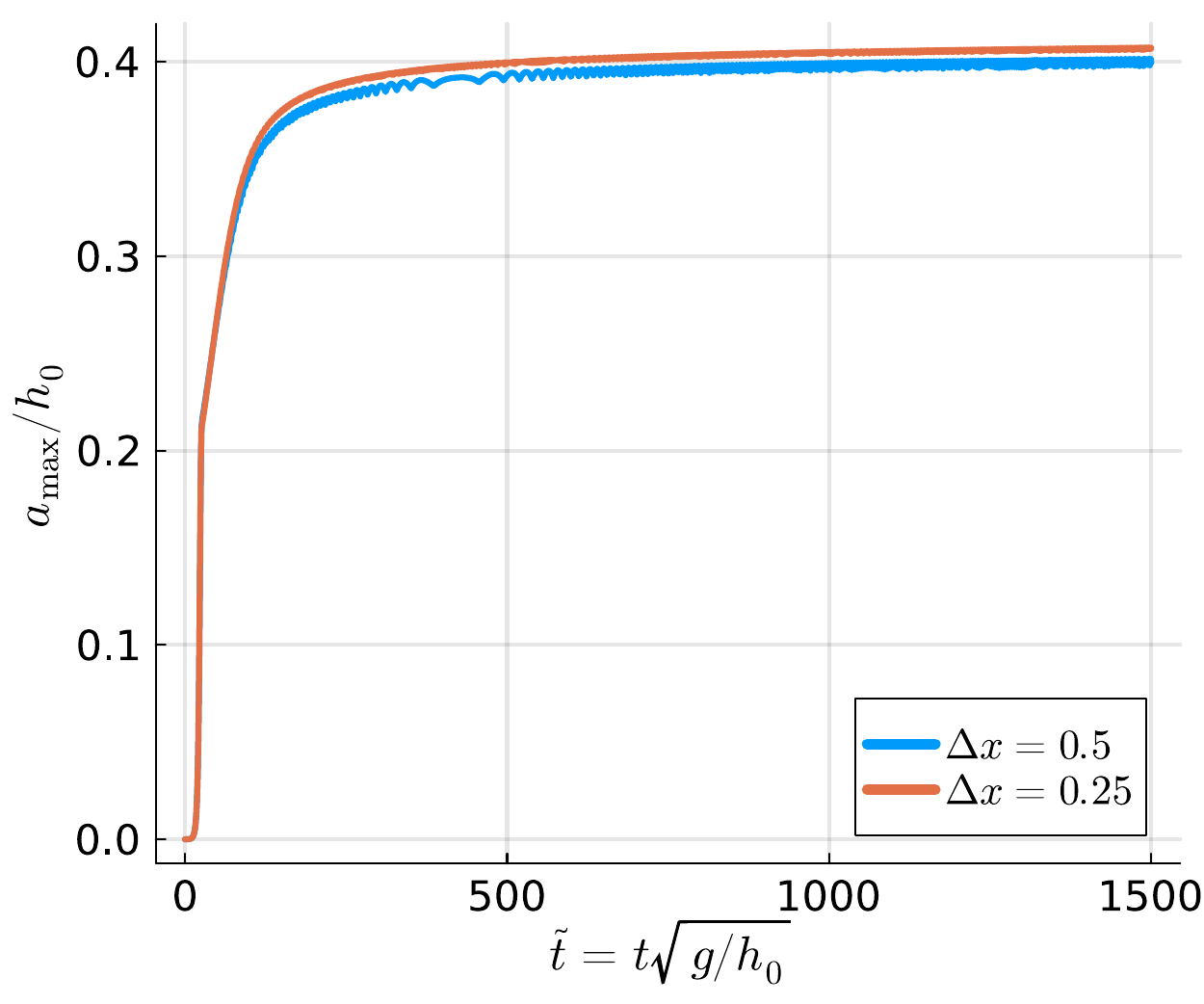}
    \end{subfigure}%
    \hspace{\fill}
    \begin{subfigure}{0.32\textwidth}
    \centering
        \includegraphics[width=\textwidth]{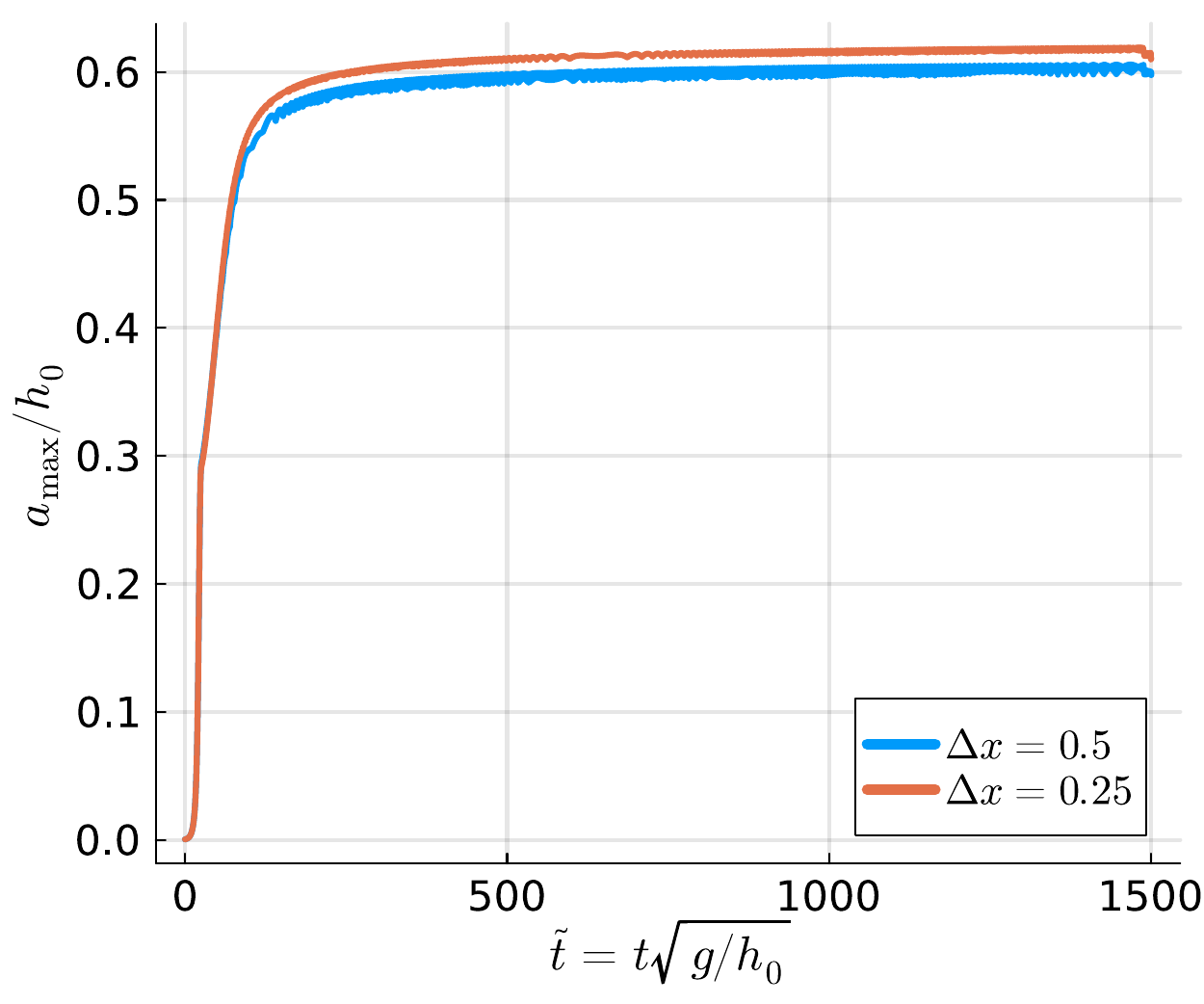}
    \end{subfigure}%
    \\
    \medskip
    \begin{subfigure}{0.32\textwidth}
    \centering
        \includegraphics[width=\textwidth]{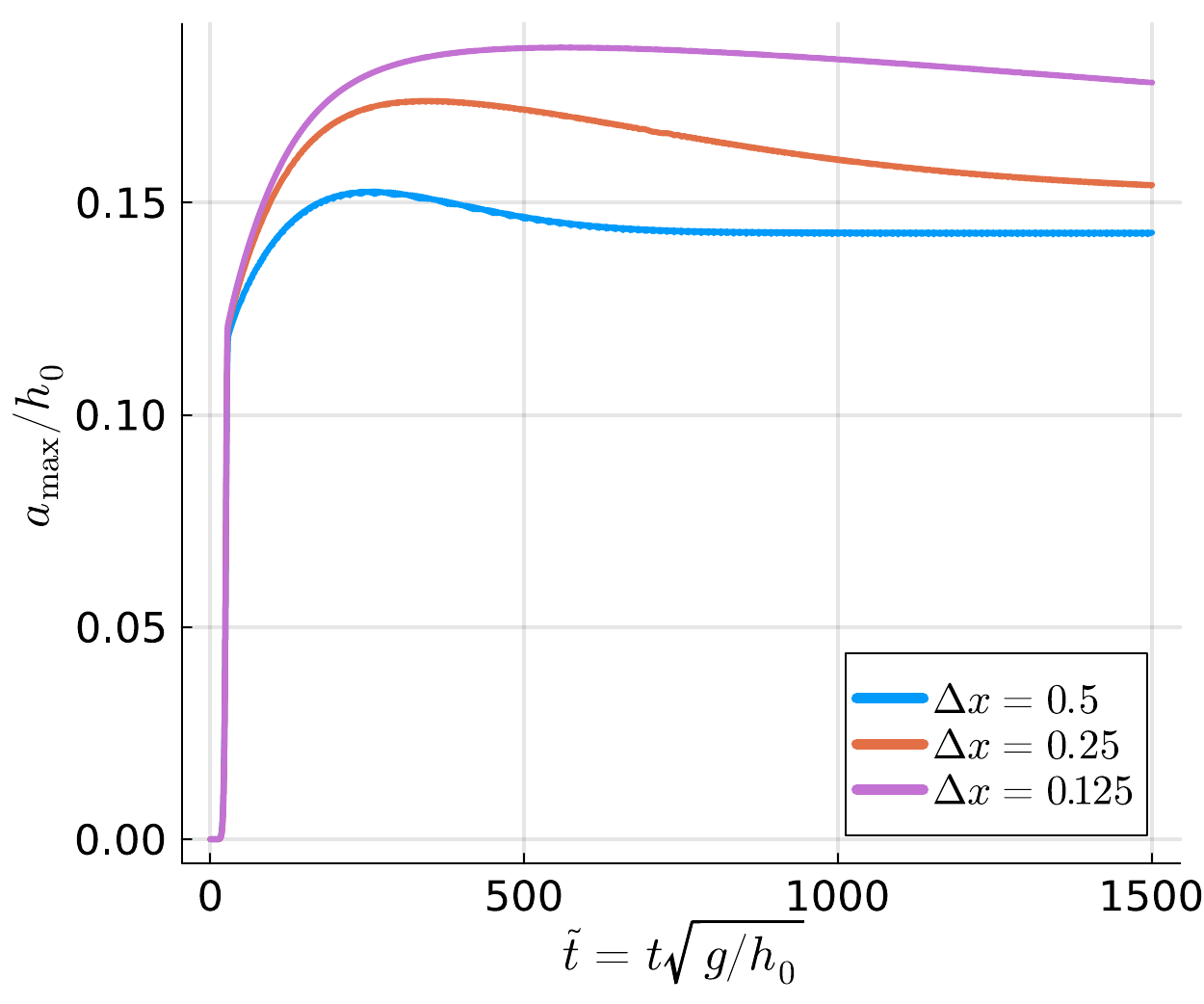}
    \end{subfigure}%
    \hspace{\fill}
    \begin{subfigure}{0.32\textwidth}
    \centering
        \includegraphics[width=\textwidth]{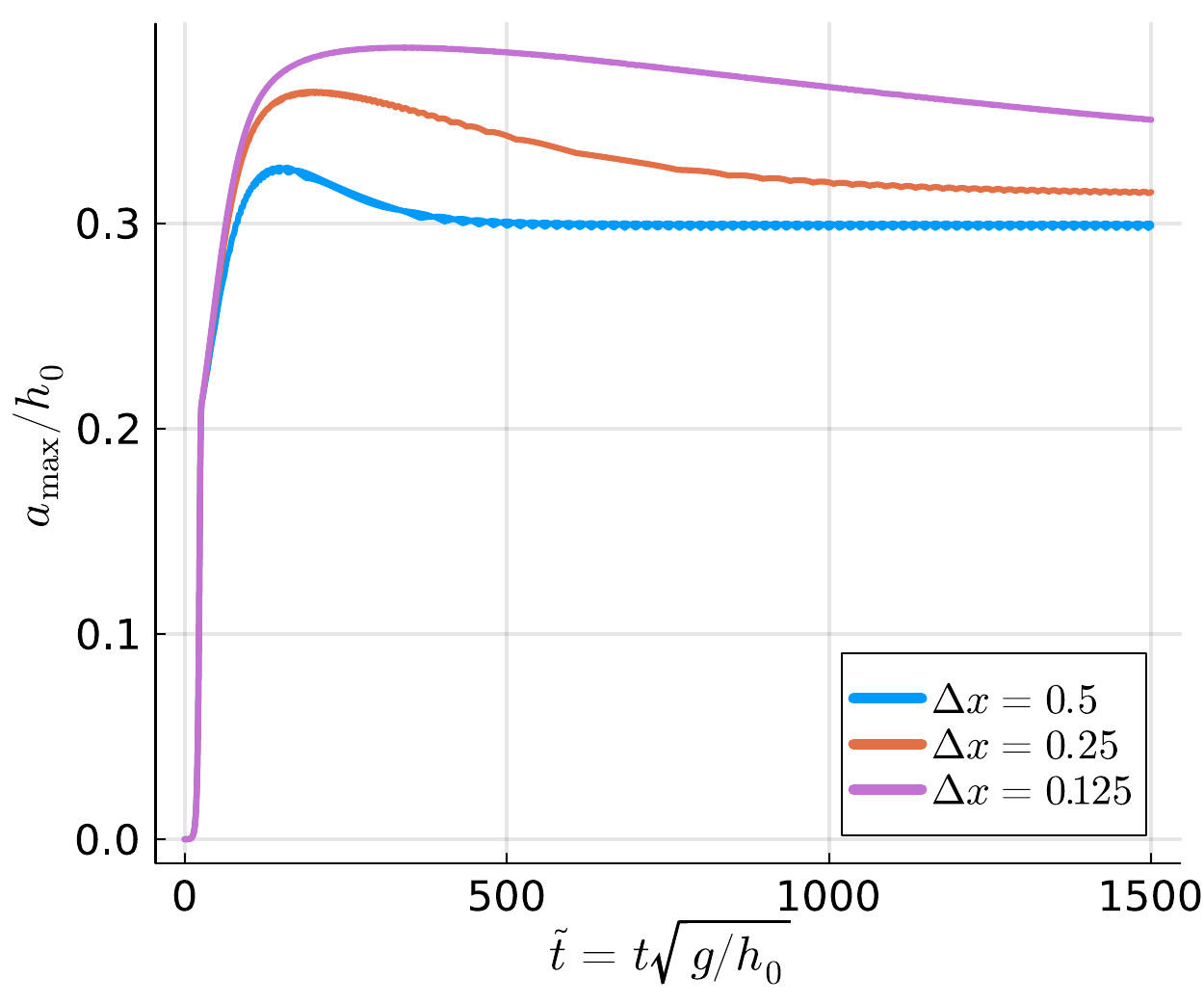}
    \end{subfigure}%
    \hspace{\fill}
    \begin{subfigure}{0.32\textwidth}
    \centering
        \includegraphics[width=\textwidth]{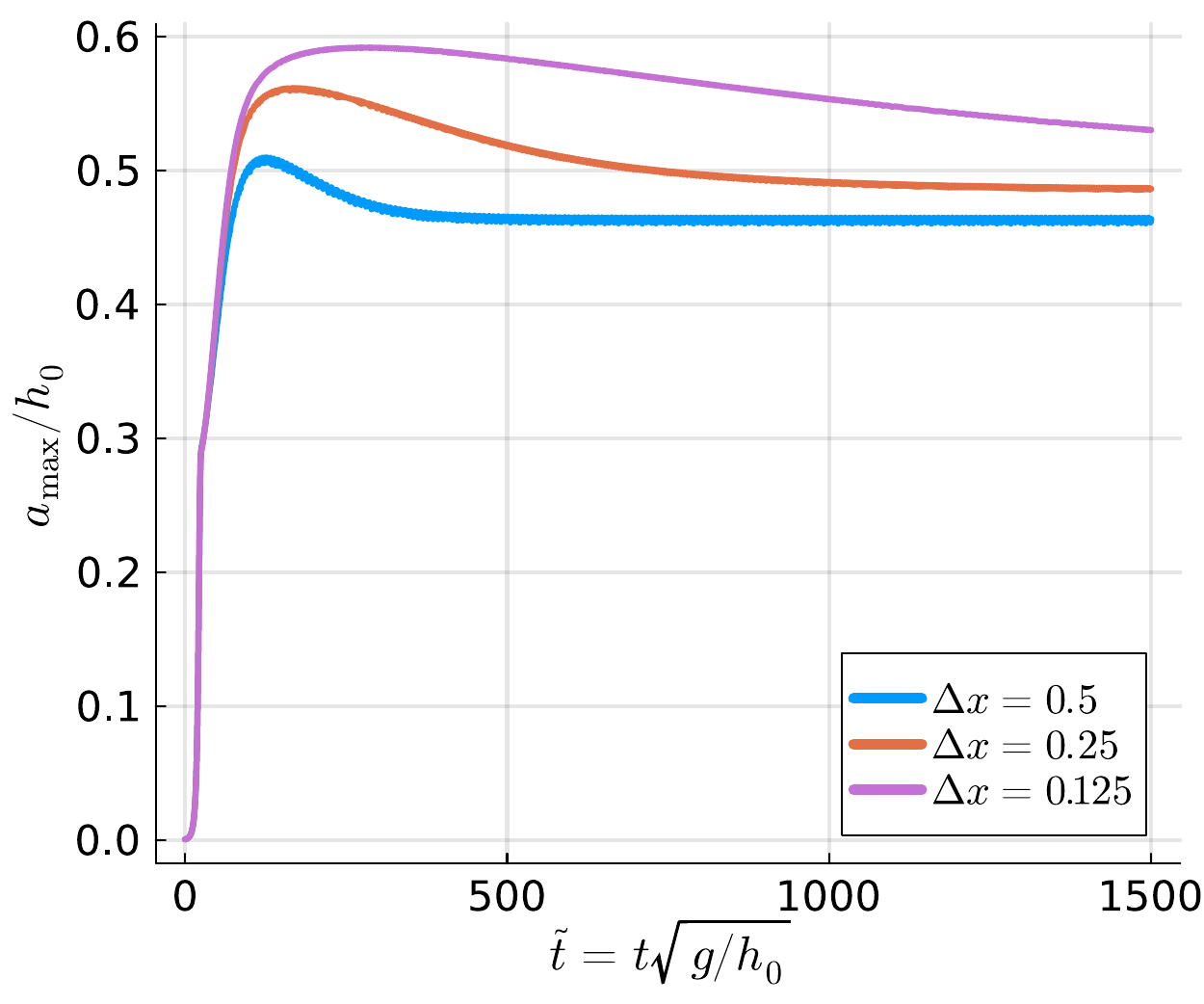}
    \end{subfigure}%
    \caption{Favre waves. Evolution of the maximum amplitude   obtained from  the original SGN system   with  nonlinearities $\epsilon =0.1$ (left),
    $\epsilon =0.2$ (center), and $\epsilon =0.3$ (right).
 Top: structure-preserving second-order finite difference scheme. Bottom: second-order finite difference  with artificial viscosity. }
    \label{fig:favre_waves-long1-amax1}
\end{figure}

\begin{figure}[htbp]
\centering
    \begin{subfigure}{0.32\textwidth}
    \centering
        \includegraphics[width=\textwidth]{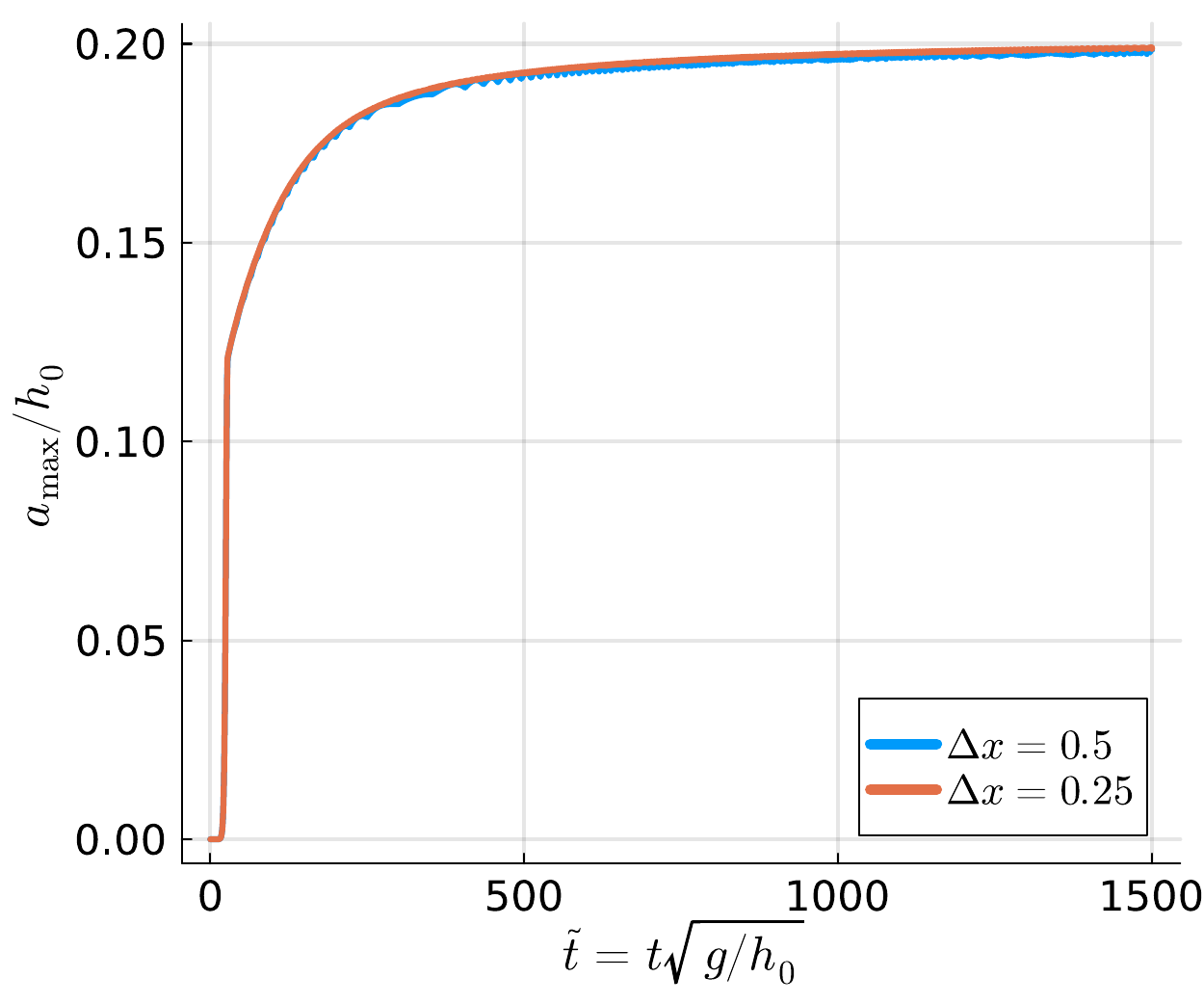}
    \end{subfigure}%
    \hspace{\fill}
    \begin{subfigure}{0.32\textwidth}
    \centering
        \includegraphics[width=\textwidth]{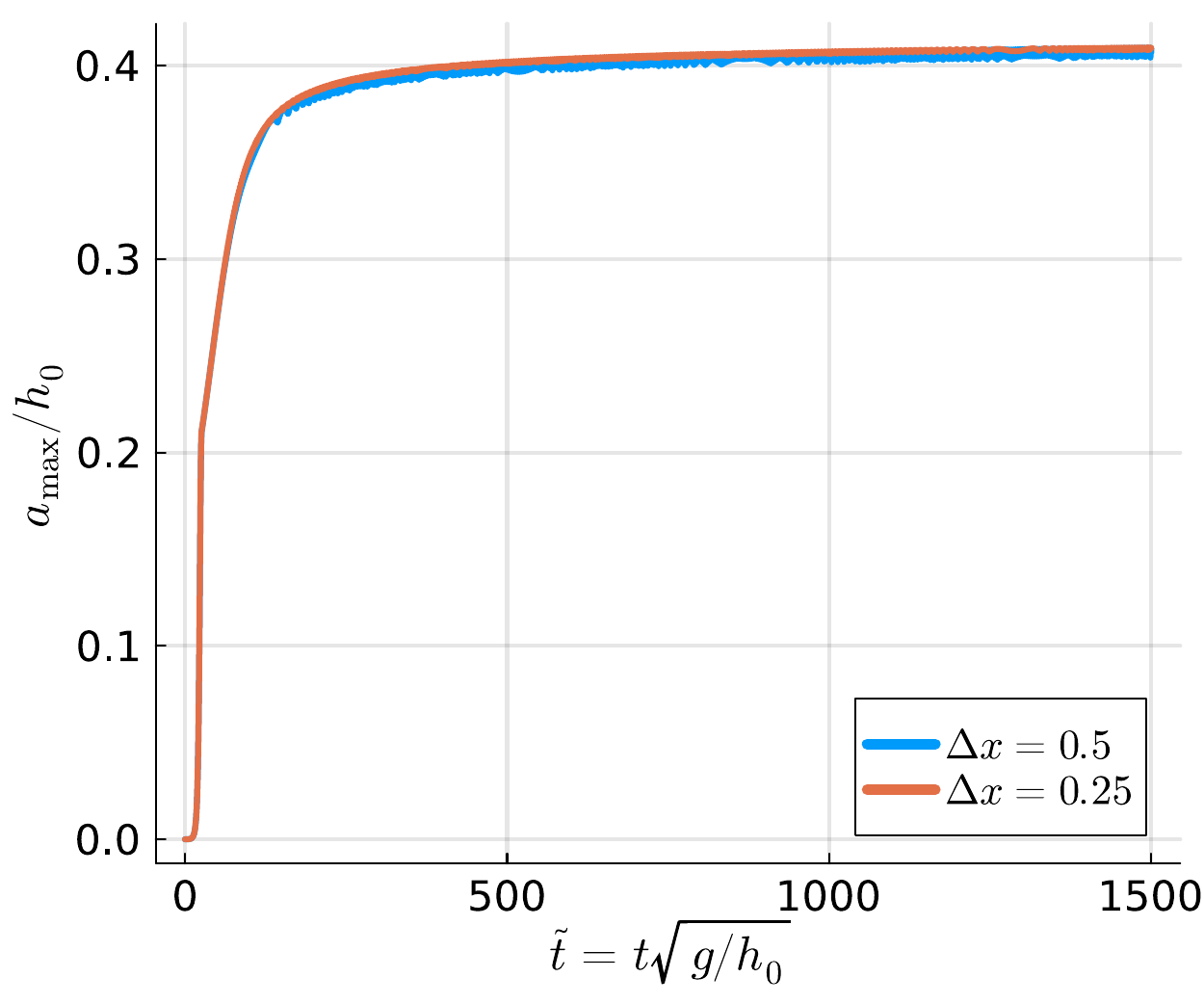}
    \end{subfigure}%
    \hspace{\fill}
    \begin{subfigure}{0.32\textwidth}
    \centering
        \includegraphics[width=\textwidth]{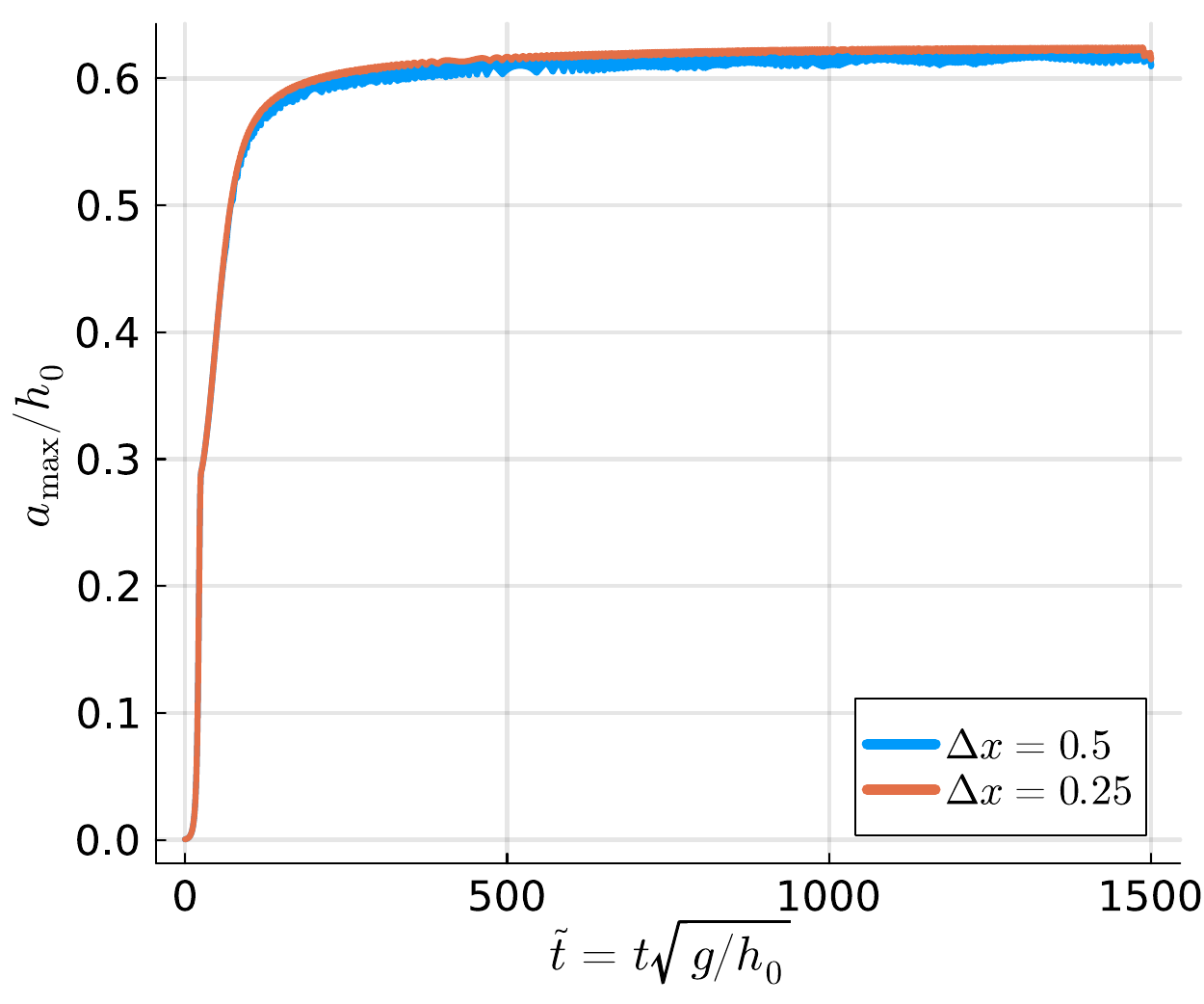}
    \end{subfigure}%
    \\
    \medskip
    \begin{subfigure}{0.32\textwidth}
    \centering
        \includegraphics[width=\textwidth]{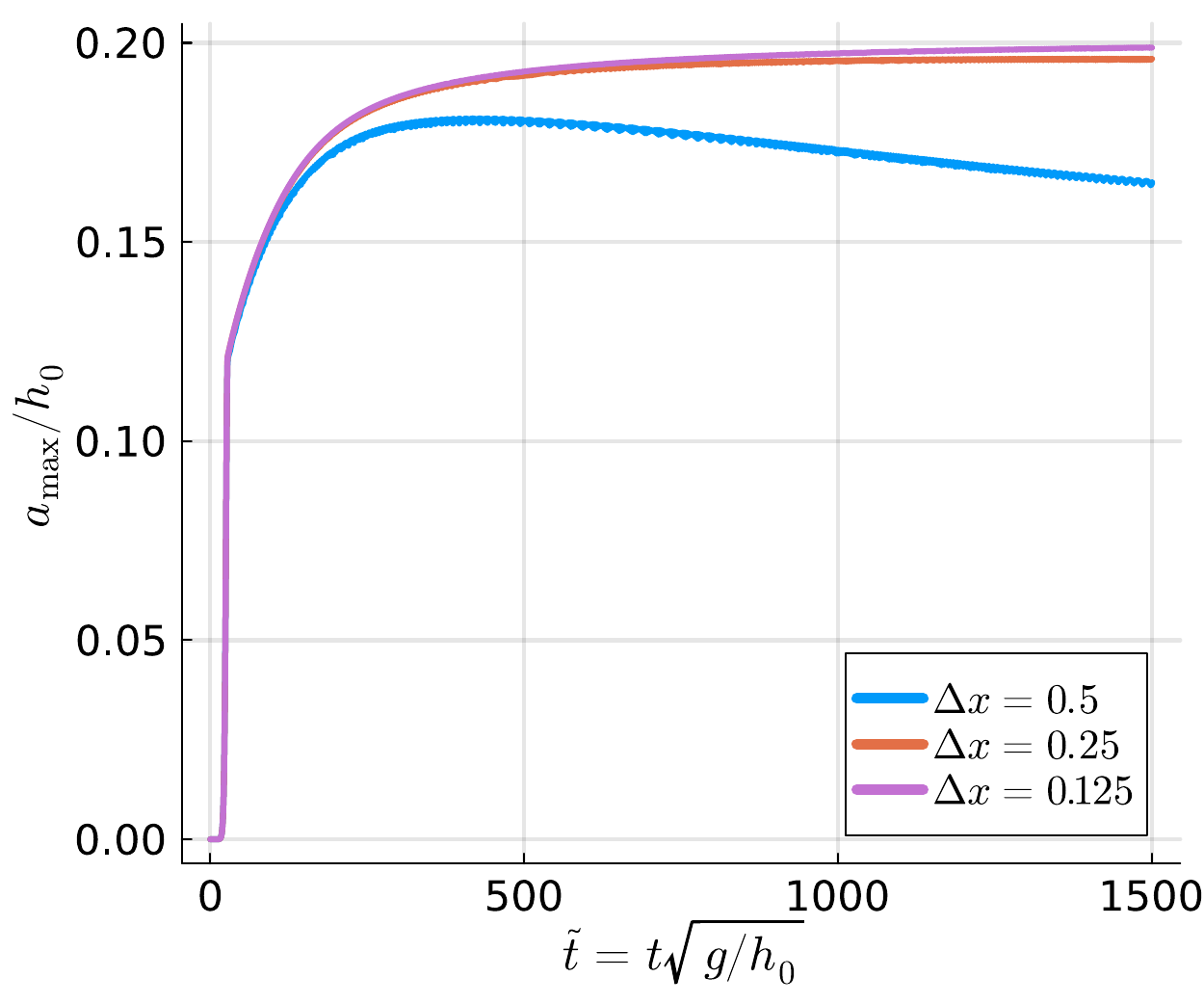}
    \end{subfigure}%
    \hspace{\fill}
    \begin{subfigure}{0.32\textwidth}
    \centering
        \includegraphics[width=\textwidth]{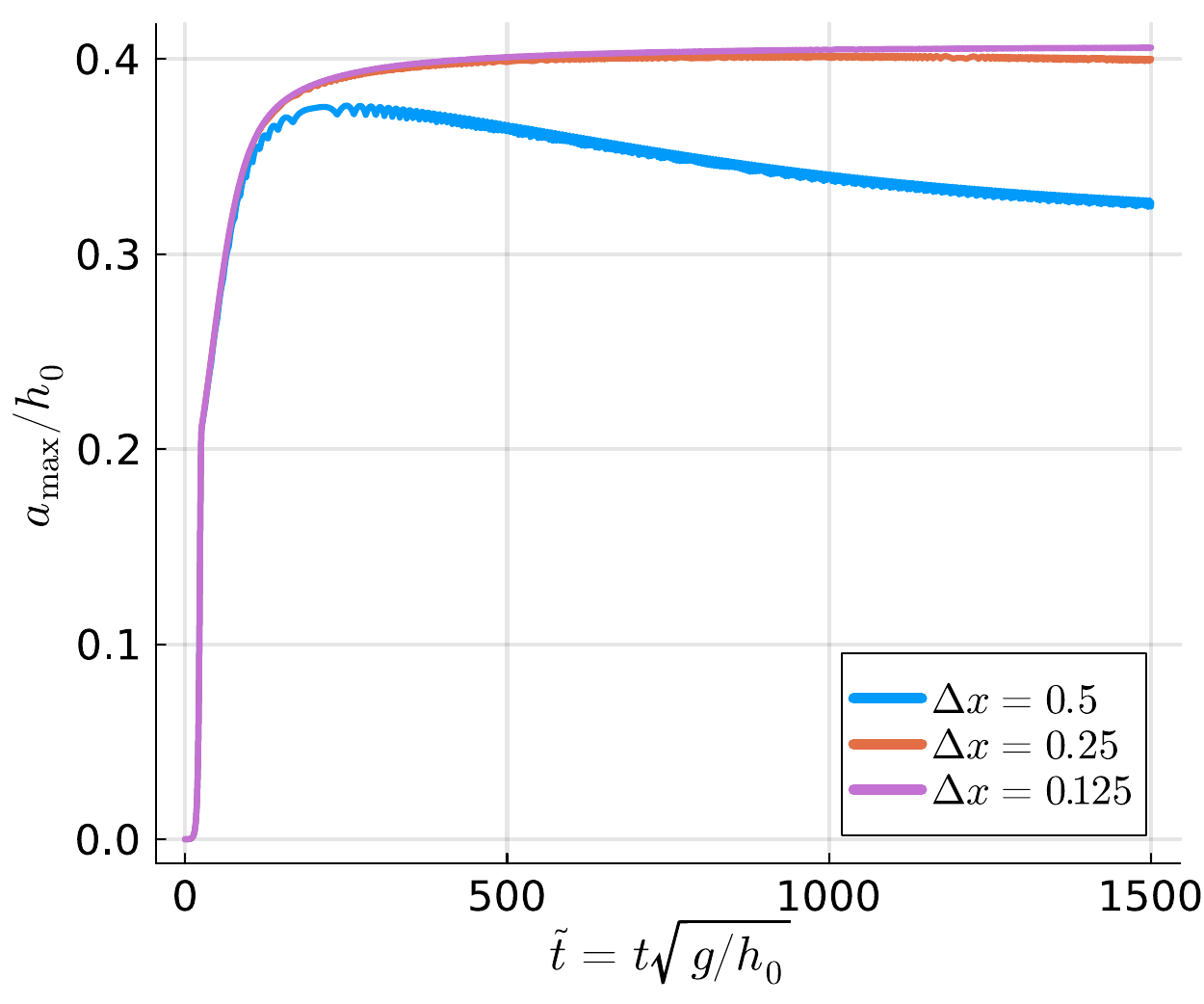}
    \end{subfigure}%
    \hspace{\fill}
    \begin{subfigure}{0.32\textwidth}
    \centering
        \includegraphics[width=\textwidth]{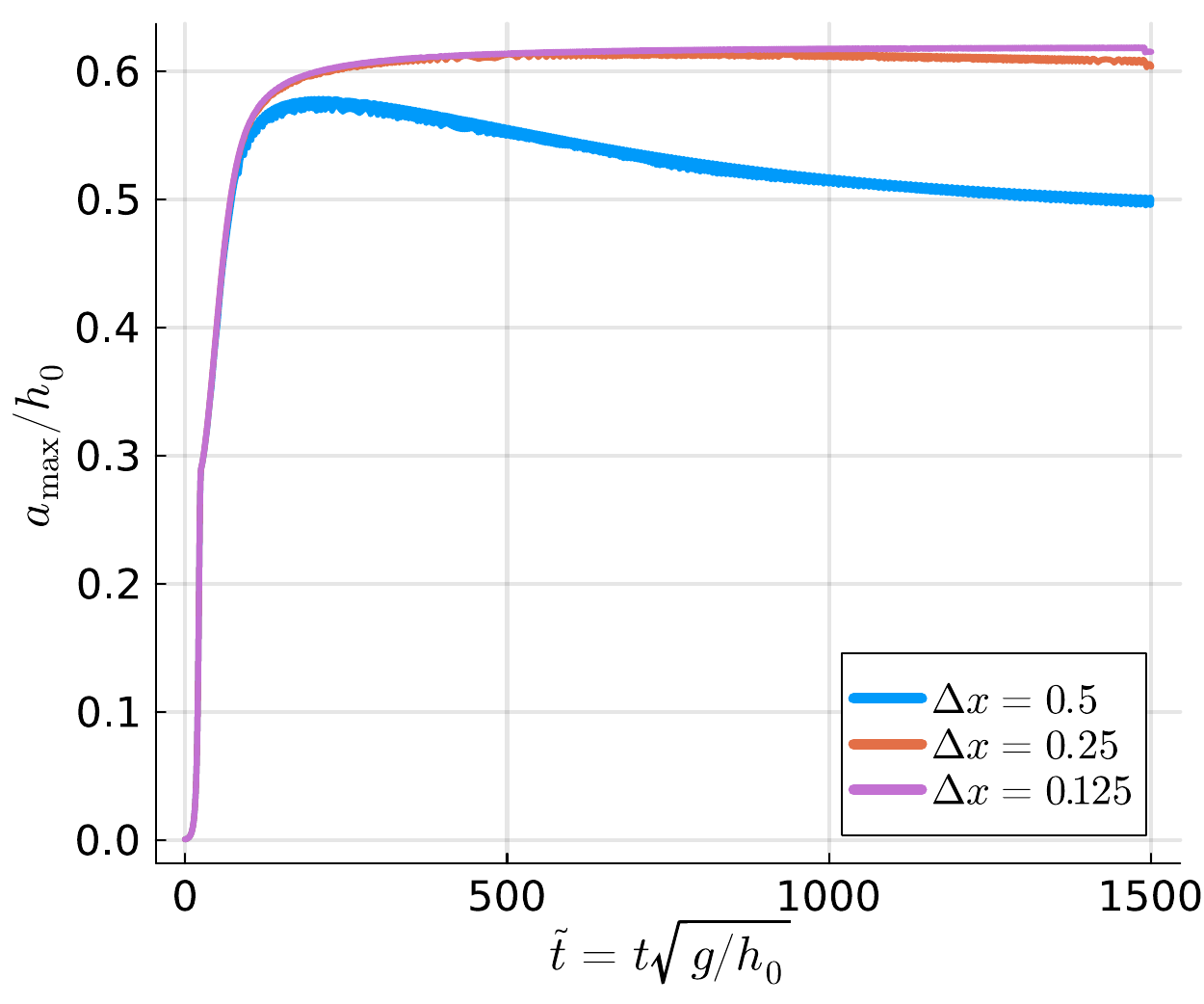}
    \end{subfigure}%
    \caption{Favre waves. Evolution of the maximum amplitude   obtained from  the original SGN system   with  nonlinearities $\epsilon =0.1$ (left),
    $\epsilon =0.2$ (center), and $\epsilon =0.3$ (right).
 Top: structure-preserving fourth-order finite difference scheme. Bottom: fourth-order finite difference with artificial viscosity. }
    \label{fig:favre_waves-long1-amax2}
\end{figure}

These observations are confirmed by the plots of the maximum amplitude in Figures~\ref{fig:favre_waves-long1-amax1}--\ref{fig:favre_waves-long1-amax2}.
The structure-preserving schemes provide already on the coarsest mesh an excellent approximation of the converged height.
We can see that two more refinements would be required with a
second-order scheme to match this value, while one refinement is required when using a fourth-order method.
These results are qualitatively in line with those of \cite{jouy_etal24}.
They generalize such results to genuinely structure-preserving discretizations
of the Serre-Green-Naghdi equations.  We refer to  the last reference   for similar results when physical dissipation (friction) is included.\\

We do not show the wave lengths $\lambda$ of the Favre waves, since they increase
over time (in the absence of friction/dissipation).
This is in accordance with the results we have observed for the soliton fission problem in Section~\ref{sec:soliton_fission}.
However, we  compare for completeness the numerical maximal amplitudes
$a_\mathrm{max}$   with the data by \cite{favre1935,treske1994}.
To this end,  we compute $a_\mathrm{max}$  for different Froude numbers
$$
    \mathrm{Fr} = \frac{\sigma}{\sqrt{g h_0}} = \sqrt{ (1 + \epsilon)  ( 1 + \epsilon / 2 )  }
$$
by choosing $\epsilon \in \{0.02, 0.06, 0.10, \dots, 0.30\}$ for $\Delta x = 0.02$,
where $\sigma$ is the non-dispersive/average bore speed.
We stop the simulations when the first wave (with amplitude $a_\mathrm{max}$) has travelled
roughly the same distance as in the experiments, i.e., $\SI{63.5}{m}$.
The results are shown in Figure~\ref{fig:favre_amplitude_over_froude}.
The numerical and experimental data agree very well for Froude numbers $\mathrm{Fr} < 1.25$.
For larger Froude numbers, wave breaking modelling is required to capture the correct amplitudes.

\begin{figure}[htbp]
\centering
    \begin{subfigure}{0.49\textwidth}
    \centering
        \includegraphics[width=\textwidth]{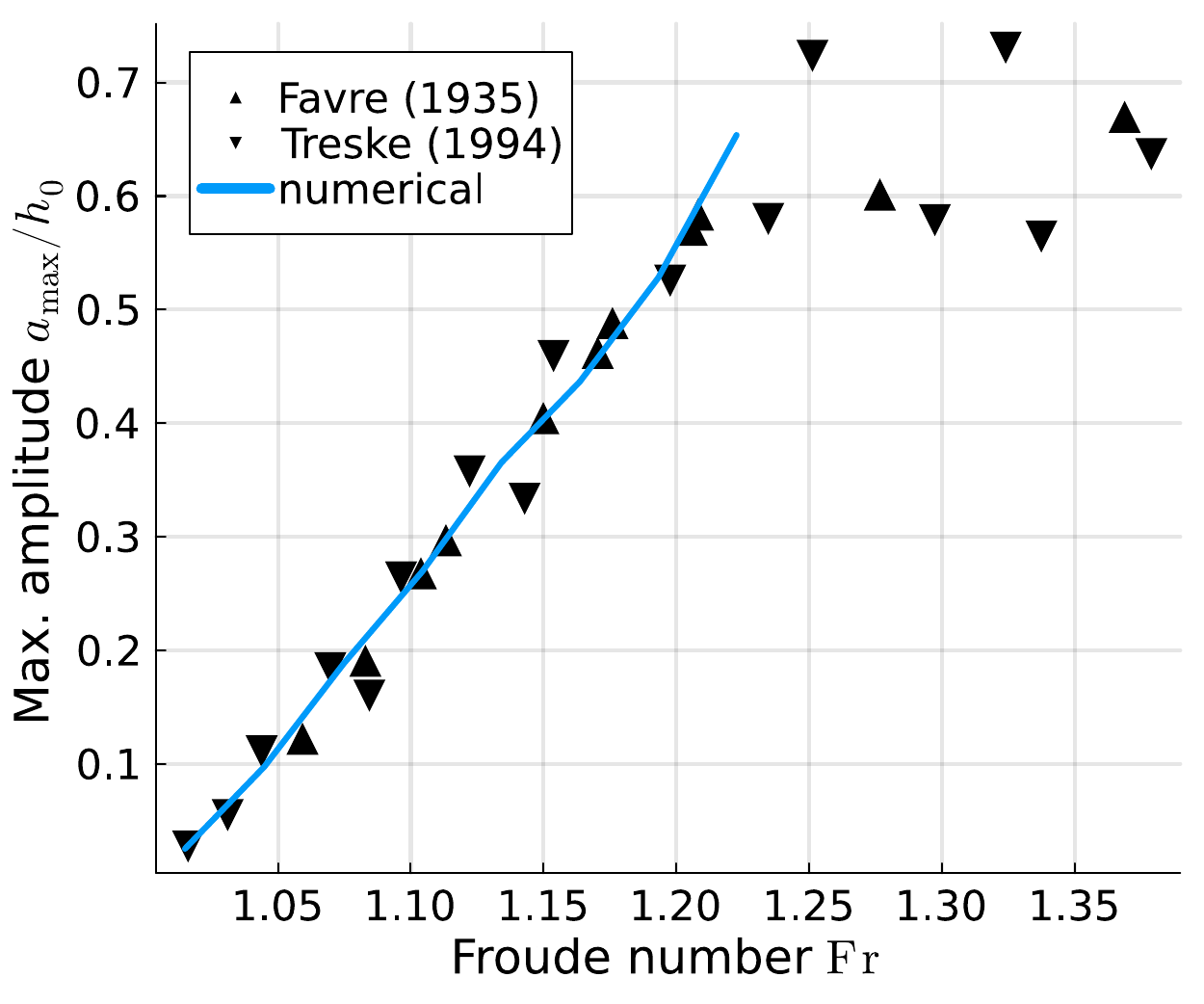}
        \caption{Hyperbolic approximation with $\lambda = 500$.}
    \end{subfigure}%
    \hspace{\fill}
    \begin{subfigure}{0.49\textwidth}
    \centering
        \includegraphics[width=\textwidth]{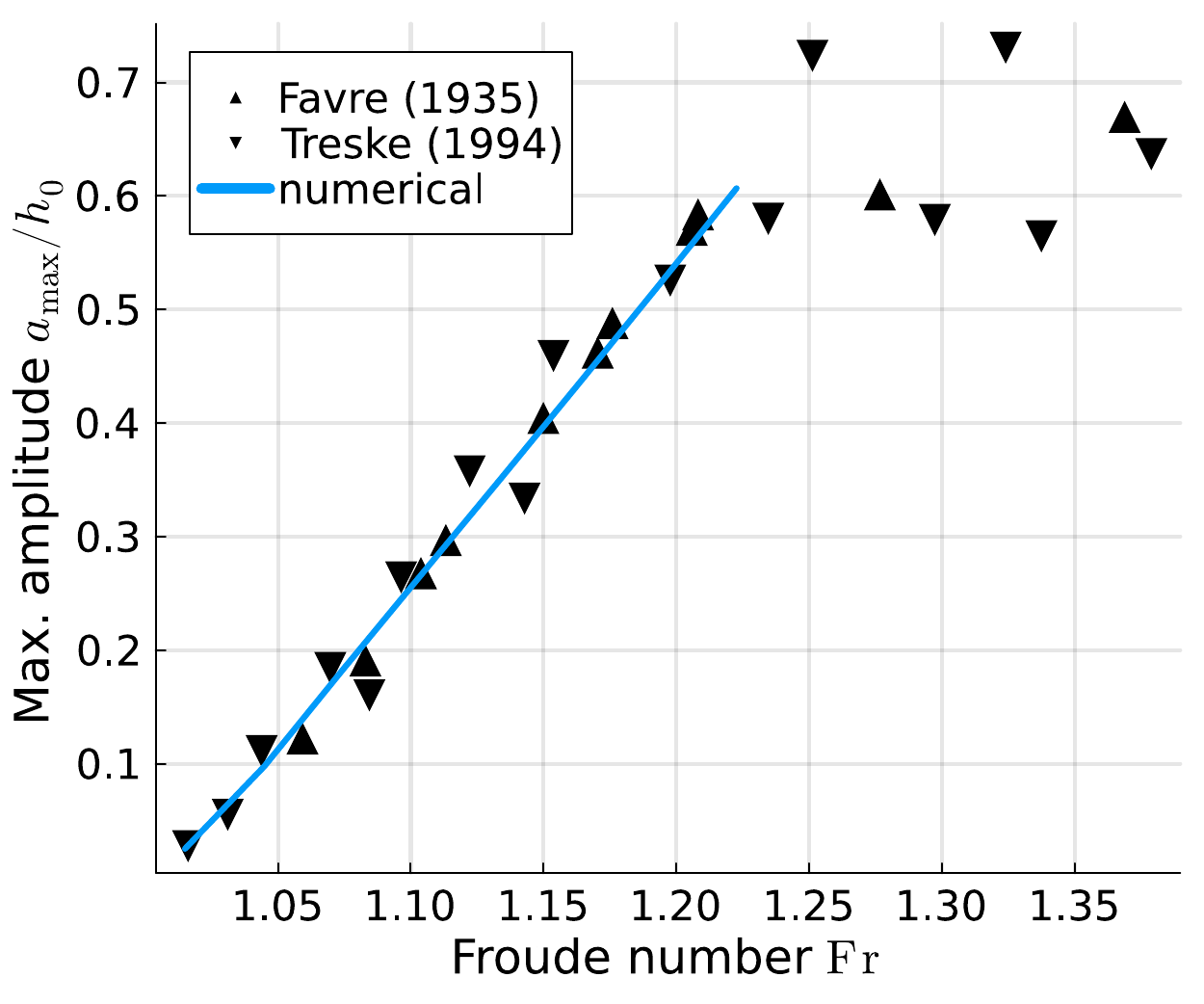}
        \caption{Serre-Green-Naghdi equations.}
    \end{subfigure}%
    \caption{Favre waves: maximal amplitude $a_\mathrm{max}$ against
             Froude number $\mathrm{Fr} = \sigma / \sqrt{g h_0}$. Comparison with experimental data of
             Favre \cite{favre1935} and Treske \cite{treske1994}.
             Numerical results with  fourth-order structure-preserving finite differences
              with central operators for
             the hyperbolic approximation and upwind operators for the
             original Serre-Green-Naghdi equations.}
    \label{fig:favre_amplitude_over_froude}
\end{figure}

\subsection{Dingemans experiment}
\label{sec:dingemans_experiment}

In this section, we compare numerical results obtained with our new
energy-conserving methods with experimental data from
\cite{dingemans1994comparison,dingemans1997water}. This setup is
similar to classical test cases such as \cite{beji1993experimental}
that have been used to validate numerical models for water waves, e.g.,
\cite{madsen1996boussinesq,kazolea2024}.

\begin{figure}[htbp]
\centering
    \includegraphics[width=0.6\textwidth]{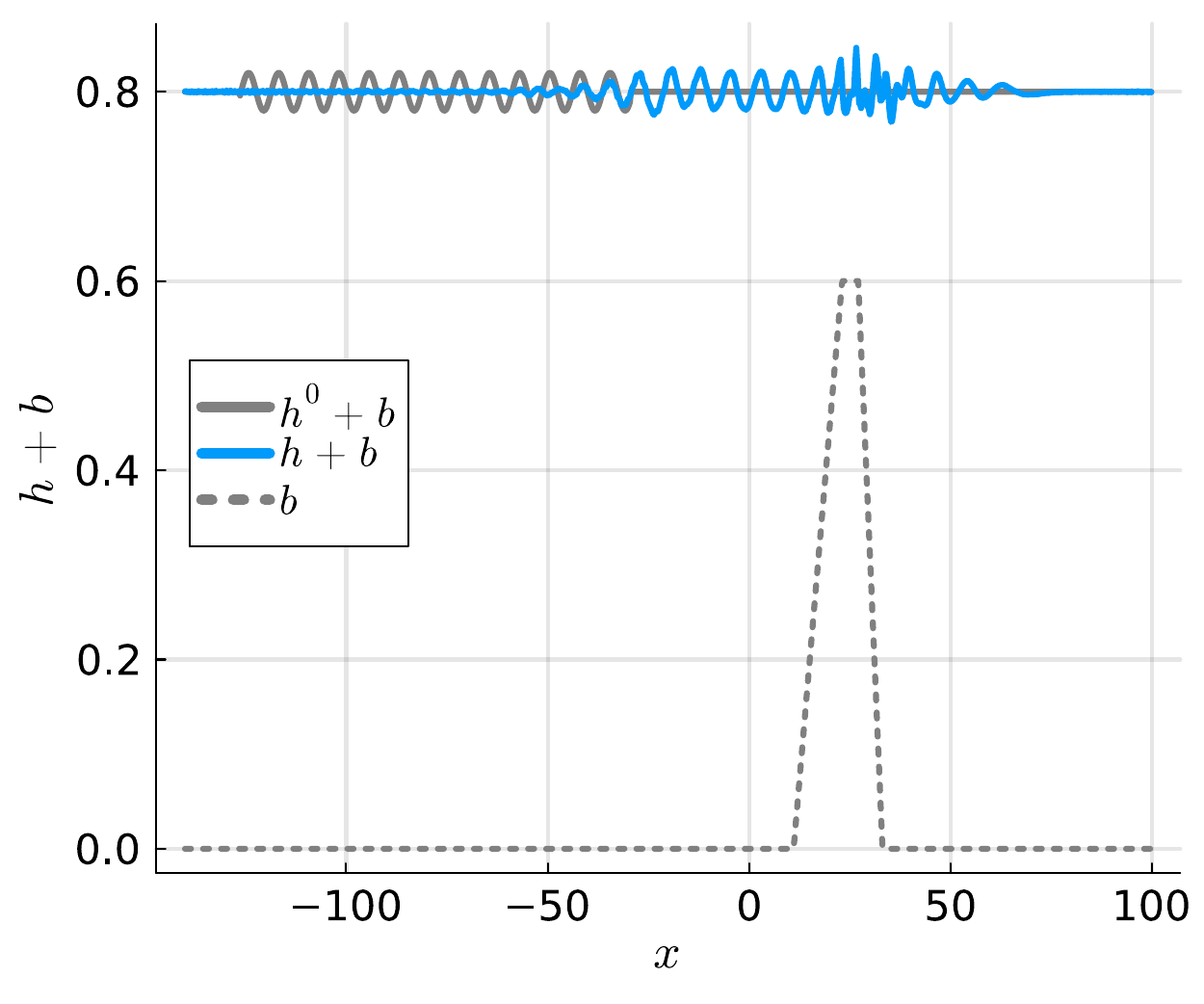}
    \caption{Initial setup and numerical solution at time
             $t = 40$ obtained with
             fourth-order accurate upwind finite difference methods
             with $\Delta x = 0.24$
             applied to the mild-slope approximation of the
             Serre-Green-Naghdi equations for the Dingemans experiment.}
    \label{fig:dingemans_solution}
\end{figure}

The initial setup as well as a numerical solution of the
mild-slope approximation are shown in Figure~\ref{fig:dingemans_solution}.
The original experiment of Dingemans
\cite{dingemans1994comparison,dingemans1997water} used a wave maker at
$x = 0$ to produce water waves with an initial amplitude of $A = 0.02$
moving to the right. For the numerical simulations, choose the spatial
domain $[-140, 100]$ and initialize the numerical solution with
a sinusoidal perturbation of the still water height $h = 0.8$
with amplitude $A = 0.02$. The phase of the perturbations and the
corresponding velocity perturbation are chosen based on the dispersion
relation of the Euler equations as in
\cite{svard2023novel,lampert2024structure}. The offset of the perturbation
is chosen manually such that the phase at the first wave gauge matches the
experimental data reasonably well.

The bottom is flat except a trapezoidal bar starting at $x = 11.01$.
Between $x = 11.01$ and $x = 23.04$, the bottom increases linearly
from $b = 0$ to $b = 0.6$. The bottom has a small plateau between
$x = 23.04$ and $x = 27.04$ with $b = 0.6$ and decreases linearly
from $b = 0.6$ to $b = 0$ between $x = 27.04$ and $x = 33.07$.

\begin{figure}[htbp]
\centering
    \includegraphics[width=0.7\textwidth]{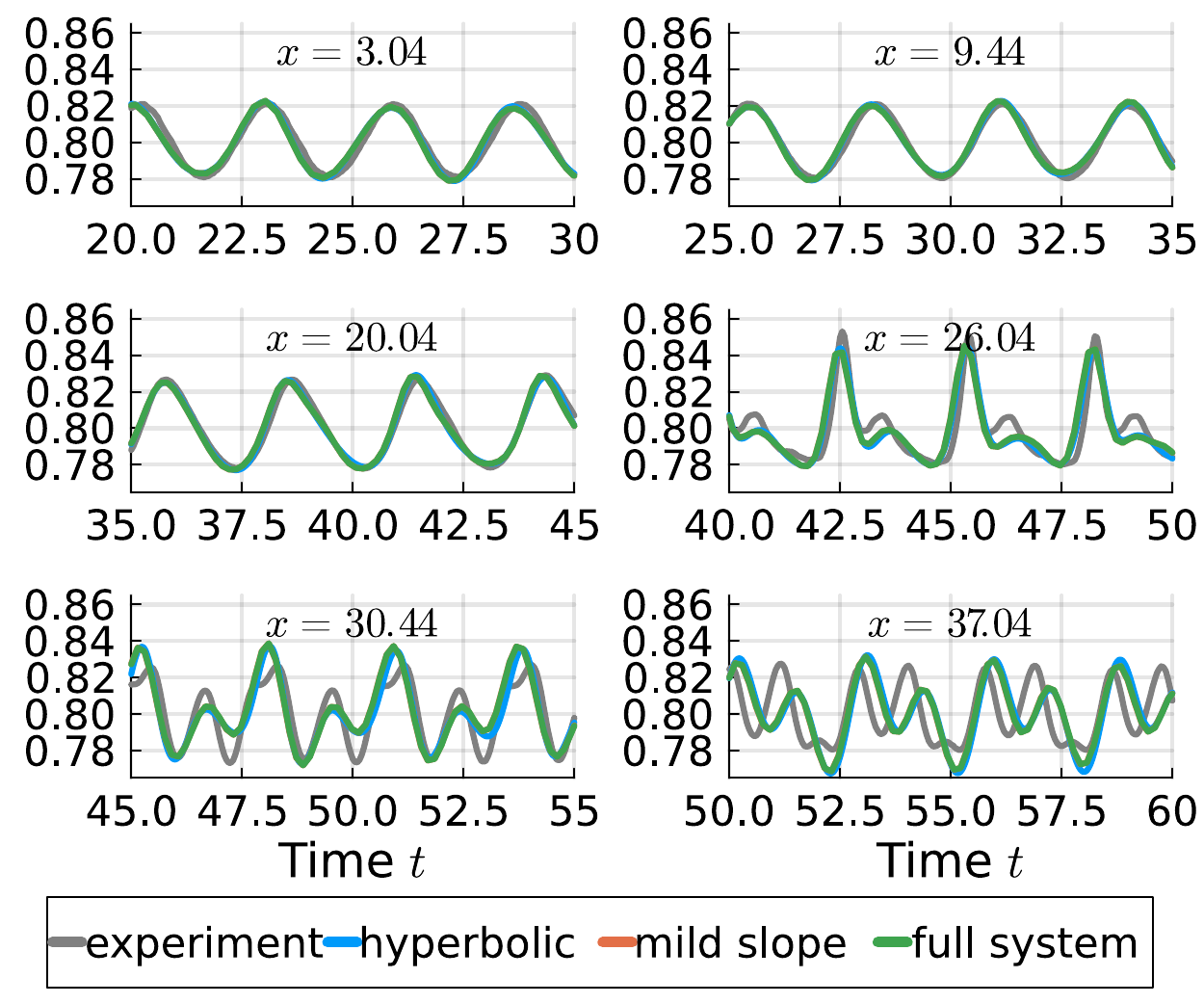}
    \caption{Experimental data of
             \cite{dingemans1994comparison,dingemans1997water}
             and total water height $h + b$ of the numerical solutions at the
             wave gauges over time obtained with fourth-order accurate
             finite difference methods with $\Delta x = 0.24$. The hyperbolic
             approximation uses $\lambda = 500$.}
    \label{fig:dingemans_solutions_at_gauges}
\end{figure}

The values of the numerical solutions are compared to the experimental data
at six wave gauges in Figure~\ref{fig:dingemans_solutions_at_gauges}.
First, we observe that the numerical solutions agree very well with
each other --- the results obtained using the hyperbolic approximation and
the original SGN equations with/without mild-slope
approximation are nearly indistinguishable.
Moreover, the numerical results agree very well with the experimental data
at the first three wave gauges. The agreement is less good but still
qualitatively correct at the remaining wave gauges above and to the right
of the plateau of the trapezoidal bar. This is within the limitations
of the model used in terms of dispersion relation \cite{rf14,filippini2015nonlinear}. However,  the amplitudes
of the numerical solutions are still in  agreement with the experiments.

\subsection{Preliminary comparison of runtime efficiency}
\label{sec:computational_costs}

The relative computational costs of discretizations of the
hyperbolic approximation and the original SGN equations
depend strongly on the parameter $\lambda$. For most numerical
results presented above, we have chosen $\lambda = 500$ as a compromise
between accuracy and computational costs since the
numerical solutions obtained with the hyperbolic approximation and this
value of the parameter $\lambda$ are visually (nearly) indistinguishable
from the results obtained with the original SGN equations.
The only exception is the Riemann problem in Section~\ref{sec:riemann_problem},
where a value of $\lambda = 1000$ is required to obtain visually
indistinguishable results.

\begin{table}[htbp]
\begin{center}
    \caption{Benchmarks of the total runtime of the numerical solutions
             obtained with the hyperbolic approximation and the original
             Serre-Green-Naghdi equations using finite differences with
             $N$ nodes.}
    \label{tab:runtime}
    \begin{subfigure}{\textwidth}
        \centering
        \caption{Setup of the conservation tests with variable bathymetry
                 and second-order central operators
                 (Sections~\ref{sec:upwind_vs_central} and
                 \ref{sec:conservation_tests}).}
        \begin{tabular}{c cc cc}
            \toprule
                & hyperbolic
                & hyperbolic
                & original
                & original \\
            $N$ & ($\lambda = 500$)
                & ($\lambda = 1000$)
                & (mild slope)
                & (full system) \\
            \midrule
            1000 & \SI{73.0 \pm 21.0}{ms}
                 & \SI{94.0 \pm 3.3}{ms}
                 & \SI{178.3 \pm 2.3}{ms}
                 & \SI{189.7 \pm 4.2}{ms} \\
            2000 & \SI{234.0 \pm 38.0}{ms}
                 & \SI{301.1 \pm 7.2}{ms}
                 & \SI{359.1 \pm 8.5}{ms}
                 & \SI{378.0 \pm 11.0}{ms} \\
            3000 & \SI{441.4 \pm 4.9}{ms}
                 & \SI{611.0 \pm 18.0}{ms}
                 & \SI{536.2 \pm 7.4}{ms}
                 & \SI{581.0 \pm 29.0}{ms} \\
            4000 & \SI{776.8 \pm 5.1}{ms}
                 & \SI{1.089 \pm 0.048}{s}
                 & \SI{724.4 \pm 3.0}{ms}
                 & \SI{781.0 \pm 32.0}{ms} \\
            5000 & \SI{1.345 \pm 0.031}{s}
                 & \SI{1.869 \pm 0.051}{s}
                 & \SI{967.0 \pm 27.0}{ms}
                 & \SI{941.0 \pm 20.0}{ms} \\
            \bottomrule
        \end{tabular}
    \end{subfigure}%
    \\
    \medskip
    \begin{subfigure}{\textwidth}
        \centering
        \caption{Setup of the Favre waves with $\epsilon = 0.2$ and
                 fourth-order central/upwind operators for the
                 hyperbolic/original equations
                 (Section~\ref{sec:favre_waves}).}
        \begin{tabular}{c cc cc cc}
            \toprule
                & hyperbolic
                & hyperbolic
                & original
                & original
                & original \\
            $N$ & ($\lambda = 500$)
                & ($\lambda = 1000$)
                & (flat bottom)
                & (mild slope)
                & (full system) \\
            \midrule
            1000 & \SI{68.3 \pm 2.5}{ms}
                 & \SI{96.6 \pm 5.2}{ms}
                 & \SI{132.8 \pm 4.1}{ms}
                 & \SI{163.5 \pm 6.1}{ms}
                 & \SI{162.3 \pm 1.9}{ms} \\
            2000 & \SI{236.0 \pm 11.0}{ms}
                 & \SI{339.0 \pm 23.0}{ms}
                 & \SI{300.1 \pm 4.3}{ms}
                 & \SI{358.5 \pm 6.6}{ms}
                 & \SI{382.0 \pm 12.0}{ms} \\
            3000 & \SI{517.0 \pm 40.0}{ms}
                 & \SI{741.3 \pm 3.9}{ms}
                 & \SI{517.6 \pm 5.8}{ms}
                 & \SI{667.5 \pm 4.8}{ms}
                 & \SI{664.0 \pm 13.0}{ms} \\
            4000 & \SI{1.1796 \pm 0.0073}{s}
                 & \SI{1.895 \pm 0.014}{s}
                 & \SI{753.0 \pm 10.0}{ms}
                 & \SI{964.0 \pm 12.0}{ms}
                 & \SI{971.4 \pm 8.3}{ms} \\
            5000 & \SI{2.0021 \pm 0.006}{s}
                 & \SI{3.128 \pm 0.025}{s}
                 & \SI{1.109 \pm 0.003}{s}
                 & \SI{1.338 \pm 0.009}{s}
                 & \SI{1.368 \pm 0.012}{s} \\
            \bottomrule
        \end{tabular}
    \end{subfigure}
\end{center}
\end{table}

To give a first impression of the computational costs of the methods,
we benchmark the total runtime (wallclock time) required to compute
the numerical solutions from Sections~\ref{sec:upwind_vs_central}
and \ref{sec:favre_waves}
on a single core of a MacBook (M2 chip) using the Julia package
BenchmarkTools.jl \cite{chen2016robust}. The results are reported in
Table~\ref{tab:runtime}. While we do not aim to present a detailed
performance study including the effects of various constraints and
effects, the results show that there is no significant difference
between the two versions of the original Serre-Green-Naghdi equations
with variable bathymetry. Moreover, the computational costs of the
hyperbolic approximation appear to increase faster with the number of
grid nodes than the costs of the original system. In particular,
the hyperbolic approximation with $\lambda = 500$ is faster than the
original system (by a factor of roughly three) for $N = 1000$ nodes. For
$N \in \{2000, 3000\}$, the runtimes of the hyperbolic approximation
with $\lambda = 500$ are still smaller than the runtimes of the original
systems. This changes around $N = 4000$ nodes; for $N = 5000$ nodes,
the original systems are faster than the hyperbolic approximation.

However, these performance benchmarks are done with the research code
we have implemented for this article. This code is not optimized for
performance. While we expect that the efficiency of the hyperbolic
version should be reasonably good, the elliptic solves required for the
original system are likely to be suboptimal. In particular, most of the
total runtime is spent assembling (multiplying sparse/diagonal matrices)
and solving (Cholesky factorization of SuiteSparse) the elliptic problems.

\section{Summary and conclusions}
\label{sec:summary}

We have developed structure-preserving numerical methods for the
Serre-Green-Naghdi equations in their original formulation and
the first-order hyperbolic approximation of \cite{favrie2017rapid,busto21}.
Starting with the hyperbolic approximation for flat bathymetry in
Section~\ref{sec:SGN_hyperbolic_flat}, we have derived the methods
for models with increasing complexity, including variable bathymetry
for the hyperbolic approximation (Section~\ref{sec:SGN_hyperbolic_variable})
and the original Serre-Green-Naghdi equations
(Sections~\ref{sec:SGN_original_mild} and \ref{sec:SGN_original_full}).
All methods conserve the total water mass, the total energy, and are
well-balanced with respect to the lake-at-rest steady state. Moreover,
the numerical methods discretizing the original Serre-Green-Naghdi equations
conserve the total momentum for flat bathymetry.\\

We have demonstrated the suitability of the novel structure-preserving
numerical methods in a range of numerical experiments, including academic
test cases such as convergence tests. We have also demonstrated the
importance of energy-conserving methods for long-time simulations of
solitary waves in Section~\ref{sec:error_growth}, where energy conservation
reduces the error growth in time from quadratic to linear.
Even without exact preservation in time, we have also shown
the impact of energy conservation in providing correct predictions
of wave heights in long-time propagation.
Moreover, we have shown that our numerical methods reproduce experimental
data, e.g., for Favre waves (Section~\ref{sec:favre_waves}) and
the flow over a trapezoidal bar (Section~\ref{sec:dingemans_experiment}).\\

Preliminary performance benchmarks show that the hyperbolic approximation
can be very efficient on coarse meshes. On finer meshes, the original
formulations of the Serre-Green-Naghdi equations are more efficient
(cf.\ Section~\ref{sec:computational_costs}).\\

Similar to Jouy et al.\ \cite{jouy_etal24}, we have found that
structure-preserving numerical methods for the Serre-Green-Naghdi
equations can efficiently capture the qualitatively correct behavior
of water waves in a range of scenarios. Compared to methods including
artificial dissipation/viscosity, the energy-preserving methods
show the correct long-time behavior of the amplitude and shape
of waves even on coarse meshes with low-order discretizations.
This poses the question of the
appropriateness of using entropy/energy dissipation as a stability criterion.
Certainly, structure-preserving methods are a promising approach for
the numerical simulation of water waves on coarse meshes (required
in practice) and long-time simulations, typical of the operational context.\\

Our investigations can be extended in several directions. For example,
we could check whether there are more energy-conservative split forms
starting from the two-parameter family of energy-conservative split forms
of the classical shallow water equations of \cite{ranocha2017shallow}.
Moreover, we could investigate the influence of possible other split forms
of the non-hydrostatic pressure term of the Serre-Green-Naghdi equations, other upwind
versions of the non-hydrostatic pressure term,
and more general narrow-stencil second-derivative SBP operators
instead of upwind operators, e.g., in
Lemma~\ref{lem:SGN_original_mild_prim_SBP_upwind}.
Further extensions include other boundary conditions, e.g., a
reflective/wall boundary condition \cite{noelle2022class}.

\appendix

\section*{Acknowledgments}

HR was supported by the Deutsche Forschungsgemeinschaft
(DFG, German Research Foundation, project numbers 513301895 and 528753982
as well as within the DFG priority program SPP~2410 with project number 526031774)
and the Daimler und Benz Stiftung (Daimler and Benz foundation,
project number 32-10/22).
MR is a member of the Cardamom team, Inria at University of Bordeaux.
We thank Oswald Knoth and Joshua Lampert for finding typos in the preprint.

\printbibliography

\end{document}